\documentclass[10pt,reqno,dvipsnames]{amsart}
\usepackage{mathtools}
\usepackage{graphicx}
\usepackage{caption}
\usepackage{subcaption}
\usepackage{afterpage}
\usepackage{tikz-cd}
\usepackage{todonotes}
\usepackage[width=5.7in,height=8.5in, centering]{geometry}
\usepackage{comment}
\usepackage[all,cmtip]{xy}
\definecolor{linkcolor}{rgb}{0.65,0,0}
\definecolor{citecolor}{rgb}{0,0.65,0}
\definecolor{urlcolor}{rgb}{0,0,0.65}
\usepackage[colorlinks=true, linkcolor=linkcolor, urlcolor=urlcolor, citecolor=citecolor]{hyperref}

\theoremstyle{plain} 
\newtheorem{theorem}{Theorem}

\newtheorem{proposition}[theorem]{Proposition}
\newtheorem{lemma}[theorem]{Lemma}
\newtheorem{corollary}[theorem]{Corollary}
\newtheorem{observation}[theorem]{Observation}

\theoremstyle{definition}
\newtheorem{definition}[theorem]{Definition}

\theoremstyle{remark}
\newtheorem{remark}[theorem]{Remark}

\numberwithin{theorem}{section}

  {\end{list}}

  {\end{list}}

\newcommand{\tightoverset}[2]{%
  \mathop{#2}\limits^{\vbox to -.5ex{\kern-1.05ex\hbox{$#1$}\vss}}}

\newcommand{\bigslant}[2]{{\raisebox{.2em}{$#1$}\left/\raisebox{-.2em}{$#2$}\right.}}

\newcommand{\C}{\mathbb{C}}
\newcommand{\R}{\mathbb{R}}

\newcommand{\QP}{\mathbb{QP}}
\newcommand{\CP}{\mathbb{CP}}
\newcommand{\RP}{\mathbb{RP}}

\newcommand{\PSL}{\mathrm{PSL}}
\newcommand{\SO}{\mathrm{SO}}
\newcommand{\GL}{\mathrm{GL}}

\newcommand{\SP}{\mathrm{SP}}
\newcommand{\Coefs}{\mathrm{Coefs}}
\newcommand{\Roots}{\mathrm{Roots}}
\newcommand{\RootMap}{\mathcal{R}}

\newcommand{\acosh}{\operatorname{acosh}}

\newcommand{\UT}{\mathrm{UT}}
\newcommand{\CCubic}{\Delta^-}

\renewcommand{\a}{\alpha}
\renewcommand{\b}{\beta}

\newcommand{\f}{\phi}

\renewcommand{\r}{\rho}

\newcommand{\z}{\zeta}

\newcommand{\CC}{\mathbb{C}}
\newcommand{\FF}{\mathbb{F}}

\newcommand{\HH}{\mathbb{H}}
\newcommand{\NN}{\mathbb{N}}
\newcommand{\PP}{\mathbb{P}}
\newcommand{\QQ}{\mathbb{Q}}
\newcommand{\RR}{\mathbb{R}}
\renewcommand{\SS}{\mathbb{S}}

\newcommand{\ZZ}{\mathbb{Z}}

\renewcommand{\GL}{\operatorname{GL}}

\newcommand\SL{\operatorname{SL}}

\definecolor{myorange}{rgb}{0.9, 0.55, 0.3}
\definecolor{mygreen}{rgb}{0.35, 0.71, 0.0}

\newcommand{\TITLE}{Algebraic Number Starscapes}
\newcommand{\TITLERUNNING}{Starscapes}
\newcommand{\DATE}{\today}

\title[\TITLERUNNING]{\vspace*{-1.3cm} \TITLE}

\author{Edmund Harriss, Katherine E. Stange, Steve Trettel}

\date{\DATE}
\address{%
Department of Mathematics, University of Colorado,
Campus Box 395, Boulder, Colorado 80309-0395}
\email{}

\thanks{
	This material is based upon work supported by the National Science Foundation under Grant No. DMS-1439786 and the Alfred P. Sloan Foundation award G-2019-11406 while the authors were in residence at the Institute for Computational and Experimental Research in Mathematics in Providence, RI, during the Illustrating Mathematics program.   Katherine E. Stange is supported by NSF CAREER CNS-1652238.
}

\begin{document}


\begin{abstract}
	We study the geometry of algebraic numbers in the complex plane, and their Diophantine approximation, aided by extensive computer visualization.  Motivated by the resulting images, which we have called \emph{algebraic starscapes}, we describe the geometry of the map from the coefficient space of polynomials to the root space, focussing on the quadratic and cubic cases.  The geometry describes and explains the notable features of the illustrations, and motivates a geometric-minded recasting of fundamental results in the Diophantine approximation of the complex plane.  Meanwhile, the images provide a case-study in the symbiosis of illustration and research, and an entry-point to geometry and number theory for a wider audience.  In particular, the paper is written to provide an accessible introduction to the study of homogeneous geometry and Diophantine approximation.
	
	We investigate the homogeneous geometry of root and coefficient spaces under the natural $\PSL(2;\CC)$ action.  Hyperbolic geometry and the discriminant play an important role in low degree. In particular, we rediscover the quadratic and cubic root formulas as isometries of $\HH^2$ and its unit tangent bundle, respectively.  Utilizing this geometry, we determine when the map sending certain families of polynomials to their complex roots (our starscape images) are embeddings.

	We reconsider the fundamental questions of the Diophantine approximation of complex numbers by algebraic numbers of bounded degree, from the geometric perspective developed.  In the quadratic case (approximation by quadratic irrationals), we consider approximation in terms of \emph{hyperbolic} distance between roots in the complex plane and the \emph{discriminant} as a measure of arithmetic height on a polynomial.  In particular, we determine the supremum on the exponent $k$ for which an algebraic target $\alpha$ has infinitely many approximations $\beta$ whose hyperbolic distance from $\alpha$ does not exceed $\acosh(1+1/|\Delta_\beta|^k)$.  It turns out to fall into two cases, depending on whether $\alpha$ lies on the image of a plane of rational slope in coefficient space (a \emph{rational geodesic}).  The result comes as an application of Schmidt's subspace theorem.  Our results recover the quadratic case of results of Bugeaud and Evertse, and give some geometric explanation for the dichotomy they discovered \cite{BugeaudEvertse}.  Our statements go a little further in distinguishing approximability in terms of whether the target or approximations lie on rational geodesics.

	The paper comes with accompanying software, and finishes with a wide variety of open problems.
\end{abstract}

\maketitle
\newpage

\tableofcontents

\section{Introduction}

We begin (and indeed this research began) with the images in Figure \ref{fig:Initial}.  Take a minute to look at them before continuing. 

\begin{figure}[h!tbp]
\centering
\begin{subfigure}[b]{0.42\textwidth} 
    \centering
	\includegraphics[width=\textwidth]{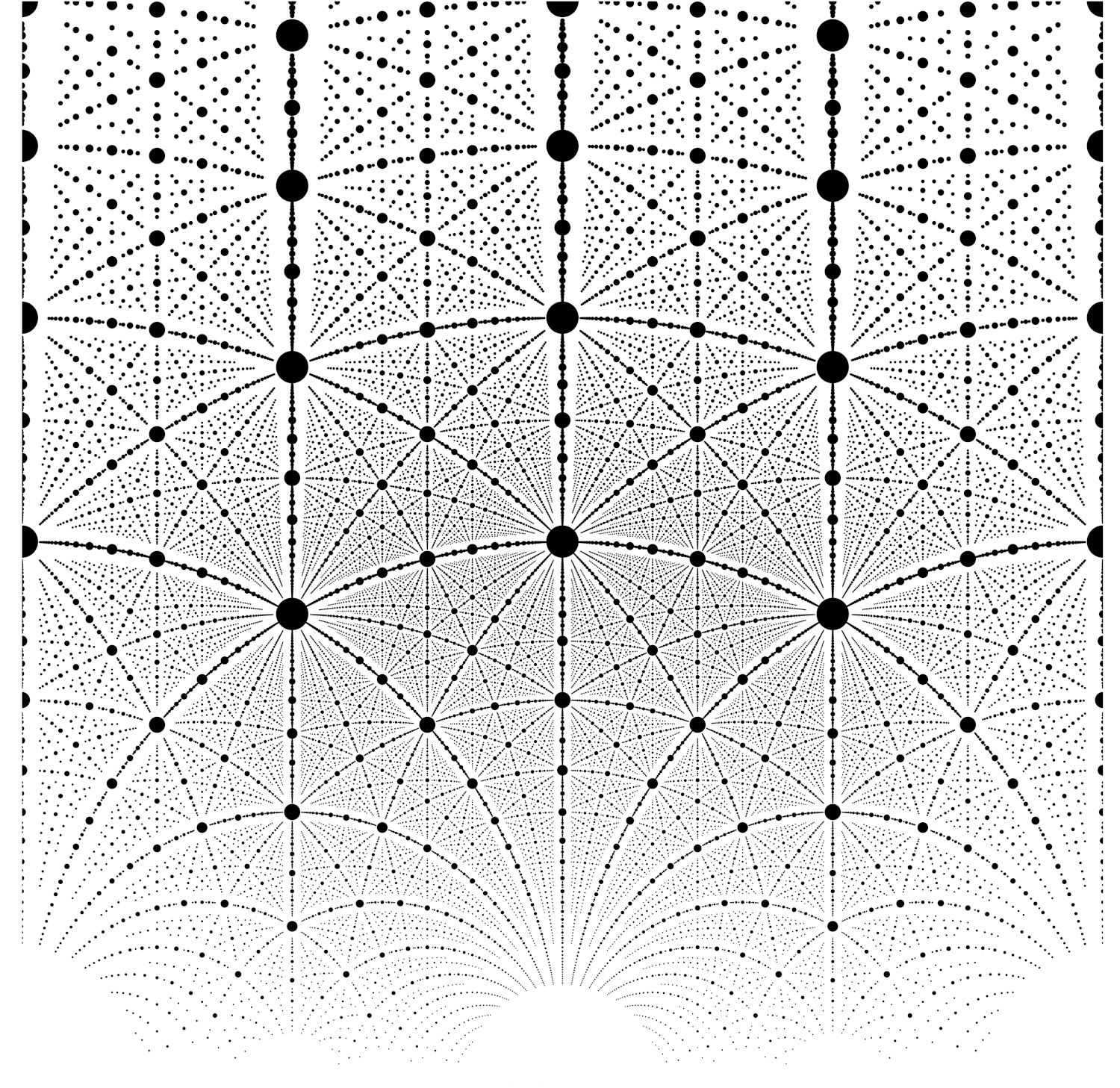}
	\caption{}
    \label{fig:Initial_quadratics}
\end{subfigure}
\begin{subfigure}[b]{0.42\textwidth} 
    \centering
	\includegraphics[width=\textwidth]{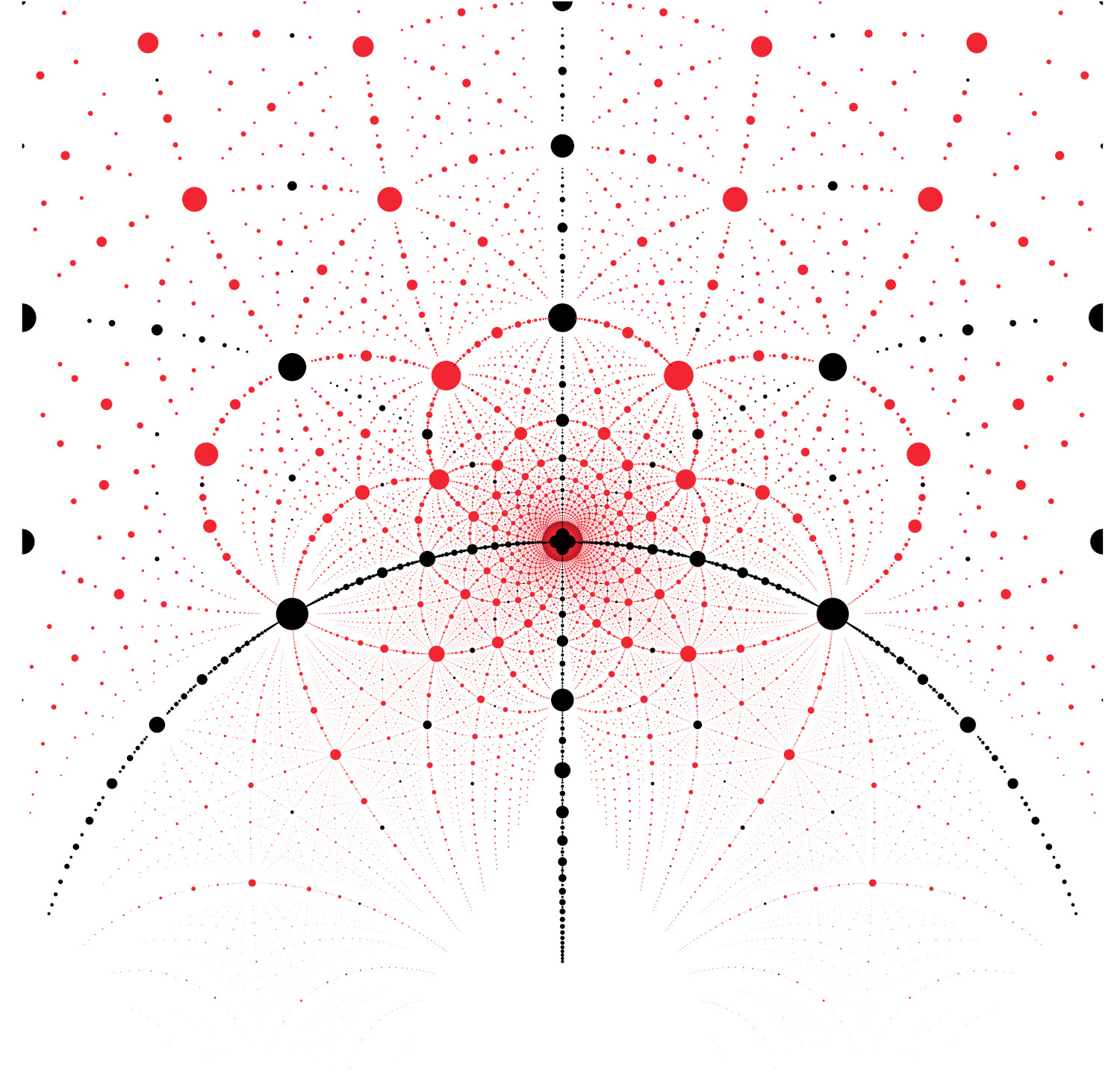}
	\caption{}
    \label{fig:Initial_cubic_family}
\end{subfigure}
\begin{subfigure}[b]{0.42\textwidth} 
    \centering
	\includegraphics[width=\textwidth]{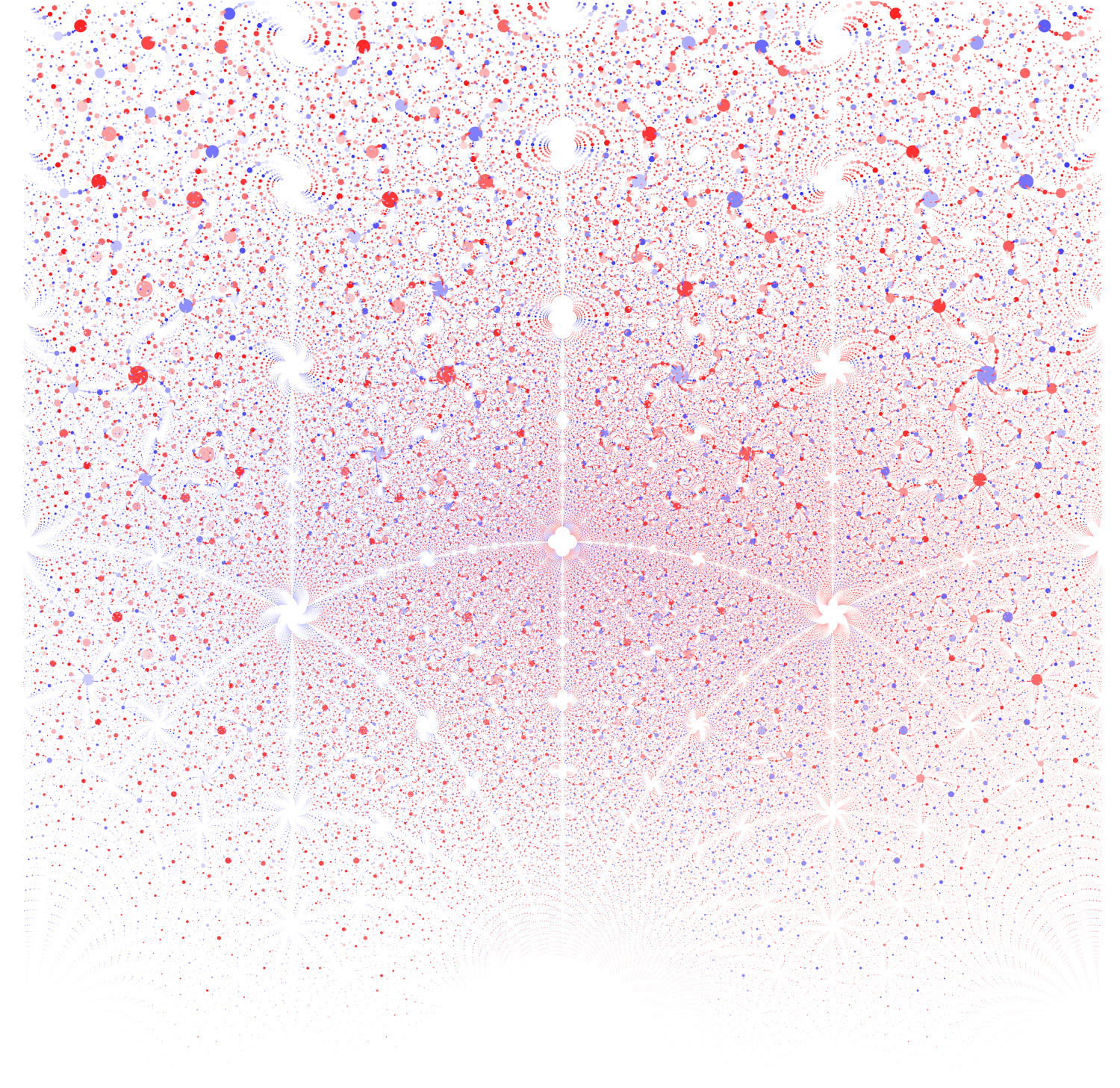}
	\caption{}
    \label{fig:Initial_cubics}
\end{subfigure}
\begin{subfigure}[b]{0.42\textwidth} 
    \centering
	\includegraphics[width=\textwidth]{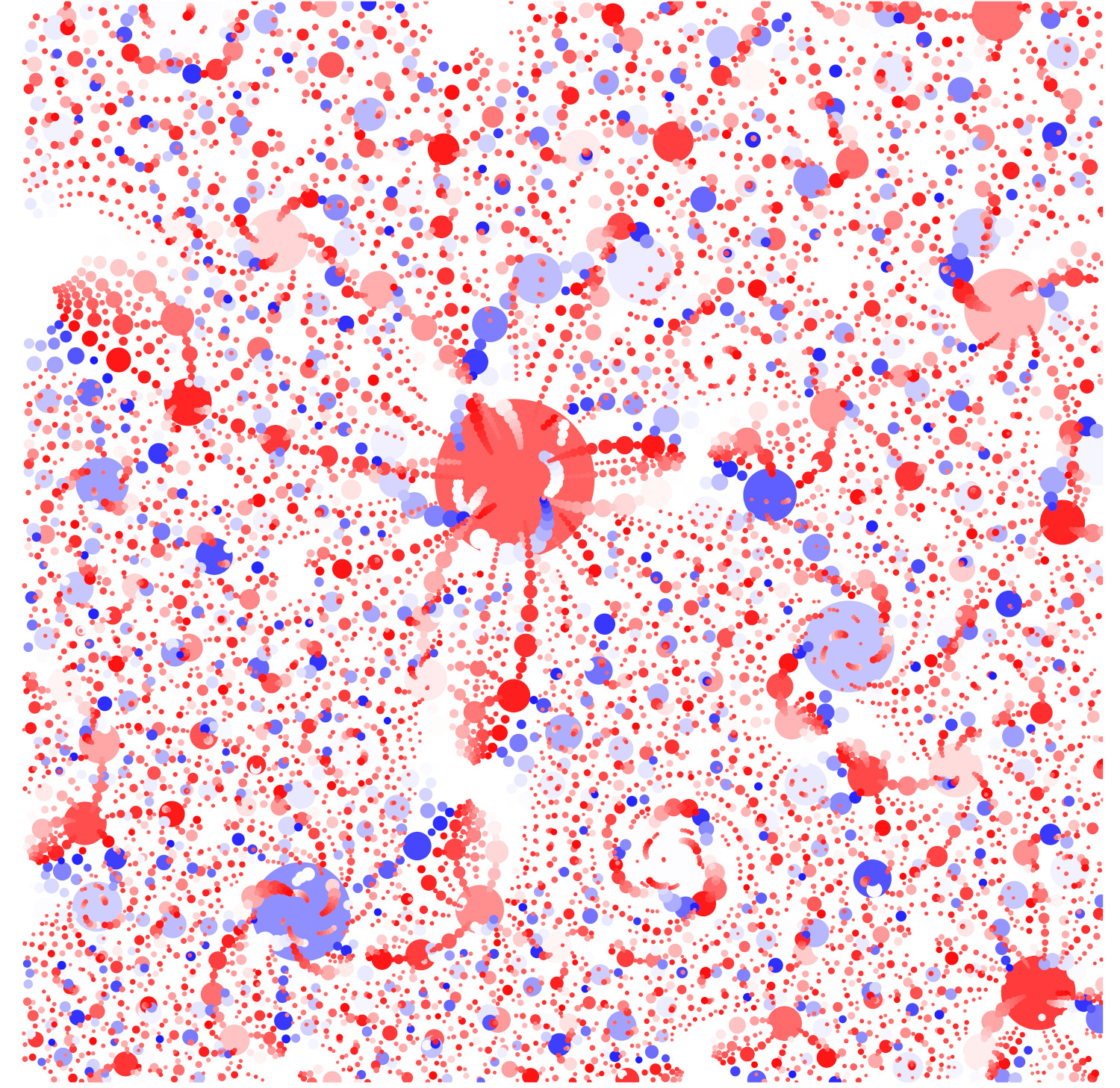}
	\caption{}
    \label{fig:Initial_Cubics_detail}
\end{subfigure}
\caption{Complex algebraic numbers sized by the inverse of the discriminant of their minimal polynomial in the hyperbolic metric. All quadratics (\ref{fig:Initial_quadratics}). All roots (quadratic in black, cubic in red) of the polynomials $a x^3 + c x^2 + b x + c = 0$ (\ref{fig:Initial_cubic_family}). All cubics, coloured based on the value of the real root (\ref{fig:Initial_cubics}) and the detail of the cubics around a root of $x^3 + x + 1 = 0$. The first three images, as many other figures here, are plotted from $-1$ to $1$ in the real axis and $0$ to $2$ in the imaginary axis.}
\label{fig:Initial}
\end{figure}

On the top left (\ref{fig:Initial_quadratics}) you see complex quadratic algebraic numbers plotted and sized by discriminant in the hyperbolic metric\footnote{For the polynomial $a x^2+bx+c = 0$ the dot plotted will be at $\frac{-b+\sqrt{4ac-b^2} i}{2a}$ on the complex plane. Note this only considers polynomials with negative discriminant and thus complex roots. A similar image can be created for polynomials with only real roots, see Figure~\ref{fig:RealQuadratics}. The radius of the dot is proportional to $\frac{1}{\sqrt{4ac-b^2}}$ (one over the root discriminant) times the height above the real axis $\frac{\sqrt{4ac-b^2}}{2a}$ to adjust the radius to the hyperbolic metric. This gives dots with radius proportional to $\frac{1}{2a}$. If you are plotting the points yourself it can be useful to adjust the scale of the dots (keeping the same proportions) as more are added.}. On the top right (\ref{fig:Initial_cubic_family}) are the complex roots of polynomials of the form $a x^3 + c x^2 + b x + c = 0$, and on the bottom all complex cubics coloured by their real conjugate (\ref{fig:Initial_cubics}), with zoomed in detail around the complex root of $x^3+x+1$ (\ref{fig:Initial_Cubics_detail}). In each case the dot size is inversely proportional to the root discriminant (the discriminant to the root of the degree), and the points are plotted in the hyperbolic metric (meaning, the radius is hyperbolic in the upper half-plane model). This paper grew out of our excitement at the beauty and detail present in these images and the search for mathematics that could be both seen and conjectured from their structure.

\subsection{A note on our expository approach}  We believe that the images presented here provide a motivated path into several topics in geometry and number theory. Mathematical beauty can be incredibly hard to communicate to people not familiar with the details of the subject, yet these images have already appeared in an art exhibit in Iceland~\cite{Harriss:AS} and been used for engagement in a workshop with the Math Club for Battle Creek Area Math and Science Center in Michigan.  We have therefore attempted to build a paper containing two distinct paths for two distinct audiences.  For those with less background, we provide a leisurely introduction to the subject keeping the description as accessible as possible, introducing even well-known terminology as we build toward more sophisticated mathematics.  We hope the paper might provide insight and interest to a motivated high school student and a mathematician in these research areas alike, and inspire REU projects.  For the researcher who wishes to access the results directly, we have endeavoured to make a bypass that avoids the more leisurely parts of the paper, and provide an alternative, condensed introductory section that allows the reader to pass directly to the results.  

\subsection{How to access this paper for a wider audience}

For a motivated student of mathematics, this paper should be read in chronological order, skipping Section~\ref{sec:researcher}.  In Section \ref{sec:gallery}, we follow the visual investigation with a gallery of some of the images produced, and discussion of some of the structural aspects visible in the images. We develop those observations in the next two sections studying the geometry of the roots (in Section~\ref{sec:Geometry}) and their number theory (in Sections~\ref{sec:DiophantineApproximation} and \ref{sec:Dio-quad}).  These (geometry and number theory) are somewhat parallel, in that a reader may wish to read both to their natural stopping point (dictated by the reader's mathematical background).  Each one starts from a lower level of background and ramps up.

\subsection{How researchers should read this paper}

For a working mathematician with experience in the relevant background, Section~\ref{sec:researcher} should be read first, followed by Section~\ref{sec:Geometry}, and then Section~\ref{sec:Dio-quad}, dipping into Section~\ref{sec:DiophantineApproximation} for background as needed.

\subsection{Previous illustration}

The notion of plotting algebraic numbers is not new, and there are many beautiful mathematical visualisations made from them, many in the world of blogs and online mathematical discussions.  We curate those we are aware follows at the end of this introduction.  Our images were studied initially from the perspective of their aesthetic, by asking, ``What makes images that look interesting?''\ (without trying too hard to define \emph{interesting}).  However, they rapidly became both a tool to illuminate existing mathematics, including hyperbolic and projective geometry, representations of $\PSL(2;\CC)$, and Diophantine approximation, and also a source of new mathematics.  

The richest previous investigation of imagery grows out of the study of the roots of Littlewood polynomials~\cite{Littlewood:SPRCA, Odlyzko:ZPWC, Shmerkin:ZPSCLSS}, which are polynomials with coefficients $\pm 1$, and more generally polynomials with coefficients from a finite set. This produces the oldest images we have found, particularly in the work of Peter and Jonathan Borwein~\cite{Bailey:EMA, Borwein:CEANT, Borwein:PWCFF, Borwein:VSNT, Pickover:M}. These collections of polynomials have the nice feature that all polynomials up to a certain degree can be considered. Other investigations of the geometry and images related to such polynomials (that we will not describe in detail) include the curiously named Thurston's master teapot~\cite{bray2019shape} and the eigenvalues of Bohemian Matrices~\cite{Bohemian, CHAN202072}.

In particular, a lot of interest was generated by the work of Dan Christiansen~\cite{Christiansen}, shared and described by John Baez, including an incredibly intricate image created by Sam Derbyshire~\cite{Baez1, Baez2, SciAm}.

Many other people, some directly inspired by this, have also created, or discussed, the images, including Paul Nylander~\cite{Nylander}, Greg Egan~\cite{Egan}, Andrej Bauer~\cite{Bauer}, Vincent Pantaloni~\cite{Pantaloni}, Daniel Wiegreffe~\cite{Wiegreffe}, 
Bernat Espigul\'{e}~\cite{Espigule}, Jwalin Bhatt~\cite{Bhatt}, Jonathan Lidbeck~\cite{Lidbeck}, and Jordan Ellenberg~\cite{Ellenberg}.

Pictures not limiting coefficients so strongly, closer to those we present here, are rarer. The most notable are the images of Stephen Brooks used on Wikipedia~\cite{Brooks}, that were further developed and optimised by David Moore~\cite{Moore}. These images also inspired a  Wolfram Demonstration from Enrique Zeleny~\cite{Zeleny}. Other interesting images we have found come from David Marciel~\cite{Marciel} and Deviant Art user Fauxtographique~\cite{Fauxtographique}. There have also been discussions on Stackexchange~\cite{Groff,iadvd}, with the latter linking to further amazing images on the Flickr account of Stackexchange user ``DumpsterDoofus''~\cite{MS}.  The ease of creating these images means that people quickly start making their own:  for example, a twitter thread started by the first-named author \cite{Harriss} quickly prompted variations by Dan Anderson, Michael Pershan and Peter Farrell \cite{Anderson, Anderson2, Pershan}.  

This list is extensive, but probably not exhaustive. We are interested in other versions of such patterns, especially earlier ones (before 2010, and even more before 2000) so please send us any you know of. 

\subsection{Software and image generation}
We encourage the reader to explore along with us, throughout the paper, using the accompanying software, available as a Sage Mathematics Software \cite{sagemath} notebook, at \url{algebraicstarscapes.com}.

In general the images here are produced by a rather simple three step process. We first generate a list of polynomials whose coefficients lie in a region about the origin (for example, a box or ball).  These polynomials are then solved to give the collection of roots, with additional data (such as polynomial discriminant) attached to the roots, data which is eventually to be used for sizing.  Finally, that list of data is converted into a collection of points and plotted.  Most of the images drawn involve over $50{,}000$ dots, but some get to over $250{,}000$.

\subsection{Acknowledgements}  
The authors are grateful to the Institute for Computational and Experimental Research in Mathematics in Providence, RI, and to the semester organizers, for the opportunity to be in residence for Fall 2019 at the Illustrating Mathematics program, where this work was initiated in the grand tradition of just being in the right place at the right time.  The authors would also like to thank their respective home institutions for their help in making semester residency possible.  Thanks are also due to the many participants in that program for helpful discussions, including Arthur Baragar and Joseph H. Silverman.  Special thanks go to Pierre Arnoux and David Dumas for especially inspiring and detailed conversations as these ideas developed.

\section{A technical introduction}
\label{sec:researcher}

\subsection{Algebraic starscapes}

The images central to our story we have called \emph{algebraic starscapes}.  Formally, these images consist of dots centred at all algebraic roots of a family of polynomials, with radius a function of the coefficients of the relevant polynomial.  More specifically, these families are chosen by fixing a bound on the polynomial degree and allowing the coefficients to range through all integer points in some affine subspace of the full vector space of coefficients.  The sizing is typically chosen from various measures of arithmetic complexity, such as Weil height or polynomial discriminant, so that big dots correspond to low complexity.  A \emph{linear starscape} is formed when the family of polynomials is two-dimensional; these appear as beaded necklaces (see Figure \ref{fig:Lines}); \emph{planar starscapes} are formed from three-dimensional families (see Figure \ref{fig:Starscapes}).  Planar starscapes contain infinitely many linear starscapes.  The reader is invited to examine the examples in Figures \ref{fig:Starscapes_a} through \ref{fig:Starscapes_f}.  These pictures can all be considered generalisations of Figure \ref{fig:RationalStarscape}.  The `repulsion' of large dots from one another illustrates an analog to Dirichlet's approximation theorem (Theorem \ref{thm:Dirichlet_intro}), stating that rational numbers cannot well approximate other rational numbers (see also Figure \ref{fig:Dirichlet}).

As an aside, although we restrict our attention to complex roots, there's no reason one cannot seek analogous visualizations for real roots.  For example, Figure \ref{fig:RealQuadratics} shows the real pairs that are roots of quadratics.  Or, for cubics with one real root, one could parametrize polynomials by a complex root in the upper half plane together with a real root on its ideal boundary:  this gives us a starscape picture living in a solid torus as in Figure \ref{fig:cubicsin3d}.  See Section \ref{sec:futurework} for further avenues.

The purpose of this paper is to study algebraic numbers in the complex plane, including their Diophantine approximation properties, as a function of the homogeneous geometry of the coefficients-to-roots map.  That is, we aim to describe exactly the map from coefficient space (the space of coefficient vectors of polynomials of fixed degree), and its affine subspaces, to the root space (collections of points in the complex plane).  Starscapes represent the images of affine subspaces (strictly speaking, the starscape is formed by plotting all the roots occurring in the root space image).  The philosophy is that, with a sufficient understanding of the geometry, we can state and prove concrete Diophantine results in low degree.  

\subsection{Geometry}

The arithmetic complexity of an algebraic number, although measured in a variety of ways (see Figures \ref{fig:DotSizesCoeff} and \ref{fig:DotSizesRoots}), is typically correlated to the size of the integer coefficients of its minimal polynomial.  The big dots in the image, then, are the image of the short integer vectors in the affine subspace.  If the geometric map from coefficient space to root space is sufficiently convoluted, these dots may end up close together in the complex plane.  If we can control the behaviour of the geometric map, we can control this effect and prove Diophantine results.

There are several key geometric points:

\begin{enumerate}
	\item The geometric description of the coefficients-to-roots map is studied in detail in degrees two and three.  The spaces of coefficients and roots decompose as a union of homogeneous spaces for $\PSL(2;\RR)$, and in these low degrees we may describe the cases of interest (real polynomials with complex roots) completely in terms of the geometry of the hyperbolic plane and its unit tangent bundle. See Sections \ref{ssec:QuadraticGeometry} and \ref{ssec:CubicGeometry}.
	
	\item In degree two, we demonstrate that the coefficients-to-roots map (more colloquially known as the quadratic formula) realizes an \emph{isometry between two models of the hyperbolic plane}. See Theorem \ref{thm:hypIso}, Corrolary \ref{cor:QuadFormula} and the associated Figure \ref{fig:hypIso} for a precise formulation.
While surely classically known, none of the authors had previously encountered this surprising aspect of the quadratic formula, and so we have provided details accessible to students of hyperbolic geometry. 

\item In degree three, the $\PSL(2,\RR)$ action naturally identifies the space of real cubics possessing a complex root with the unit tangent bundle to the hyperbolic plane.
This identification allows a recasting of many familiar algebraic results in geometric terms.
As a particular example we see the fact that every real cubic has a real root provides a preferred trivialization $\UT\HH^2$ (Proposition \ref{prop:UTH2}), and the cubic formula can be interpreted as a means of explicitly identifying cubics in coefficient-space with their coordinates with respect to this trivialization (Theorem \ref{thm:CubicFactor}).

	\item Concerning the roots map in degree three, the projection onto complex roots naturally identifies with the bundle map $\UT\HH^2\to\HH^2$. We study the interaction of the lattice of integer points with this map:  for example, can rational affine subspaces be contained in the fibres?  We discuss this, and its implications for starscapes, in Section \ref{sec:applications-cubic}.
\end{enumerate}


\subsection{Diophantine Approximation}

Diophantine approximation can be described as the quantitative study of the trade-off that is required to approximate a real number from a set of approximations (for example, rational numbers), namely between the precision of the approximation and the complexity of the approximant.  Most of the classical story lives on the real line, so that, given $\alpha \in \RR$, we ask for $p/q \in \QQ$ (if approximating with rationals), so that
\[
	\left| \alpha - \frac{p}{q} \right| < \frac{1}{q^k}
\]
for various positive $k$.  Dirichlet's Theorem (Theorem \ref{thm:Dirichlet_intro}) asserts that for $\alpha \notin \QQ$, and $k=2$, there are infinitely many such approximations, while for algebraic $\alpha$, and $k > 2$, Roth's Theorem (Theorem \ref{thm:roth}) asserts that there are only finitely many.  Thus, the exponent $k=2$ is a critical exponent for approximation of algebraic numbers by rationals.  We might choose approximations from other sets, such as algebraic numbers of bounded degree.  Koksma defines $k_d(\alpha)$ to be the supremum of all $k$ such that there are infinitely many algebraic $\beta$ of degree $\le d$ satisfying
\[
	\left| \alpha - \beta \right| < \frac{1}{H(f_\beta)^k}.
\]
Here, $H$ refers to the na\"ive height, and $f_\beta$ to the minimal polynomial of $\beta$.

These questions have naturally been extended to the rest of the complex plane, typically by reference to the euclidean distance in the complex plane.  Thus one may define $k_d(\alpha)$ for $\alpha \in \CC$ also.  Sprind\u{z}uk showed that the real and complex cases are essentially different:  $k_d(\alpha) = d+1$ for almost all $\alpha \in \RR$, but $k_d(\alpha) = (d+1)/2$ for almost all $\alpha \in \CC$ (see Theorem \ref{thm:sprindzuk}).  
Bugeaud and Evertse where able to determine $k_d(\alpha)$ for most algebraic $\alpha \in \CC \backslash \mathbb{R}$ (Theorem \ref{thm:bugeaud-evertse}).  In particular, in the quadratic case they discovered that $k_2(\alpha)=2$ or $3/2$, depending upon whether the quantities $1$, $\alpha \overline{\alpha}$ and $\alpha + \overline{\alpha}$ are linearly dependent or not, respectively.

However, motivated by the geometry of the coefficients-to-roots map which describes the visual features of the algebraic starscapes, we propose that it make sense to ask about Diophantine approximation in the hyperbolic metric and with a sensitivity to the action of $\PSL(2;\ZZ)$ implied by the homogeneous geometry.  Thus we consider instead the critical exponent $k$ when asking for approximations in the following sense:
\[
	d_{hyp}(\alpha,\beta) \le \operatorname{arcosh} \left( 1 + \frac{1}{|\Delta_\beta|^k} \right).
\]
Here, $\Delta_\beta$ is the discriminant of the minimal polynomial of $\beta$, and $d_{hyp}$ represents the hyperbolic distance.  This inequality is preserved under $\PSL(2;\ZZ)$ transformations, and we demonstrate, using Schmidt's Subspace Theorem, that the critical exponent is either $2$ or $3/2$, the difference controlled as in Bugeaud and Evertse's result.  The dependence of $1$, $\alpha \overline{\alpha}$ and $\alpha + \overline{\alpha}$ can be reinterpreted as the condition that $\alpha$ lie on a \emph{rational geodesic} (see Section~\ref{sec:applquad}), i.e. the image of a plane of rational slope in the coefficient space.  These are exactly the linear starscapes of Figure~\ref{fig:Initial_quadratics}, which are more densely packed with quadratic irrationals.

This essentially recovers (up to some nuances discussed in Section \ref{ssec:QuadraticDirichletonRationalGeodesics}) the result of Bugeaud and Evertse for degree $2$.  But it begs the question if higher degree cases (some of which are not completely settled) can be similarly described in terms of the geometry of the coefficients-to-roots map.

\subsection{Reading Guide}

In Section~\ref{sec:Geometry}, we describe in detail the map from coefficient space to root space.  A research audience may wish to skip Section \ref{sec:ProjGeo}, briefly looking at Section~\ref{sssec:RootsMap} for the roots map in the linear case, and begin with Section \ref{sec:QuadHypGeo}, which concerns quadratic polynomials. Section~\ref{sec:applquad} concerns the consequences of the quadratic geometric story for integer points, which is relevant to the Diophantine approximation we will do later.
Section~\ref{ssec:CubicGeometry} tells a similar story for cubics using the geometry of the unit tangent bundle to the hyperbolic plane, and interprets the cubic formula from this perspective \ref{thm:CubicFactor}).

In Section~\ref{sec:DiophantineApproximation}, we give the necessary background to Diophantine approximation.  The expert will not find anything new in Section \ref{ssec:ClassicDiophantineApproximation}--\ref{ssec:NaiveHeightComplexity}, but we place the Diophantine results in context in Section \ref{ssec:fixdeg}, and argue for the importance of $\PSL(2;\ZZ)$ in Section \ref{ssec:DophantineSl2R}.  Sections \ref{ssec:WeilHeight} and \ref{ssec:WeilHeightRepulsion} can again be skipped by the expert.  Section \ref{ssec:DiscriminantComplexity}--\ref{ssec:SizingStarscapes} discuss the appropriate choice of measure of arithmetic complexity.

The number theorist interested in the quadratic results will therefore wish to read Sections \ref{sec:QuadHypGeo}--\ref{sec:applquad}, Sections \ref{ssec:fixdeg}--\ref{ssec:DophantineSl2R} and then focus on Section \ref{sec:Dio-quad} (possibly returning to Sections \ref{ssec:CubicGeometry} for the cubic geometric story, although our Diophantine results are restricted to quadratics).  In Section \ref{sec:Dio-quad}, we revisit the description of the critical exponent for Diophantine approximation of complex numbers by algebraic numbers of degree $2$.  We use Schmidt's subspace theorem, and we find that the exponent depends upon whether the complex numbers lies on a rational hyperbolic geodesic (defined in Section~\ref{sec:applquad}).  The main theorems are Theorems~\ref{thm:dirichlet-geo} and \ref{thm:dirichlet-quad} (analogous to Dirichlet's Theorem, asserting infinitely many good quadratic irrational approximations) and Theorem \ref{thm:quad-roth} (analogous to Roths' Theorem, asserting that there are only finitely many better quadratic irrational approximations).

\section{Gallery}
\label{sec:gallery}

\subsection{In a galaxy far, far away}
\label{sec:Starscapes}

\begin{figure}[h!tbp]
\centering
\begin{subfigure}[b]{0.42\textwidth} 
    \centering
	\includegraphics[width=\textwidth]{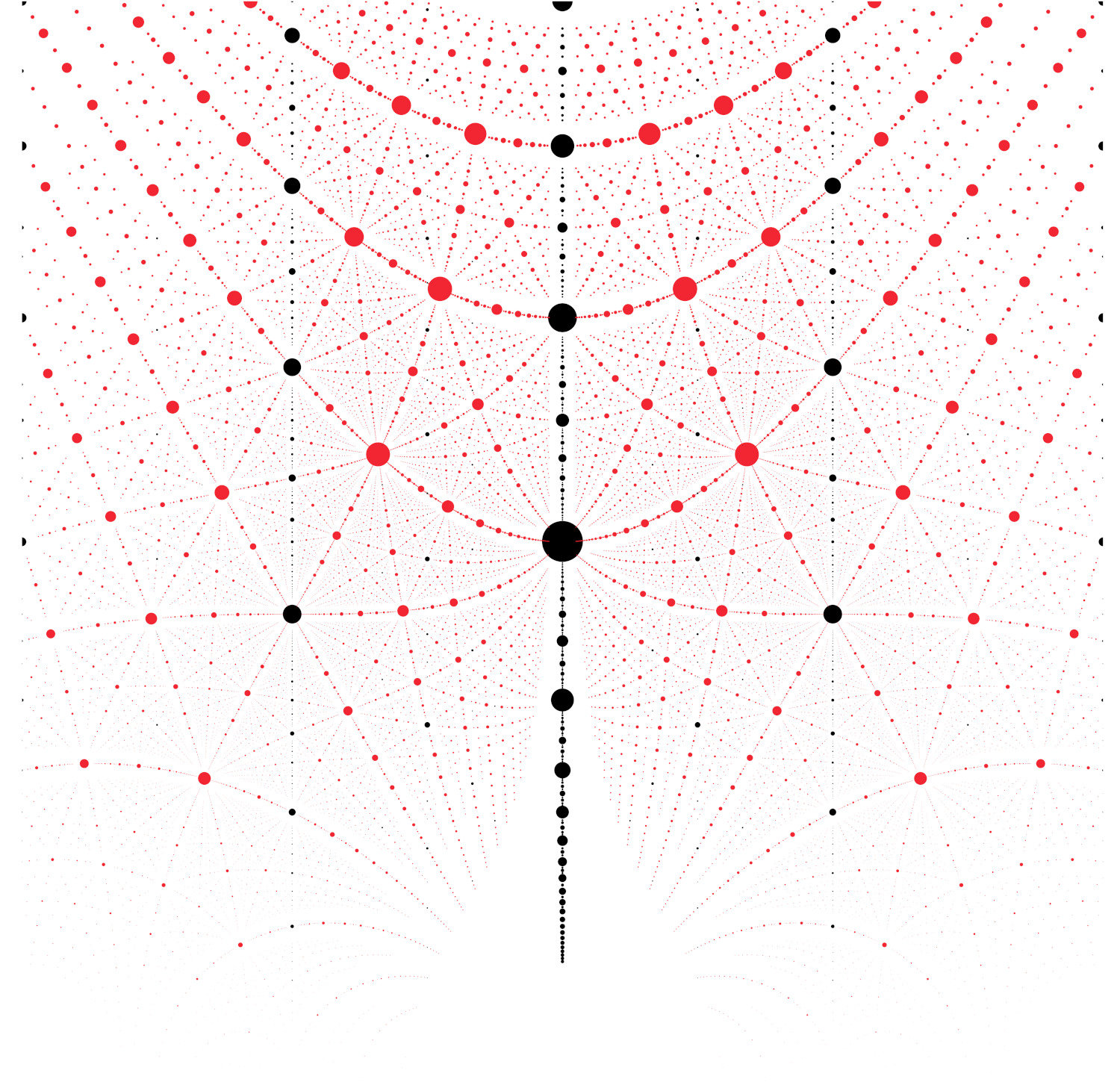}
	\caption{$a x^3 + b x + c = 0$}
    \label{fig:Starscapes_a}
\end{subfigure}
\begin{subfigure}[b]{0.42\textwidth} 
    \centering
	\includegraphics[width=\textwidth]{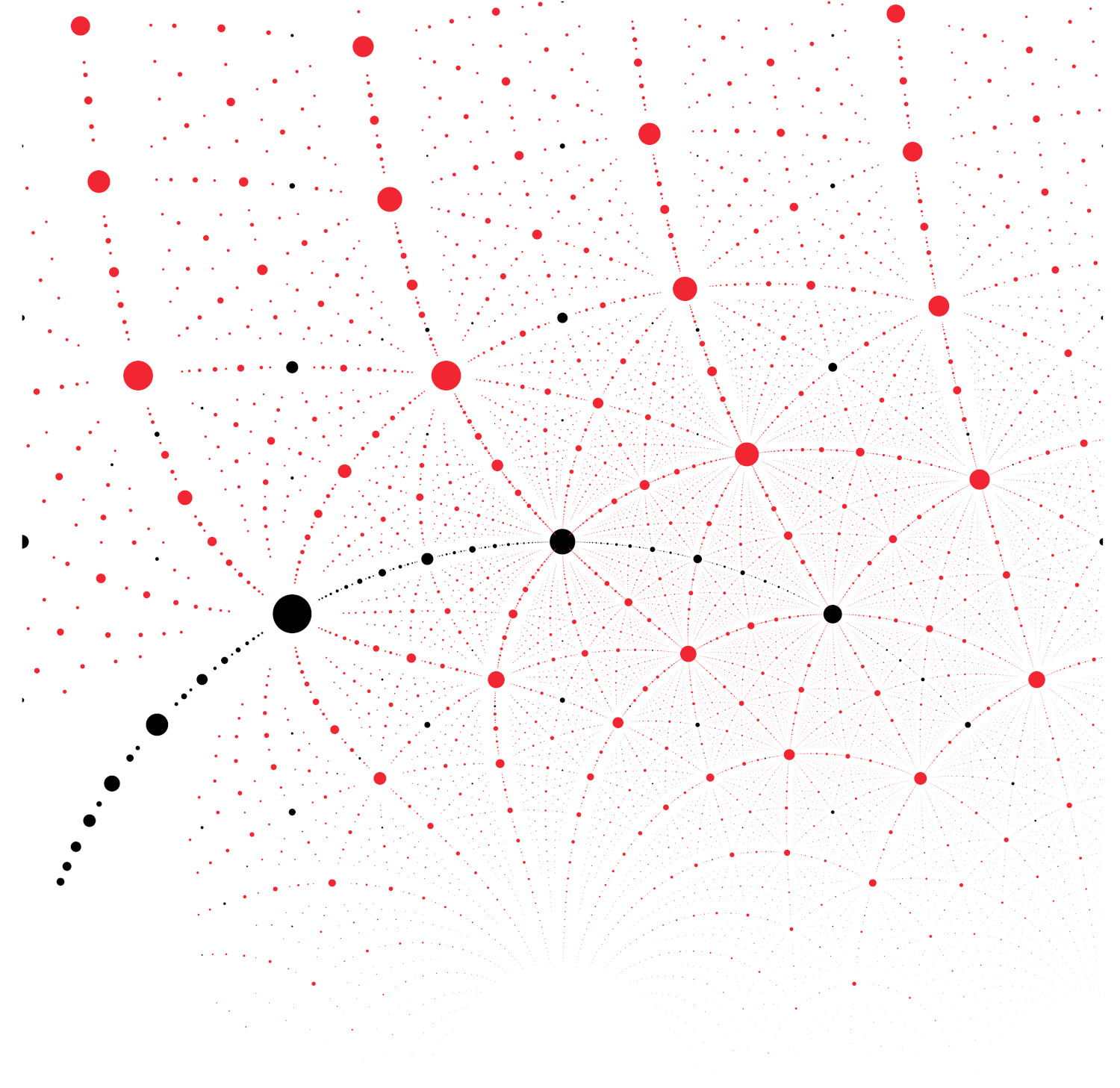}
	\caption{$a x^3 + b x^2 + c x + a = 0$}
    \label{fig:Starscapes_b}
\end{subfigure}
\begin{subfigure}[b]{0.42\textwidth} 
    \centering
	\includegraphics[width=\textwidth]{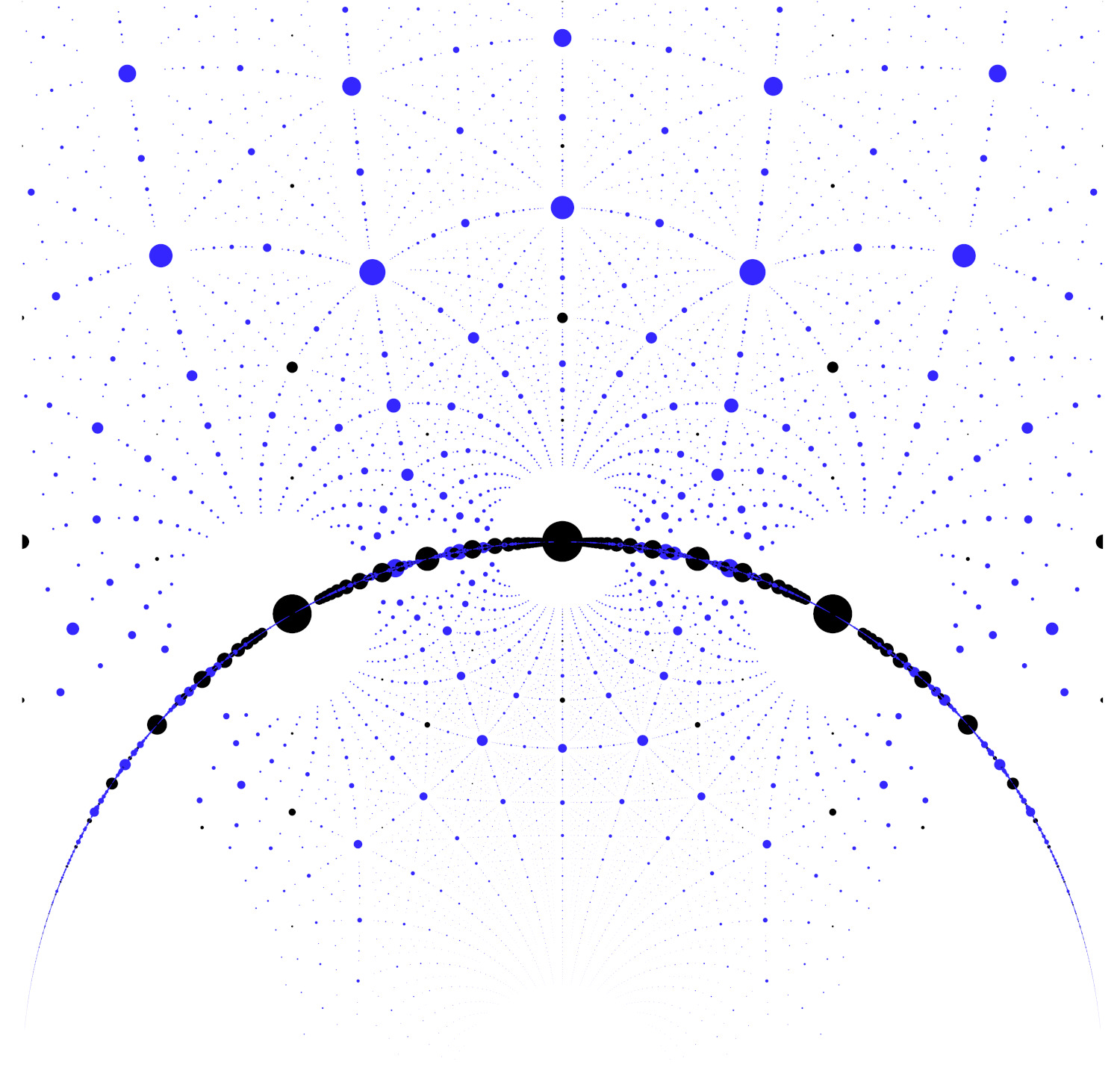}
	\caption{$a x^4 + b x^3 + c x^2 + b x + a = 0$}
    \label{fig:Starscapes_c}
\end{subfigure}
\begin{subfigure}[b]{0.42\textwidth} 
    \centering
	\includegraphics[width=\textwidth]{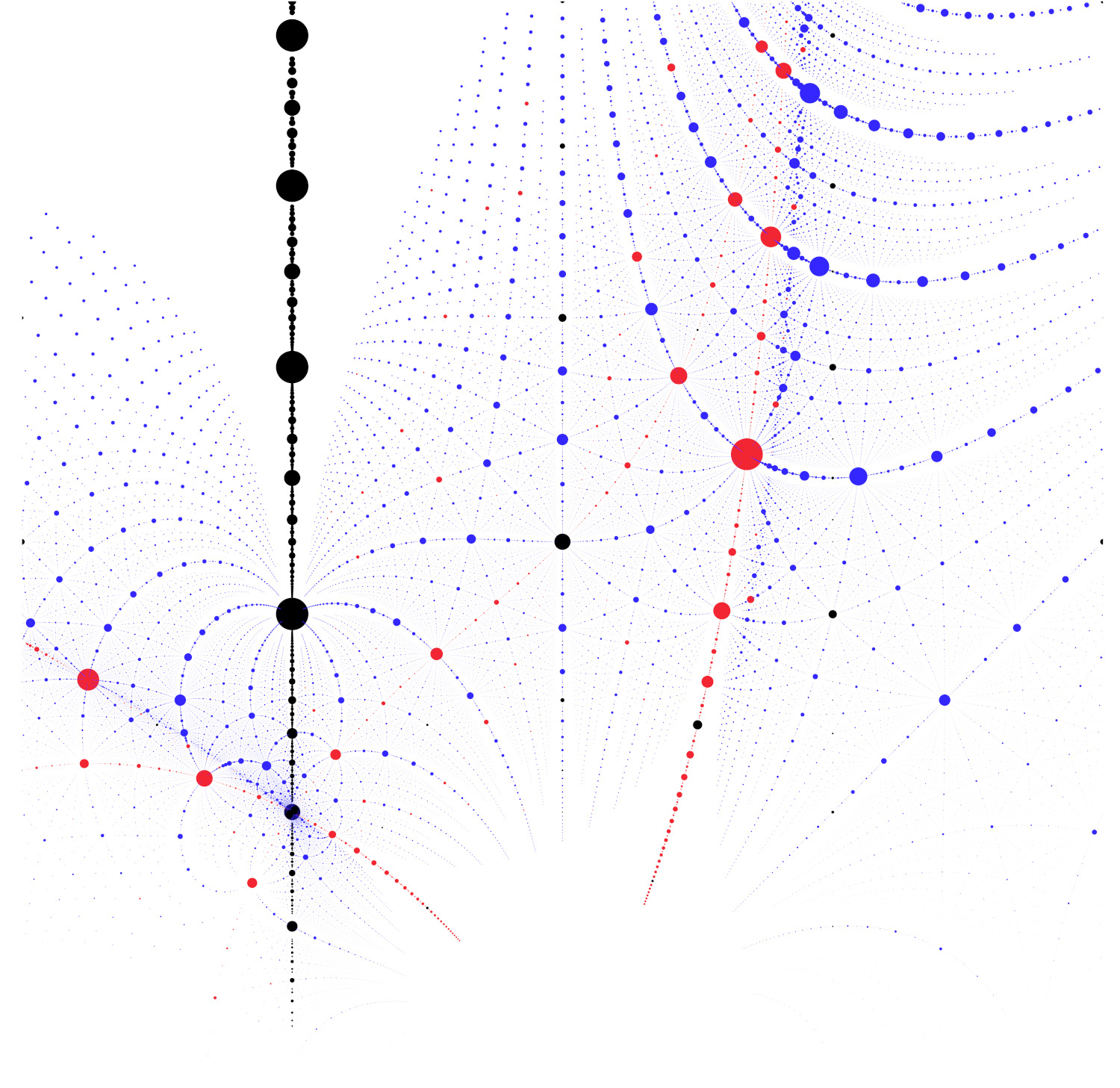}
	\caption{$a x^4 + b x^2 + b x + c = 0$}
    \label{fig:Starscapes_d}
\end{subfigure}
\begin{subfigure}[b]{0.42\textwidth} 
    \centering
	\includegraphics[width=\textwidth]{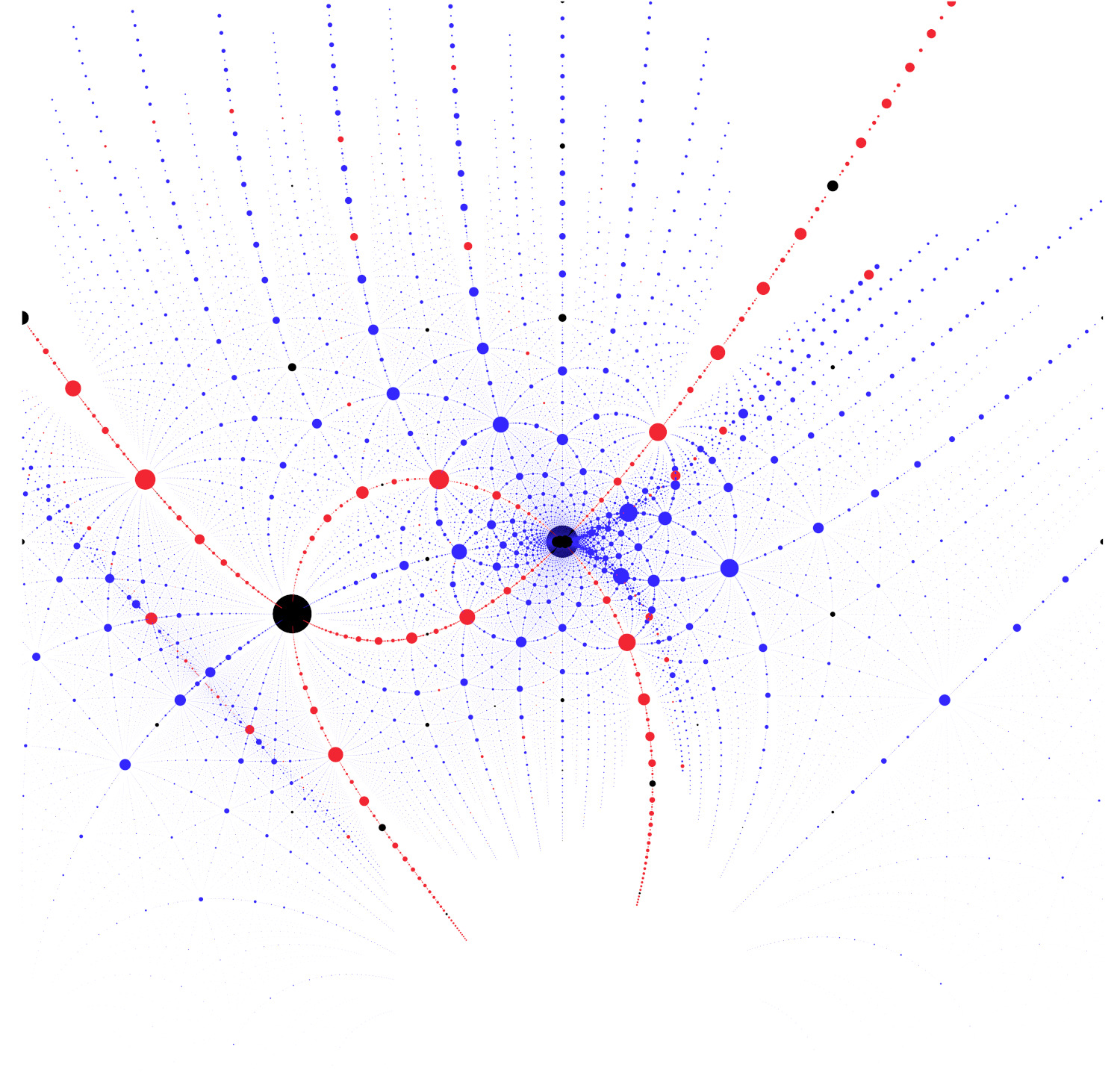}
	\caption{$a x^4 + b x^3 + b x^2 + b x + c = 0$}
    \label{fig:Starscapes_e}
\end{subfigure}
\begin{subfigure}[b]{0.42\textwidth} 
    \centering
	\includegraphics[width=\textwidth]{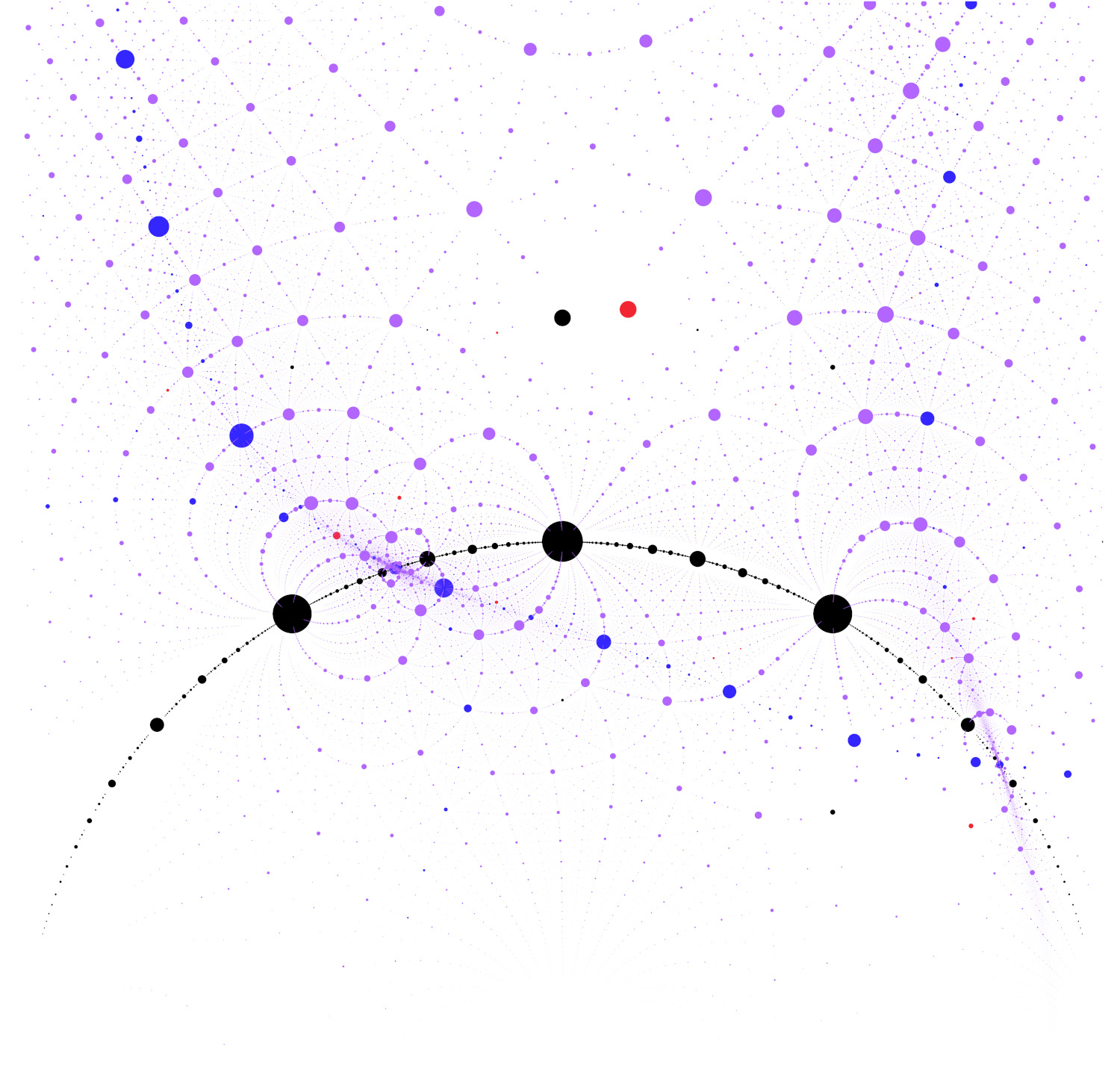}
	\caption{$a x^5 + b x^3 + c x^2 + b x + a = 0$}
    \label{fig:Starscapes_f}
\end{subfigure}
\caption{Complex algebraic numbers (roots of each polynomial family with $a,b,c \in \ZZ$) sized by the root discriminant in the hyperbolic metric. Quadratic numbers are black, cubics red, quartics blue and quintics purple. }
\label{fig:Starscapes}
\end{figure}

The starscapes shown in Figure \ref{fig:Initial} of the quadratics (\ref{fig:Initial_quadratics}) and the family of cubics (\ref{fig:Initial_cubic_family}) have a striking self-similar curvilinear structure:  the figures seem to be populated with beaded necklaces of discs subdividing the plane into smaller regions criss-crossed with similar, finer, necklaces. 

This basic pattern occurs quite generally. To give all polynomials of a given degree each coefficient must be allowed to change freely. This defines topological space, the  ``coefficient space,'' of polynomials where each coefficient gives a dimension\footnote{This space starts as $\RR^n$, but with polynomials multiplying all coefficients by a constant does not change the roots of the polynomials. It is natural, therefore to consider such polynomials as equivalent. This gives projective geometry one dimension lower, as discussed in Section~\ref{sec:ProjGeo}.}. The total dimension of all polynomials of a given degree is one greater than the degree. It is natural to think of subspaces of this space and the pattern we describe above seems to appear any time one works with a three dimensional linear subspace in coefficient space; more examples appear in Figure \ref{fig:Starscapes}. 

As the degree increases, the curves and their relationships become more complicated, but the basic motif continues.  We hope these images justify the term ``algebraic starscapes.''  When the family of polynomials is two-dimensional, we obtain a ``linear starscape'' (that is, a single beaded necklace threading through the plane; see Figure \ref{fig:Lines}).  When a family of polynomials is three-dimensional, we obtain ``planar starscapes'' such as those in Figure \ref{fig:Starscapes}.  One may continue this to higher dimension, but in four-dimensional families such as all cubics (Figure \ref{fig:Initial_cubics}), the collapse to the complex plane produces much more complicated, less immediately patternful, although nonetheless enticing, pictures.

\begin{figure}[h!tbp]
\centering
\includegraphics[width=.45\textwidth]{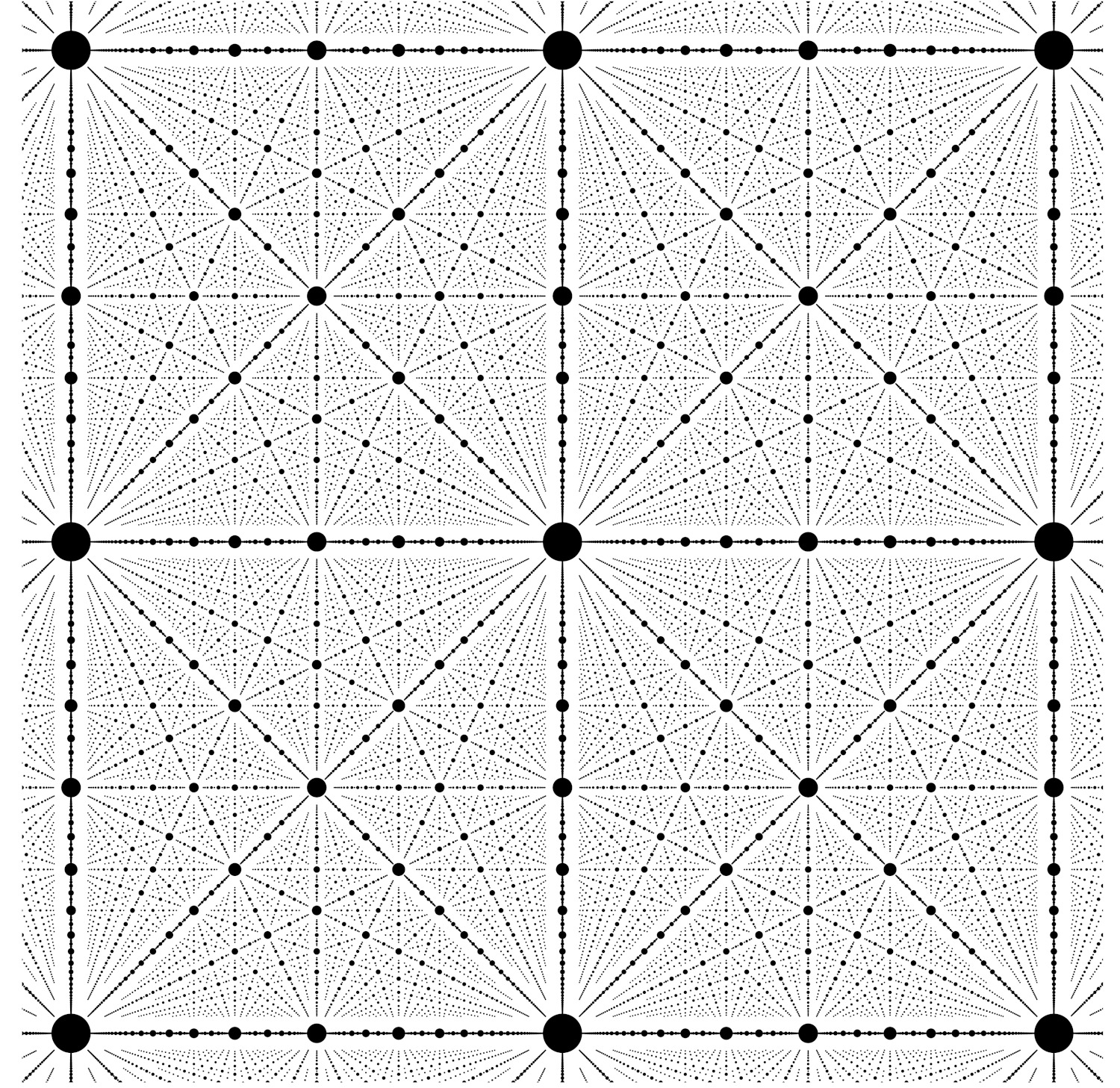}
	\caption{The points $(b/a,c/a)$, for $a,b,c \in \ZZ$ with radius proportional to $1/a$. To see how this links to the view of a lattice for an observer, see Figure \ref{fig:RPn}.}
\label{fig:Rationals}
\end{figure}

The simplest version of the basic motif that seems to pervade these images appears when looking at a lattice in perspective, as you can see in Figure \ref{fig:Rationals}. In this case the points only form straight lines. The one dimensional version of this effect is sometimes called the orchard illusion and can be experienced when driving past an orchard planted on a grid (Figure \ref{fig:Orchard}). 

\begin{figure}[h!tbp]
\centering
\includegraphics[width=.8\textwidth]{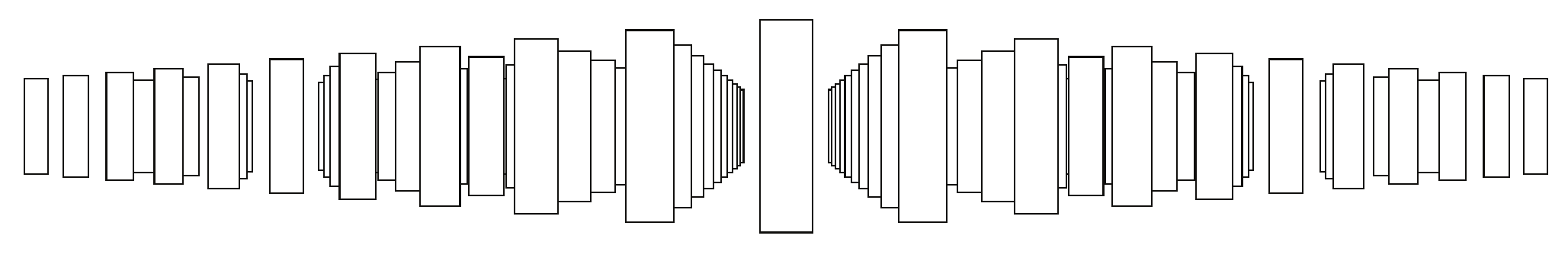}
\caption{A sideways view of a grid of rectangles in perspective forming the orchard illusion.}
\label{fig:Orchard}
\end{figure}

\begin{figure}[h!tbp]
\centering
\includegraphics[width=.8\textwidth]{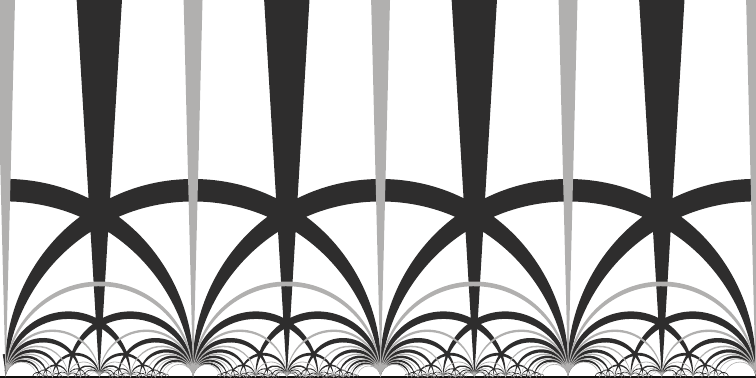}
\caption{A tiling by infinite polygons on the hyperbolic plane (a partition of the hyperbolic plane by the action of $\PSL(2;\ZZ)$). The curves seen here are some of those seen in Figure \ref{fig:Initial_quadratics}.}
\label{fig:Modular}
\end{figure}

In the quadratic case (Figure \ref{fig:Initial_quadratics}) the lines might actually seem familiar to geometers as they are the straight lines (geodesics), not in euclidean geometry, but in the upper half plane model of the hyperbolic plane. In fact Figure \ref{fig:Initial_quadratics} is a regular tiling of the upper half plane by a single tile, under the action of a group of symmetries preserving the hyperbolic distance.  More precisely, the action is that of the \emph{modular group} $\PSL(2;\ZZ)$ (Figure \ref{fig:Modular}).

\begin{figure}[h!tbp]
\centering
\includegraphics[width=.42\textwidth]{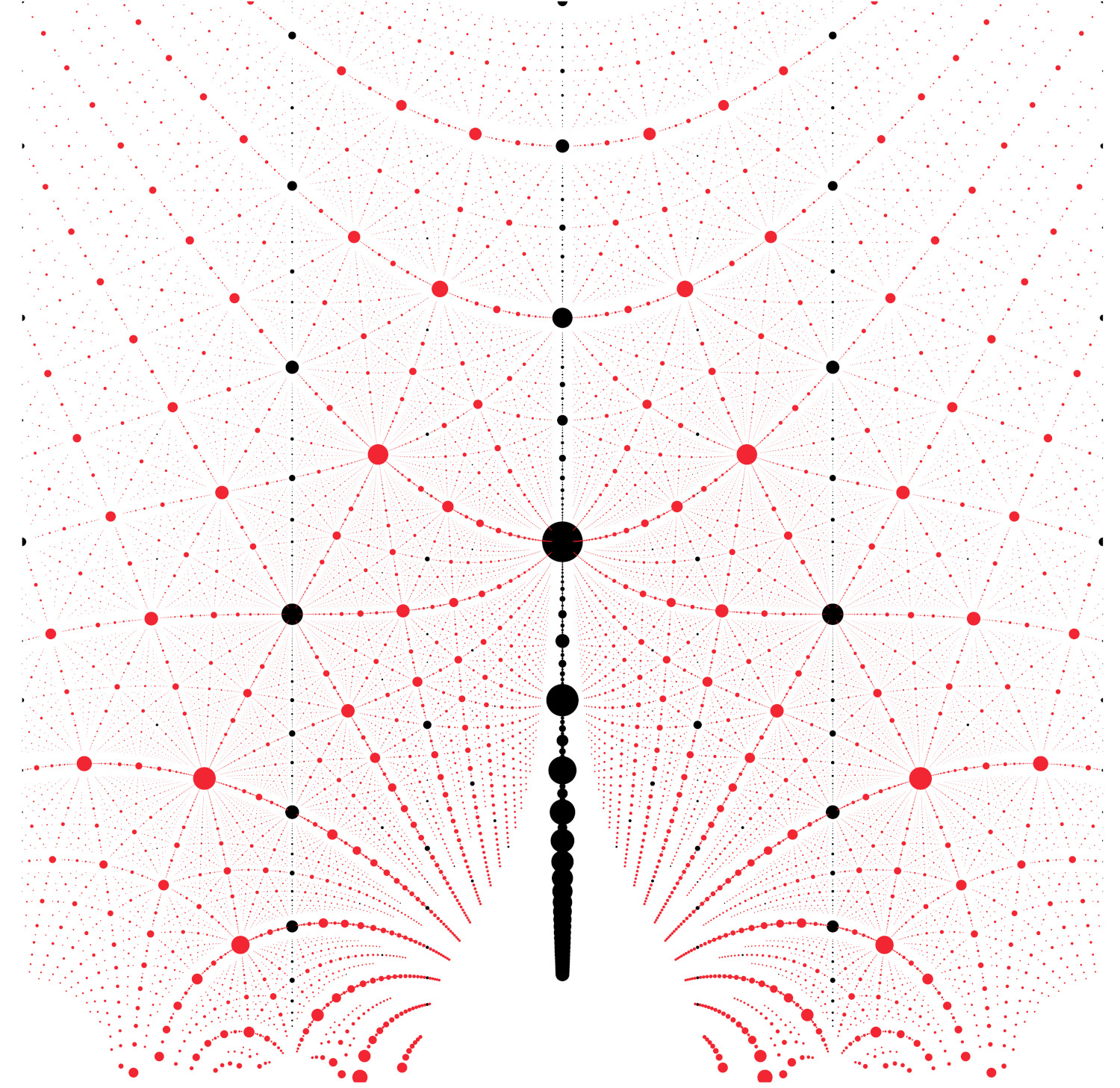}
	\caption{The depressed cubics, sized by the root discriminant, plotted in the euclidean metric (meaning, radii are interpreted as euclidean). Note how the plot is very dense at the bottom, but light at the top.}
\label{fig:CubicFamilyA0BCEuclidean}
\end{figure}

While the link to hyperbolic geometry and $\PSL(2;\ZZ)$ is especially strong in the quadratic case, it is more generally useful in understanding starscapes.  From an aesthetic perspective, using the hyperbolic metric improves the look of the images close to the real line, as seen in Figure \ref{fig:CubicFamilyA0BCEuclidean}.  There are more mathematical justifications for this approach described in later sections.

The observations we have described so far constitute a na\"ive visual approach to the images produced:  we are simply asking what it is we are seeing. The underlying geometric explanation, especially for the quadratic and cubic cases, is developed in far more detail in Section \ref{sec:Geometry}, and the deeper connections to the study of Diophantine approximation in Sections \ref{sec:DiophantineApproximation} and \ref{sec:Dio-quad}.

\subsection{Mostly harmless}

\begin{figure}[h!tbp]
\centering
\begin{subfigure}[b]{0.42\textwidth} 
    \centering
	\includegraphics[width=\textwidth]{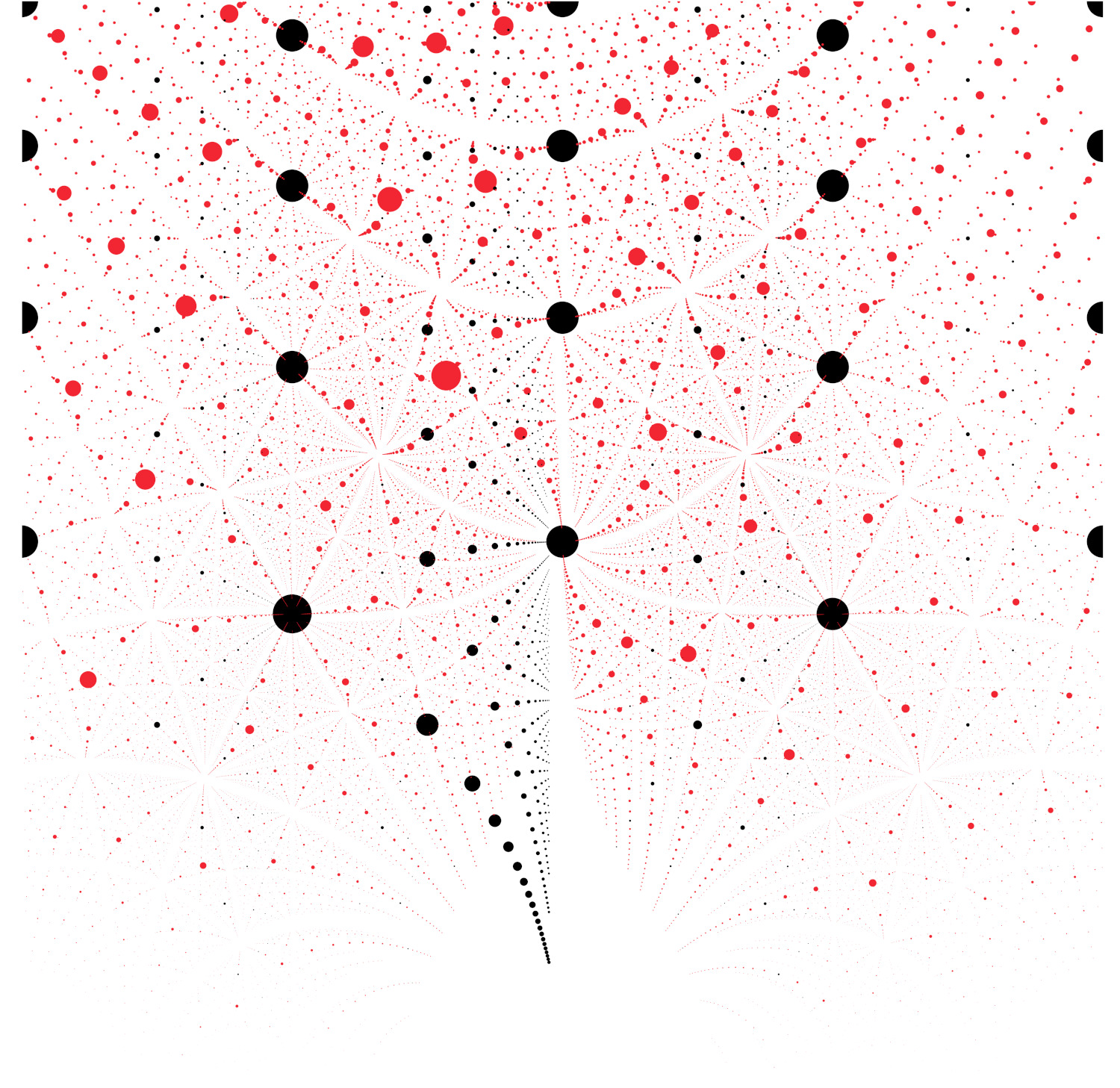}
	\caption{$a x^3 + x^2 + b x + c = 0$}
    \label{fig:AffineStarscapes_a}
\end{subfigure}
\begin{subfigure}[b]{0.42\textwidth} 
    \centering
	\includegraphics[width=\textwidth]{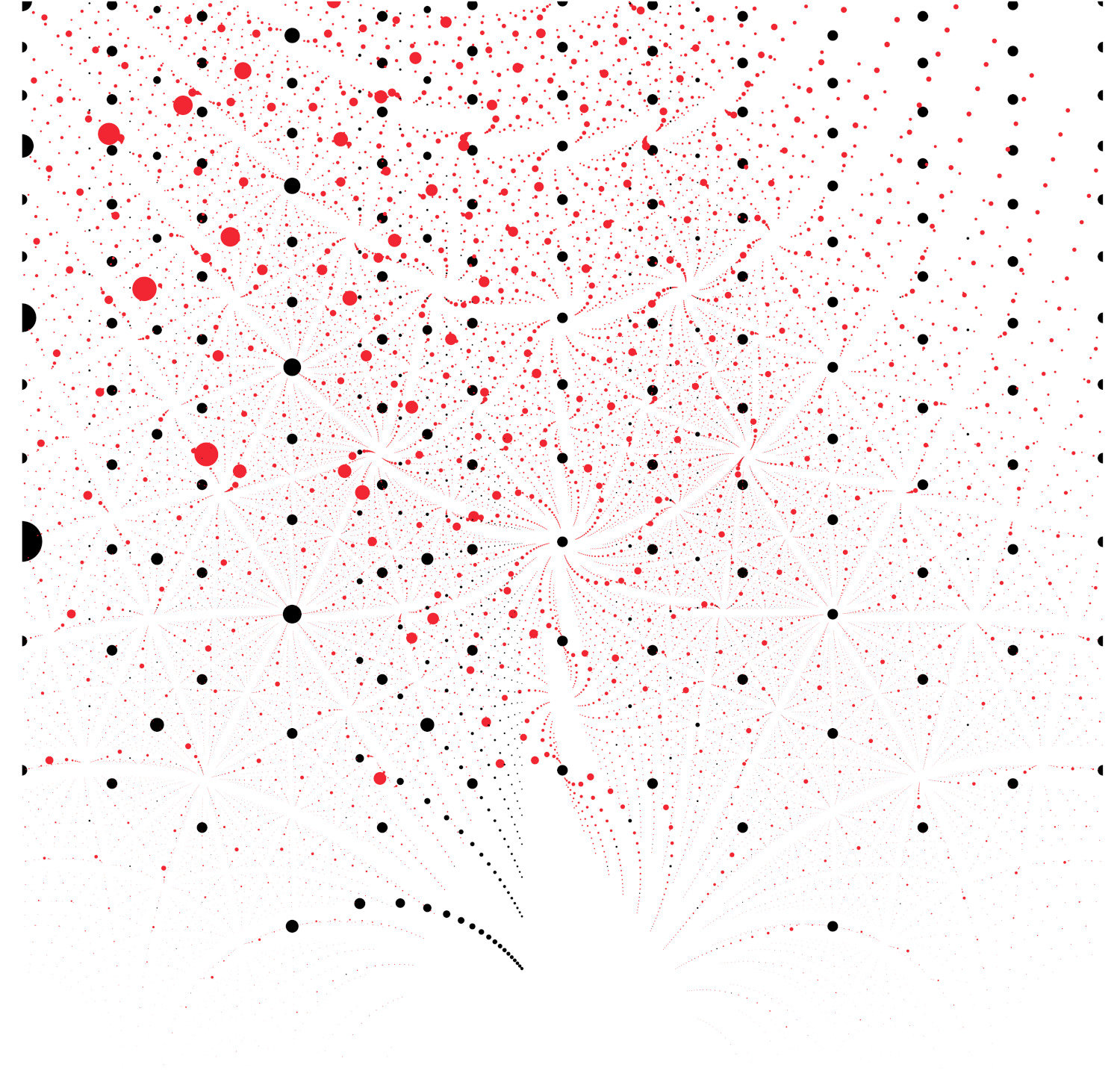}
	\caption{$a x^3 + 3 x^2 + b x + c = 0$}
    \label{fig:AffineStarscapes_b}
\end{subfigure}
\begin{subfigure}[b]{0.42\textwidth} 
    \centering
	\includegraphics[width=\textwidth]{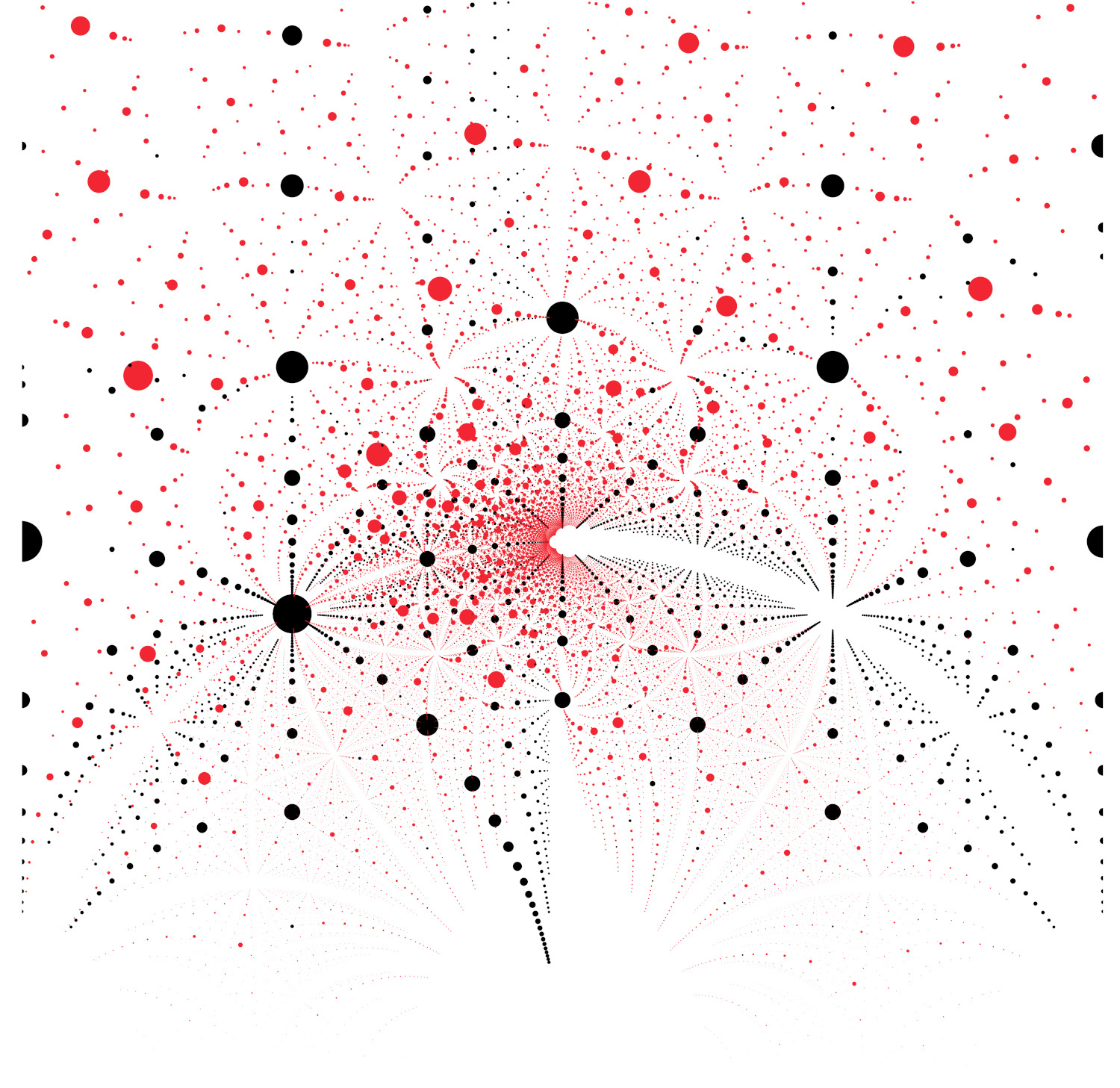}
	\caption{$a x^3 + (c+1) x^2 + b x + c = 0$}
    \label{fig:AffineStarscapes_c}
\end{subfigure}
\begin{subfigure}[b]{0.42\textwidth} 
    \centering
	\includegraphics[width=\textwidth]{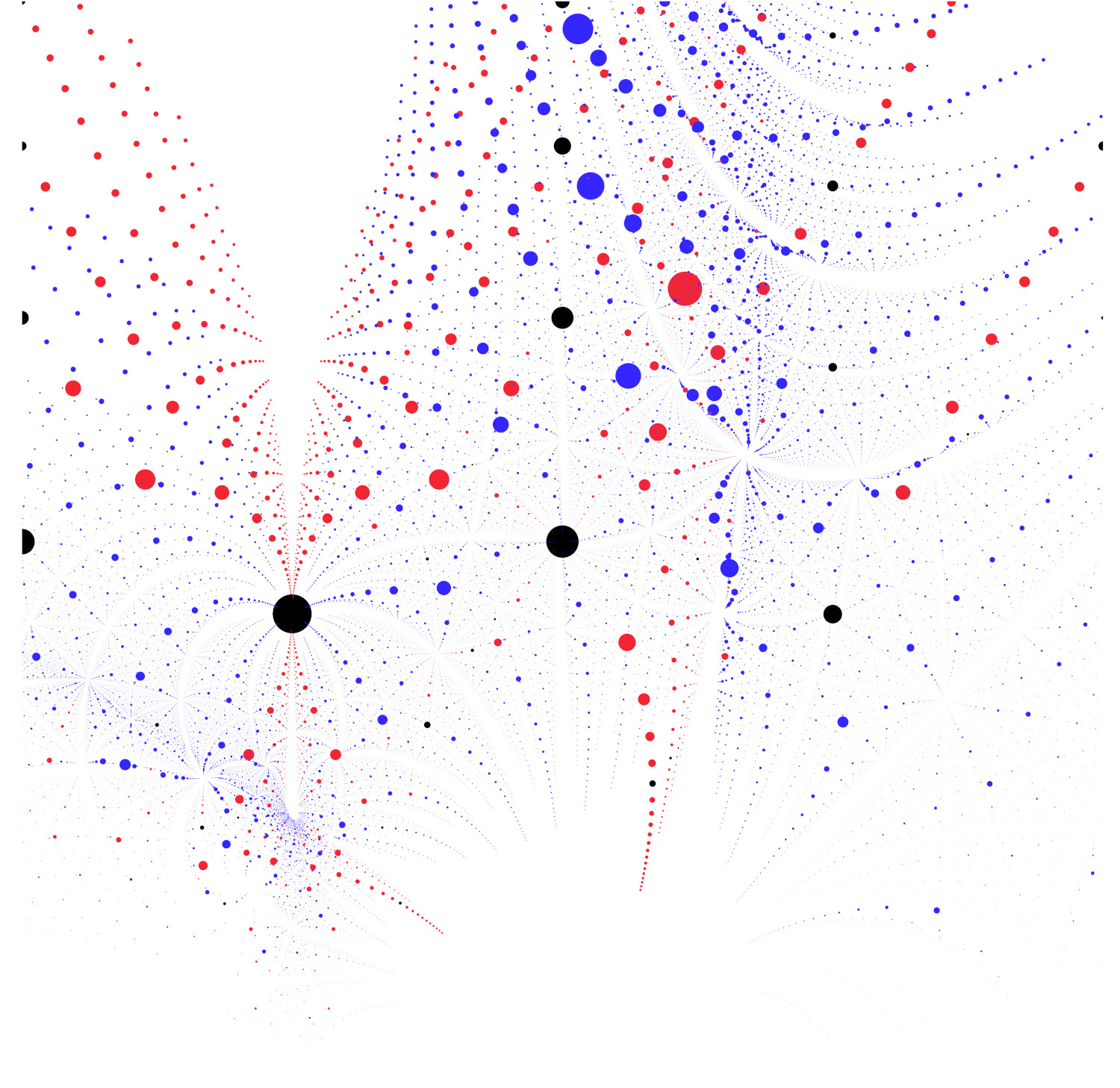}
	\caption{$a x^4 + x^3 + b x^2 + b x + c = 0$}
    \label{fig:AffineStarscapes_d}
\end{subfigure}
\caption{Roots on affine planes in coefficient space.}
\label{fig:AffineStarscapes}
\end{figure}

The images shown in Section \ref{sec:Starscapes} are compelling enough to encourage wider investigation. For example, one might move the families considered away from 0 to give affine subspaces of the coefficient space, as shown in Figure \ref{fig:AffineStarscapes}. 

\begin{figure}[h!tbp]
\centering
\begin{subfigure}[b]{0.42\textwidth} 
    \centering
	\includegraphics[width=\textwidth]{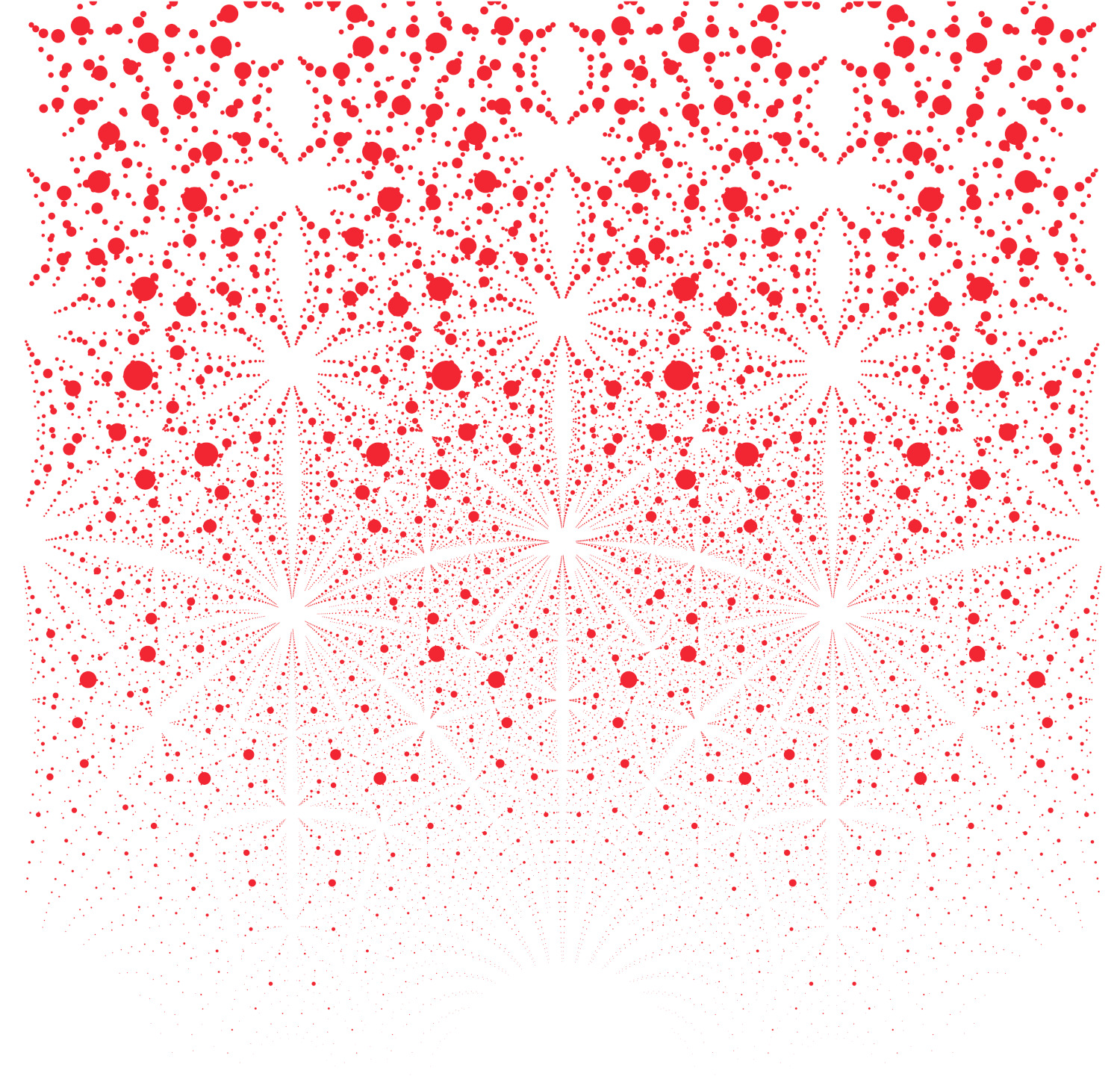}
	\caption{Cubic algebraic integers}
    \label{fig:CubicIntegers}
\end{subfigure}
\begin{subfigure}[b]{0.42\textwidth} 
    \centering
	\includegraphics[width=\textwidth]{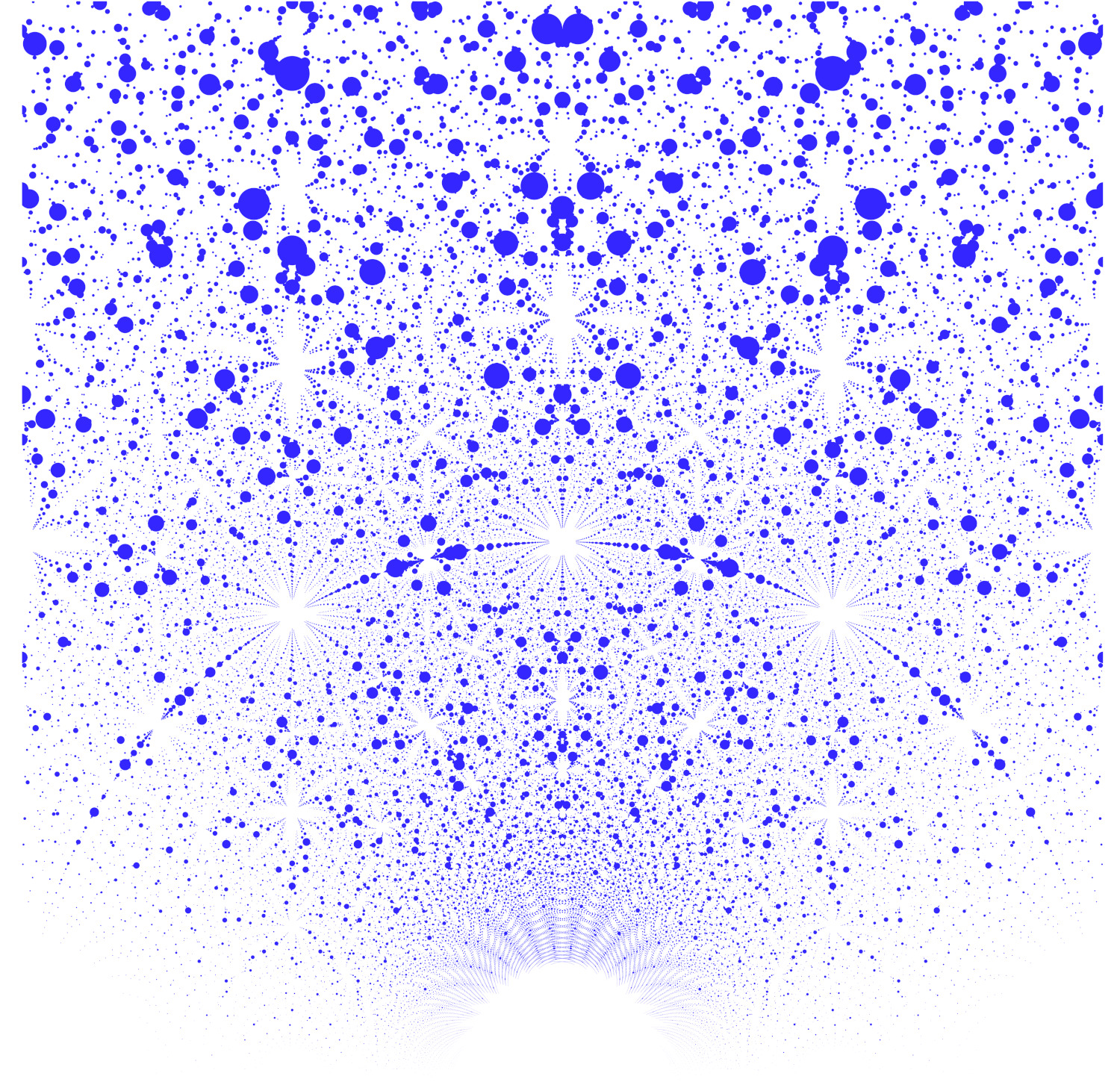}
	\caption{Quartic \emph{unit} algebraic integers}
    \label{fig:QuarticUnits}
\end{subfigure}
	\caption{Algebraic integers and unit integers.}
\label{fig:IntegersAndUnits}
\end{figure}

In these affine subspaces, a couple have particular interest:  the algebraic integers (where the leading coefficient is 1) and the algebraic integer units (where the leading and constant coefficients are both 1). These are shown for the cubics and quartics in Figure \ref{fig:IntegersAndUnits}.

\begin{figure}[h!tbp]
\centering
\includegraphics[width=.45\textwidth]{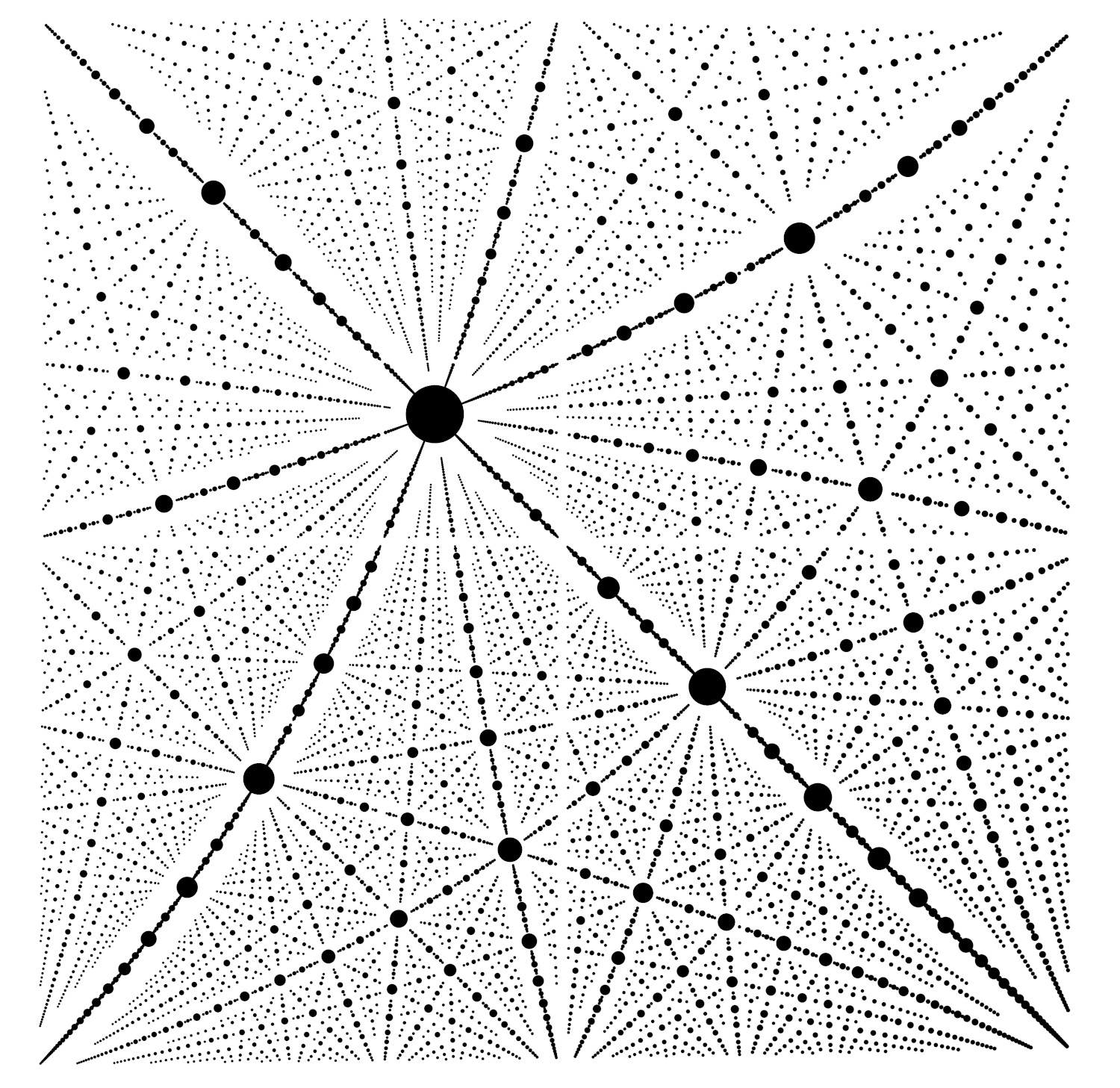}
	\caption{The real quadratics plotted against their algebraic conjugate, with one root (x-axis) between $-1$ and $0$ and the other between $1$ and $2$. The large dot is the Golden ratio $(\frac{1-\sqrt{5}}{2},\frac{1+\sqrt{5}}{2})$.}
\label{fig:RealQuadratics}
\end{figure}

Another approach is to consider real roots, or tuples of roots.  For example, quadratics with real roots can be plotted in $\R \times \R$ (Figure \ref{fig:RealQuadratics}), though the lack of an ordering on the roots will cause each to appear twice. 

A powerful property of the images of the quadratics with complex roots is that they show \emph{all} the information about the roots, as the complex roots come as a complex conjugate pair (so that all the information is shown with just one of them). To extend this to the cubics requires an additional dimension. In this case the cubic polynomials with complex roots always have an additional real root. 

\begin{figure}[h!tbp]
\centering
\begin{subfigure}[b]{0.45\textwidth} 
    \centering
	\includegraphics[width=\textwidth]{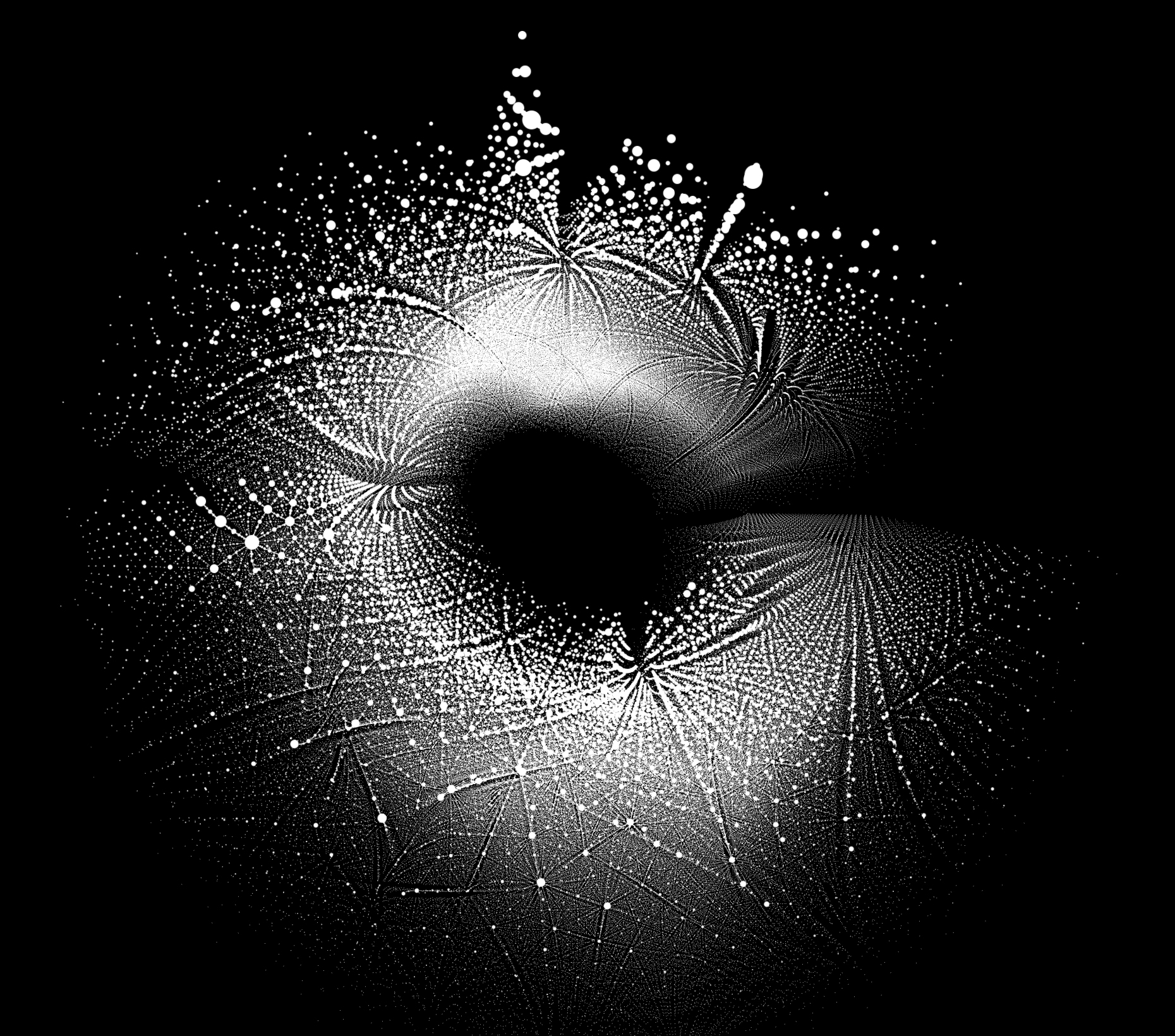}
	\caption{}
    \label{fig:AllCubicsin3d}
\end{subfigure}
\begin{subfigure}[b]{0.45\textwidth} 
    \centering
	\includegraphics[width=\textwidth]{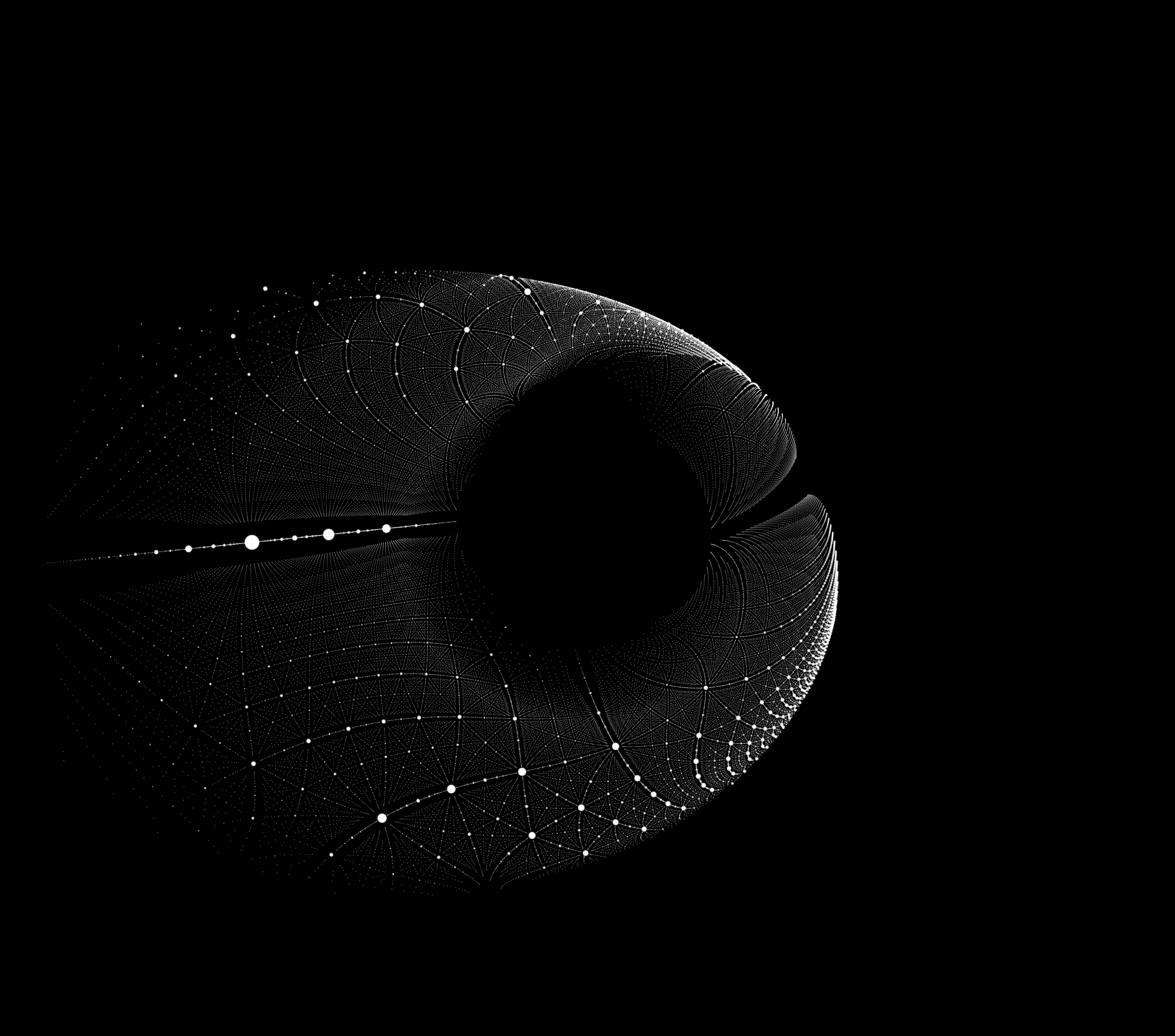}
	\caption{}
    \label{fig:CubicFamilyABACin3d}
\end{subfigure}
	\caption{All cubic polynomials shown in the unit tangent bundle to $\HH^2$ (\ref{fig:AllCubicsin3d}) and the roots of $a x^3 + b x^2 + a x + c = 0$ (\ref{fig:CubicFamilyABACin3d}), forming a M\"{o}bis strip within it. Both are shown from the same angle so if you look carefully you can see the second image in the first. Images are screenshots from SL(View) by David Dumas \cite{slview}.}
\label{fig:cubicsin3d}
\end{figure}

To take into account this additional information, we could take a space in $\C\times\R$ given by the upper half-plane for the complex root and the real line for the real root. We can then add some rather beautiful geometry to this space, described in detail in Section \ref{ssec:CubicGeometry}. This corresponds to considering the upper half plane of $\C$ as the hyperbolic plane. In this model the real line in $\C$ lies on the boundary of the hyperbolic plane, which contains one additional point, the point at infinity that pulls the real line back into a circle by connecting the two ends. This can be considered as the difference between considering the slope and angle of a line on the plane. As the slope gets more positive or negative the line gets closer to vertical. As a slope this is not obtainable, but it is can be considered as an angle. The complex root can therefore look towards any point on this circular boundary. 

Considering the pair of a point and direction in a space gives the geometry of the ``unit tangent bundle''. For hyperbolic geometry this can be considered to be a solid torus, where the circular slices are a disk model of the hyperbolic plane. This torus is a finite region of three dimensional space so, although we took a bit of a journey the result gives a powerful way to see the roots of cubic polynomials, as shown in Figure \ref{fig:cubicsin3d}. 

The cubic families we have considered nicely embed into this picture as 2d surfaces.  The polynomials $a x^3 + b x^2 + a x + c = 0$ having a complex root even create a M\"{o}bius strip (Figure \ref{fig:CubicFamilyABACin3d}). Both of these images are even more powerful when you can manipulate them yourself in 3d, and we encourage you to check out David Dumas' beautiful software SL(View) that we used to make these images ourselves \cite{slview}; see \url{algebraicstarscapes.com} for the datasets.

A different approach to using the unit tangent bundle is to draw arrows, rather than dots on the plane. Such images are shown in Figures~\ref{fig:Arrows} and \ref{fig:CoeffCubics}. 

There are many other spaces to explore that have the potential to reveal many aspects of the structure of algebraic numbers and illustrate various geometric ideas. Using the colour of the dots has the potential to give pictures with up to 6 dimensions of information (3 spatial and 3 colour (red, green and blue, for example)). As an example, more of the structure of Figure \ref{fig:Starscapes_d} is revealed when real roots are used to colour points, as shown in Figure \ref{fig:RealColoring}.

\begin{figure}[h!tbp]
\centering
\includegraphics[width=.42\textwidth]{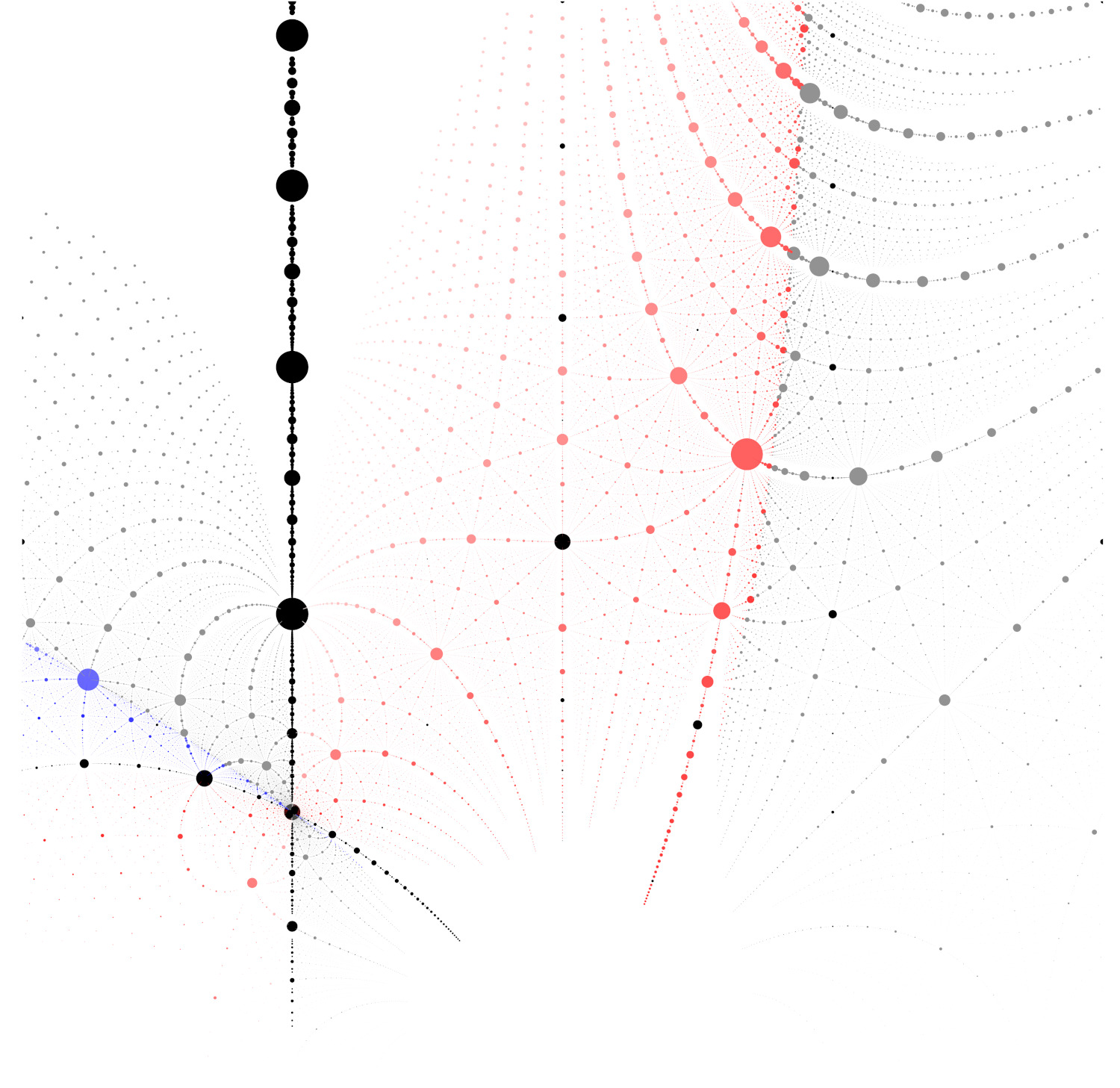}
	\caption{The starscape shown in Figure \ref{fig:Starscapes_d} coloured by a real root (simply the first given by Sage), red for negative and blue for positive, each fading to white as the absolute value increases. Quadratic Points are black and points with no real conjugate are shown in gray.}
\label{fig:RealColoring}
\end{figure}

\begin{figure}[h!tbp]
\centering
\begin{subfigure}[b]{0.6\textwidth} 
    \centering
	\includegraphics[width=\textwidth]{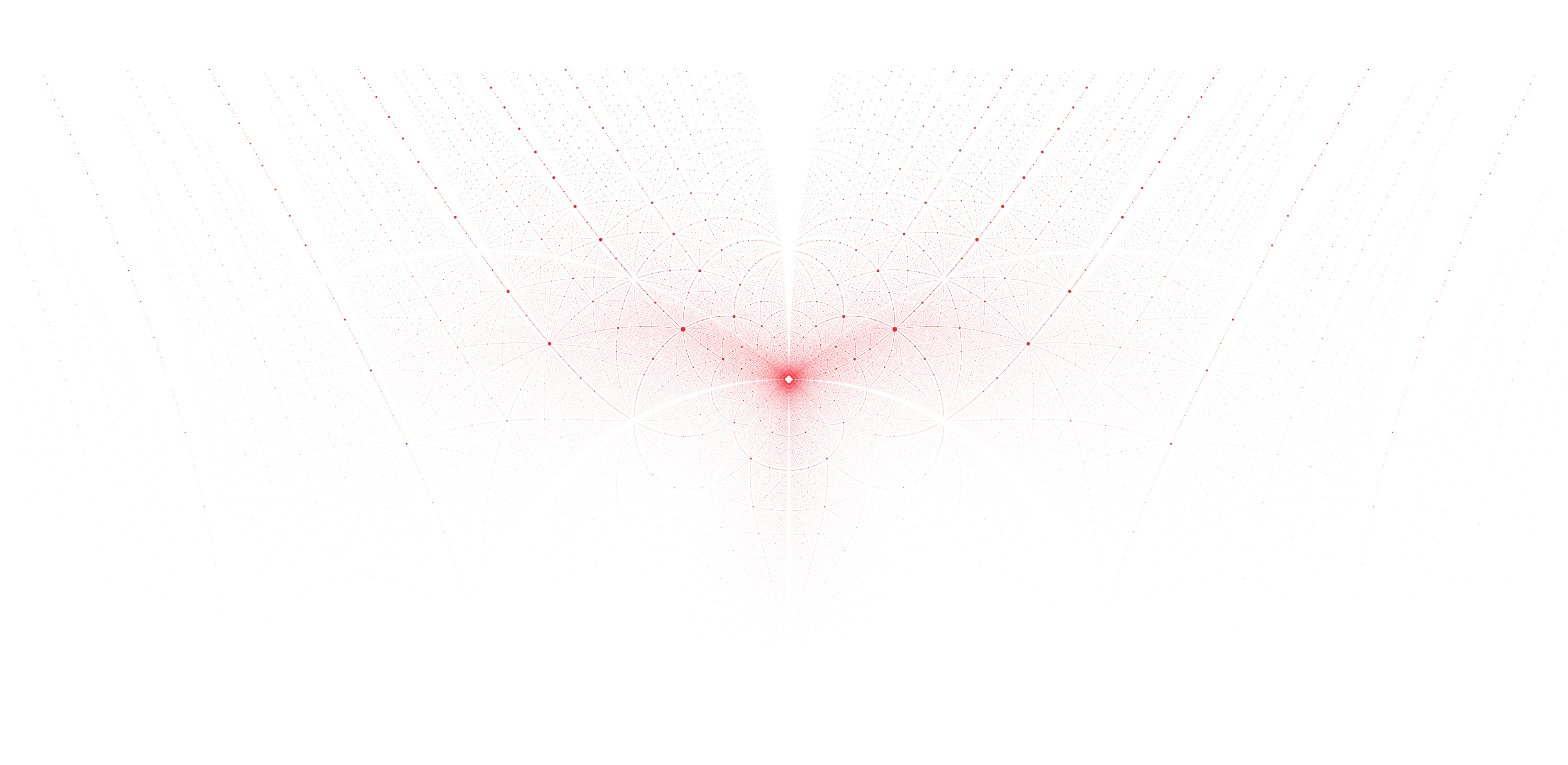}
	\caption{The 1-Norm on the coefficients (sum of absolute value).}
    \label{fig:DotSizes1Norm}
\end{subfigure}
\begin{subfigure}[b]{0.6\textwidth} 
    \centering
	\includegraphics[width=\textwidth]{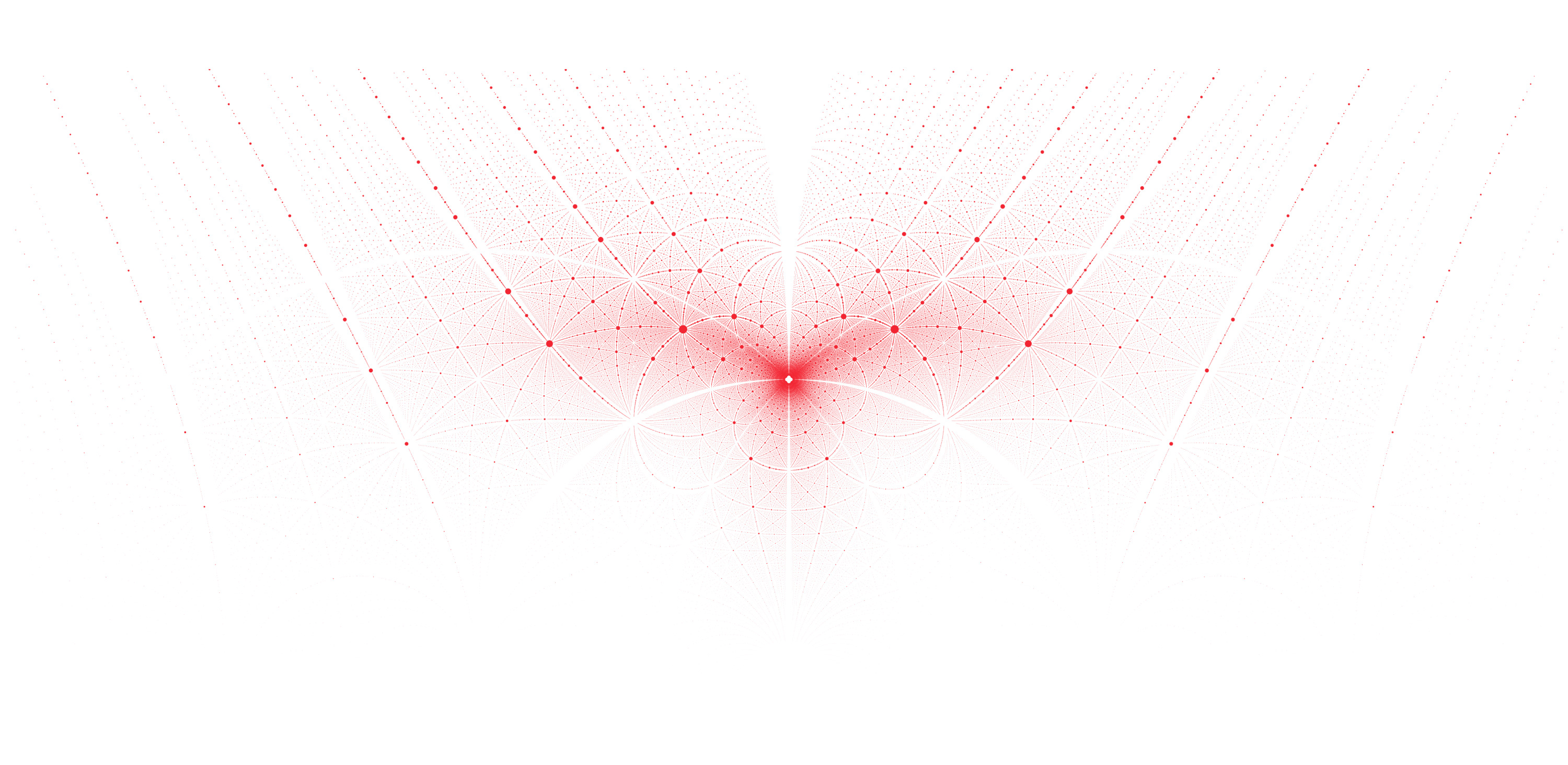}
	\caption{The standard euclidean norm on the coefficients.}
    \label{fig:DotSizes2Norm}
\end{subfigure}
\begin{subfigure}[b]{0.6\textwidth} 
    \centering
	\includegraphics[width=\textwidth]{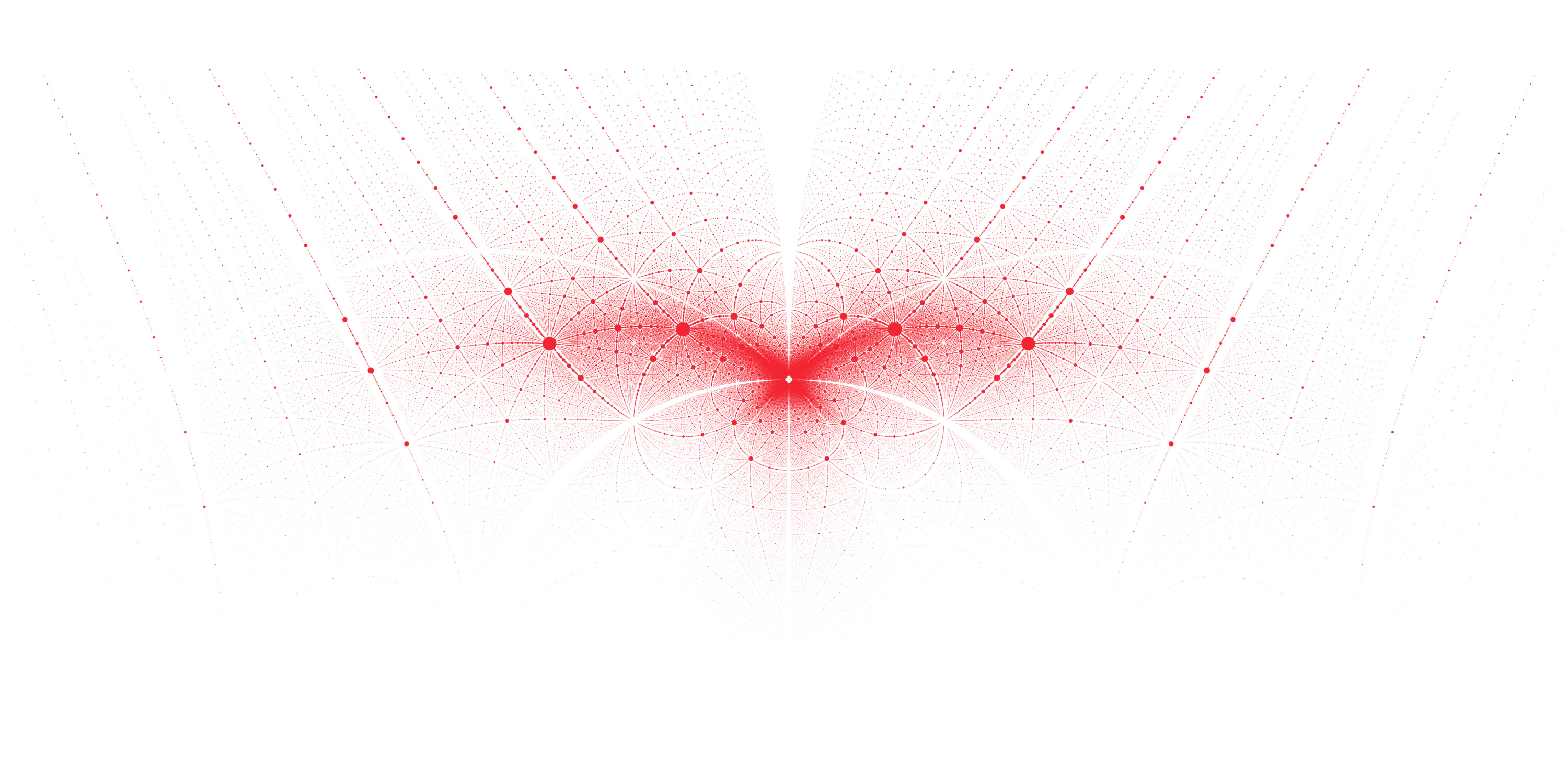}
	\caption{The $\infty$-Norm giving the largest absolute coefficient.}
    \label{fig:DotSizesInftyNorm}
\end{subfigure}
\caption{Complex roots of polynomials of the form $a x^3 + b x^2 + c x + a = 0$, using different dot sizing from the geometry of the coefficients, given by the inverse of the given quantity. All points with $a,b,c$ between $-40$ and $40$ are plotted, in the region between $-2.5$ and $2.5$ on the real axis and $0$ and $2$ on the imaginary axis, for a total of $231710$ roots.}
\label{fig:DotSizesCoeff}
\end{figure}

The challenge is to produce images that are attractive, informative or, ideally, both. From our explorations so far we have found that chasing the first is a surprisingly reliable (though not guaranteed) path to the second. 

\begin{figure}[h!tbp]
\centering

\begin{subfigure}[b]{0.6\textwidth} 
    \centering
	\includegraphics[width=\textwidth]{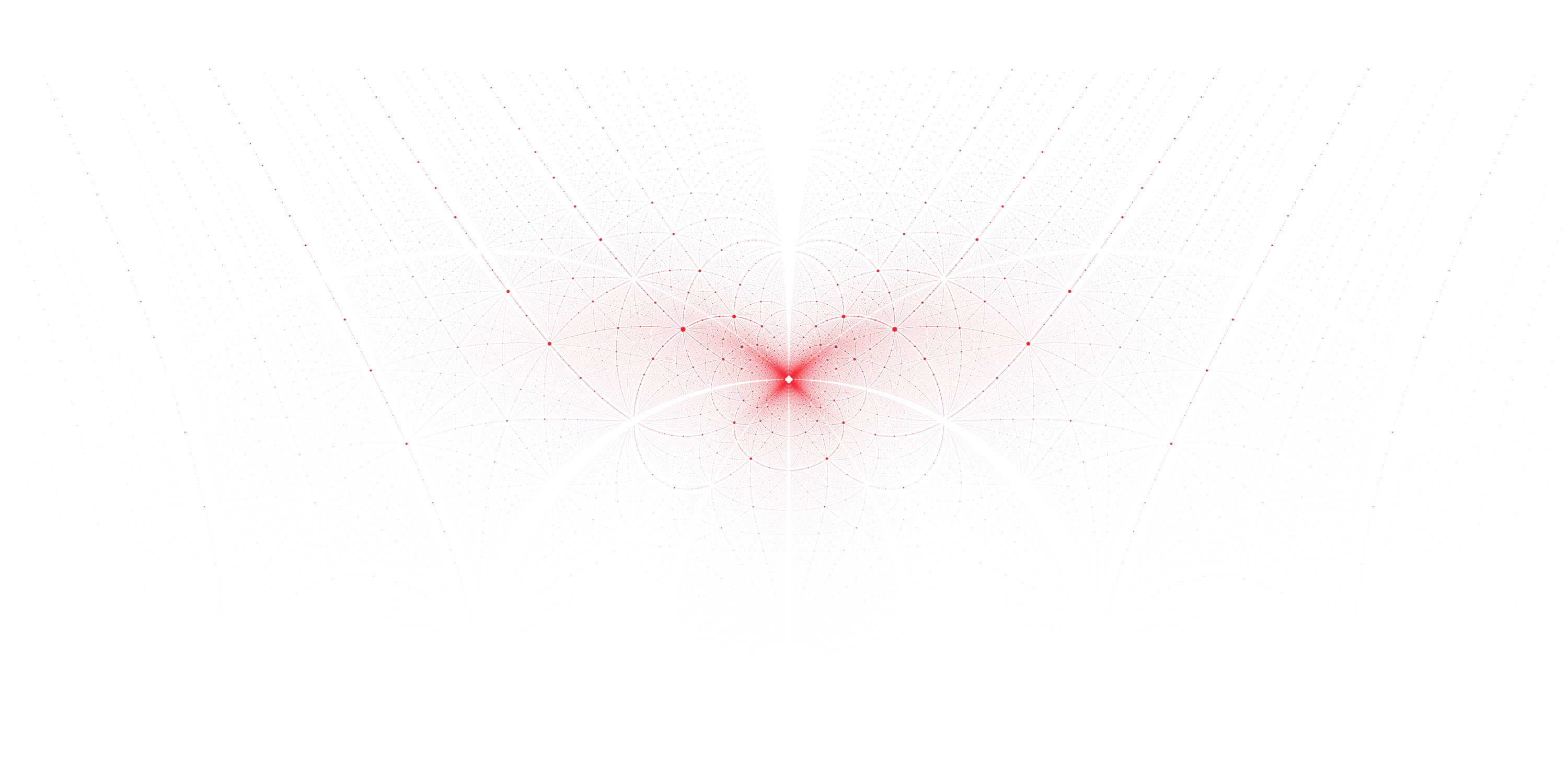}
	\caption{The Weil Height, or Mahler measure.}
    \label{fig:DotSizesMahler}
\end{subfigure}
\begin{subfigure}[b]{0.6\textwidth} 
    \centering
	\includegraphics[width=\textwidth]{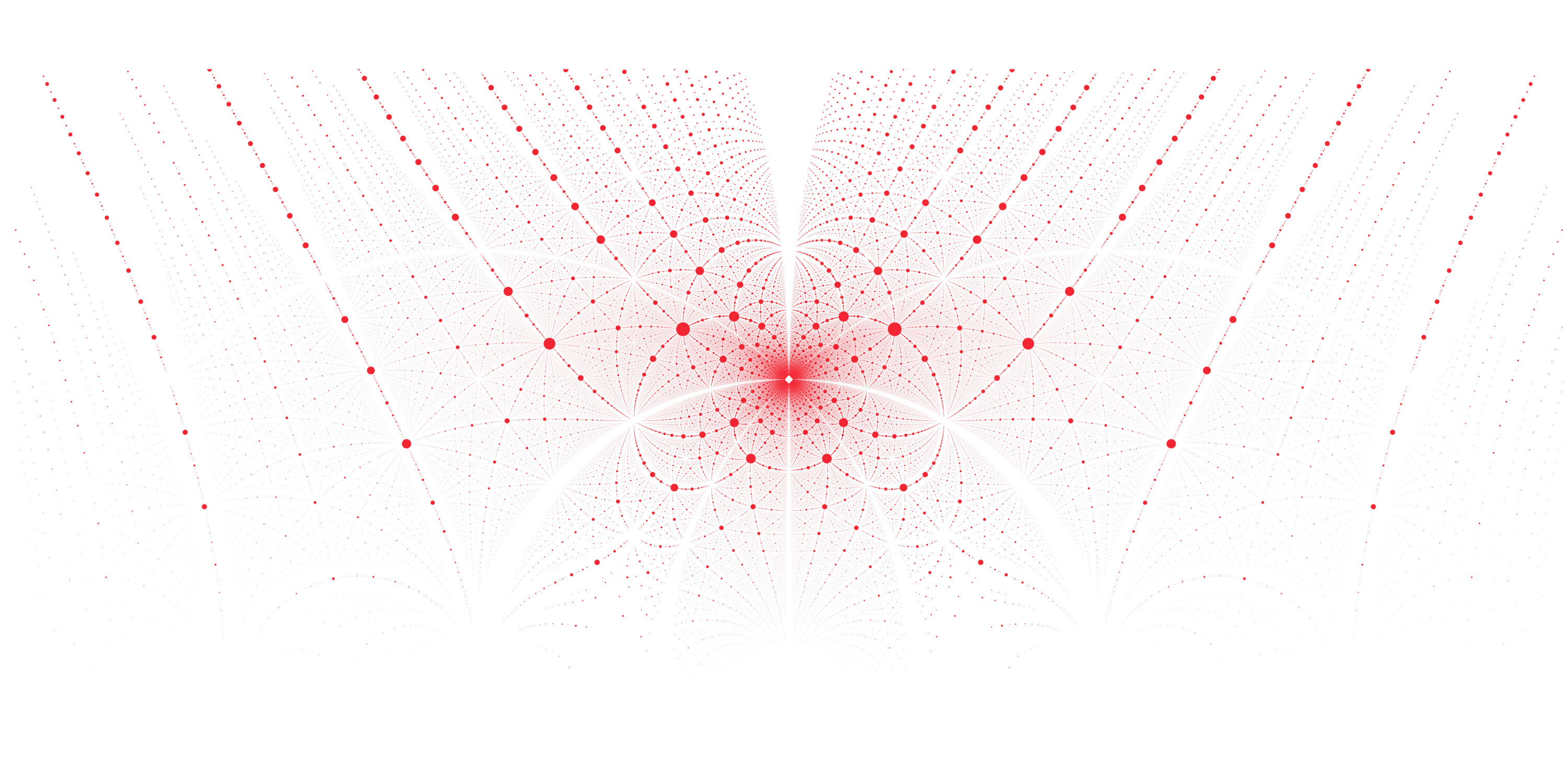}
	\caption{The root discriminant as used more generally in the paper.}
    \label{fig:DotSizesDisc}
\end{subfigure}

	\caption{Companion to Figure \ref{fig:DotSizesCoeff}, showing two sizings based on the polynomial and number theory of the roots.}
\label{fig:DotSizesRoots}
\end{figure}

\subsection{A celestial dance}

In all our images so far, there is a feeling that larger dots repel each other, like binary stars locked in a mutual orbit, unable to approach.  We also see remarkable spiraling trajectories of shrinking dots surrounding large ones, like the arms of small galaxies, as in Figure \ref{fig:Initial_Cubics_detail}.  We have been using the root discriminant as the way to size the dots:  this provides a rough measure of the ``complexity'' of an algebraic number.  The notion of needing increasingly complicated algebraic numbers to successively better and better approximate a target is a classic idea in number theory, and is explored in the field of \emph{Diophantine approximation}. 

The paradigm is that as the polynomial gets more complicated, the dots representing its roots are drawn smaller. 
There are many notions of complexity one might use to size the points. 
Ideas of ``complicated'' are generally tied to the coefficients (as in the discriminant); Figure \ref{fig:DotSizesCoeff} shows some classic metrics on the coefficients.  In Figure \ref{fig:DotSizesMahler} we show the Weil height (or Mahler measure), a more sophisticated way to show complexity. These are all discussed in more detail in Section \ref{sec:DiophantineApproximation}.  Compare with the discriminant in Figure \ref{fig:DotSizesDisc}.

From a geometric perspective, we've seen that this type of repulsion effect can arise from a lattice in projection, with the dots appearing smaller in some sense proportional to their distance. If we consider the lattice of points $(p,q)$ for $p,q \in \ZZ$, we can map to a line by taking $\frac{p}{q}$ and get the rational numbers, as in Figure \ref{fig:QP1} and its result, Figure \ref{fig:RationalStarscape}.  In this setup, the points are sized inversely to their denominators.  But as it turns out, the traditional measure of arithmetic complexity of a rational number is in fact, simply its denominator!  The geometric aligns with the arithmetic.  This is the most elementary example of the relationship between geometry and arithmetic that plays out throughout the paper. 

\begin{figure}[h!tbp]
\centering
\includegraphics[width=.8\textwidth]{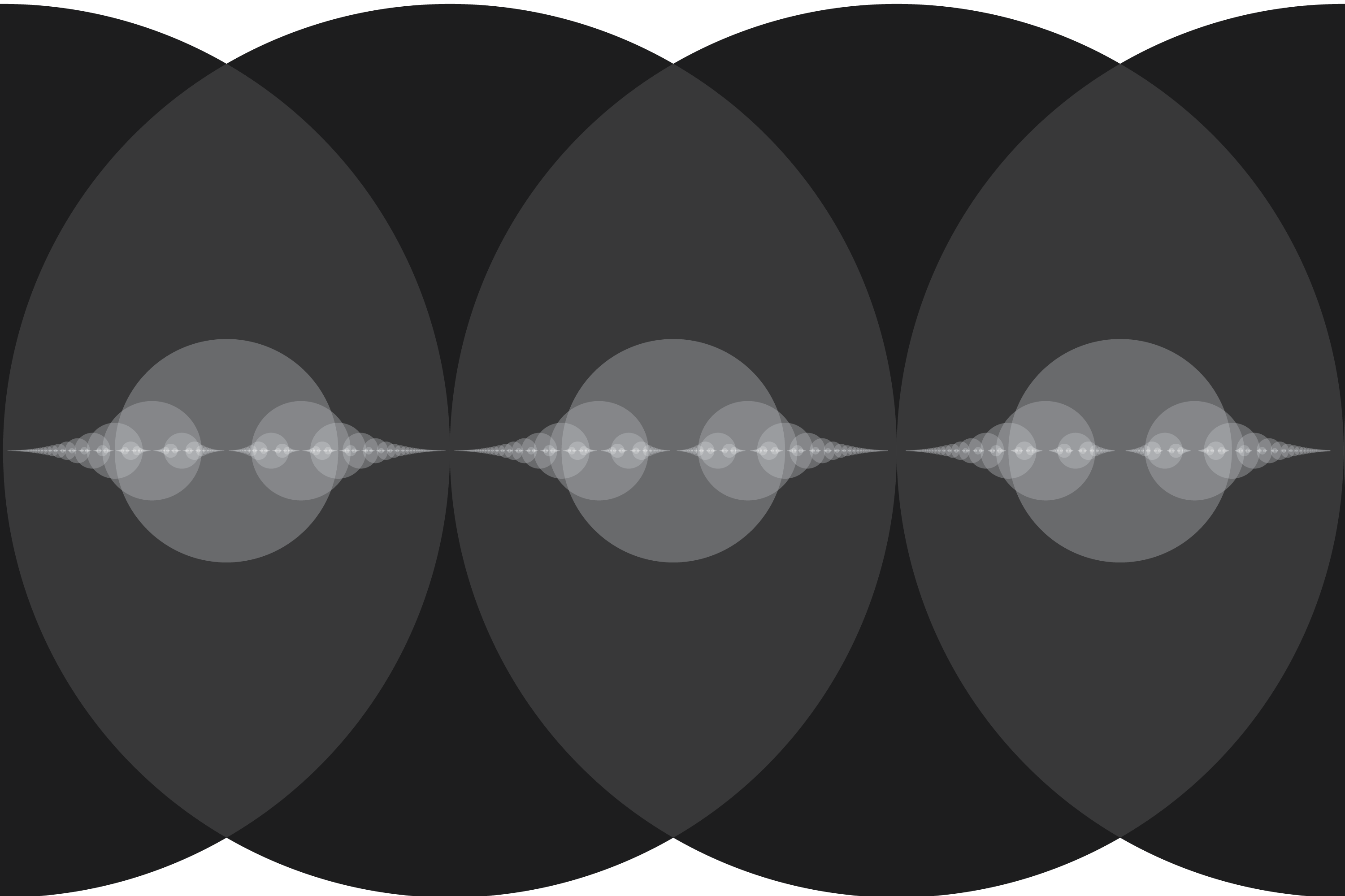}
	\caption{Disks centred at rational numbers $\frac{p}{q}$ between $0$ and $3$ with radius $\frac{1}{q^2}$, illustrating Dirichlet's theorem.}
\label{fig:Dirichlet}
\end{figure}

A natural arithmetic question is to ask how well a real number can be approximated by rationals. The first answer to this, of course, is that the rational numbers are dense, so you can get approximations as close as you wish.  On the other hand, one might consider $\frac{22}{7}$ to be a surprisingly good approximation for $\pi$, given the simplicity of the numerator and denominator.  After all, $\left|\pi - \frac{22}{7}\right| < 0.0013$.  To quantify this, it's convenient to imagine a cost (the complexity) of a rational number, with higher denominators considered more expensive. We can then ask which approximations are particularly good value for money. 

To illustrate such approximations, we could surround each rational number with a disk, where more expensive points have smaller disks; then we could call an approximation to $\alpha \in \RR$ \emph{good} if the corresponding disk covers $\alpha$.  This notion is, of course, very dependent on the disk sizing.  It turns out a sizing of $\frac{1}{q^2}$ is a sort of cusp in behaviour, as demonstrated by the following theorem\footnote{The theorem as stated here is most commonly known as Dirichlet's Approximation Theorem, but this theorem was already known to Legendre \cite{Legendre} as a result of the study of continued fractions.  The proof we give in Section \ref{ssec:ClassicDiophantineApproximation} is Dirichlet's and actually gives a stronger \emph{asymptotic approximation} result.} illustrated in Figure \ref{fig:Dirichlet}.

\begin{theorem}[Dirichlet's Approximation Theorem]
\label{thm:Dirichlet_intro}
  For any $\alpha \in \RR$, $\alpha$ is irrational if and only if there exist infinitely many distinct $p/q \in \QQ$ such that
  \[
    \left| \alpha - p/q \right| < 1/q^2.
  \]
\end{theorem}

\begin{figure}[h!tbp]
\centering
\includegraphics[width=.8\textwidth]{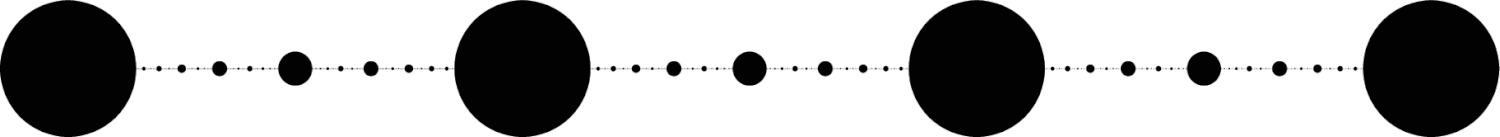}
\caption{Disks centred at rational numbers $\frac{p}{q}$ between $0$ and $3$ with radius $\frac{3}{20 q^2}$.}
\label{fig:RationalStarscape}
\end{figure}

This gives good approximations to irrational numbers (the $p/q$ that satisfy Theorem \ref{thm:Dirichlet_intro}) and also shows that the rational numbers cannot be easily approximated by other rationals, demonstrating the idea that the points are repelling each other. By scaling the points down, the repulsion rather than the notion of approximation becomes clearer, and we have an image close to a one dimensional version of the starscapes (Figure \ref{fig:RationalStarscape}).

These ideas and pictures lie at the heart of Diophantine approximation.  Linking geometry and number theory, they provide many of the key ideas that can be extended to the algebraic numbers.  It is therefore with geometry that we will start the next section, moving from a study of images to the mathematical ideas expressed in them.  In later sections, we will return to Diophantine approximation, motivated by the geometry to recast and reprove some variations on and extensions of standard results.

\section{Geometry}
\label{sec:Geometry}
The gallery images show some structures in the lattice of polynomials over the integers, viewed in two dimensions by drawing their roots in the complex plane.
Because polynomials are basic and important objects across mathematics, it should perhaps be no surprise (although it was one to the authors!) that their geometric story involves many familiar characters: from projective geometry and discrete groups, to hyperbolic space, the representation theory of $\mathrm{SL}_2(\mathbb{C})$ and symmetric powers of the sphere.
A recurring theme:

\begin{quote}
\emph{Patterns in the distribution of algebraic numbers are shadows of the geometry of their lattice of minimal polynomials.}
\end{quote}

Throughout the geometry section, we will formulate precise versions of the above statement in degrees two and three, where there is a beautiful connection of the geometry of polynomials to the geometry of the hyperbolic plane.  We attempt to increase the required prerequisites only gradually, to allow readers at many levels to chart their own paths through the material.
In particular, researchers may want to read the high level overview \emph{Section \ref{sec:BirdsEye}: A birds eye view} immediately below, and then move directly to any theorem of particular interest.
Students may wish instead to skip this fast-paced summary and begin reading at Section \ref{sec:ProjGeo}, which provides an introduction to projective space.

\subsubsection{A Bird's Eye View}
\label{sec:BirdsEye}
The main actor in this story is the roots map  $\RootMap$ sending the coefficients of a complex polynomial to its multi-set of roots.
This provides a two-way bridge between the space of polynomial solutions and their coefficients, relating the structures in these images directly to structures already present in their sets of minimal polynomials.
Making use of this bridge requires understanding what kind of information (topological, representation-theoretic, geometric) survives the trip.
We summarize three main themes here:

\begin{itemize}
  \item  The fundamental theorem of algebra implies that the roots map is a \emph{homeomorphism} between the spaces of roots and coefficients over $\CC$.
  Thus topological properties of collections of algebraic numbers are equivalent to topological properties of their corresponding sets of minimal polynomials.
  \item For $\FF\in\{\RR,\CC\}$, the roots map $\RootMap$ is equivariant with respect to natural $\PSL(2;\FF)$ actions on the spaces of roots and coefficients.
  This equips each of these with a notion of geometry preserved by $\RootMap$.
  Thus, geometric properties of the space of minimal polynomials determine geometric properties of the algebraic numbers.
  \item In small degree, this action has finitely many orbits, decomposing the space of polynomials into a union of homogeneous geometries for $\PSL(2;\FF)$.
\end{itemize}

By `geometric' throughout, we mean in the sense of homogeneous geometry, following Felix Klein.
In a vast generalisation of euclidean geometry, Klein proposed in his 1872 Erlagen Programm that a geometry is \emph{defined as a space $X$, together with the transitive action of a group $G$}.  This group action is  interpreted as the allowable, or `rigid' motions of this geometry, its transitivity implies the geometry is \emph{homogeneous} or behaves the same at every point\footnote{From this perspective the euclidean plane is $\R^2$ equipped with the group of rotations, translations and (glide) reflections.  Projective geometry, hyperbolic geometry, and de Sitter space are other common examples, which we will encounter throughout our journey.}.

In the sections that follow, we give a detailed analysis of the roots map in low degree, and introduce the necessary geometric objects as they arise.
 In particular, we use the geometry of homogeneous spaces for $\PSL(2;\RR)$ to 
 give an interpretation of the quadratic and cubic formulas in terms of familiar geometric spaces:

\begin{itemize}
    \item The space of real quadratics (those with real coefficients) with complex roots identifies with the hyperbolic plane.  Restricted to this subset, the roots map (the quadratic formula) is an isometry between the projective model (in coefficient space) and upper half plane model (in root space).
    \item The space of real cubics with complex roots is topologically a solid torus, and identifies with the unit tangent bundle to the hyperbolic plane.  The roots map is an isometry between the models of this geometry constructed from coefficients and roots, respectively.
    \item  Viewing the space of these cubics as the unit tangent bundle to $\HH^2$ gives a geometric factorization of the cubic formula.
Computing the complex and real roots amounts to a projection onto the base and fiber (with respect to a fixed trivialization) respectively.
    \end{itemize}

As the algebro-geometric study of polynomials of low degree spans centuries, these are almost certainly not new, but we do not know of a reference.
We prove them in this section (Theorems \ref{thm:QuadRoots}, \ref{cor:QuadFormula}, \ref{thm:cubicMain}, \ref{thm:cubicIsom} and \ref{thm:CubicFactor}) for the benefit of the reader.
These geometric interpretations provide both insight into the gallery images and new perspectives on results from number theory.  
We summarize some of these insights here.

\begin{itemize}
	\item The starscapes are naturally interpreted through hyperbolic geometry. Inversion in the unit circle and translation along a horocycle through $\infty$ preserve integer polynomials, explaining the evident $\SL(2;\ZZ)$ symmetry in figures such as Figure \ref{fig:Initial_quadratics} and Figure \ref{fig:Initial_cubics}.
\item The roots map is an isometry, so we can measure distances between quadratic numbers explicitly using the discriminant quadratic form on their minimal polynomials, which we will use in Section \ref{sec:Dio-quad}.
\item The identification of cubics having complex conjugate roots with the unit tangent bundle to $\HH^2$ suggests new ways of visualizing cubic numbers: as subsets of the solid torus (Figures \ref{fig:AllCubicsin3d}, \ref{fig:UTRoots}, \ref{fig:UTCoefs}) and as vector fields in $\CC$ (Figure \ref{fig:Arrows}).
\item Under this identification, two parameter families of cubics embed in $\CC$ via the roots map when everywhere transverse to the fibers of the unit tangent bundle.  This provides a condition on precisely when the projection onto the complex root is not a homeomorphism: compare and contrast Figures \ref{fig:AffineStarscapes_c} and \ref{fig:Fam7C}.
\end{itemize}

\subsubsection{Notation}
We briefly collect here some useful notation and conventions used throughout.
A \emph{complex polynomial of one variable} is a function $f\colon \CC\to\CC$ of the form $f(x)=a_nx^n+\cdots+a_1x+a_0$ for $a_0,\ldots a_n\in\CC$.
Throughout this section we focus instead on \emph{homogeneous polynomials}, which are polynomials in $x,y$ where each term has constant total degree, i.e.\ of the form $a_nx^n + a_{n-1}x^{n-1}y \cdots + a_1xy^{n-1} + a_0y^n$ for $a_0, \ldots, a_n \in \CC$. 
These generalize the single variable case (setting $y=1$ returns its single variable counterpart) with several advantages.
Chiefly, binary forms of degree $n$ naturally include all single variate polynomials of lower degree\footnote{For example the linear polynomial $2x+3$ can be thought of as the homogeneous linear polynomial $2x+3y$, or the homogeneous quadratic $2xy+3y^2$, or the homogeneous cubic $2xy^2+3y^3$, etc.}, providing spaces to study all algebraic numbers of bounded degree.

For a fixed degree $n$, we let $\Coefs_n$ be the space of all coefficients of all degree $n$ binary forms, and $\Roots_n$ be the space of all multi-sets of their roots (we drop the subscript when no confusion results).
For binary forms, it is convenient to allow roots to lie in the extended complex plane $\CP^1=\C\cup\{\infty\}$.
In this way, every polynomial of degree $n$ has exactly $n$ roots (with multiplicity)\footnote{Again taking the univariate polynomial $2x+3$ as an example, thought of as a linear homogeneous equation this has a single root.  But thought of as a quadratic, such linear equations have an additional root ``at infinity".}.
The roots map $\RootMap_n\colon\Coefs_n\to\Roots_n$ returns all roots of a polynomial as a function of its coefficients.
As the roots of a polynomial are unchanged by rescaling all coefficients by a constant, we abuse notation and also write $\RootMap_n$ for the roots map on scaling classes $\RootMap_n\colon\PP\Coefs_n\to\Roots_n$.
Similarly, we write $\Delta_n\colon\Coefs\to\CC$ for the discriminant, dropping $n$ when unambiguous. 

Often, we will identify a certain subset of polynomials with some geometric or topological space $X$.
Depending on our perspective, we may find it useful to think of these as either being built out of the polynomial's coefficients or its roots.
To keep these two conceptually distinct, we will often decorate $X$ accordingly, writing $X_\Coefs\subset \Coefs$ and $X_\Roots\subset\Roots$.
Some of the important spaces that arise in this way (and will be defined in the following sections) are the projective spaces $\RP^1,\CP^1$, their generalisations $\RP^n,\CP^n$, the symmetric powers of the sphere $\SP^n(\CP^1)$, and the hyperbolic plane $\HH^2$.

\subsection{Projective geometry and rational numbers}
\label{sec:ProjGeo}
An analysis of the rational numbers $\QQ$ provides a gentle introduction to the recurring themes of this section, and an excellent first introduction to projective geometry.
Secretly of course, this is just the story of linear polynomials over $\ZZ$ and their roots.
The simplicity of the roots map $\mathcal{R}(ax-b)=b/a$ allows us to suppress much of this formalism and focus on examples of our main idea:
\emph{complex patterns that reveal themselves when viewed as shadows of some higher dimensional structure.}

After introducing projective geometry, we provide two examples of this.
We give the geometric explanation for the otherwise strange mediant addition law $\tfrac{p}{q}\oplus\tfrac{r}{s}=\tfrac{p+r}{q+s}$ used in the construction of Farey sequences, and enumeration of the rationals via the Stern-Brocot tree.

\subsubsection{Projective Space}
Recall that every rational number $r$ is the quotient of two integers $r=p/q$, so it is natural to expect our analysis of $\QQ$ to be deeply intertwined with the lattice $\ZZ\times\ZZ$.  
However, to actually build $\QQ$ from $\ZZ\times\ZZ$, there are two minor complications to be resolved: (1) a pair $(p,q)\in\ZZ\times\ZZ$ does not uniquely determine a rational number, as $\tfrac{p}{q}=\tfrac{np}{nq}$, and (2) not all pairs $(p,q)\in\ZZ\times\ZZ$ determine rational numbers, as division by zero is undefined.
Both of these issues evaporate upon the realization that as \emph{ratios} of integers, we should not model rationals by \emph{points in} $\ZZ\times\ZZ$ but rather as \emph{slopes}.
This is more memorably stated as follows:

\begin{quote}
\emph{The rational numbers $\QQ$ are what the integer lattice $\ZZ\times\ZZ$ looks like if you stand at the origin $(0,0)$ and look around (Figure \ref{fig:QP1}).}
\end{quote}

\begin{figure}[h!tbp]
\centering
	\includegraphics[width=0.8\textwidth]{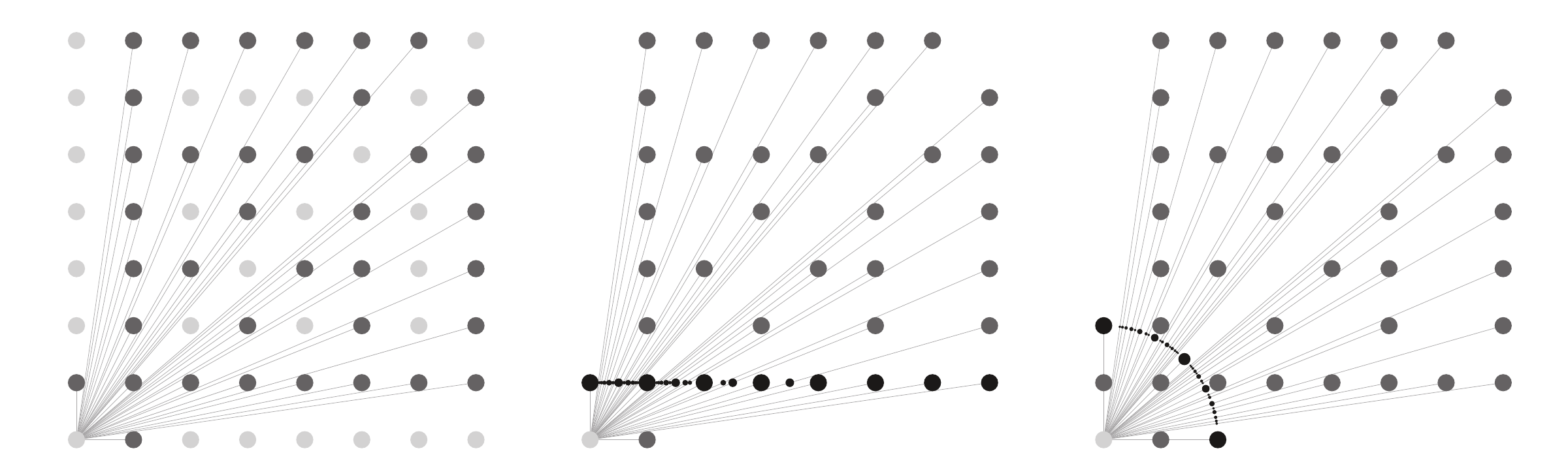}
	\caption{The rational numbers as the projection of integer lattice (left) onto an affine patch (middle) and visual sphere (right).  The darker point on the left represent the first point on each rational slope, as viewed from the origin (lower left dot).}
\label{fig:QP1}
\end{figure}

Formally, considering objects up to scaling is called \emph{projectivization}, and scaling classes of pairs of integers form the \emph{rational projective line} or $\QP^1$:

$$\QP^1=\bigslant{\{(p,q)\neq\vec{0}\mid p,q\in\ZZ\}}{(p,q)\sim (np,nq)}.$$
The points of $\QP^1$ are denoted $[p:q]$.
 Note that by definition these points are unchanged by a global scaling of their coordinates, resolving issue (1) above: for example $[p:q]=[2p:2q]=[\tfrac{p}{q}:1]$.
 Thus, one may recover the usual perspective of the rational numbers as points on a line from this, by imagining a screen (called an \emph{affine patch}) placed a unit distance in front of one's eyes.
 The rational number $[p:q]$ projects onto this screen to $[\tfrac{p}{q}:1]$ (Figure \ref{fig:QP1}, middle).

This definition automatically resolves point (2) as well, and gives a rigorous interpretation for $1/0$ as the projective point $[1:0]\in\QP^1$ associated to the horizontal axis in the plane.
This line does not intersect the affine patch $\{[x:1]\mid x\in\QQ\}$, so it is not a rational number but a \emph{point at infinity}, often denoted $[1:0]=\infty$.
The rational projective line fixes the asymmetry of $\QQ$'s relation to $\ZZ\times\ZZ$ by adding a single point:
$\QP^1=\QQ\cup\{\infty\}$.
There are many contexts in which scaling classes are the right objects to consider, and this construction generalizes far beyond the rational numbers.

\begin{definition}
Let $\FF$ be a field, and $n\in\NN$.  Then $n$-dimensional projective space over $\FF$, denoted $\FF\PP^n$, is given by scaling classes of the nonzero elements of $\FF^{n+1}$ up to elements of $\FF^\times$:
$$\FF\PP^n=\bigslant{\{\mathbf{x}=(x_1,\ldots, x_{n+1})\neq\vec{0}\mid x_i\in\FF\}}{\left(\mathbf{x}=a \mathbf{x},\;\; a\in\FF^\times\right)}.$$
An equivalence class in $\FF\PP^n$ is denoted $[\mathbf{x}]=[x_1:x_2:\cdots:x_{n+1}]$.
Given $X\subset\FF^{n+1} \backslash \vec{0}$, we write $\PP X=\{[x]\mid x\in X\}\subset\FF\PP^n$ for the scaling classes of all elements in $X$.
An \emph{affine patch} of $\FF\PP^n$ is a subset homeomorphic to $\FF^n$ equivalent to $\{[x_1:x_2:\cdots:x_n:1]\}$ under some change of coordinates.
\end{definition}

\begin{figure}[h!tbp]
\centering
	\includegraphics[width=0.75\textwidth]{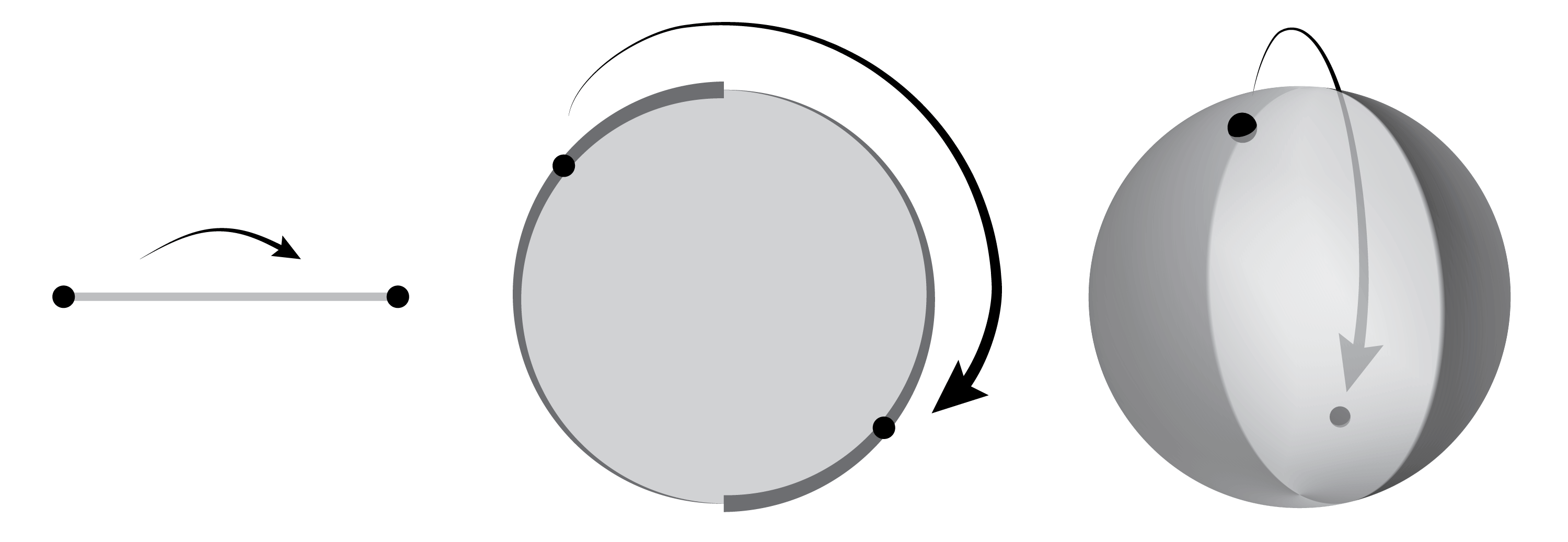}
	\caption{Real projective space can be constructed from a solid ball by identifying its boundary via the antipodal map.  Here we see this in dimensions 1, 2 and 3.}
\label{fig:RPn}
\end{figure}

\begin{remark}\label{rem:FP1}
For any field $\FF$, we may see that $\FF\PP^1\cong\FF\cup\{[1:0]\}=\FF\cup\{\infty\}$ by the same argument as for $\QQ$.
In fact, $\FF\PP^1$ is the \emph{one point compactification of $\FF$}. For example, $\RP^1$ is a circle and $\CP^1$ is a sphere.
\end{remark}

For us, a main advantage of this perspective\footnote{In addition to managing to rigorously make sense of $1/0$ of course!} is to aid in switching back and forth between thinking about objects in $\FF\PP^n$ and objects in the higher dimensional space $\FF^{n+1}$.
We give two simple, but surprisingly beautiful applications of this below.

\subsubsection{$\QP^1$ and the Stern-Brocot Tree}
\label{sec:Stern-Brocot}
A means of representing the lowest-terms representative of each rational number by an infinite binary tree was discovered independently by Stern and Brocot in the mid 1800s.
This tree is depicted in Figure \ref{fig:TreeMediant}.
The construction of the tree is inductive, where the elements of the $n^{th}$ row are produced from those in prior rows by taking \emph{mediants}.
The \emph{mediant of two fractions} $\tfrac{p}{q}$ and $\frac{r}{s}$ in lowest terms is the fraction $\frac{p+q}{r+s}$, which appears to be a rather algebraically unnatural construction\footnote{However, it is certainly a visually natural thing to try given our representation of rationals as quotients - indeed probably the most common mistake when first learning arithmetic is to add rationals by taking their mediant!}.
Taking mediants is a common step in such enumeration sequences (including the Farey sequence, a source of much beautiful mathematics on $\QP^1$), but is best understood not as an \emph{algebraic operation} on $\QQ$ but rather as a \emph{geometric operation} on $\QP^1$.
Indeed, before projectivizing, the points $(p,q)$ and $(r,s)$ are vectors in $\RR^2$, and here the mediant operation is simply \emph{vector addition}!
This is completely natural geometrically on $\QP^1$: given two points, lift to their representatives in $\ZZ^2$ which are closest to the origin.
These two vectors determine two sides of a parallelogram, whose main diagonal connects you directly to the mediant.

 \begin{figure}[h!tbp]
\centering
	\includegraphics[width=0.9\textwidth]{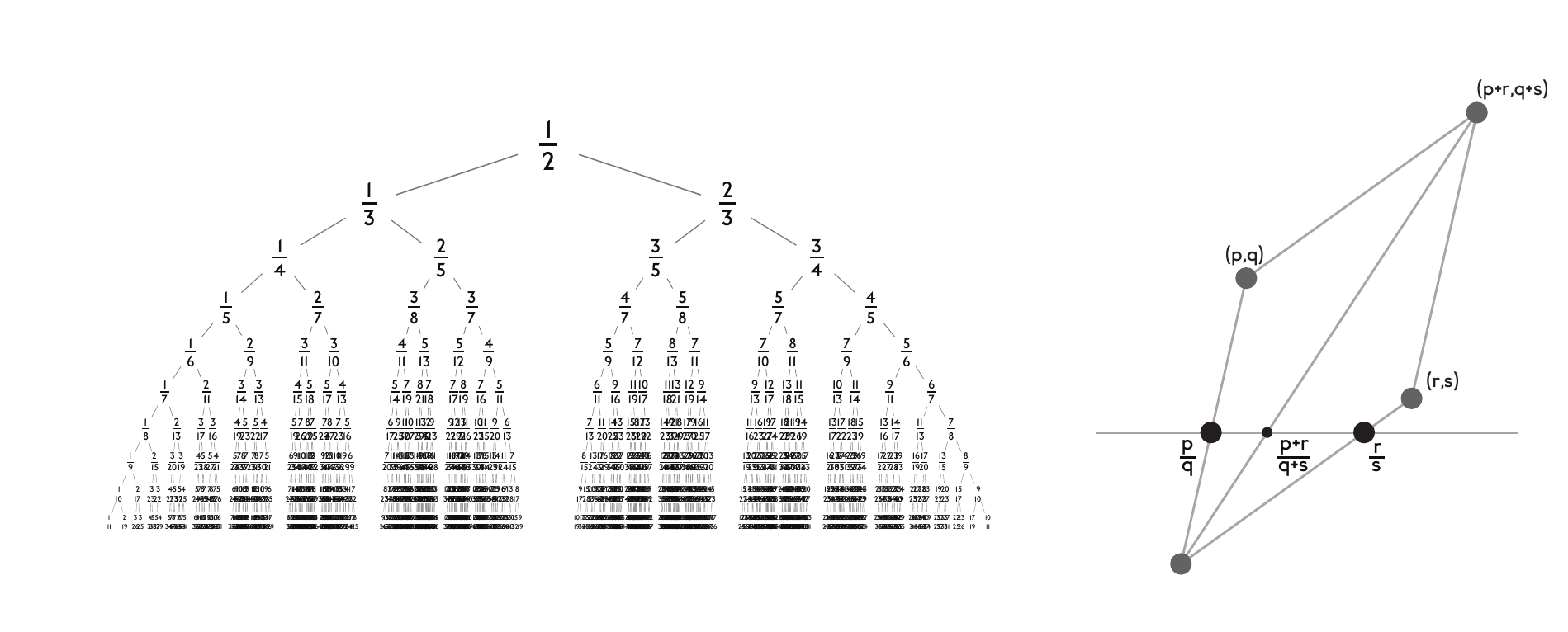}
	 \caption{The Stern-Brocot Tree enumerating the rationals, and a geometric description of the mediant of two fractions, essential to its construction.}
\label{fig:TreeMediant}
\end{figure}

\subsubsection{The Roots Map}
\label{sssec:RootsMap}
We briefly record the roots to linear equations $ax+b$ in our standard notation, for later use when discussing cubics.
Over $\CC$, the space of coefficients identifies naturally with $\CC^2=\{(a,b)\mid a,b\in\CC\}$, and so up to scaling we have $\PP\Coefs=\CP^1$.
The space of their roots is also $\CP^1$, thought of as the extended complex plane; as $ax+b$ (thought of as the homogeneous equation $ax+by$ on $\CP^1$), has the unique root $[-b:a]=[-b/a:1]$ when $a\neq 0$ and otherwise $\infty=[1:0]$.
Thus the roots map here is a linear isomorphism of $\CP^1$ with itself:
$$\RootMap_1\colon \CP^1_\Coefs\to\CP^1_\Roots\hspace{1cm}[a:b]\mapsto[-b:a]$$
This gives a natural identification between the spaces of roots and (projectivized) coefficients, allowing us to simplify the stories above.
In the following sections, $\RootMap$ will continue to be a homeomorphism $\PP\Coefs\to\Roots$, and topologically the mapping from a polynomial's coefficients to its solutions is still equivalent to projectivization.
However, $\RootMap$ is no longer such a simple isomorphism, and much of the work involved in accurately transferring information from coefficients to roots involves a careful analysis of $\RootMap$ and the symmetries it preserves.

\subsection{Hyperbolic geometry and quadratic numbers}
\label{ssec:QuadraticGeometry}
Some of the striking images in the gallery involve \emph{quadratic numbers}, or solutions in $\C$ to degree-two polynomials with integer coefficients $ax^2+bx+c$.
The suggestively ``hyperbolic'' nature of these (Figure \ref{fig:Initial_quadratics}) is no accident; and the goal of this section is to make this connection explicit.
In particular, we prove the following.

\begin{theorem}
Let $\HH^2_{\Coefs}$ be the projectivized set of coefficients of real quadratics with complex roots and $\HH^2_\Roots$ be the set of their root-sets, equipped with the following metrics:
\begin{itemize}
\item $\HH^2_\Coefs=\{[a:b:c]\mid b^2<4ac\}$ is given the projectively invariant metric it inherits as a convex subset of $\RP^2$,
\item $\HH^2_{\Roots}=\{\{x\pm iy\}\mid x,y\in\R,y>0\}$ is given the M\"{o}bius-transformation invariant metric from identification with the upper half plane $\subset\CP^1$.
\end{itemize}

Both of these spaces are isometric to the hyperbolic plane.
Furthermore, the quadratic formula, given as the map $[a:b:c]\mapsto \left\{\frac{-b\pm i\sqrt{4ac-b^2}}{2a}\right\}$ is an isometry between these two metrics.
\label{thm:QuadRoots}
\end{theorem}

This theorem is our main goal in Section \ref{ssec:QuadraticGeometry}; we will take some time to explain the metrics and invariances just referenced. 
Throughout we attempt to emphasize how how natural symmetry considerations might lead one to have conjectured this theorem in the first place.

\subsubsection{Complex Quadratics}
Irreducible quadratics over $\RR$ have pairs of complex conjugate roots, which makes the complex numbers integral to our discussion.
As such, we begin with a discussion of homogeneous quadratics over $\CC$.
This has several advantages\footnote{As $\CC$ is algebraically closed, the space of roots is easy to describe, defining the roots map does not require passing to a field extension, and the $\PSL(2;\CC)$ symmetry can be exhibited at its most natural level of generality.}, and after getting comfortable here we will restrict back to real coefficients to prove Theorem \ref{thm:QuadRoots}.
In this section we precisely define the space of coefficients, space of roots, and the map $\RootMap\colon\Coefs\to\Roots$ for complex homogeneous quadratics.
This lays the foundation for a bridge between properties of quadratic numbers and properties of their minimal polynomials, by the following observation.

\begin{observation}
\label{obs:RootsHomeo}
The roots map $\RootMap\colon\Coefs\to\Roots$ is a homeomorphism from the space of complex homogeneous quadratics to the multi-sets of their roots in $\CP^1$.
\end{observation}

The polynomial $f=ax^2+bxy+cy^2$ is determined by its coefficients $a,b,c\in\CC$, so $\Coefs\cong\CC^3$.
We denote the scaling class of a quadratic $f=ax^2+bxy+cy^2$ by $[f]$, with coefficients $[a:b:c]\in\PP\Coefs=\PP(\CC^3)=\CP^2$ in the complex projective plane.
The roots map $\RootMap$ is just the quadratic formula; taking the polynomial with coefficients $[a:b:c]$ with $a\neq 0$ to its solutions $[\tfrac{-b}{2a}\pm\tfrac{\sqrt{b^2-4ac}}{2a}:1]$, or $[-b\pm\sqrt{b^2-4ac}:2a]$ after clearing denominators\footnote{
It is quick to check that as $r\to\infty$ the coefficients of the polynomial $(x-ry)(ax+by)$ converge projectively to $[0:a:b]$ and the roots to $\{[1:0],[-b:a]\}=\{\infty,-b/a\}$.
Thus the quadratic formula extends continuously to linear equations interpreted as `quadratics with roots at infinity' as noted in Section \ref{sec:ProjGeo}.}.
The space of roots comes with no natural ordering.
Thus, $\Roots$ does not identify with the space of \emph{ordered pairs} $\CP^1\times\CP^1$ of points in the extended complex plane, but rather the set of \emph{unordered pairs in} $\CP^1$.
 This space is called the \emph{second symmetric power of} $\CP^1$, and we write $\Roots=\SP^2(\CP^1)$.
 An explicit construction of this space results from the quotient of $\CP^1\times\CP^1$ by the action of $\ZZ/2\ZZ$ swapping coordinates.

 \begin{remark}
 \label{rem:MobiusBand}
It may be helpful to pause here to gain some intuition from the lowest-dimensional example of a nontrivial symmetric power.
The second symmetric power of a circle $\SP^2(\SS^1)$ is the quotient of a torus $\SS^1\times\SS^1$ by the involution $(x,y)\mapsto(y,x)$
 This fixes the diagonal $(x,x)$, and the quotient is a M\"{o}bius strip with these points as the boundary curve (Figure \ref{fig:Mob}).

 \end{remark}
 
 \begin{figure}[h!tbp]
\centering
\includegraphics[width=0.95\textwidth]{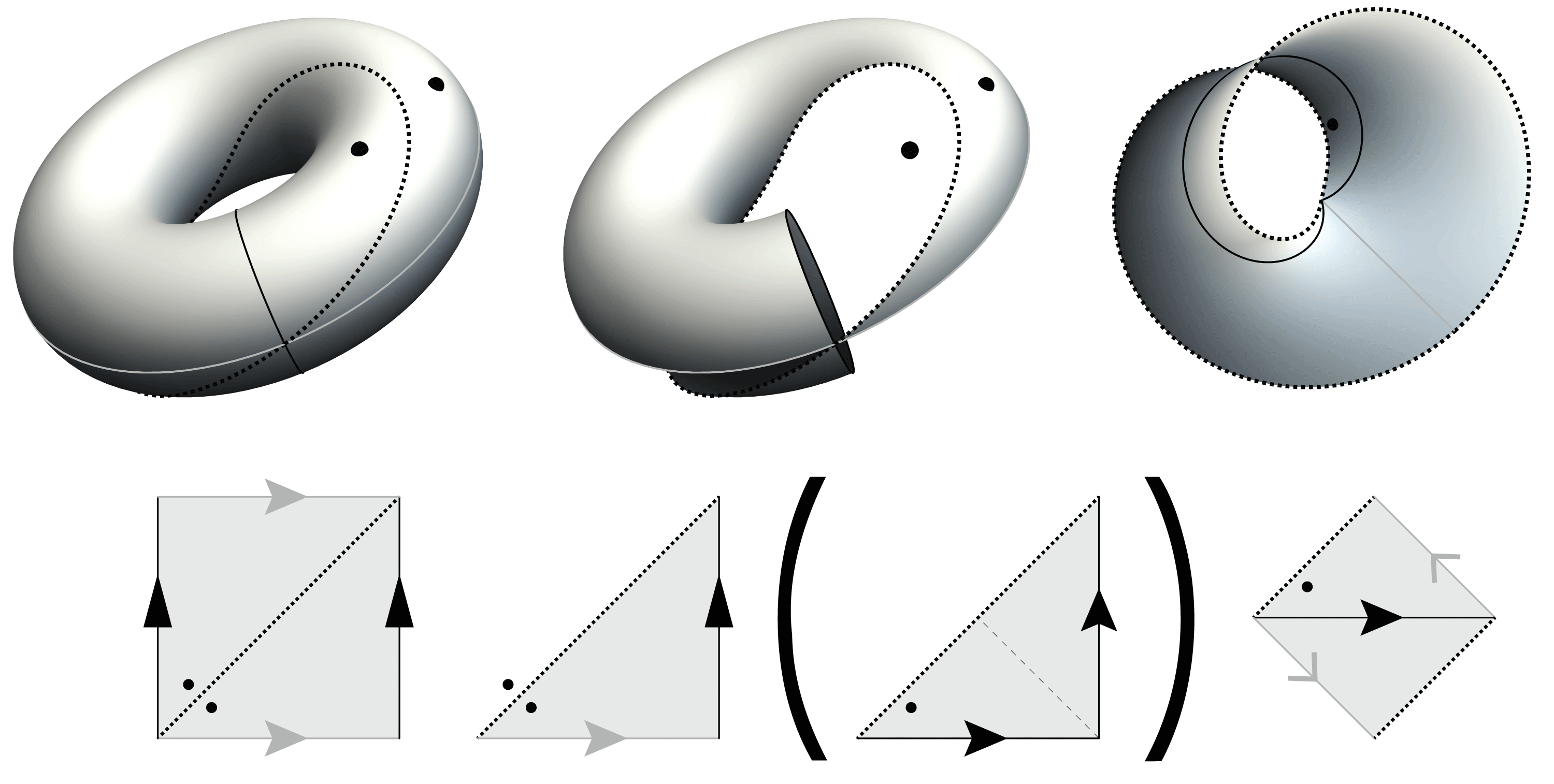}
\caption{The space $\SP^2(\RP^1)$ of unordered pairs of points on the circle $\RP^1 \cong \SS^1$ is homeomorphic to a M\"{o}bius band.}
\label{fig:Mob}
\end{figure}

The space of quadratic polynomials parameterized by their roots is $\SP^2(\CP^1)$, a space built from the sphere in much the same way as the M\"{o}bius strip $\SP^2(\mathbb{S}^1)$ was built from the circle in Remark \ref{rem:MobiusBand}.  See the text following Observation \ref{obs:RootsHomeo}, for the discussion of this.
Not only do every pair of points in $\CP^1$ determine a scaling class of quadratics, but as a consequence of the fundamental theorem of algebra, every scaling class has two roots (with multiplicity) determining it.
Together with the continuity of $\RootMap$, this implies that the roots map is a homeomorphism\footnote{This is an important argument in its own right, for it shows topologically the symmetric power $\SP^2(\SS^2)$ is the complex projective plane in disguise.  In particular, it is a closed manifold.  Compare this with Remark \ref{rem:MobiusBand} where $\SP^2(\SS^1)$ is a manifold with boundary.  This generalizes to arbitrary degree $n$, and the roots map provides a homeomorphism $\RootMap_n\colon\CP^n\stackrel{\sim}{\to}\SP^n(\CP^1)$.
} from $\Coefs=\CP^2$ to $\Roots=\SP^2(\CP^1)$.
Thus topological properties of (projectivized) collections of quadratic polynomials completely determine the topology of their collection of roots.
To strengthen this connection, we next focus on natural symmetries of the space of polynomials which are preserved by the roots map. We will see these symmetries are actually \emph{isometries} of natural choices of metrics on both the space of coefficients and roots, which will be instrumental in our proof of the main theorem, \ref{thm:QuadRoots}.

\subsubsection{$\PSL(2;\CC)$ Symmetry}
\label{sec:PSL_Sym}
From both the roots and coefficients perspectives, quadratics are intimately tied to complex projective spaces.
But these spaces are only half the story. To work geometrically, we must also describe their groups of allowable motions.
As projective space $\CP^n$ is just scaling classes of vectors in $\CC^{n+1}$, the most natural group of symmetries is the linear group $\GL(n+1;\CC)$ acting on these scaling classes via $A.[v]=[Av]$ for $A\in \GL(n;\CC)$, $[v]\in\CP^n$.
For simplicity, we may consider these matrices up to constant multiples as well, and take the \emph{projective special linear group} $\PSL(n+1;\CC)=\{[A]\mid \det{A}=1\}$ as the symmetries of $\CP^n$.
For points in $\CP^1$ thought of as $\C\cup\{\infty\}$, these are known as \emph{M\"{o}bius transformations}: the transformation corresponding to $\left(\begin{smallmatrix}p&q\\r&s\end{smallmatrix}\right)$ acts on $z\in\CC$ as
$$
\begin{pmatrix}
p&q\\r&s
\end{pmatrix}.z
=\left[\begin{matrix}
p&q\\r&s
\end{matrix}\right].[z:1]
=\left[\begin{pmatrix}
p&q\\r&s
\end{pmatrix}
\begin{pmatrix}
z\\1
\end{pmatrix}
\right]
=\left[\begin{pmatrix}
pz+q\\rz+s
\end{pmatrix}\right]
=\left[\frac{pz+q}{rz+s}:1\right]
=\frac{pz+q}{rz+s}.$$

Thus, $\PP\Coefs\cong\CP^2$ naturally has the symmetries of $\PSL(3;\CC)$.
 As $\Roots=(\CP^1\times\CP^1)/\ZZ_2$ is built from the complex projective line, the natural symmetries of the space of roots are $\PSL(2;\CC)$, inherited from $\CP^1$ via $[A].\{[z],[w]\}=\{[Az],[Aw]\}$.
Initially these two notions of geometry are completely independent, constructed from our models of $\Roots$ and $\Coefs$ respectively.
However, given that $\RootMap\colon\Coefs\to\Roots$ is a homeomorphism, we may use it to compare the symmetries on one side to the other.
The main observation of this section is that these two actions are actually \emph{compatible}\footnote{This picture may already be familiar from representation theory: the action of $\PSL(V)$ on a 2-dimensional complex vector space $V$ induces actions on symmetric powers of $V$, which are the irreducible representations of $\PSL(V)$.  Projectivizing this picture under an identification $V=\CC^2$ yields the result we exposit here.} with each other.

\begin{observation}
	\label{obs:proj}
Moving the space of roots by a (projective) linear transformation is represented on the space of coefficients by a (projective) linear transformation as well.
That is, conjugation of this action by $\RootMap^{-1}$ embeds $\PSL(2;\CC)$ into $\PSL(3;\CC)$.
\end{observation}

One may then directly convert any knowledge about this representation into geometric statements binding the spaces of roots and coefficients yet closer together.
This $\PSL(2;\C)$ action on the space of coefficients has a natural description: conjugation by the roots map simply declares that 
$[A]\in\PSL(2;\CC)$ send the polynomial with roots $\{[z],[w]\}$ to the polynomial with roots $\{[Az],[Aw]\}$.
Writing this out explicitly, let $A=\left(\begin{smallmatrix}p&q\\r&s\end{smallmatrix}\right)\in\PSL(2;\CC)$ and $f$ be a homogeneous quadratic with coefficients $[f]=[a:b:c]$.
Then $[A].[f]=[f\circ A^{-1}]$, which gives the following.

$$
f\left(
\left[\begin{matrix}
s&-q\\-r&p
\end{matrix}\right]
\left[\begin{matrix}
x\\y
\end{matrix}\right]
\right)
=f\left(
\left[\begin{matrix}
sx-qy\\py-rx
\end{matrix}\right]\right)
=
a(sx-qy)^2+b(sx-qy)(py-rx)+c(py-rx)^2
$$

Expanding and collecting like terms in $x^iy^j$ shows that the new coefficient vector is \emph{a linear transformation of $[a:b:c]$ involving $p,q,r,s$}, which confirms Observation \ref{obs:proj}. 
Specifically, the coefficients of $[A].[f]$ satisfy
\begin{equation}
\label{eqn:Rep}
 \left[
\begin{matrix}p&q\\r&s\end{matrix}
\right].
[a:b:c]
=\left[
\begin{pmatrix}
s^2&-rs&r^2\\
-2qs&qr+ps&-2pr\\
q^2&-pq&p^2
\end{pmatrix}
\begin{pmatrix}a\\b\\c\end{pmatrix}
\right].
\end{equation}

This $3\times 3$ matrix is the \emph{representation} of $\left(\begin{smallmatrix}p&q\\r&s\end{smallmatrix}\right)$ acting on $\Coefs=\CP^2$.
Let $\rho\colon\PSL(2;\CC)\to\PSL(3;\CC)$ denote this representation.
 By construction, for each $A\in\PSL(2;\CC)$ we have $f(A.[x\nolinebreak:\nolinebreak y])=(\rho(A).f)([x:y])$.
This implies the roots map is $\rho$-\emph{equivariant}, satisfying the following important identity:
\begin{equation}
\label{eqn:Rho_Equivariant}
   \RootMap\left(\rho(A).f\right)=A.\RootMap(f).
\end{equation}

In principle, this formula for the action lets us completely compute anything we desire about the relationship between the spaces of roots and coefficients (including the quadratic formula itself, Corollary \ref{cor:QuadFormula}).
This $\PSL(2;\CC)$ action divides the space of quadratics into two components: those with a double root, and (the generic case) those with distinct roots. 
Each of these components forms a singe $\PSL(2;\CC)$ orbit, which we may see as follows.
For quadratics with a double root, note for any $a=[a_1:a_2]\in\CP^1$ the symmetry\footnote{indeed there are many choices for $A$: if $a\in\CC$ then $A=\left(\begin{smallmatrix}1&a\\0&1\end{smallmatrix}\right)$ also works.} $A=\left(\begin{smallmatrix}a_2&-a_1\\a_1&a_2\end{smallmatrix}\right)\in\PSL(2;\CC)$ takes $f=x^2$ to the quadratic $(a_2x-a_1y)^2$ with double root $a$. More abstractly, one may deduce this for the component containing quadratics with distinct roots from the fact that $\PSL(2,\CC)$ acts freely and transitively on distinct triples in $\CP^1$ (see \cite{richter2011perspectives} for this and other useful facts in projective geometry).  That is, given any quadratic $f$ with roots $a,b\in\CP^1$, choosing any arbitrary $c\not\in\{[0:1],[1:1],a,b\}$ there is a unique $A\in\PSL(2;\CC)$ taking $(a,b,c)$ to $([0:1],[1:1])$, and thus taking $f$ to the fixed quadratic $x(x-y)$.
Because the action of $\PSL(2;\CC)$ is transitive on each of these components, they can be interpreted as homogeneous geometries in the sense of Klein.
This allows us to use geometric properties to understand the behavior of the roots map.

\subsubsection{Real Quadratics} 
\label{ssec:RealQuad} We now turn to quadratics with real coefficients.
The spaces of real quadratics naturally inherit their topology and geometry as subsets of the corresponding spaces over $\CC$.
This reduces the space of coefficients from $\CC^3$ to $\RR^3$, so we work with the projective plane $\PP\Coefs=\RP^2$, and its symmetries $\PSL(3;\RR)$.
The space of roots is more complicated to describe due to the fact that real quadratics may have complex roots.
However, as the roots map is a homeomorphism over $\CC$, we can immediately determine its topology: $\Roots=\RootMap(\PP\Coefs)=\RootMap(\RP^2)\cong\RP^2$.
We will denote these two models of projective space as  $\RP^2_\Coefs$, $\RP^2_\Roots$ in what follows.

$$
\RP^2_\Coefs =\left\{[a:b:c]\mid (a,b,c)\in\RR^3\backslash \mathbf{0}\right\}\subset\CP^3
$$
$$
\RP^2_\Roots =\left\{\{x,y\}\mid x,y\in\RP^1\right\}\bigcup\left\{\{x\pm i y\}\mid x,y\in\RR,y>0\right\}\subset\SP^2(\CP^1)
$$

The natural symmetry group on the roots restricts to  $\PSL(2;\RR)$, and equation \eqref{eqn:Rep} confirms that the representation faithfully translates this to a subgroup of $\PSL(3;\RR)$ on the coefficients.
Thus $\RootMap\colon\RP^2_\Coefs\to\RP^2_\Roots$ remains equivariant with respect to the $\PSL(2;\RR)$ symmetries on each side.
Recall that over $\CC$, this action divided the space of quadratics into two orbits.
Any two polynomials with distinct roots are related to one another by a symmetry transformation, as are any two with a double root; but $\PSL(2;\CC)$ cannot convert one type into the other.
Over $\RR$ the story gets more interesting, as the generic case (polynomials with distinct roots) splits.

\begin{observation}
The $\PSL(2;\RR)$ action divides the space of real quadratics into three orbits\footnote{On $\RP^2_\Coefs$, the orbit of quadratics with double root is exactly the discriminant locus (the set of polynomials whose discriminant is equal to zero); its complement is the union of the other two orbits.}: (1) quadratics with a double real root, (2) quadratics with distinct real roots, and (3) quadratics with complex conjugate roots. 
\end{observation}

The orbit (1) is homeomorphic to a circle\footnote{In fact, from here, elementary topology completes the story as all separating circles divide $\RP^2$ into a M\"obius strip and a disk, giving the topological type of (2) and (3) respectively.}
as any point $[p:q]\in\RP^1$ determines a quadratic $f=(px-qy)^2$ with double root at $[p:q]$.
The orbit (2) of quadratics with distinct real roots identifies with the space of unordered pairs of points of $\RP^1$, which we saw in Remark \ref{rem:MobiusBand} to be the M\"obius band.
We are mainly interested in orbit (3), which consists of unordered pairs of \emph{complex conjugates} in $\CC\backslash\RR$.
Each such pair contains a unique point $x+iy$ with $y>0$, so we identify this with the upper half plane in $\CC$.
As the $\PSL(2;\RR)$ action is transitive on each orbit, all three of these pieces inherit the structure of a homogeneous geometry, listed below.

\begin{figure}[h!tbp]
\centering
\begin{subfigure}[b]{0.22\textwidth} 
    \centering
	\includegraphics[width=\textwidth,page=1]{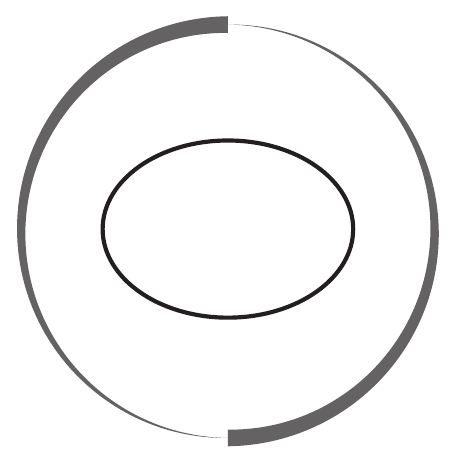}
	\caption{Double Root\\~}
    \label{fig:QuadDoubleRoot}
\end{subfigure}
\begin{subfigure}[b]{0.22\textwidth} 
    \centering
	\includegraphics[width=\textwidth,page=3]{Figures/QuadraticDivisions.pdf}
		\caption{2 Real Roots\\~}
    \label{fig:QuadPosDisc}
\end{subfigure}
\begin{subfigure}[b]{0.22\textwidth} 
    \centering
	\includegraphics[width=\textwidth,page=2]{Figures/QuadraticDivisions.pdf}
	\caption{2 Complex Conjugate Roots}
    \label{fig:QuadNegDisc}
\end{subfigure}
\caption{The four $\PSL(2;\RR)$? orbits on the $\RP^2$ of real quadratics, as subsets of $\RP^2$.  These are the three homogeneous spaces of Observation \ref{obs:Quadratic_HomogeneousSpace}
As in Figure \ref{fig:RPn}, the topology of each component is recovered by identifying points on the boundary circle via the antipodal map.
Compare to the analogous decomposition of $\RP^3$ for cubics in Figure \ref{fig:CubicRp3}.
}
\label{fig:QuadraticRp2}
\end{figure}

\begin{observation}
\label{obs:Quadratic_HomogeneousSpace}
The $\PSL(2;\RR)$ action on the projective plane of real quadratics divides it into three disjoint homogeneous geometries.
\begin{enumerate}
    \item The geometry of quadratics with a double root is the familiar geometry of the real projective line.
    \item The geometry of quadratics with distinct real roots is \emph{de Sitter space}, a two dimensional geometry relevant to relativistic physics\footnote{De Sitter geometry is a particular geometry not of space, but rather of \emph{spacetime}.  In this particular case, De Sitter 2-space describes a world with one space and one time dimension of positive curvature.  Geometrically, this is just the hyperboloid of 1-sheet $x^2+y^2-z^2=1$ in $\RR^3$ equipped with the action of $\SO(2,1)$.}.
    \item The geometry of quadratics with complex conjugate roots is  a hyperbolic plane (the disk in Figure     \ref{fig:QuadNegDisc} equipped with the projective action of $\PSL(2;\RR)$.
\end{enumerate}
\end{observation}

\subsubsection{Hyperbolic Geometry and the Roots Map}
\label{sec:QuadHypGeo}

It is the real quadratics with complex roots which are responsible for some of the beautiful images in the gallery such as Figure \ref{fig:Initial_quadratics}, so we study the hyperbolic geometry which underlies them\footnote{Hyperbolic geometry, commonly denoted $\HH^2$, is the unique two dimensional geometry with constant negative curvature, and was the first non euclidean geometry discovered.
Negative curvature implies that $\HH^2$ violates Euclid's fifth postulate with an infinitude of parallel lines to a given line through any point not on it.
For an introductory treatment of the hyperbolic plane, see \cite{anderson1999hyperbolic}.}.
An abstract understanding of hyperbolic space is insufficient for our goals, which rely on an explicit understanding of the geometry of the roots map $\RootMap$.  
Thus we need to consider both the model $\HH^2_\Coefs$ of hyperbolic geometry given by the coefficients of these polynomials, and model $\HH^2_\Roots$ formed by their roots.  For an excellent exposition of these models and more, see \cite{HyperbolicExposition}.

We begin with $\HH^2_\Roots$.  Identifying this space with the upper half plane $\{x+iy\mid y>0\}\subset\CC$, the Riemannian metric\footnote{A Riemannian metric is a choice of inner product for each tangent space, which allows one to measure the length of vectors, and hence the arc length of curves.} for hyperbolic geometry is $ds^2=(dx^2+dy^2)/y^2$, which follows directly from the $\PSL(2;\R)$ action, translating the standard metric $dx^2+dy^2$ at $i$ around by the group action.

The length of paths in the hyperbolic plane is computed via integrating this infinitesimal arc length: given a curve $\gamma(t)=(x(t),y(t))$ defined for $t\in[a,b]$, its length is $\operatorname{Length}(\gamma)=\int_a^b\frac{x^\prime(t)^2+y^\prime(t)^2}{y(t)^2}dt$.
The geodesics of hyperbolic geometry are given by arcs of semicircles whose centers lie on the boundary $y=0$, together with vertical lines (circles of infinite radius).
These geodesics determine all length minimizing segments in the hyperbolic plane: if $p,q$ are two points in hyperbolic space and $\gamma$ is a geodesic passing through $p$ and $q$, then the segment of $\gamma$ connecting them achieves the minimum distance among all curves joining $p$ to $q$.
This allows us to compute explicitly the distance function on hyperbolic space: if $x_1+iy_1$ and $x_2+iy_2$ in $\HH^2_\Roots$, the distance between them is given by

\begin{equation}
\label{eqn:RootsMetric}
    d_\Roots(x_1+iy_i,x_2+iy_2)=
    \operatorname{acosh}\left(1+\frac{(x_2-x_1)^2+(y_2-y_1)^2}{2y_1y_2}\right)
\end{equation}

Hyperbolic geometry has a well-defined notion of an \emph{ideal boundary}, consisting of points at infinity.
For the upper half plane model, we may describe these points as the idealized endpoints of geodesics: if $\gamma$ is any hyperbolic geodesic, the limits $\lim_{t\to\pm\infty}\gamma(t)$ lie on the ideal boundary.
Concretely, this consists of all points on the real line $y=0$ (the endpoints of all semicircle geodesics, and one endpoint of each vertical geodesic) together with a single additional point traditionally labeled $\infty$ (denoting the idealized endpoint of all vertical geodesics not lying on the real line).  See \cite{anderson1999hyperbolic}, Chapter 1 for more details on the upper half plane model and its ideal boundary.

Next, we turn to $\HH^2_\Coefs$.
Here the basic geometry is likely familiar from the study of conic sections, often studied in high school mathematics. 
Generic real quadratics have either two real roots or a pair of complex conjugate roots, as determined by the discriminant $\Delta(ax^2+bx+c)=b^2-4ac$ being positive or negative, respectively.
In fact, this and more can be explicitly recovered from studying the $\PSL(2;\RR)$ action given by the representation $\rho$ on $\RP^2_\Coefs$.
We focus here on the polynomials with complex roots; similar reasoning applies in the other case.
As $\PSL(2;\RR)$ acts transitively on $\HH^2_\Roots$, we may recover the entire space as the orbit of any point.
Using the defining property $\RootMap(\HH^2_\Coefs)=\HH^2_\Roots$ and equation \eqref{eqn:Rho_Equivariant}, we see that
\begin{equation}
\HH^2_\Roots=\PSL(2;\RR).\{\pm i\}\implies \HH^2_\Coefs=\rho\left(\PSL(2;\RR)\right).[1:0:1],
\end{equation}
as $[1:0:1]$ are the coefficients of $x^2+1$, with roots $\{\pm i\}$.
Computing this orbit\footnote{The entire action of $\PSL(2;\RR)$ on $\RR^3$ via the representation $\rho$ preserves the quadratic form $\Delta$, and the inner product $\langle (a_1,b_1,c_1),(a_2,b_2,c_2)\rangle=b_1b_2-2a_1c_2-2a_2c_1$ for which $\Delta(v)=\langle v,v\rangle$.
Thus the symmetries of $\HH^2_\Coefs$ are contained in the \emph{special orthogonal group} of this form $\SO(\Delta)$,  the orthogonal group of a quadratic form $q$ is the group of all matrices with determinant 1 whose action leaves $q$ invariant: $\SO(q)=\{A\in\SL(n;\RR)\mid q(v,w)=q(Av,Aw)\}$.  The indefinite orthogonal group $\SO(\Delta)$ has two components, determined by whether or not a symmetry preserves or swaps the two sheets of the hyperboloid.  The representation $\rho$ is actually an isomorphism onto the connected component of the identity: $\PSL(2;\RR)\cong \SO_0(\Delta)$.},
    the hyperbolic space $\HH^2_\Coefs$ is the projectivization of the negative cone\footnote{A similar description can be given for the two other geometries of quadratic polynomials.
    The $\RP^1$ of polynomials with double roots is the zero set of $\Delta$, or the projectivization of the light cone $b^2=4ac$.
   The space of polynomials with distinct real roots corresponds to points on which $\Delta$ is positive, which projectively forms a M\"obius band.  This space also has a natural notion of geometry, coming not from the hyperbolic plane but from relativity (it is called $1+1$ dimensional de Sitter space, but is beyond the scope of this paper.)
} of the discriminant $\Delta$; that is $\HH^2_\Coefs=\{[a:b:c]\mid 4ac>b^2\}$.
Topologically $\HH^2_\Coefs$ is a disk, but depending on the affine patch of $\RP^2$ that we choose it may take different forms.
Indeed, the patch $[1:b:c]$, corresponding to taking monic representatives, represents $\HH^2_\Coefs$ as the interior of a paraboloid, with one point on its ideal boundary at infinity: $\HH^2_\Coefs=\{[1:b:c]\mid 4c>b^2\}$.

\begin{figure}[h!tbp]
\centering
	\includegraphics[width=0.5\textwidth]{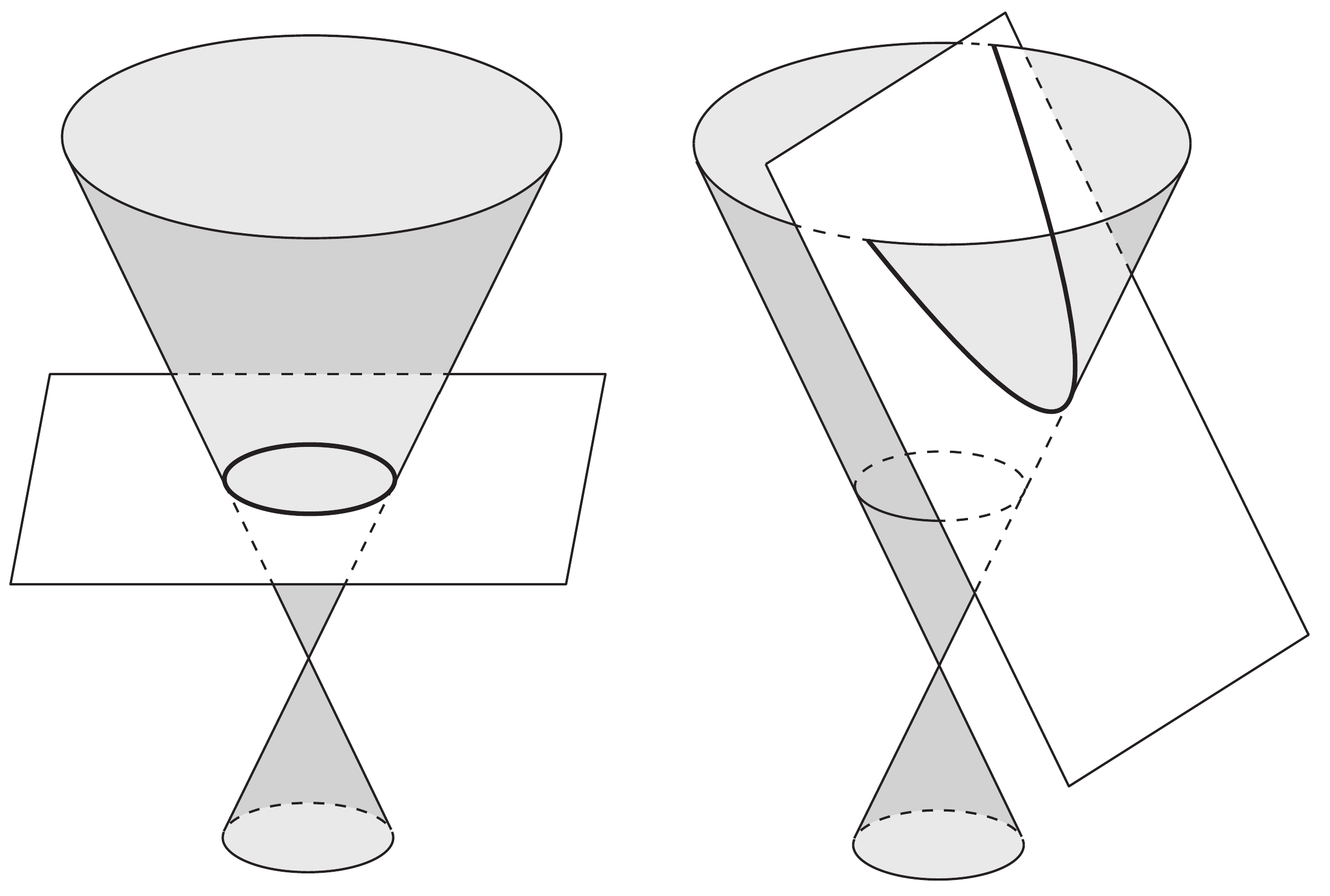}
	\caption{The cone determined by the discriminant in $\mathbb{R}^3$, together with two affine patches giving the usual Klein model (left) and a parabola model (right) of $\mathbb{H}^2$.
	}
\label{fig:Initial-reallyanotherone}
\end{figure}

As $\HH^2_\Coefs$ is a properly convex set in $\RP^2$, it can be endowed with a natural metric invariant under projective transformations\footnote{This is called the Hilbert metric.  Such a metric may be defined for any convex subset $\Omega\subset\RP^2$ not containing any entire projective line, and realizes a model of hyperbolic geometry precisely when $\Omega$ is bounded by a nonsingular conic section.}.  See \cite{HyperbolicExposition} for a more detailed analysis of this projective, or \emph{Klein Model} of the hyperbolic plane.

As the discriminant carries all the geometry associated with the $\PSL(2;\RR)$ symmetries, there is a nice description of this metric in terms of $\Delta$.
This is most apparent for the Riemannian metric after pulling back to the hyperboloid in $\RR^3$: if $v\in\RR^3$ is any tangent vector to the hyperboloid, than its infinitesimal arc length is simply
$$ds^2=\Delta(v).$$

The geodesics in this metric are straight line segments in any affine patch containing $\HH^2_\Coefs$, and the distance between two points $f_1=[a_1:b_1:c_1]$ and $f_2=[a_2:b_2:c_2]$ is given by
\begin{equation}
\label{eqn:CoefsMetric}
d_\Coefs(f_1,f_2)=\acosh{\left(\frac{-\langle f_1,f_2\rangle}{\sqrt{\Delta(f_1)\Delta(f_2)}}\right)}
=\acosh{\left(
\frac{2a_1c_2+2a_2c_1-b_1b_2}
{
\sqrt{(4a_1c_1-b_1^2)(4a_2c_2-b_2^2)}
}
\right)}
\end{equation}

Bringing this all together, we have seen the natural $\PSL(2;\RR)$ actions on both the spaces of roots and coefficients endow the space of real quadratics with complex conjugate roots with the homogeneous geometry of the hyperbolic plane.
We are now in a position to prove the main theorem of this section, following the outline proposed at the beginning.

\begin{theorem}
\label{thm:hypIso}
Let $\HH^2_{\Coefs}=\{[a:b:c]\in\RP^2\mid 4ac>b^2\}$ be equipped with the projectively invariant metric (equation \ref{eqn:CoefsMetric}) it inherits as a convex subset of $\RP^2$, and $\HH^2_{\Roots}=\{\{x\pm iy\}\mid x,y\in\R,y>0\}$ be equipped with the hyperbolic metric (equation \ref{eqn:RootsMetric}) arising from its identification with the upper half plane.
Then the restricted roots map $\RootMap\colon\HH^2_\Coefs\to\HH^2_\Roots$ is an isometry.
\end{theorem}
\begin{proof}

Observation \ref{obs:RootsHomeo} implies the roots map is a homeomorphism on the total space of projectivized quadratics and their roots, so restricting to polynomials with a complex conjugate pair of roots, $\mathcal{R}$ remains a homeomorphism from $\HH^2_\Coefs$ to $\HH^2_\Roots$.  
Equipped with their respective metrics $d_\Coefs$ and $d_\Roots$ referenced in the theorem statement, we show that $\mathcal{R}$ is an isometry by proving for each pair of quadratics $f,g$ with complex roots,
$$d_\Coefs(f,g)=d_\Roots(\mathcal{R}(f),\mathcal{R}(g)).$$
Fix such an $f$ and $g$.  Because hyperbolic geometry is homogeneous, there is some isometry $A$ of $\mathbb{H}^2_\Coefs$ which takes $f$ to any point of our choosing.  To leverage this symmetry in our computations, we choose $A$ to be the isometry taking $f$ to $[1:0:1]$, which represents the polynomial $x^2+1$.  But furthermore the hyperbolic plane is isotropic (looks the same in every direction); consequently we may find another isometry $B$ which fixes $Af=[1:0:1]$ and rotates about it, taking $A.g$ to a point of the form $[1:0:a^2]$ lying on the geodesic $[1:0:t]$ through $[1:0:1]$.
As $C=BA$ is an isometry it leaves distances invariant, and so
$$d_\Coefs(f,g)=d_\Coefs(C.f,C.g)=d_\Coefs\left([1:0:1],[1:0:a^2]\right)$$
This final quantity is straightforwward to compute directly from the definition of the distance function:
\begin{equation}d_\Coefs\left([1:0:1],[1:0:a^2]\right)=\acosh\left(\frac{2+2a^2-0}{\sqrt{(-4)(-4a^2)}}\right)=\acosh\left(\frac{1}{2}\left(a+\frac{1}{a}\right)\right)
\label{eqn:RootMapIsom_CoefSide}\end{equation}
Now, we turn to the computation of $d_\Roots(\mathcal{R}(f),\mathcal{R}(g))$.  The symmetry $C$ we leveraged in Equation \ref{eqn:RootMapIsom_CoefSide} is a linear transformation preserving the hyperboloid $\Delta(f)=-1$ in $\RR^3$, and hence $C$ lies in the image of the representation $\rho$, so we may write $C=\rho(M)$ for some $M\in\PSL(2;\RR)$.
Now, we can use the equivariance of the root map $\mathcal{R}$ to simplify things: 
$$d_\Roots\left(\mathcal{R}(f),\mathcal{R}(g)\right)=d_\Roots\left(M.\mathcal{R}(f),M.\mathcal{R}(g)\right)=d_\Roots\left(\mathcal{R}\left(\rho(M).f\right),\mathcal{R}\left(\rho(M).g\right)\right)$$
where the first equality follows as $M$ is an isometry of $d_\Roots$, and the second is equivariance (Equation \ref{eqn:Rho_Equivariant}).
But this is just the distance between the roots of $C.f=x^2+1$ and $C.g=x^2+a^2$; which is the length of the hyperbolic geodesic connecting $i$ to $ia$ in the upper half plane.
Using the expression for $d_\Roots$ in Equation \ref{eqn:RootsMetric}, we see
\begin{equation}d_\Roots(i,ia)=\acosh\left(1+\frac{(0-0)^2+(1-a)^2}{2(1)(a)}\right)=\acosh\left(\frac{1}{2}\left(a+\frac{1}{a}\right)\right)
\label{eqn:RootMapIsom_RootSide}\end{equation}
Putting these two computations together, we see that for any $f,g$ the distance $d_\Coefs(f,g)$ in the domain is equal to the distance $d_\Roots(\mathcal{R}(f),\mathcal{R}(g))$ in the range, so $\mathcal{R}$ is an isometry as claimed.

\end{proof}

We may use this to derive the quadratic formula from hyperbolic geometry.
Fix any quadratic $f$ (say, $x^2+1$) with complex root $r$ (here $i$) in the upper half plane, and let $A\in\PSL(2;\R)$.
By equivariance, the polynomial with root $A.r$ is $\rho(A).f$ for $\rho\colon\PSL(2;\R)\to\PSL(3;\R)$ the representation from Equation \ref{eqn:Rep}.
Writing this out, we see the polynomial with roots $x\pm iy=\left(\begin{smallmatrix}y&x\\0&1\end{smallmatrix}\right).\{\pm i\}$ has coefficients $\rho\left(\left(\begin{smallmatrix}y&x\\0&1\end{smallmatrix}\right)\right).[1:0:1]=[1:-2x:x^2+y^2]$.
This is an explict formula for the inverse of the roots map, $\RootMap^{-1}\colon\HH^2_\Roots\to\HH^2_\Coefs$, sending the roots $\{x\pm iy\}$ to the polynomial with (projectivized) coefficients $[1:-2x:x^2+y^2]$.
Inverting the relation $[1:-2x:x^2+y^2]=[a:b:c]$ gives $x=-b/2a$ and $y=\sqrt{c/a-x^2}$, or
$$x+iy=\frac{-b}{2a}+i\frac{\sqrt{4ac-b^2}}{2a}.$$

Theorem \ref{thm:hypIso} allows us to compute any geometric quantity of interest using either the roots or coefficients.
This simplifies certain calculations.
In particular, if $\alpha,\beta\in\CC$ are complex roots of the real quadratics $f_\alpha,f_\beta$ respectively, we may avoid using equation \eqref{eqn:RootsMetric} to compute $d_\Roots(\alpha,\beta)$, and instead compute $d_\Coefs(f_\alpha,f_\beta)$ using equation \eqref{eqn:CoefsMetric}.
This is used for the results in Section \ref{sec:Dio-quad}.

\begin{figure}[h!tbp]
\centering
	\includegraphics[width=0.85\textwidth]{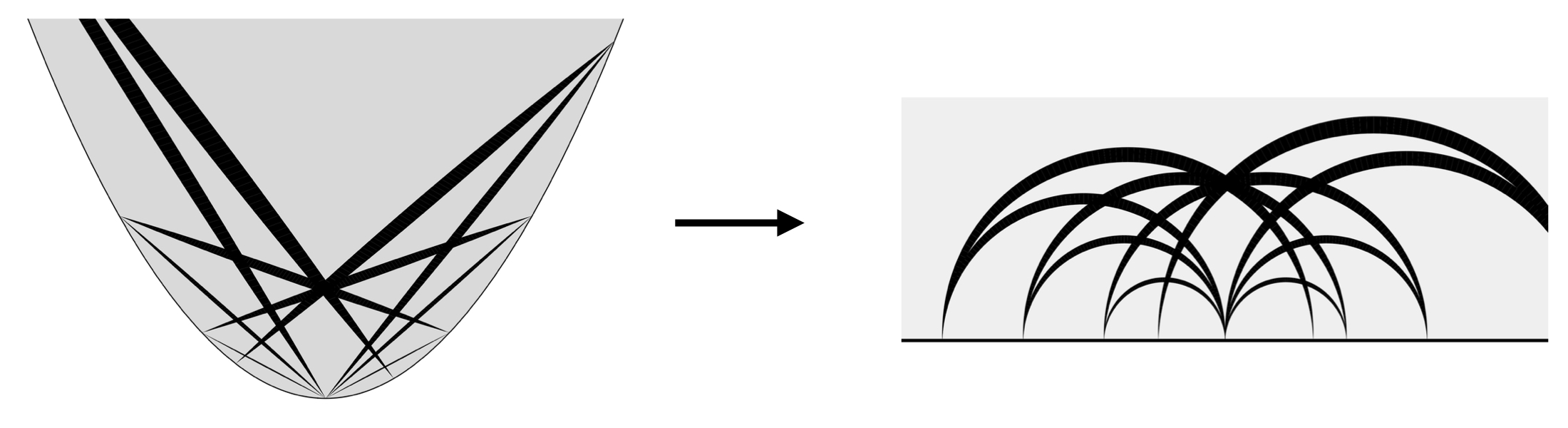}
\caption{Isometry from the projective model $\{[a:b:c]\mid 4ac>b^2\}$ and the upper half plane model of the hyperbolic plane, given by the quadratic formula.}
\label{fig:hypIso}
\end{figure}

\begin{corollary}
\label{rem:QuadFormula}
\label{cor:QuadFormula}
After the projective change of coordinates $\RP^2\to\RP^2$ given by $[a:b:c]=[\tfrac{w+u}{2}:v:\tfrac{w-u}{2}]$, the quadratic formula is precisely the usual isometry from the Klein disk model of $\HH^2$ to the upper half plane model, 
$$[u:v:1]\mapsto \frac{-v\pm i\sqrt{1-u^2-v^2}}{1+u}.$$
\end{corollary}

Qualitatively, this provides a complete understanding of the roots map\footnote{As particular examples, the 1-parameter families of quadratics with coefficients $[1:0:t]$, and $[1:t:1]$ project under the roots map to the vertical geodesic and unit circle through $i\in\C$ respectively.
}: affine lines in the space of coefficients are geodesics in the projective model of $\mathbb{H}^2$, and so the roots of such a 1-parameter family of polynomials form a geodesic in the upper half plane model: generalized circles orthogonal to $\R\subset\C$.

There is also a very nice geometric interpretation of the roots map for quadratics with two real roots, a patch of which is visualized in Figure \ref{fig:RealQuadratics}.
We do not give many details here, as these quadratics do not occur in the starscape images.
Nonetheless, we cannot resist telling the beginning of the story.

\begin{observation}
    \label{obs:quad_PosDisc}
    The roots map for quadratics with a double root is the continuous extension of $\RootMap\colon\HH^2_\Coefs\to\HH^2_\Roots$ to the ideal boundary of the hyperbolic plane.
For polynomials of positive discriminant, the roots map decomposes geometrically as follows:
\begin{enumerate}
    \item Send the polynomial with coefficient vector $v=[a:b:c]$, thought of as a point in the M\"obius band $\RP^2\backslash\HH^2_\Coefs$, to the set $\{w_1,w_2\}\in\partial_\infty\HH^2$, where the line through $v$ and $w_i$ is tangent to the ideal boundary. See figure \ref{fig:PosDiscQF}.
    \item Follow by applying to each $w_i$ the homeomorphism $\partial_\infty\HH^2_\Coefs\to\partial_\infty\HH^2_\Roots$ sending the projectivized lightcone to the extended real line in $\CP^1$.  The resulting two points are the roots of $f(x)=ax^2+bx+c$.
\end{enumerate}
\end{observation}

We may reduce the proof of this observation to checking its truth at a single point using two facts: the transitivity of 
this $\PSL(2;\RR)$ action (via $\rho$) on the space of quadratics with positive discriminant, and the fact that as linear linear transformations preserving the lightcone, $\rho(A)$ preserves the colledtion of tangent lines to the ideal boundary for each $A\in\PSL(2;\RR)$.
Choosing a point at which to verify the assertion:
note the polynomial $f(x)=(x+1)(x-1)$ has coefficients $v=[1:0:-1]$ and roots $\{1,-1\}\in\RR$.
These roots are identified with the polynomials $(x\pm 1)^2$ on the lightcone, so $w_i=[1:\pm 2:1]$, and the lines through $v$ and $w_i$ are easily verified to be tangent to the discriminant locus, as claimed.

\begin{figure}[h!tbp]
\centering
	\includegraphics[width=0.4\textwidth]{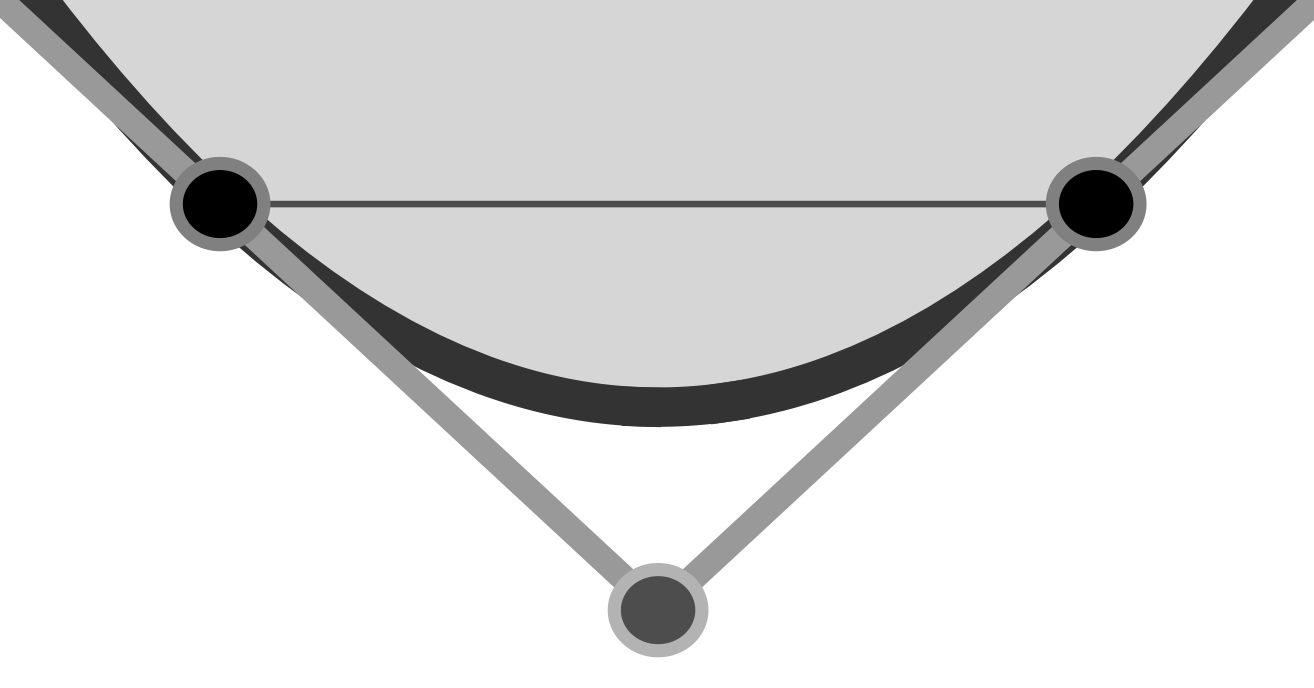}
\caption{The roots map for quadratics with two real roots associates a point exterior to the hyperbolic plane in $\RP^2_\Coefs$ to its two points of tangency with $\partial_\infty\HH^2_\Coefs$.}
\label{fig:PosDiscQF}
\end{figure}

\subsubsection{Applications to quadratic algebraic numbers}
\label{sec:applquad}

We now can apply this to the original case of interest: quadratic algebraic numbers and integer quadratic polynomials.
The integer polynomials form a $\mathbb{Z}^3$ lattice in the space $\R^3$ of coefficients, and its image in $\RP^2$ can be interpreted as what it would look like to see the integer lattice from the origin (much like we saw for $\QP^1$ and $\ZZ^2$ in Figure \ref{fig:QP1}).
The quadratics of interest lie inside of the cone cut out by the discriminant, with planes through the origin projecting to lines in the disk $\HH^2_\Coefs$.
And as the roots map realizes an isometry onto the upper half plane model (for specificity, by selecting the root in the complex conjugate pair with positive imaginary part), we know the image of these geodesics are also geodesics - here represented by circles orthogonal to the boundary.
This explains the small scale patterns visible everywhere in the picture of the integer quadratic numbers - they are just the perspective view of a cubic lattice, distorted by the isometry taking the Klein model to the upper half plane model.

\begin{figure}[h!tbp]
\centering
    \includegraphics[width=0.95\textwidth]{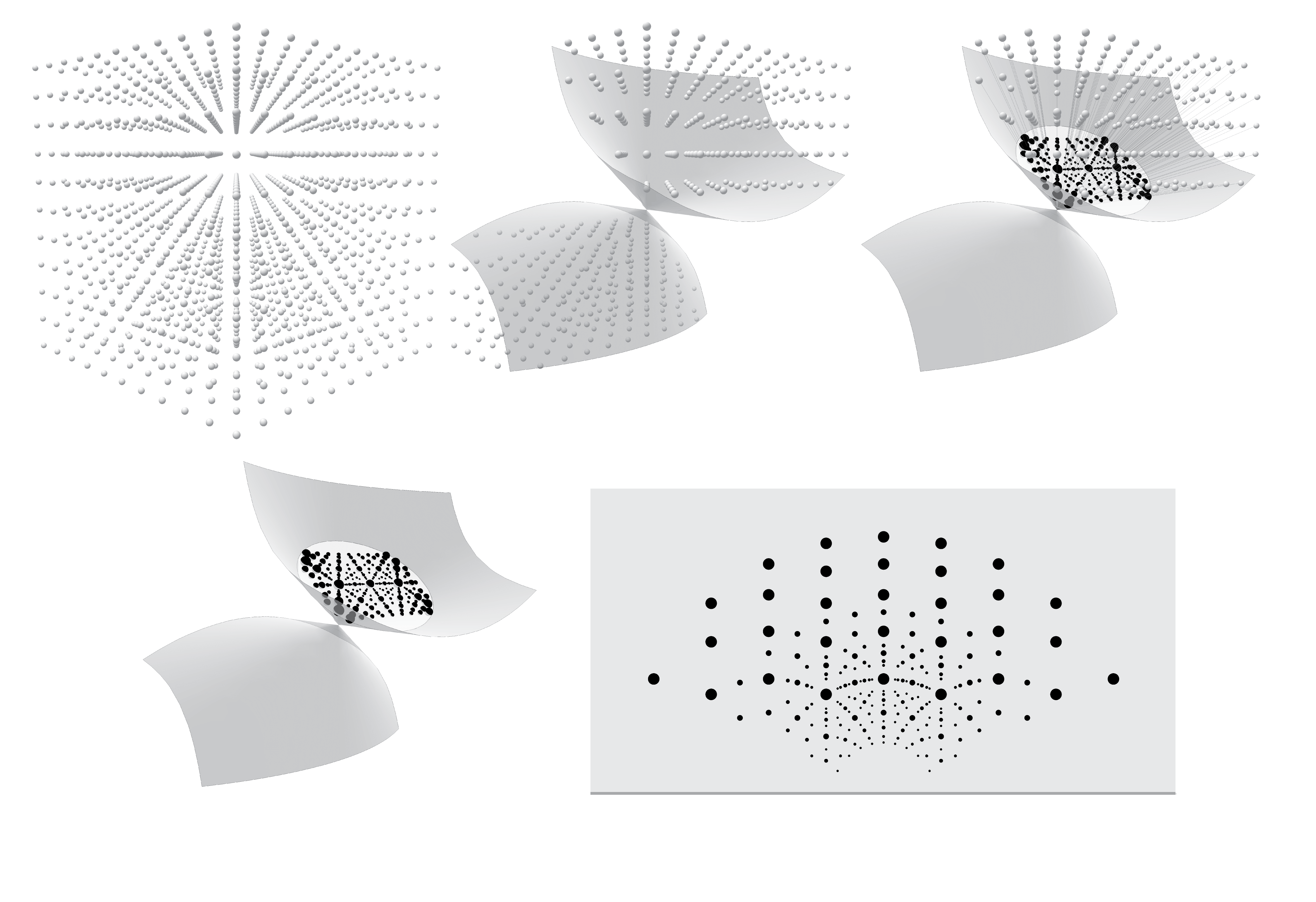}
\caption{
The lattice of integer quadratic polynomials, and its image under the roots map, decomposed as a sequence of geometric steps, from top left to bottom right. The lattice points lying within the light cone of the discriminant represent quadratics with a complex conjugate roots.  Their projectivization is a collection of points in a projective model of the hyperbolic plane, and the roots map is an isometry onto the upper half plane model.  See Theorem \ref{thm:hypIso}.
}
\label{fig:latticetoalgebraics}
\end{figure}

The natural $\PSL(2;\CC)$ action on roots and coefficients does not preserve the $\ZZ^3$ lattice.  The subgroup which does is isomorphic to $\PSL(2;\ZZ)$ (it is the intersection of $\rho(\PSL(2;\CC))$ with $\PSL(3;\ZZ)$, and it will play an important role in Sections \ref{sec:DiophantineApproximation} and \ref{sec:Dio-quad}).

The two-dimensional sublattices of $\mathbb{Z}^3$ will play a special role in what is to come.  The planes such a sublattice can span are exactly planes with rational normal vector.  These project to lines in $\HH^2_\Coefs$ which we will call \emph{rational geodesics}.  These are the dominant features in Figure \ref{fig:Initial_quadratics}.  There are a few important facts to collect about rational geodesics; these are just immediate consequences of the geometry.

\begin{observation}
	\label{obs:rat-geo}
	\begin{enumerate}
		\item For $z \in \CC$, its corresponding point $\{z, \overline{z}\}$ in root space pulls back to $[1:-z-\overline{z}:z\overline{z}]$, and so $z$ lies on a rational geodesic if and only if $1$, $z + \overline{z}$ and $z\overline{z}$ are $\QQ$-linearly dependent.
		\item Any two quadratic irrationalities $z, w \in \CC$ share a unique rational geodesic, since any two points in $\RP^2$ determine a unique line.  Any quadratic irrationality lies on infinitely many rational geodesics, and any two rational geodesics intersect at a quadratic irrationality.
		\item The images of rational geodesics under the roots map are exactly the hyperbolic geodesics of the upper half plane given by the upper half circles centred on a rational number, whose radius squared is rational.  In other words, the limit points form a conjugate pair of points in a real quadratic field, or a pair of rational points.
		\item The group $\PSL(2;\ZZ)$ acts as change of variables on the quadratic form associated to the geodesic\footnote{In this way the orbits of rational geodesics under $\PSL(2;\ZZ)$ are identified with the narrow ideal classes of real quadratic fields $K$.  See \cite[Section B.7]{Cox}.}. 
	\end{enumerate}
\end{observation}


The geometric description of the roots map for quadratics with two real roots (Observation \ref{obs:quad_PosDisc}) has a nice interpretation for geodesics. 
Realizing such a geodesic $\gamma$ as the intersection of a plane $\mathcal{S}$ with the cone of positive discriminant, denote the normal to $\mathcal{S}$ by $n$ (well defined up to scaling; lying outside the lightcone) and the endpoints of $\gamma$ by $x,y$ on the ideal boundary of $\HH^2_\Coefs$ (the projectivized lightcone of the discriminant).
In the projective model one may recover these endpoints $x,y$ (and hence the geodesic $\gamma$ itself) directly from the normal $n$ via a purely geometric construction.
There are precisely two planes containing the line defining $n$ which are tangent to the lightcone, and projectivizing these lines of tangency gives the ideal endpoints $\{x,y\}$ of the geodesic associated to $n$.
But by Observation \ref{obs:quad_PosDisc}, the map sending $n$ to its two points of tangency with $\partial_\infty\HH^2$, followed by the isometry from the projective to upper half plane model of $\HH^2$ is none other than the roots map on quadratics of positive discriminant.
Thus, if we think of a geodesic as determined by its normal vector $n=(n_2,n_1,n_0)$ (often computationally an attractive thing to do), we may directly recover\footnote{Topologically, this describes a map which takes the projectivization of the exterior of the lightcone (a M\"obius band) to the set of unordered pairs of distinct points on the circle (or $\SP^2(\mathbb{S}^1)$, which is also a M\"obius band, as depicted in Figure \ref{fig:Mob}).}
the geodesic from $n$ as its endpoints are precisely the roots of $n_2x^2+n_1x+n_0=0$.  This\footnote{Geometrically one may tell a beautiful story here quite analogous to Theorem \ref{thm:hypIso}, where the roots map is an isometry between a pair of Lorentzian metrics defined on each of these M\"obius bands, though investigation of this would take us too far afield from the goals of this paper.} is visible in Figure \ref{fig:PosDiscQF}.

\subsection{Beyond discworld: the geometry of cubics}
\label{ssec:CubicGeometry}
Similarly to the quadratic case, we begin with the space of all complex homogeneous polynomials together with the $\PSL(2;\CC)$ action arising from precomposition with M\"obius transformations.
Cubics are the highest degree\footnote{This action has finitely many orbits as $\PSL(2;\CC)$ acts simply transitively on ordered triples of distinct elements of $\CC\PP^1$. Thus, some of these orbits are open.  In higher degree, there are a continuum of orbits, which are parameterized by the moduli of $n$ possibly indistinct points in $\CP^1$ up to projective transformations (see also Remark \ref{rem:SymmDim}).} where this action is enough to equip the various open subsets of generic cubics (components of the complement of the discriminant locus) with the structure of homogeneous geometries.
Because $\PSL(2;\RR)$ is so tightly linked to $\HH^2$, the hyperbolic plane remains a prominent actor in this story.
Indeed, restricting to real cubics with negative discriminant, the complex root determines a point in $\HH^2$ and the real root a direction - identifying this geometry as the unit tangent bundle\footnote{Given a manifold $X$, the set of all tangent vectors to $X$ at a point $p$ is called the \emph{tangent space} to $X$ at $p$.  You can think of this as a 'linear approximation' to $X$ near that point.  If we restrict our attention to only unit vectors, we define the \emph{unit tangent space} at $p$ (for 2-dimensional geometries $X$, the unit tangent space at every point is just a circle).
Collecting all the tangent spaces for every point of $X$ gives the \emph{tangent bundle} $TX$ to $X$, and collecting only all unit vectors gives the \emph{unit tangent bundle} $\UT X$.  For an introduction smooth manifolds and their tangent bundles, see \cite{lee2019introduction}.}
to the hyperbolic plane.
Both the spaces of coefficients and roots form models of this geometry, related by the roots map, resulting in an analogous theorem to Theorem \ref{thm:QuadRoots}.

\begin{theorem}
\label{thm:UTH2}
	Let $\{[a:b:c:d]\mid \Delta_3(a,b,c,d)<0\}$ be the set of real cubics with exactly one real root, where $\Delta_3$ is the discriminant for cubics, and 
	$\{\{r,z,\overline{z}\}\mid r\in\RP^1,z\in \CC\backslash\RR\}$ be the set of their root-sets.
	Each of these spaces admits a natural $\PSL(2;\RR)$ action (the former by precomposing the polynomial with a linear transformation, the latter by applying a M\"obius transformation to each root).
	Finally, equipped with these actions, each of these spaces is isomorphic to the unit tangent bundle to the hyperbolic plane.
\end{theorem}

We will prove this theorem in two pieces, Proposition \ref{prop:UTH2} and Corollary \ref{prop:UT_Coefs}.
Beyond this, we will see that there is a natural way to equip each of these spaces with a Riemannian metric and with respect to these metrics, the roots map $\mathcal{R}$ is actually an \emph{isometry} (exactly analogous to the quadratic case).  After developing the necessary pieces, this is stated precisely and proven in Theorem \ref{thm:cubicIsom}.
Utilizing this geometry both on the spaces of roots and coefficients provides a geometric description of the cubic formula, which we again pre-emptively state here.
We state and prove the full version in Theorem \ref{thm:CubicFactor}.

\begin{theorem}
\label{thm:cubicMain}
On the set $X\subset\PP\Coefs$ of polynomials which have exactly one real root, we define the \emph{complex-root-map} $\RootMap_\CC$ as follows.
For $f\in X$ a polynomial, let $\RootMap_\CC(f)=z$ be the unique complex root of $f$ with positive imaginary part.
Then, equipping domain with the geometry of the hyperbolic plane's unit tangent bundle, the map $\RootMap_\CC$ factors geometrically as the projection $\UT\HH^2\to\HH^2$ onto the projective model $\mathbb{H}^2_\Coefs$ of the hyperbolic plane, followed by the isometry $\mathbb{H}^2_\Coefs\to\mathbb{H}^2_\Roots\subset\CC$ of Theorem \ref{thm:hypIso}.

\begin{center}
\begin{tikzcd}
\mathrm{Coefs} \arrow[r, "\cong" description] \arrow[rrrd, "\mathcal{R}_\mathbb{C}" description, bend left=49] & \mathrm{UT}\mathbb{H}^2 \arrow[d, "\pi" description]   &                                             &            \\
                                                                                                               & \mathbb{H}^2_\mathrm{Coefs} \arrow[r, "\mathcal{R}_2"] & \mathbb{H}^2_\mathrm{Roots} \arrow[r, hook] & \mathbb{C}
\end{tikzcd}
\end{center}
\end{theorem}

Said briefly, the map sending a real cubic to its complex root in the upper half plane used to draw cubic starscapes is topologically conjugate to the projection $\UT\HH^2\to\HH^2$ defining the hyperbolic plane's unit tangent bundle.  We may use this geometry to understand some of the interesting interesting 1 and 2 dimensional families of cubic polynomials, producing linear and planar starscapes respectively. 
Some such linear starscapes are highlighted in the planar starscapes below (Figure \ref{fig:Lines}).
In particular, this gives a simple condition for when a starscape (thought of as a projective subspace of the set of cubics) embeds in the complex plane under the projection onto the complex root (Figure \ref{fig:A0BClines}) and when it is singular, collapsing some curve to a point (Figure \ref{fig:ACBClines}).

\begin{figure}[h!tbp]
\centering
\begin{subfigure}[b]{0.42\textwidth} 
    \centering
	\includegraphics[width=\textwidth]{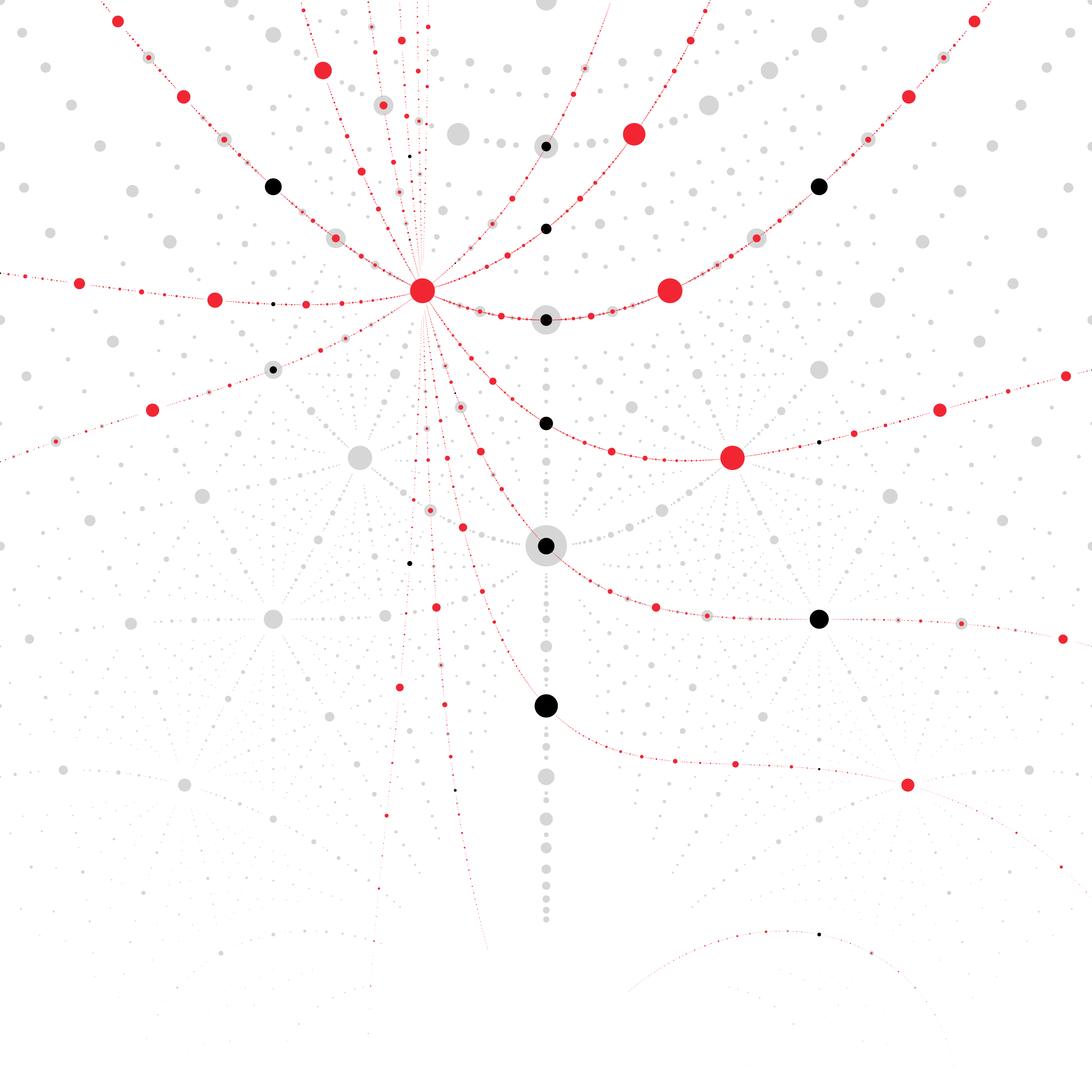}
	\caption{}
    \label{fig:A0BClines}
\end{subfigure}
\begin{subfigure}[b]{0.42\textwidth} 
    \centering
	\includegraphics[width=\textwidth]{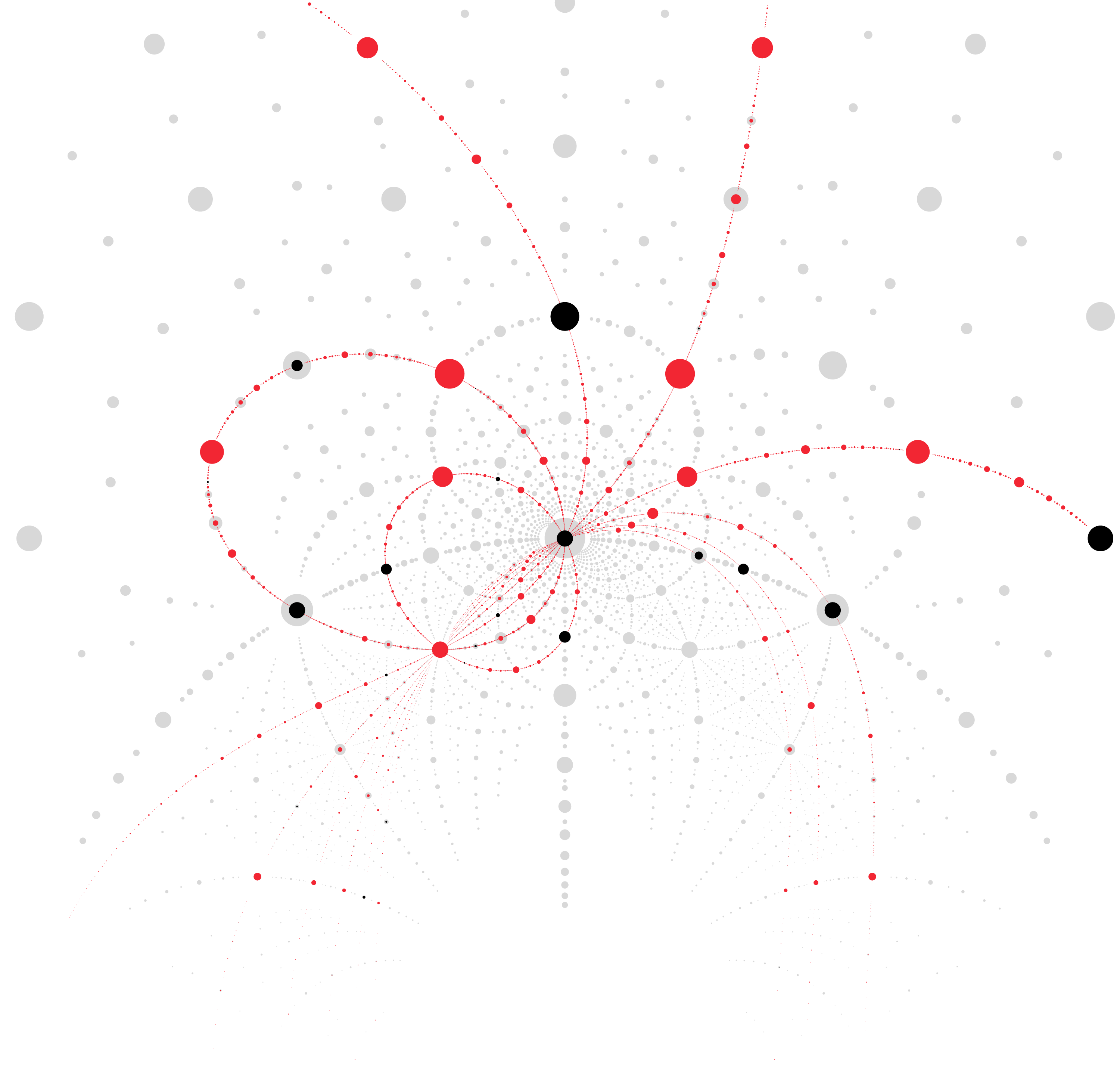}
	\caption{}
    \label{fig:ACBClines}
\end{subfigure}
\caption{{Some linear starscapes, shown in red and black, within planar starscapes, the other points are shown in grey. The black points are quadratic roots, where the real root is rational. 
In (\ref{fig:A0BClines}) some of the projective lines (corresponding to 2d sublattices of the lattice of coefficients) through the root of $x^3-2x^2+1$ are plotted in the depressed cubic starscape (also shown in Figure \ref{fig:Starscapes_a}). 
In (\ref{fig:ACBClines}) some of the lines through the root of $x^3-x^2-1$ are plotted in the starscape for the family $a x^3 + c x^2 + bx + c$ (also shown in Figure \ref{fig:Initial_cubic_family}). In the latter, these lines all also pass through $i$, which is the projection of a fibre $(d x + e)(x^2+1) = d x^3+e x^2+d x + e$.  
This can be seen in a different way on the right in Figure \ref{fig:Fam7C}.}}
\label{fig:Lines}
\end{figure}
\begin{corollary}
\label{cor:StarscapeEmbed1}

Let $\mathcal{S}$ be a one or two dimensional affine space of real cubics in the space of projectivized coefficients, and $\mathcal{R}_\mathbb{C}|_\mathcal{S}\colon \mathcal{S}\to\CC$ the projection onto the complex root.
Under the geometric identification of this space with the unit tangent bundle of the hyperbolic plane (Theorem \ref{thm:UTH2}), the space of coefficients is foilated by simple closed curves (fibers of the unit tangent bundle).
Then $\mathcal{R}_\mathbb{C}|_\mathcal{S}$ is an embedding
if and only if $\mathcal{S}$ is everywhere transverse to this fibration.

\end{corollary}

	For this failure of embedding to actually be \emph{visible} in a cubic starscape, it must occur at some cubic or quadratic number, and accounting for this gives the more refined statement of Corollary \ref{cor:transverse}.
Quadratics are too low-degree for any interesting analogs of this behavior to occur\footnote{For quadratics, the roots map is a homeomorphism on the entire space: there are no 2-dimensional subfamilies, and all 1-dimensional curves are geodesics.}, but it persists in higher degree (Figure \ref{fig:Starscapes}), making cubics an important testing ground.
We fill in the details of this picture below.

\subsubsection{General Complex and Real Coefficients}
\label{subsec:CubicCoefs}
	A homogeneous cubic in two variables has the form $ax^3+bx^2y+cxy^2+dy^3$, so the sets of coefficients of all such cubics naturally identifies with $\CC^4$.
As previously, we are mostly concerned with cubics only up to a global scaling, and so take the space of projectivized coefficients $\PP\Coefs=\CP^3$, together with the natural action of $\PSL(4;\CC)$ as our starting point.

From the perspective of their solutions, cubics may be identified with unordered triples of (possibly coincident) points in the extended complex plane, so $\Roots=\SP^3(\CP^1)$ is the third symmetric power of the sphere.
The symmetric power admits a natural action of $\PSL(2;\CC)$ coming directly from its usual action on $\CP^1$ by M\"obius transformations.
More precisely, the action of $A=\left(\begin{smallmatrix}p&q\\r&s\end{smallmatrix}\right)$ on the triple of points $\{z_1,z_2,z_3\}$ is

\begin{equation}
\label{eqn:ProjActionSP3}
A.\{z_1,z_2,z_3\}=\left\{\frac{pz_1+q}{rz_1+s},\frac{pz_2+q}{rz_2+s},\frac{pz_3+q}{rz_3+s}\right\}.\end{equation}

Again, the roots map $\RootMap\colon\PP\Coefs\to\Roots$ realizes a homeomorphism from $\PP\Coefs=\CP^3$ to $\Roots=\SP^3(\CP^1)$.
Conjugating the $\PSL(2;\CC)$ action on $\Roots$ by $\RootMap$ gives an action on $\PP\Coefs$, which is compatible with its natural $\PSL(4;\CC)$ action: it is the projectivization of the representation $\SL(2;\CC)\to\SL(4;\CC)$.
As in Section \ref{sec:PSL_Sym}, the explicit formula for this representation is defined by $\rho(A).f(x)=f(A^{-1}.x)$:
\begin{equation}
\label{eqn:IrrepActionRP3}
\begin{pmatrix}p&q\\r&s\end{pmatrix}\stackrel{\rho}{\longrightarrow}
\begin{pmatrix}s^3&-r s^2& r^2s &-r^3\\
-3 q s^2& 2q r s+ps^2&-qr^2-2prs&3pr^2\\3q^2s & -q^2r-2pqs&2pqr+p^2s&-3p^2r\\
q^3&pq^2&-p^2q&p^3\end{pmatrix}\end{equation}
Recall the $\PSL(2;\CC)$ action on the plane acts triply transitively (see \cite{richter2011perspectives} for an introduction to real and complex projective geometry).  This implies that the space of cubics divides up into three distinct orbits under this action: cubics with three distinct roots, cubics with a pair of coincident roots, and cubics with a triple root.
As each component is an orbit of the action, $\PSL(2;\CC)$ acts transitively on each.

\begin{remark}
\label{rem:SymmDim}
Here the (real) dimension of the space of homogeneous cubics is	6, which coincides with the dimension of $\PSL(2;\C)$.
Thus, cubics are the last dimension on which this $\PSL(2;\C)$ action remains transitive on the open subset of generic cubics, giving it the structure of a $\PSL(2;\C)$ homogeneous space.
\end{remark}

Restricting to real cubics replaces the space of coefficients $\CP^3$ with $\RP^3$, and its image under the roots map is an embedding of real projective 3-space in $\SP^3(\CP^1)$.
The restricted $\PSL(2;\RR)$ action divides $\RP^3$ into four components.
The discriminant locus $\Delta^0:=\PP\Delta^{-1}(0)$ consists of the union of the orbits with triple and double roots, and is homeomorphic to a torus\footnote{Note this torus is not smoothly embedded in the space of coefficients, and is singular along the circle parameterizing cubics with a triple root.}.
 Over $\RR$ the generic case splits into polynomials with three distinct real roots $\Delta^+$ (positive discriminant) and those with a complex conjugate pair of complex roots $\CCubic$ (negative discriminant).
 We see below each of $\Delta^\pm$ are individually homeomorphic to a solid torus, forming the standard Heegaard decomposition of $\RP^3$.

\begin{figure}[h!tbp]
\centering
\begin{subfigure}[b]{0.22\textwidth} 
    \centering
	\includegraphics[width=\textwidth]{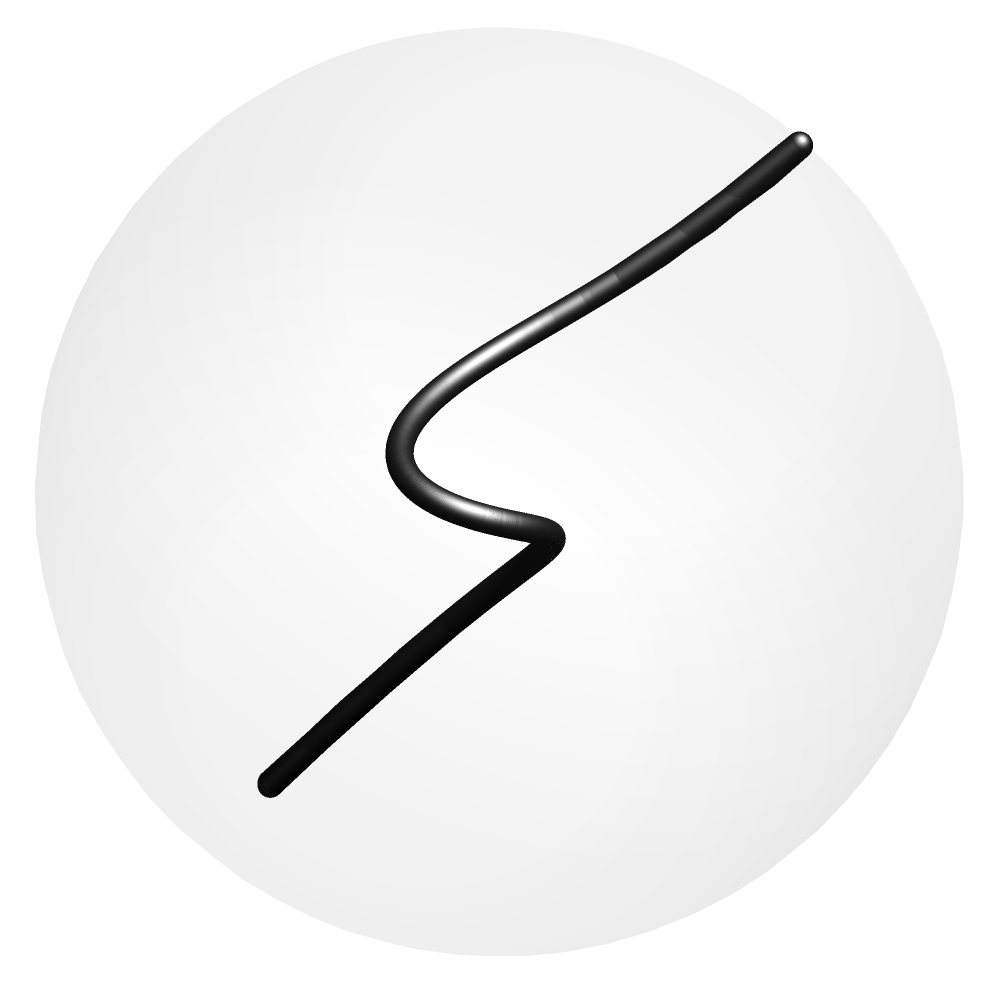}\\
		\includegraphics[width=\textwidth]{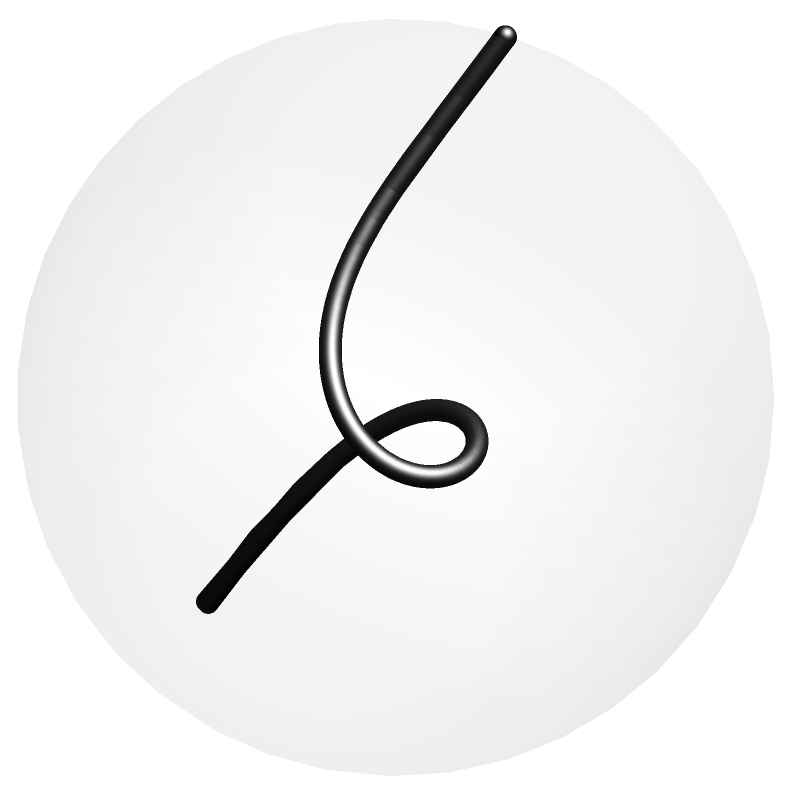}
	\caption{Triple Root\\~}
    \label{fig:CubicTripRoot}
\end{subfigure}
\begin{subfigure}[b]{0.22\textwidth} 
    \centering
	\includegraphics[width=\textwidth]{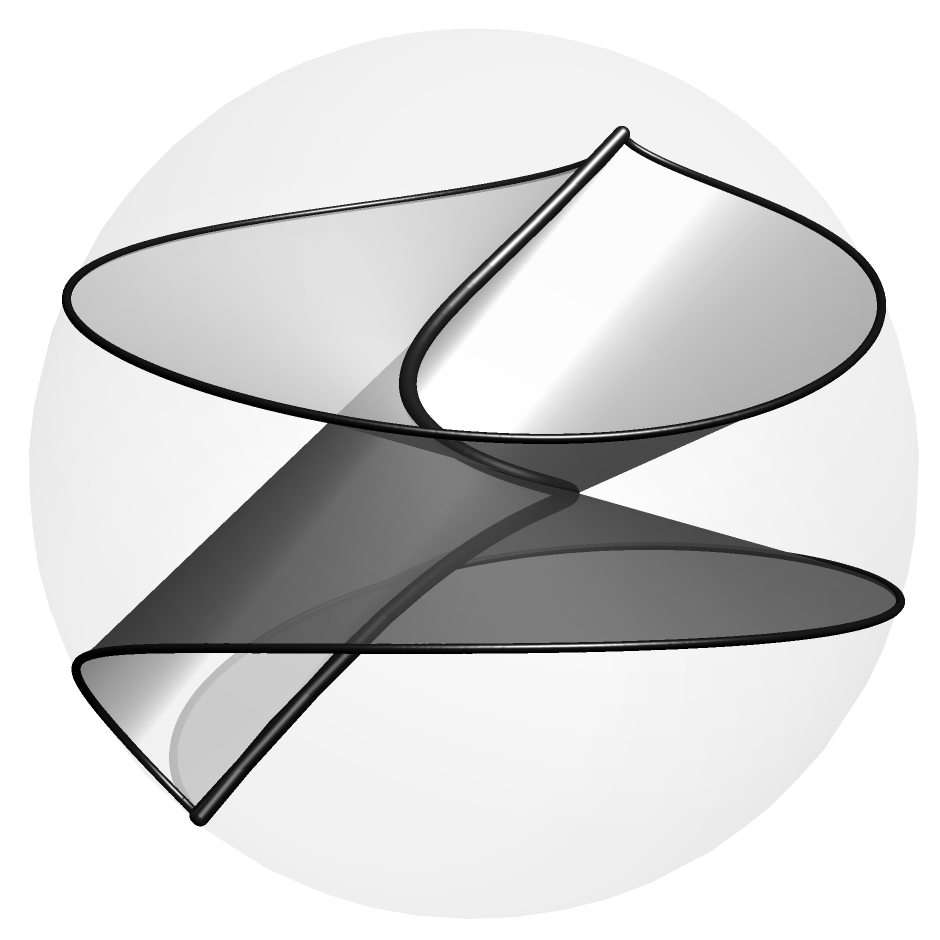}\\
		\includegraphics[width=\textwidth]{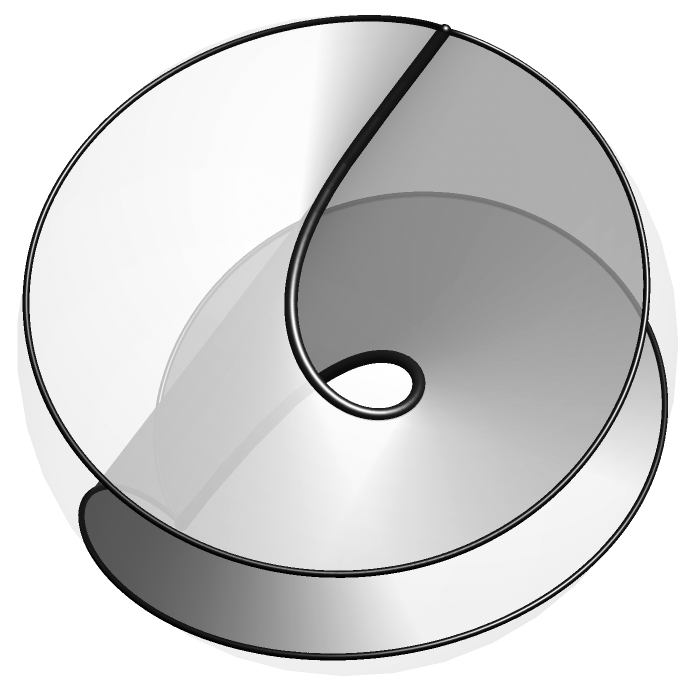}
	\caption{Double Root\\~}
    \label{fig:CubicDisc}
\end{subfigure}
\begin{subfigure}[b]{0.22\textwidth} 
    \centering
	\includegraphics[width=\textwidth]{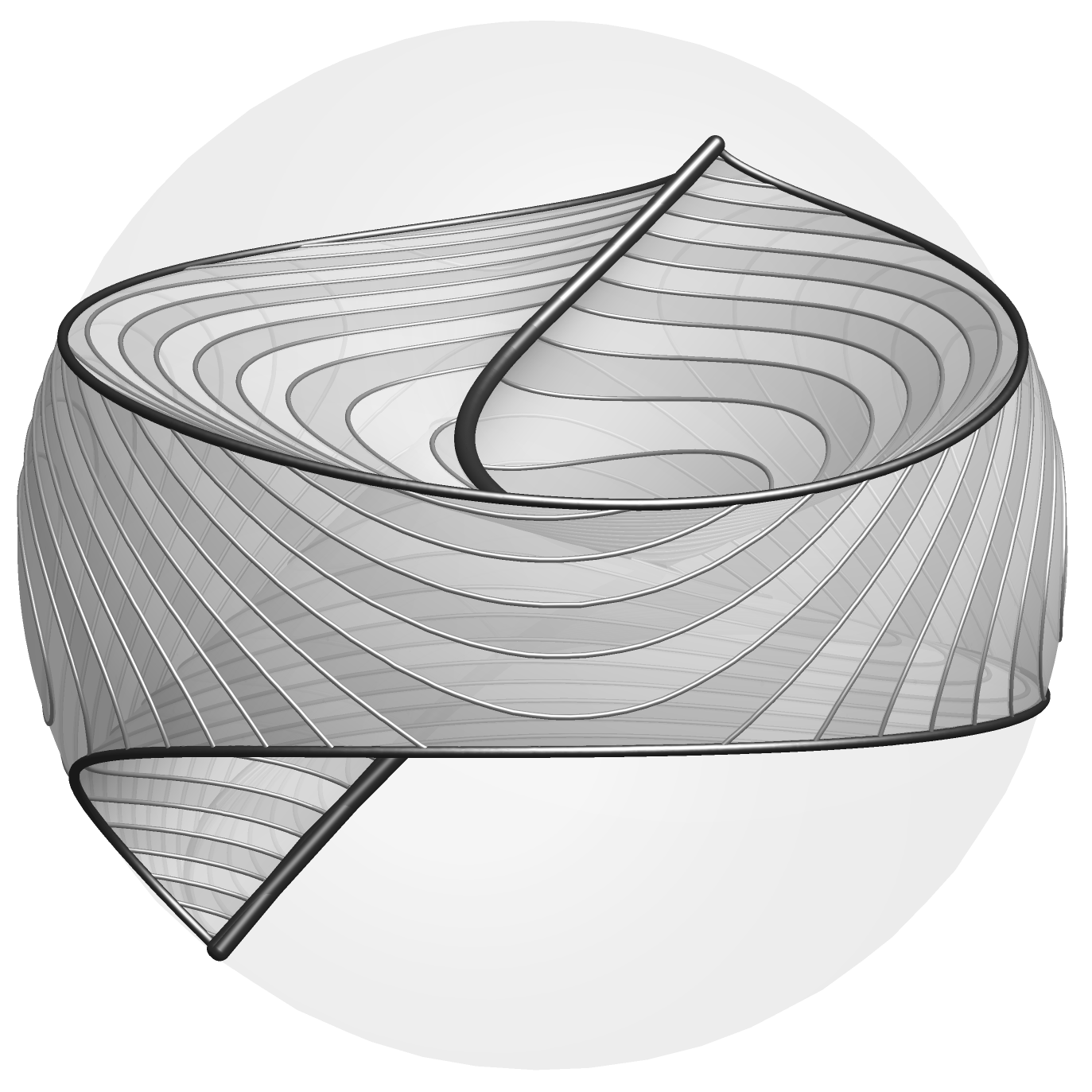}\\
	\includegraphics[width=\textwidth]{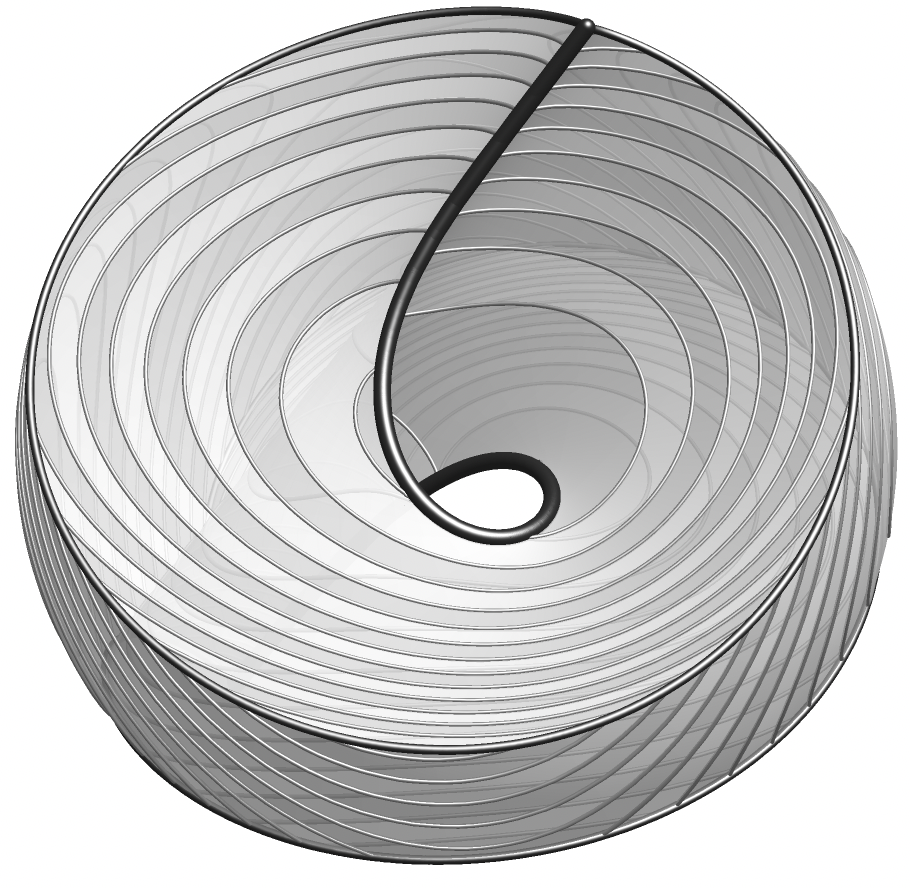}
		\caption{3 Real Roots\\~}
    \label{fig:PosDisc}
\end{subfigure}
\begin{subfigure}[b]{0.22\textwidth} 
    \centering
	\includegraphics[width=\textwidth]{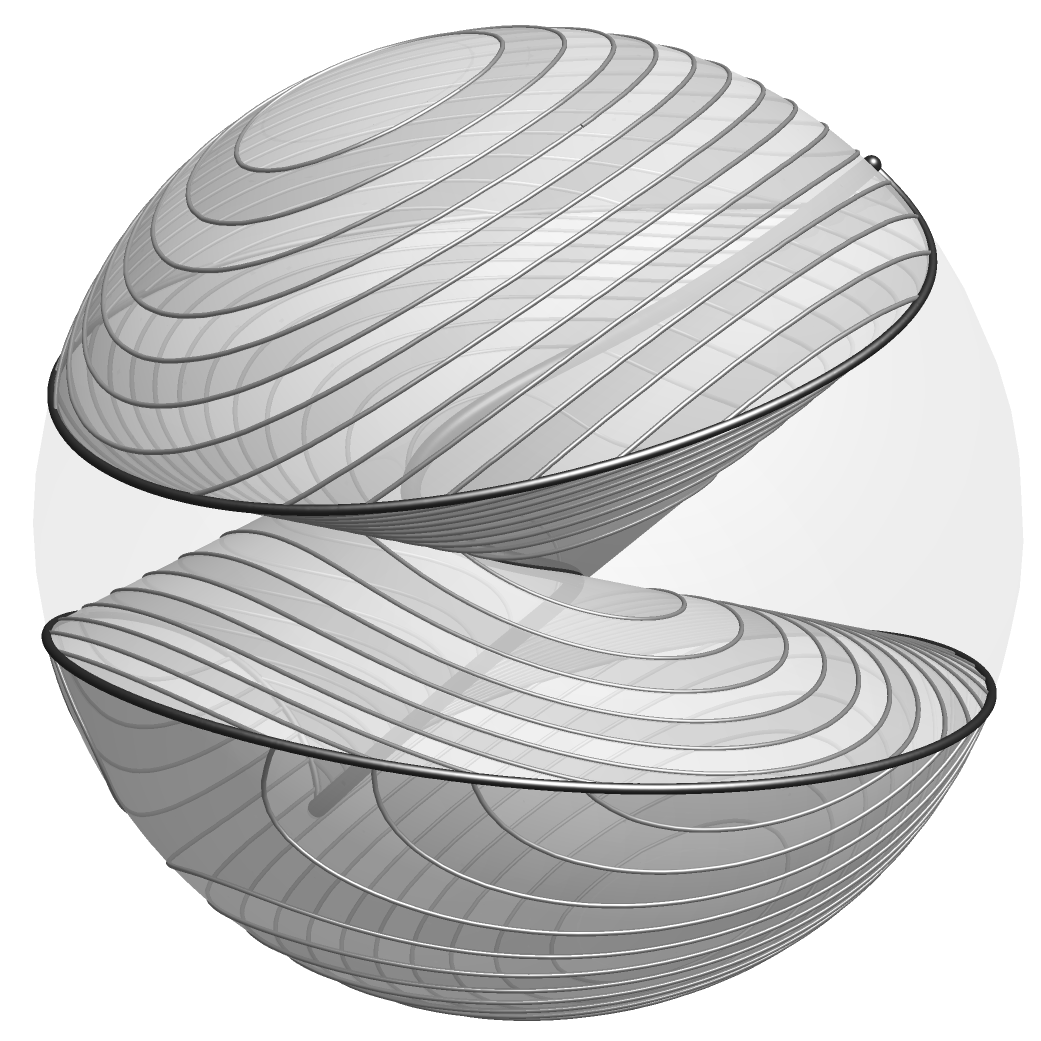}\\
	\includegraphics[width=\textwidth]{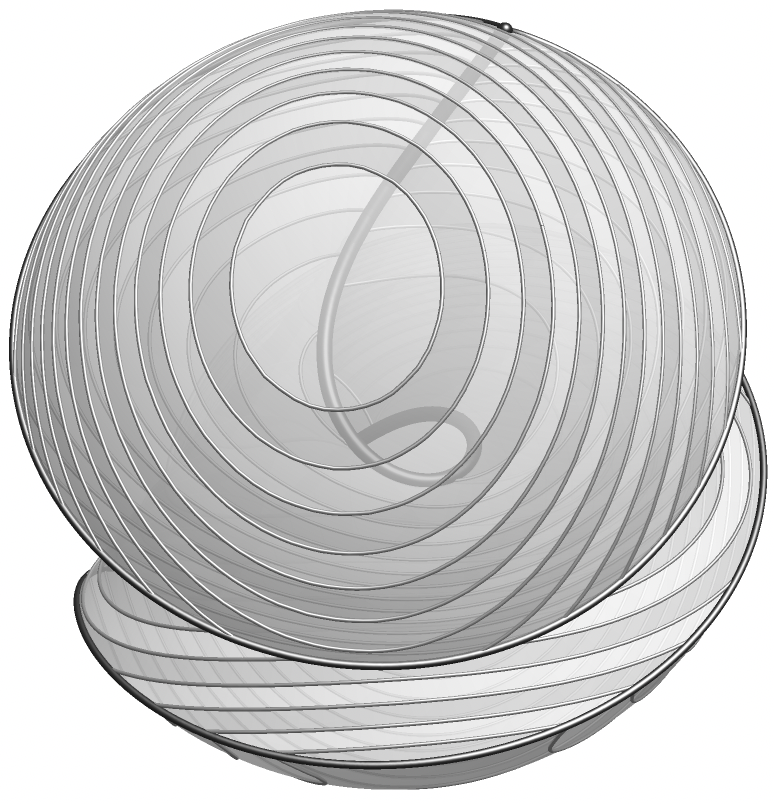}
	\caption{1 Real, 2 Complex Conjugate Roots}
    \label{fig:NegDisc}
\end{subfigure}
\caption{The four $\PSL(2;\RR)$ orbits on the $\RP^3$ of real cubics described in Section \ref{subsec:CubicCoefs}, as subsets of $\RP^3$.
Compare with the quadratic case in Figure \ref{fig:QuadraticRp2}.
As in Figure \ref{fig:RPn}, the topology of each component is recovered by identifying points on the boundary sphere via the antipodal map.
}
\label{fig:CubicRp3}
\end{figure}

The action of $\PSL(2;\RR)$ is transitive on all four components, so we may consider each as a homogeneous geometry.
Like in the complex case, via dimension count we see that $\dim \Delta^+=\dim \CCubic=\dim \PSL(2;\RR)$, so this action has at most a discrete stabilizer.
For $\Delta^+$, this stabilizer is the symmetric group on three elements\footnote{Thus the space of cubics with three distinct roots is an infinite volume three dimensional orbifold with a geometric structure modeled on $\widetilde{SL}_2(\RR)$}, as $\PSL(2;\RR)$ acts simply transitively on the space of ordered distinct triples in $\RP^1$.
For $\CCubic$ however, this action is free.
In the following section we look in detail at this geometry both from the coefficients and roots perspectives.

\subsubsection{Cubics and the Unit Tangent bundle to $\HH^2$}

Each cubic in $\CCubic$ has a real root and complex conjugate pair of complex roots.
The image under $\RootMap$ in $\SP^3(\CP^1)$ is easily described, as
$$\RootMap(\CCubic)=\left\{\{r,z,\overline{z}\}\mid r\in\RP^1,\; z=x+iy\;\textrm{for}\; x,y\in\R,y>0\right\}.$$
This is exactly an embedding of the product of $\RP^1$ and the upper half plane in $\SP^3(\CP^1)$, so this describes the homeomorphism of $\CCubic$ with a solid torus, as observed in Figure \ref{fig:NegDisc}.
As $z$ is a point in the upper half plane being acted on by $\PSL(2;\RR)$ it is tempting to think of it as a point of $\HH^2$.
From this perspective, the real root $r\in\RP^1$ lies on the ideal boundary of the upper half plane model of $\HH^2$, so we may think of it as specifying a \emph{direction}.
We make this precise in the following proposition, which is the first half of Theorem \ref{thm:cubicMain} highlighted in the introduction to this section.

\begin{proposition}

Let $\mathcal{R}(\CCubic)=\{\{r,z,\overline{z}\} \mid z=x+iy, x\in\RR,y\in\RR_+, r\in\RR\}$ denote the space consisting of the roots of all cubics in $\CCubic$, equipped with the $\PSL(2;\RR)$ action by M\"obius transformations acting on each root described above,
and let $\UT\mathbb{H}^2=\{(p,v)\mid p\in\HH^2,v\in T_p\HH^2,\|v\|=1\}$ denote the unit tangent bundle to the hyperbolic plane.
Define the map $\Phi\colon \mathcal{R}(\CCubic)\to\UT\HH^2$ by sending
each triple $\{r,z,\overline{z}\}$ to the point $z$ and the unit tangent vector $v\in T_z\HH^2$ such that $v$ is the initial tangent vector to the geodesic ray $\gamma$ starting at $z$ and limiting to the ideal point $r=\gamma(+\infty)$.
Then $\Phi$ is an isomorphism of geometries.

\label{prop:UTH2}
\end{proposition}
\begin{proof}

First we note that $\Phi$ is in fact a homeomorphism: given any point $z$ in the upper half plane and any real number $r$, there is a unique circle passing through $z$ which intersects the real axis at $r$.  Taking the tangent vector to this circle at $z$ defines the desired unit tangent $v$ and noting that geodesics are uniquely determined by these tangent vectors completes the proof that $\Phi$ is a bijection.
That $\Phi$ is in fact continuous with continuous inverse follows directly from constructions in Euclidean geometry\footnote{ Computing expressions for $\Phi$ and its inverse we see they are compositions of elementary (continuous) functions:
$\Phi(x\pm iy,r)=(v_1,v_2)_{x+iy}=\left(\frac{w}{\sqrt{1+w^2}},\frac{1}{\sqrt{1+w^2}}\right)_{x+iy}$ for $w=\frac{2x(x-r)}{x^2+y^2-r^2}$ and 
$\Phi^{-1}(v_1,v_2)_{x+iy}=\{x\pm iy, r\}=\left\{x\pm iy, x+\left(\frac{v_2}{v_1}+\sqrt{1+\left(\frac{v_2}{v_1}\right)^2}\right)y\right\}$} (since, in the upper half plane model here, all hyperbolic geodesics are simply Euclidean circles meeting the real axis along a diameter).

To see this is an isomorphism of geometries, we must further show the natural actions of $\PSL(2;\RR)$ are preserved by $\Phi$.
That is, we need that $A.\Phi(\{z,\overline{z},r\})=\Phi(A.\{z,\overline{z},r\})$ for all $A\in\PSL(2;\RR)$ and all points of $\RootMap(\Delta^-)$.
The action on $\RootMap(\CCubic)$ is given by equation \eqref{eqn:ProjActionSP3}, where $A$ acts on both $z, r$ by the same M\"obius transformation of $\CP^1$.
The action on $\UT\HH^2$ is given by the differential of its action on $\HH^2$ by isometries.
The compatibility of these actions follows immediately from the fact that geodesics (and thus their endpoints at infinity) are determined by their initial conditions.
That is, $\Phi(\{A.z,A.\overline{z},A.r\})$ is the tangent vector based at $A.z$ pointing in the direction of $A.r$, which is the image of the tangent vector $\Phi(\{z,\overline{z},r\})$ under $A$, by existence and uniqueness of solutions to the geodesic equation.

\end{proof}

\begin{figure}[h!tbp]
\centering
	\includegraphics[width=0.5\textwidth]{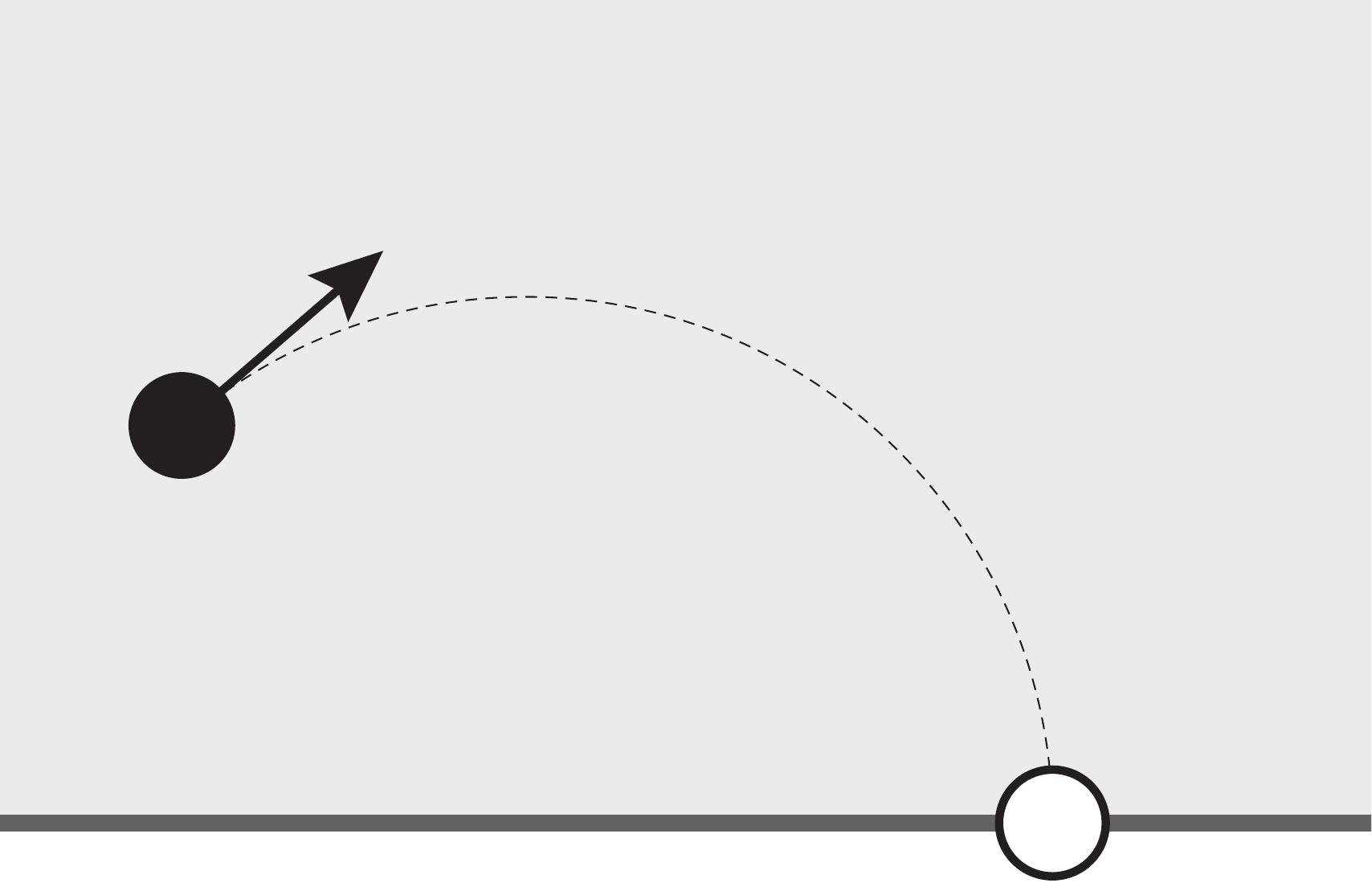}
	\caption{Cubic on upper halfplane (black point) and the real conjugate (white point), showing the geodesic between them and the direction from the complex root looking along the geodesic to the real root.}
\label{fig:cubic_tangent}
\end{figure}

Above we saw that the space $\mathcal{R}(\Delta^-)$ of roots of the cubics in $\Delta^-$ can be identified with the unit tangent bundle to $\HH^2$ by constructing a homeomorphism between these spaces which was equivariant with respect to the $\PSL(2;\RR)$ action on each side.  But, as we have already seen, the roots map $\mathcal{R}$ is itself a homeomorphism from $\Delta^-$ to $\mathcal{R}(\Delta^-)$, which is equivariant with respect to the $\PSL(2;\RR)$ actions on each.  Thus, this map defines an isomorphism of geometries $\Delta^-\cong\mathcal{R}(\Delta^-)$, and thus, transitively, between $\Delta^-$ and $\UT\mathbb{H}^2$.

\begin{corollary}

    The geometry of $\CCubic$ equipped with the $\PSL(2;\RR)$ action defined by Equation \ref{eqn:IrrepActionRP3} is isomorphic to the geometry of the unit tangent bundle to the hyperbolic plane.
    \label{prop:UT_Coefs}
    
\end{corollary}

From now on, we will denote this collection of polynomials as $\CCubic=\UT\HH^2_\Coefs$ and their associated roots as $\RootMap(\CCubic)=\UT\HH^2_\Roots$, to emphasize this geometric structure.
To be able to measure distances in these geometries (which right now are only spaces together with a group of symmetries) we need to specify a distance function on each of $\UT\HH^2_\Coefs$ and $\UT\HH^2_\Coefs$.  As in the case of the hyperbolic plane (Section \ref{ssec:RealQuad}) we specify this metric indirectly: first we define a norm on each tangent space (induced by an inner product: the Riemannian metric), then we define the length of curves by integrating the norm of their tangent vectors with respect to this, and finally we define the distance between two points to be the length of the shortest curve between them. See \cite{lee2019introduction} for an introductory account of Riemannian manifolds and their metrics.

\begin{definition}[The Metric on $\UT\HH^2_\Roots$]
\label{def:CubicRootMetric}

As a point in $\UT\HH^2_\Roots$ may be represented unambiguously as a pair $\{a,z\}$ for $a\in\RP^1$ and $z$ in the upper half plane, a tangent vector to $\UT\HH^2_\Roots$ at $\{a,z\}$ is a pair $\{u,v\}$ where $x\in\RR$ and $v\in\CC$.
Fixing the basepoint $\{0,i\}$, we define the norm squared of the tangent vector $\{u,v\}$ at $\{0,\pm i\}$ as
$$\|\{u,v\}\|^2=u^2+|v|^2.$$
That is, if we think of the tangent vector $(u,v_{\mathsf{Re}},v_{\mathsf{Im}})$ as a vector in $\RR^3$, we are employing the standard Euclidean inner product at the basepoint.
We use the free and transitive action of $\PSL(2;\RR)$ to translate this to every tangent space, equipping $\UT\HH^2_\Roots$ with a $\PSL(2;\RR)$ invariant Riemannian metric.

\end{definition}

\begin{definition}[The Metric on $\UT\HH^2_\Coefs$]
\label{def:CubicCoefMetric}
As a point in $\UT\HH^2_\Coefs$ may be represented unambiguously as a monic cubic $f$ (with coefficients the projective point $[1:a:b:c]$), a tangent vector to $\UT\HH^2_\Coefs$ at $f$ is an infinitesimal deformation of this cubic which leaves it monic.  That is, tangent vectors to $\UT\HH^2_\Coefs$ are represented by quadratic equations $u x^2+vx+w$ (we write this as $(u,v,w)$ when thinking of it as a tangent vector to the 3-dimensional affine patch $[1:a:b:c]$).
Fixing the basepoint $f=x(x^2+1)$, we define the norm squared of the tangent vector $(u,v,w)$ as 
$$\|(u,v,w)\|^2=\left(\frac{u-w}{2}\right)^2+\left(\frac{v}{2}\right)^2+w^2.$$
We use the free and transitive action of $\PSL(2;\RR)$ to translate this to every tangent space, equipping $\UT\HH^2_\Roots$ with a $\PSL(2;\RR)$ invariant Riemannian metric.
\end{definition}

We quickly comment on the form of the metric on $\UT\HH^2_\Coefs$, which is the natural choice\footnote{In fact, while any choice of inner product on the tangent space to our basepoint can be promoted to a Riemannian metric where $\PSL(2;\RR)$ acts by isometries, this metric is the most symmetric possible choice: its isometry group is 4-dimensional, whereas a generic inner product only leads to a 3-dimensional isometry group.}, with respect to the correct choice of affine patch.
We may embed the space of quadratics with complex roots into the space of cubics by taking a quadratic $f$ to the polynomial $xf(x)$ which has a unique real root at $x=0$.
Following the previous section, as the space of quadratics can be identified with the hyperbolic plane, we may expect this collection of cubics to look something like an embedding of a projective model of the hyperbolic plane.  And it does -- with respect to the original affine patch it appears as the \emph{paraboloid model} of the hyperbolic plane, but changing patch gives a round Klein disk, as in Figure \ref{fig:hypIso}: and it is with respect to an affine patch of this form that we take the metric to look Euclidean at the basepoint\footnote{More precisely, we take the Euclidean form of the metric on the affine patch where the following three curves of polynomials are orthogonal at the polynomial $x(x^2+1)$: (1)those with fixed imaginary root $i$, varying real root $t$, (2) those with fixed real root $0$, varying real part of complex root, and (3) those with real root $0$, varying imaginary part of complex root}.

These choices of Riemannian metrics will make clear the importance of hyperbolic geometry to the study of cubics.  Indeed, we will see in Proposition \ref{thm:cubicIsom} that these metrics are actually \emph{isometric to each other}, and that the roots map provides an isometry between them, in direct analogy to what the quadratic formula provided for $\mathbb{H}^2$.

To make use of this geometry in our analysis of cubics, we need to introduce some facts about the hyperbolic plane's tangent bundle.
First, the circle of unit tangent vectors at each point provides a foliation by circles: every point of $\UT\HH^2$ is the unit tangent vector to \emph{some unique point of $\HH^2$} and thus lies on a unique one of these circles.
But we may also define a collection of sections of $\HH^2\to\UT\HH^2$ of the unit tangent bundle such that every point lies on exactly one such hyperbolic plane. 
For a given ideal point $r\in\partial_\infty\HH^2$, at each $z\in\HH^2$ we select the unit vector $v_r\in T_z\HH^2$ which points\footnote{That is, choose $v$ so that the geodesic $\gamma_v$ with initial tangent $v$ at $z$ has $\lim_{t\to\infty+}\gamma(t)=r$.} to $r$.
This mapping $z\mapsto v_r\in T_z\HH^2$ provides one such section, and varying $r\in\partial_\infty\HH^2$ foliates the unit tangent bundle with translates of this.

It is straightforward to give an explicit example of one of each of these, passing through the basepoint for the space.  Considering $\UT\HH^2_\Coefs$ with $f=x(x^2+1)$ as the basepoint (the corresponding statements are equally true for $\UT\HH^2_\Roots$ and $\{0,\pm i\}$ as the basepoint), the circle fiber through $f$ is just all cubics in $\UT\HH^2_\Coefs$ with $x^2+1$ as their irreducible quadratic factor, and the hyperbolic plane through $f$ is all such cubics having $x$ as their linear factor.
Using the metric from Definition \ref{def:CubicCoefMetric}, these two spaces are seen to intersect each other orthogonally at $f$.
But as the metric on the entire space is built from the metric at this basepoint, we conclude that the same behavior is observed at every point.
We note this precisely in the following observation.

\begin{observation}	
\label{obs:Foliation_Roots}
The foliation of $\UT\HH^2_\Roots$ by circles (the fibers of the bundle) correspond to collections of cubics with a fixed complex root $\RP^1_{z}=\{\{a,z,\overline{z}\}\mid a\in\RP^1\}$, as the real root varies.
Conversely, fixing the real root $a\in\RP^1$ and letting the complex conjugate roots vary over $\CC\backslash\RR$ gives the foliation by disks $\HH^2_a=\{\{a,z,\overline{z}\}\mid z\not\in\RR\}$.
with respect to the Riemannian metric on $\UT\HH^2_\Roots$, these two foliations are orthogonal at every point.
The first consists of geodesic circles all of length $2\pi$, and the second of isometrically embedded hyperbolic planes.
\end{observation}

Together, these foliations provide a \emph{trivialization} of the unit tangent bundle: an explicit choice of homeomorphism\footnote{Note however this homeomorphism is \emph{not} an isometry of the metrics we have defined with the product metric on $\HH^2\times\RR$.} $\UT\HH^2\to\HH^2\times\SS^1$, sending each $v\in \UT\HH^2$ to its basepoint $z\in \HH^2$, and the angle $\theta$ that $v$ makes with respect to the direction field associated with some fixed ideal point.
Working in the upper half plane model, it is easiest to choose this as the direction field associated to $\infty$; as in the euclidean coordinates $z=x+iy$ this is simply the direction made with the vertical at every point.
Explicitly, the point $(z,r)$ in the space of roots, identified with the unit tangent vector $v\in T_z\HH^2$, is sent to $(z,\theta)$  for 
\begin{equation}\theta=2\arctan\left(\frac{r-x}{y}\right).
\label{eqn:Triv}
\end{equation}
\begin{remark}
\label{rem:Triv}
Using a M\"obius transformation to send the upper half plane to the unit disk, we may represent the space of roots as the interior of a solid torus of revolution in $\RR^3$, as in Figure \ref{fig:UTRoots}.  
This homeomorphism provides the beautiful pictures visible in Figure \ref{fig:cubicsin3d} produced by David Dumas' wonderful program.
\end{remark}

This geometric perspective provides a nice way of thinking about the space of roots, avoiding the complicated space $\SP^3(\CP^1)$ in which it was originally defined.
The upper half plane model of $\HH^2$ is a subset of $\CP^1$, so the map $\Phi$ in Proposition \ref{prop:UTH2} above explicitly identifies $\RootMap(\CCubic)$ with a subset of $\UT\CP^1$.
But, as $\CP^1$ is topologically the 2-sphere (Remark \ref{rem:FP1}), $\UT\CP^1$ is just the unit tangent bundle of the sphere\footnote{One shows the isometries of $\SS^2$ act freely and transitively on the unit tangent bundle, which then provides a diffeomorphism from $\mathrm{Isom}(\SS^2)=\SO(3)$ and $\UT\SS^2$. 
Finally, we recall that $\SO(3)$ is topologically real projective 3-space (for example, by noting that it is double covered by $\mathrm{SU}(2)\cong\SS^3$).}, which is topologically $\RP^3$.
As the space of coefficients is naturally a subset of $\RP^3$, this provides a uniform means of drawing both spaces, see Figures \ref{fig:UTRoots_CP1} and \ref{fig:UTCoefs}.

\begin{figure}[h!tbp]
\centering
\begin{subfigure}[b]{0.48\textwidth} 
    \centering
	\includegraphics[width=\textwidth]{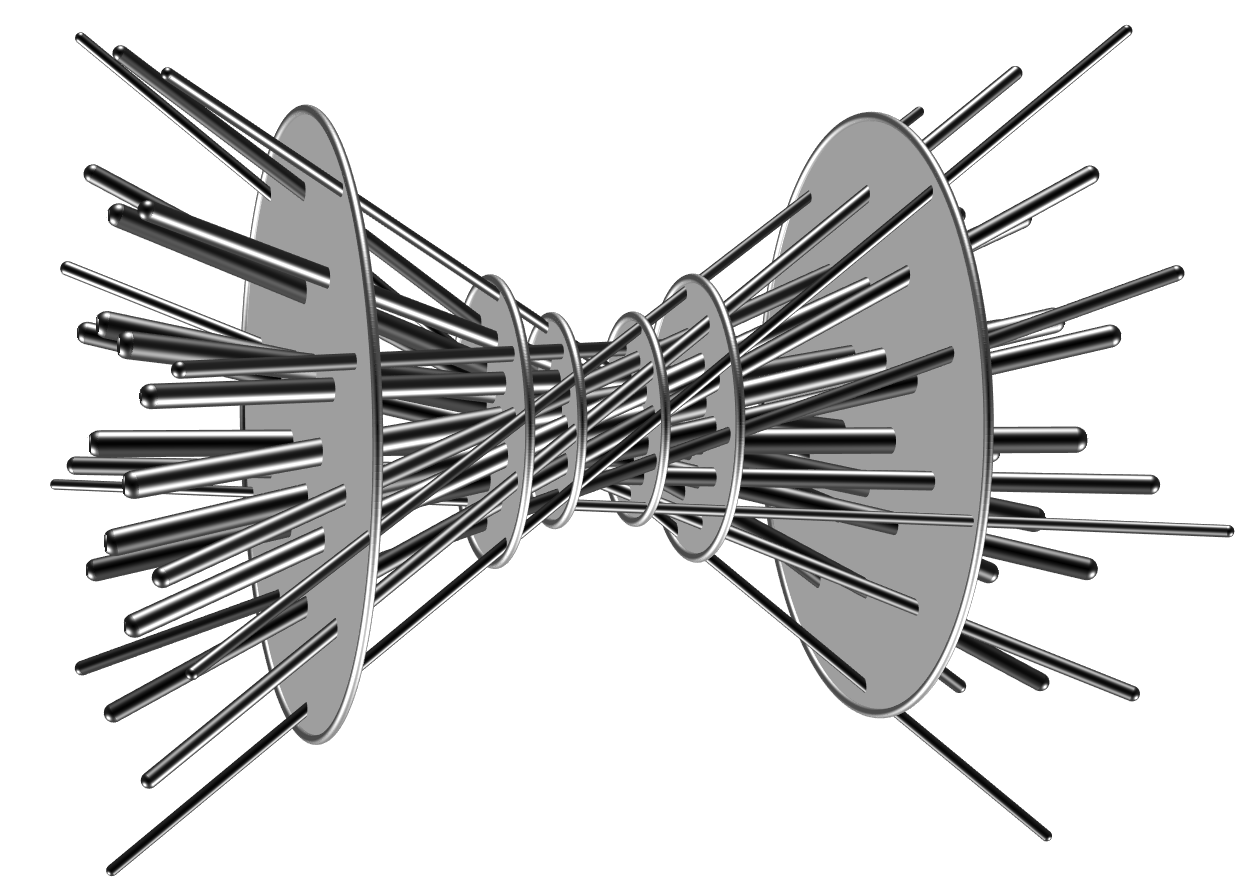}
	\caption{}
    \label{fig:UTRoots_CP1}
\end{subfigure}
\begin{subfigure}[b]{0.4\textwidth} 
    \centering
	\includegraphics[width=\textwidth]{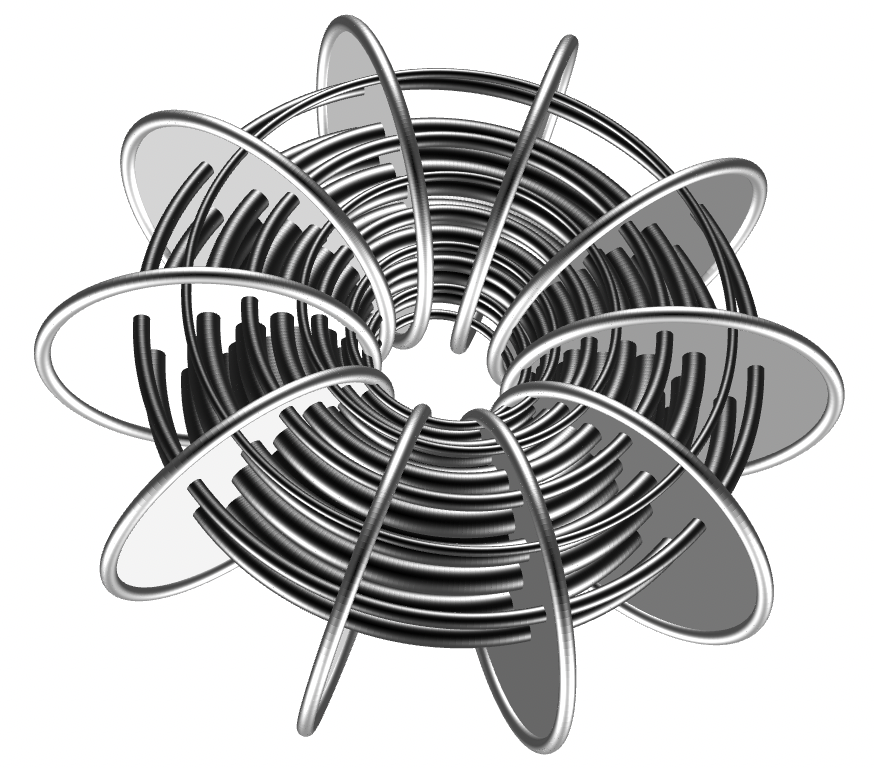}
	\caption{}
    \label{fig:UTRoots}
\end{subfigure}
	\caption{Foliation of $\UT\HH^2_\Roots$ by circles and hyperbolic planes.
In (A) we view $\UT\HH^2_\Roots$ in an affine patch of  $\UT\CP^1=\RP^3$, where the hyperbolic foliation is visible as parallel disks.
	In (B), it is realized as the interior of a torus of revolution in $\RR^3$ via Remark \ref{rem:Triv}.}
\label{fig:}
\end{figure}

\begin{remark}
	This also provides the means to represent cubics literally as unit tangent vectors in $\HH^2$: As $\theta$ is measured with respect to geodesics limiting to $\infty$ in the upper half plane, we may depict the cubic with roots $\{z,\overline{z},r\}$ by the unit vector at $z$ pointed along the geodesic\footnote{Using the coordinates $z=x+iy$ on the upper half plane, this vector is in the direction $(2y(r-x),(r-x)^2-y^2)$.} to $r$.  As this depicts a 3-dimensional space using 2-dimensions, we cannot understand the entire space of cubics this way - however it provides a useful means of looking at 2-dimensional families, such as Figure \ref{fig:Arrows}.
\end{remark}

\begin{figure}[h!tbp]
\centering
    \begin{subfigure}[b]{0.42\textwidth} 
        \centering
    	\includegraphics[width=\textwidth]{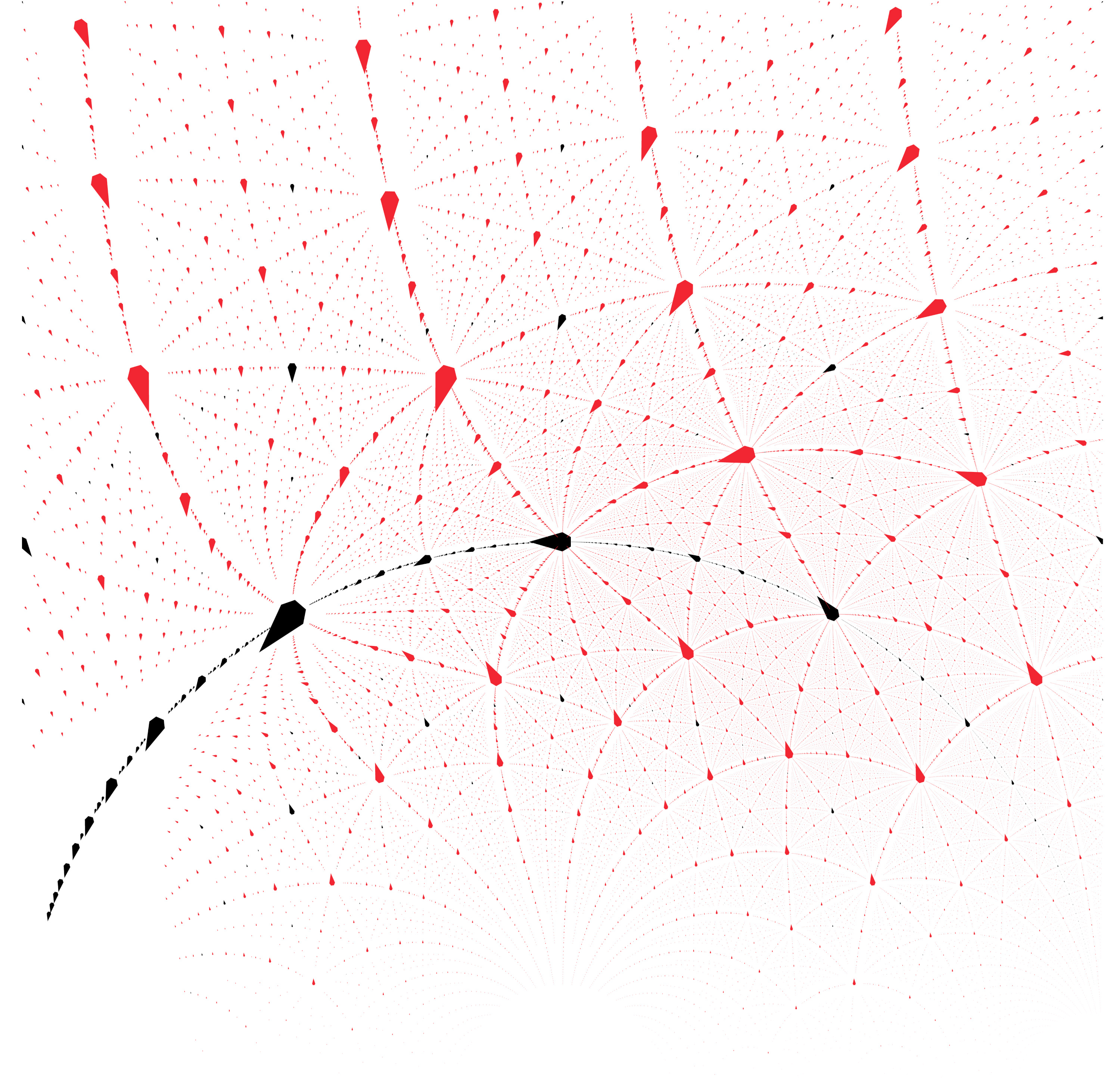}
    	\caption{}
        \label{fig:ArrowsACBC}
    \end{subfigure}
    \begin{subfigure}[b]{0.42\textwidth} 
        \centering
    	\includegraphics[width=\textwidth]{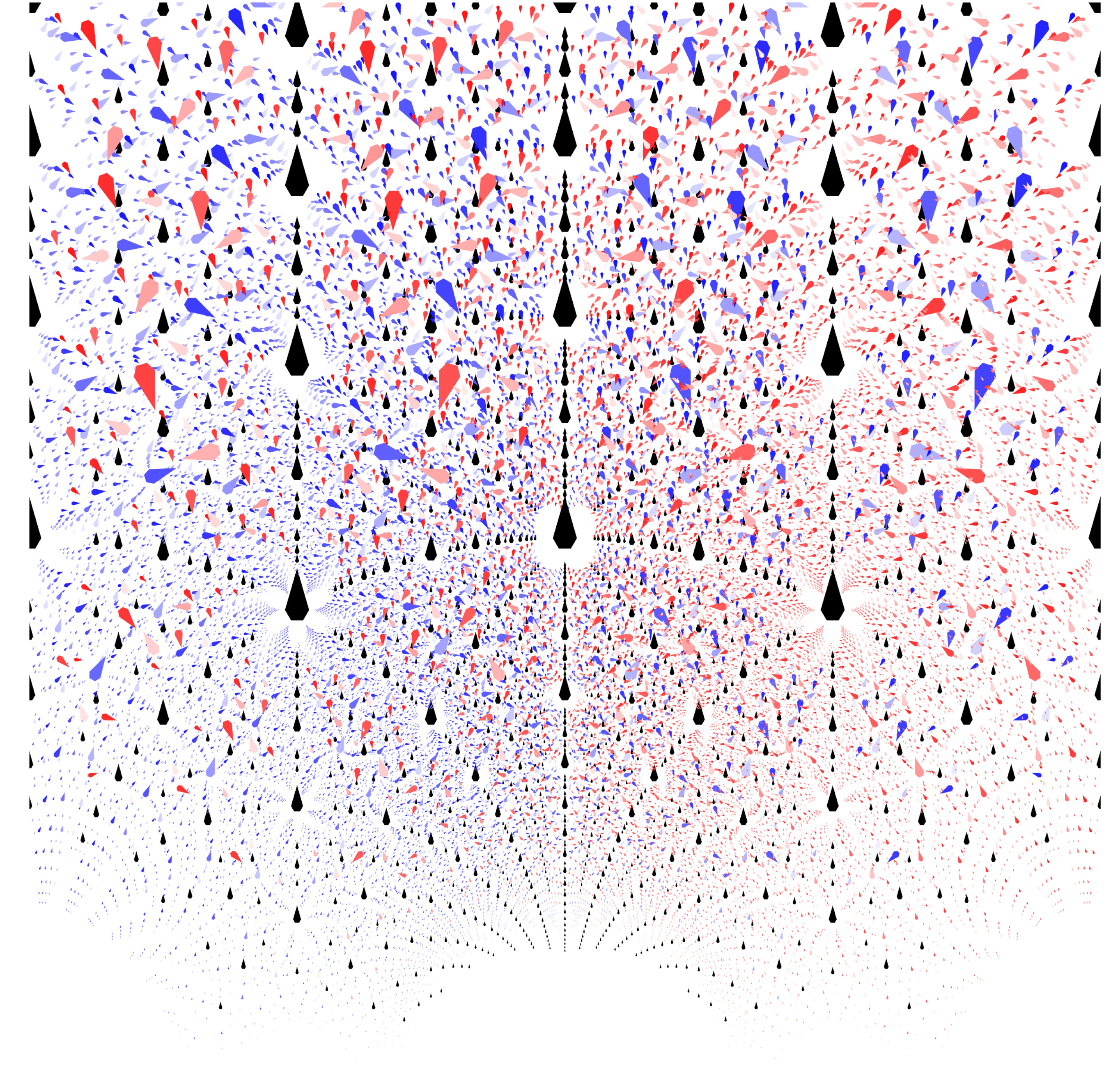}
    	\caption{}
        \label{fig:ArrowsCubics}
    \end{subfigure}
	\caption{Cubic numbers as a vector field, as shown in Figure \ref{fig:cubic_tangent}. The roots of the polynomial family $a x^3 + b x^2 + c x + a$ (\ref{fig:ArrowsACBC}), and all cubics (\ref{fig:ArrowsCubics}).}
\label{fig:Arrows}
\end{figure}

We next turn to the description of this geometry on the space of coefficients, which we likewise denote $\UT\HH^2_\Coefs$.
Here, the surface cut out by the discriminant represents the ideal boundary of the geometry, whose points correspond to the polynomials with negative discriminant, depicted in Figure \ref{fig:NegDisc}.
The important geometric information is a description of the $\SS^1$ fibers, which provide the structure of the unit tangent bundle, and a choice of section $\HH^2\to\UT\HH^2_\Coefs$ giving a trivialization.
Via Observation \ref{obs:Foliation_Roots}, the choice made in the roots model has a convenient description in terms of their corresponding polynomials, which we summarize below.

\begin{observation}	
\label{obs:Foliation_Coefs}
The trivialization in Observation \ref{obs:Foliation_Roots} is expressed in the space of coefficients as:
\begin{itemize}
	\item Each fiber of the $\SS^1$ foliation passes through a unique cubic with real root at infinity, identified with the quadratic $ax^2+bx+c\in\HH^2$.
The fiber passing through this point is parameterized by $[a:b-ra:c-rb:rc]$ in the space of coefficients, for $r\in\RP^1=\RR\cup\{\infty\}$.
\item Each fiber of the $\HH^2$ foliation has constant real root $r\in\RP^1$.
This fiber is parameterized by $[1:-2u-r:2ru-u^2-v^2:-ru^2-rv^2]$ for $u,v\in\R, v>0$ representing the varying complex root $u+iv\in\HH^2$.
\end{itemize}
Exactly as in Observation \ref{obs:Foliation_Roots}, with respect to the choice of metric in Definition \ref{def:CubicCoefMetric}, these two spaces of fibers are orthogonal at every point of intersection.
Every circle fiber has metric length $2\pi$ and every plane in the $\HH^2$ foliation is an isometrically embedded hyperbolic plane (thus justifying the name).
\end{observation}

As in the quadratic case, the natural geometry is more difficult to see at first from the coefficient perspective.
However, drawing $\UT\HH^2_\Coefs$ in the affine patch which puts quadratics at infinity, we see the hyperbolic foliation consists of copies of the now familiar \emph{parabola model} of $\HH^2$.
Choosing other affine patches may render (some of) these as copies of the more familiar Klein disk model.
These two foliations of the space of cubics, by hyperbolic planes and by circles, provide convenient means of keeping track of information about cubics.
We see some examples of this below as we seek a geometric understanding of the roots map.

\begin{figure}[h!tbp]
\centering
	\includegraphics[width=0.65\textwidth]{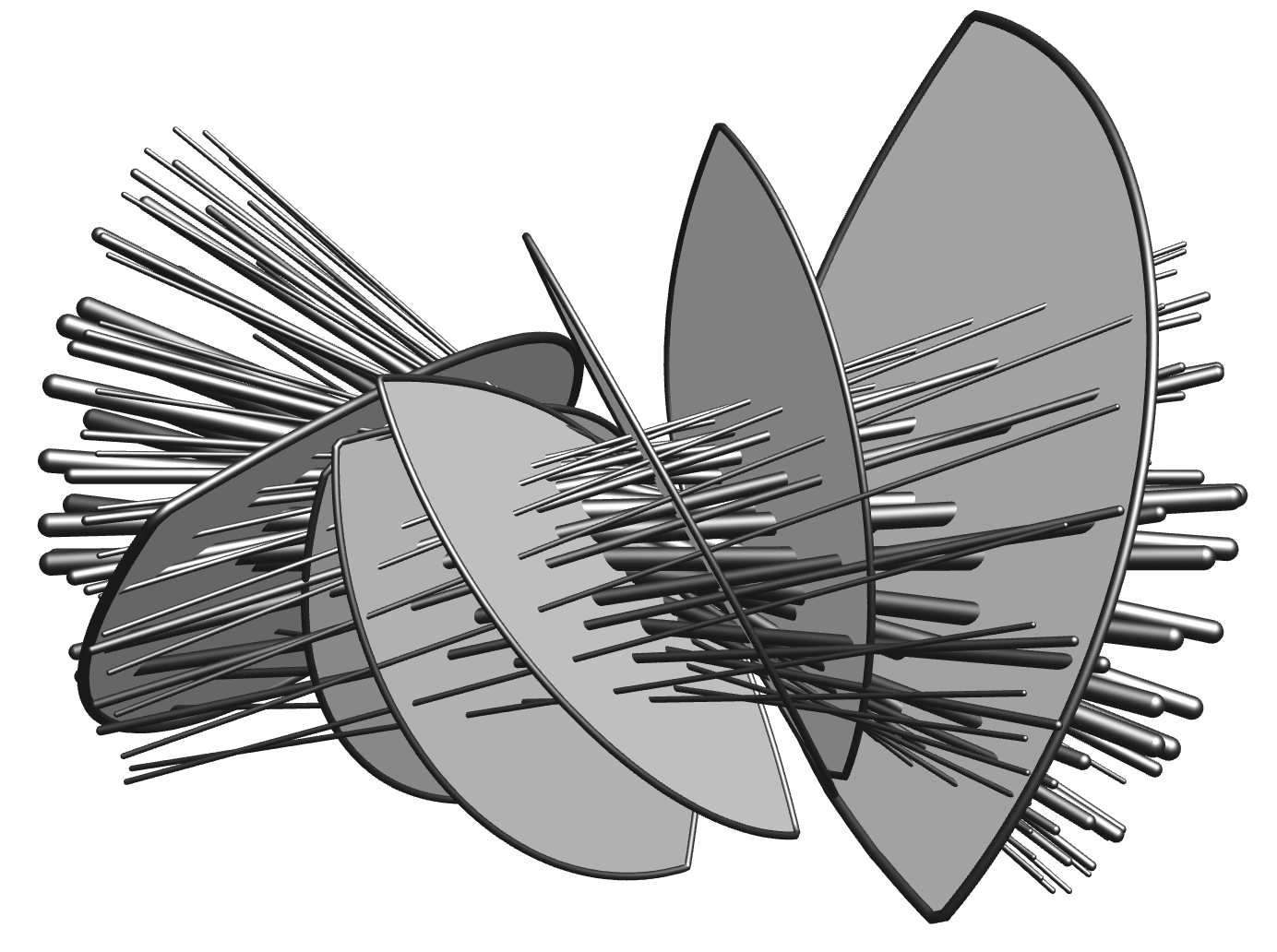}
\caption{Foliation of $\UT\HH^2_\Coefs$ by hyperbolic planes and circle fibers
}
\label{fig:UTCoefs}
\end{figure}

\subsubsection{Geometry of the Roots Map}
The roots map $\RootMap\colon\PP\Coefs\to\Roots$ restricts on the space of real cubics of negative discriminant to a homeomorphism $\UT\HH^2_\Coefs\to\UT\HH^2_\Roots$ equivariant with respect to the $\PSL(2;\RR)$ actions given by equations \eqref{eqn:ProjActionSP3} and \eqref{eqn:IrrepActionRP3}.
We noted earlier (Proposition \ref{prop:UTH2} and Corollary \ref{prop:UT_Coefs}) that this implies $\RootMap$ is an isomorphism of geometries in the sense of Klein.
Following this, we saw that both the space of roots and coefficients can be equipped with a natural choice of Riemannian metric in definitions \ref{def:CubicRootMetric} and \ref{def:CubicCoefMetric}.
We now strengthen our original proposition, and show that with respect to these two metrics in fact $\mathcal{R}$ is an isometry.
Because computing the actual geodesic metric distance here is quite challenging (see \cite{Divjak:2009aa} for a computation of the geodesic curves), we adopt a different approach than our proof of Theorem \ref{thm:hypIso} for quadratics, and work locally, leveraging the equivariance of $\mathcal{R}$ with respect to the group actions to reduce the problem to showing $\mathcal{R}$ induces an isometry of tangent spaces at a single point.
This technique relies on a useful lemma of Riemannian geometry, stated below.

\begin{lemma}
Let $(X,g_X)$ and $Y,g_Y)$ be Riemannian manifolds, each equipped with a transitive action of some Lie group $G$ by isometries.
Suppose further that $f\colon X\to Y$ is a diffeomorphism which is equivariant with respect to these $G$ actions.
Then, $f$ preserves the inner product at any point $p\in X$, it is an isometry.
\end{lemma}

The proof of this lemma is a straightforward computation using the fact that $G$ acts on both sides via isometries.  See again \cite{lee2019introduction} for an introduction to the tools utilized in such arguments.  We sketch the proof below.

\emph{Sketch.}
Let $p$ be any point in $X$. We wish to see that the map $f$ preserves the inner product at $p$.  But using the homogeneity of the $G$ action we can choose some isometry taking the fixed point $x$ in the theorem statement to $p$.
Then using equivariance of the $G$ action with respect to $f$, we see this same isometry element takes $f(x)$ to $f(p)$.
We can use this isometry to transfer any local computation at $p$ to a local computation at $x$, and similarly for $f(p)$ and $f(x)$.
Putting this all together, we see that our map $f$ preserves the inner product at $p$ if and only if it does at $x$.  But this latter condition was precisely our assumption: thus $f$ is a Riemannian isometry.

We now use this result to prove the proposition of interest.

\begin{proposition}
\label{thm:cubicIsom}
Let $\UT\HH^2_\Coefs$ be the set of real cubics with exactly one real root, $\UT \HH^2=\{[a:b:c:d]\mid \Delta_3(a,b,c,d)<0\}$, and $\UT\HH^2_\Roots=\{\{r,z,\overline{z}\}\mid r\in\RP^1,z\in \CC\backslash\RR\}$ be the set of their root-sets.
Equipping each with the $\PSL(2;\RR)$ metric identifying them with the unit tangent bundle to the hyperbolic plane, the roots map $\RootMap\colon\UT\HH^2_\Coefs\to\UT\HH^2_\Roots$ is an isometry.	
\end{proposition}
\begin{proof}
By the above proposition and the equivariance of the $\PSL(2;\RR)$ actions by isometries, it is enough to show that $\mathcal{R}$ defines a Riemannian isometry at any single point.
We choose to compute at the basepoint $f=x(x^2+1)$ in $\UT\HH^2_\Coefs$ and its image $\RootMap(f)=\{0,\pm i\}$ in $\UT\HH^2_\Roots$.
Because writing an actual formula for the roots map in degree 3, we opt instead to work with its inverse (if $\RootMap^{-1}$ is a Riemannian isometry then so is $\mathcal{R}$).
This inverse takes the point $\{r,z,\overline{z}\}$ to the polynomial that has $r$ and $z,\overline{z}$ as roots: written in the affine patch of monic cubics, this has the explicit formula
$$\mathcal{R}^{-1}(\{r,z,\overline{z}\})=(x-r)(x-z)(x-\overline{z})=(x-r)(x^2+2\mathrm{Re}(z)+|z|^2).$$

Let $\{u,v+iw\}$ denote a tangent vector to $\{0,\pm i\}$ as in Definition \ref{def:CubicRootMetric}.
We may realize this tangent vector as the derivative at $t=0$ of the path $\{tu, i+t(v+iw)\}$.
Thus we may compute the result of applying $\mathcal{R}^{-1}$ to this tangent vector by looking at the path $\RootMap^{-1}(\{tu, i+t(v+iw)\})$ and taking its derivative in the space of coefficients.
Performing this computation, we arrive at the path of polynomials
$$x^3+( 
    2 t v-t u ) x^2 +(1 - 2 t^2 u v + t^2 v^2 + 2 t w + t^2 w^2) x-(t u + t^3 u v^2 + 2 t^2 u w + t^3 u w^2)$$
differentiating at $t=0$ gives the tangent vector, realized a quadratic polynomial as in Definition \ref{def:CubicCoefMetric}.
$$(2 v-u) x^2 + 2 w x - u$$
Having these two vectors on hand, the rest of the proof is a direct computation: we compute the norm of each with respect to the Riemannian metric on each respective space, and then see the results are equal.
First, in the space of roots, using Definition \ref{def:CubicRootMetric} we see the norm square of $\{u,v\pm iw\}$ is simply its Euclidean value:
$$\|\{u,v\pm iw\}\|^2_{\Roots}=u^2+v^2+w^2.$$
Performing the analogous computation in coefficient space, we take the coefficient vector $(-(u+2v),2w,-u)$ of the quadratic above, and apply the norm squared of Definition \ref{def:CubicCoefMetric}:
\begin{align*}\|((2v-u),2w,-u)\|^2_\Coefs&=\left(\frac{(2v-u)+u}{2}\right)^2+\left(\frac{2w}{2}\right)^2+\left(-u\right)^2\\
&=v^2+w^2+u^2
\end{align*}
Thus $\RootMap^{-1}$ induces an isometry between the tangent spaces $T_{\{0,\pm i\}}\UT\HH^2_\Roots$ to the roots and $T_{x(x^2+1)}\UT\HH^2_\Coefs$ to the coefficients at our chosen basepoint, as required.
\end{proof}

After the right preliminary work, the proof of our main theorem, (the analog of Theorem~\ref{thm:hypIso} but for cubics) reduced to checking a simple computation at the basepoint.  But, we can say even more than this, and relate exactly how this isometry acts with respect to the foilations of the space of cubics by circles and hyperbolic planes (Observation \ref{obs:Foliation_Roots}).

\begin{theorem}
\label{thm:CubicFactor}
The roots map $\RootMap\colon\PP\Coefs\to\Roots$, restricted to the space $\UT\HH^2_\Coefs$ of real cubics with negative discriminant, factors as $\RootMap_3=(\RootMap_1,\RootMap_2)\circ(\pi_1,\pi_2)$ where $\pi=(\pi_1,\pi_2)$ is the trivialization of the unit tangent bundle given by the foliations by circles, and hyperbolic planes, respectively, in Observation \ref{obs:Foliation_Coefs}, and $\RootMap_1,\RootMap_2$ are the root maps for linear and quadratic polynomials.
This is best seen diagrammatically: compare the diagram below with  Figure \ref{fig:cubicrootmap}.
\begin{center}
\begin{tikzcd}
\mathbb{P}\mathrm{Coefs} \arrow[d, "\cong" description]                           &             & \mathbb{RP}^1_\mathrm{Coefs} \arrow[r, "\mathcal{R}_1"] & \mathbb{RP}^1_\mathrm{Roots} \arrow[rrd]        &&                                                                       \\
\mathrm{UT}\mathbb{H}^2_\mathrm{Coefs} \arrow[rru, "\pi_1"] \arrow[rrd, "\pi_2"'] & f \arrow[r] & \begin{pmatrix}f_1\\f_2\end{pmatrix} \arrow[r]          & \begin{pmatrix}r\\z\end{pmatrix} \arrow[r] & {\{r,z,\overline{z}\}} & \mathrm{UT}\mathbb{H}^2_\mathrm{Roots} \arrow[d, "\cong" description] \\
 &             & \mathbb{H}^2_\mathrm{Coefs} \arrow[r, "\mathcal{R}_2"'] & \mathbb{H}^2_\mathrm{Roots} \arrow[rru]         &                             & \mathrm{Roots}                                                       
\end{tikzcd}
\end{center}
\end{theorem}
\begin{proof}

This is not deep, and follows directly from our geometric interpretations of the space of quadratic and linear polynomials.
As isometries of $\UT\HH^2$ preserve the fiber bundle structure, $\RootMap$ preserves the $\RP^1$ foliation, and similarly the $\HH^2$ foliation described in Observations \ref{obs:Foliation_Roots} and \ref{obs:Foliation_Coefs}.
By Observation \ref{obs:Foliation_Roots}, we may identify these foliations with the preimages of projections $\pi_1,\pi_2$ onto the linear and irreducible quadratic factors respectively.
As these are preserved by the roots map, $\RootMap_3$ factors through the trivialization $\pi$ to a pair of maps sending these factors to their roots.
But these are already familiar: they are just the lower dimensional roots maps $\RootMap_1$ and $\RootMap_2$.
	
\end{proof}

The bottom row of this diagram gives the formula for the complex root as in Theorem \ref{thm:cubicMain}; the top row gives the analog returning the real root.
Choosing to identify the base $\HH^2_\Coefs$ with one of the sections of Observation \ref{obs:Foliation_Coefs}, one may see the solution to the cubic as a \emph{process} as follows.
(For simplicity of exposition, we identify it with the fiber of cubics with real root zero here.)
Starting with a point $f\in\UT\HH^2_\Coefs$ (recall Figure \ref{fig:UTCoefs} as a visual aid here), we slide along the fiber through $f$ until reaching the hyperbolic sheet specified by the section.
The real root is given by the distance traveled along this fiber, and the complex root is given by taking the resulting point, which is now a cubic of the form $(ax^2+bx+c)x$ and applying the isometry from the parabola model of $\HH^2$ in $[a:b:c]$ to the upper half plane.

\begin{figure}[h!tbp]
\centering
\includegraphics[width=0.9\textwidth]{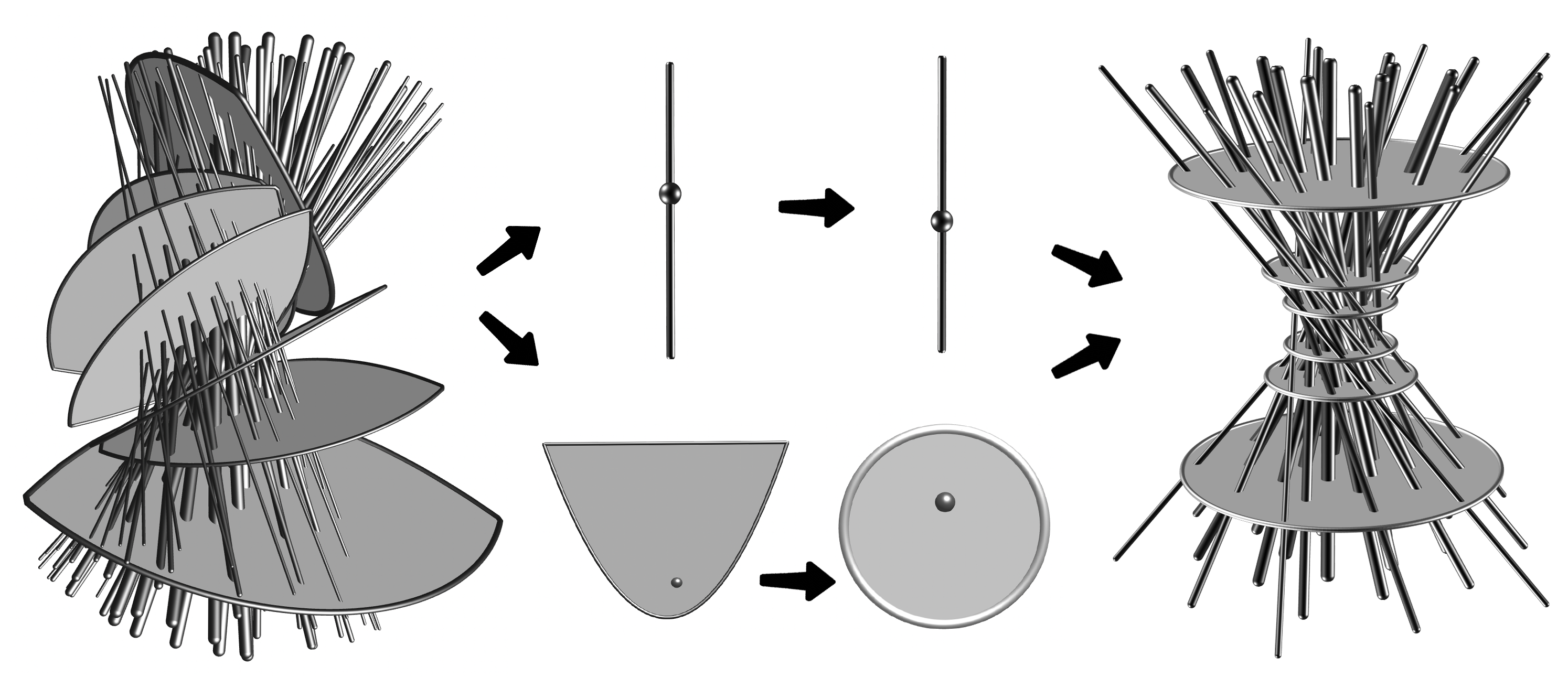}
\caption{Geometric factoring of the roots map $\RootMap_3\colon\UT\HH^2_\Coefs\to\UT\HH^2_\Roots$ of Theorem \ref{thm:CubicFactor}, using the root maps $\RootMap_1,\RootMap_2$ of lower degrees.
On the left, the space $\UT\HH^2_\Coefs\subset\RP^3$ is depicted together with the trivialization determined by fixing a real or complex root.  On the right is the space $\UT\HH^2_\Roots\subset\UT\CP^1$ of unordered roots with the same trivialization. 
Compare with the diagram in Theorem \ref{thm:CubicFactor}.
}
\label{fig:cubicrootmap}
\end{figure}

\begin{remark}
One may derive an expression for the cubic formula from this procedure, similarly to what was done in Remark \ref{rem:QuadFormula}, using the parameterization of the $\RP^1$ and $\HH^2$ factors in Observation \ref{obs:Foliation_Coefs} and attempting to invert their dependence of the coefficients $[a:b:c:d]$ on the parameters $u,v$.
Of course, the inherent messiness of the cubic formula must make this challenging at some point, which we can now specify: it lies in giving explicit formulas for the projections of $\UT\HH^2_\Coefs$ onto its foliations\footnote{In light of Theorem \ref{thm:CubicFactor} it cannot be anywhere else: as the only remaining portion to the cubic formula from this perspective is to solve the associated quadratic and linear equations: both of which have simple roots maps as we have seen before.}.
\end{remark}

\subsubsection{Applications to Cubic Numbers}
\label{sec:applications-cubic}

When reasoning about cubic numbers, we often want to deal just with the points in $\CC$ themselves, and not the abstract space of roots.
Thus, from the perspective of polynomials, we are interested in the image of the \emph{projection} $\UT\HH^2_\Roots\to\HH^2_\Roots\subset\CC$, not the space of roots itself.
This is yet another place where the circle and hyperbolic fibrations of the space of cubics are important, as they are the kernel and $1$-eigenspaces of the differential of this projection, respectively.

One place this may arise is in trying to bound distances between cubic numbers.
Thus, to control distances between cubic numbers in terms of their minimal polynomials, we do not need to use a full expression\footnote{While the Riemannian metric here is easy to describe by translating the standard euclidean metric on some tangent space to a point $p\in\UT\HH^2_\Coefs\subset\RP^3$ around by action given by the representation $\rho$, its expression is complicated, making the computation of a distance function unwieldy.}
 for the unit tangent bundle metric on $\UT\HH^2_\Coefs$ but rather just a means of measuring the hyperbolic distance between their projections to the base $\HH^2$.
 While abstractly this is given exactly by the pull back of the metric on $\HH^2_\Coefs$ by the bundle projection $\UT\HH^2\to\HH^2$, any sufficiently simple expression for this would hopefully make some of the analysis of Section \ref{sec:Dio-quad} in the quadratic case extendable to cubics.

Another place this arises is in the study of one and two parameter families of cubics.
Let $F\colon\R^n\to\PP\Coefs$ be a smooth map for $n\in\{1,2\}$ tracing out a submanifold of the space of cubics.
A natural question for the production of good images, is when does the result of drawing the complex root of each polynomial in the family produce a coherent image in $\CC$: that is, when is the composition $\RootMap_\CC\circ F$ a homeomorphism onto its image?
Given that on the space of roots, the projection onto the complex root is directly collapsing the $\RP^1_\Roots$ factor, we recall $\RootMap_\CC$ is the projection along the $\RP^1$ fibers of Observation \ref{obs:Foliation_Coefs}, followed by an isometry.
This gives an explicit condition on $F$.

\begin{observation}
	\label{obs:transverse}
Let $F\colon \RR^n\to\UT\HH^2_\Coefs$ be a family of real cubics in the space of projectivized coefficients, for $n\in\{1,2\}$.
Then the map $\RootMap_\CC\circ F$ sending a cubic to its complex root is an embedding if $F$ is everywhere transverse to the $\RP^1$ fibration of $\UT\HH^2_\Coefs$.\end{observation}

Restricting attention to planar starscapes, the relevant families $F\colon\R^2\to\RP^3$ 
are those whose image is some affine plane $\mathcal{S}\subset\RP^3$ (that is, the projectivization of some $\RR^3$ through the origin) intersected with the space of cubics of negative discriminant, as in Corollary \ref{cor:StarscapeEmbed1} highlighted in the introduction to this section.
Examples include the accompanying figure (\ref{fig:Lines}) as well as the gallery Figures \ref{fig:Starscapes_a}, \ref{fig:Starscapes_b} and \ref{fig:AffineStarscapes_a}--\ref{fig:AffineStarscapes_d}.
The fiber above any point $r\in \CC$ is a line $(\RootMap_2\pi_2)^{-1}(r)=L_r$ (parameterized as in Observation \ref{obs:Foliation_Coefs}), thus any affine plane $\mathcal{S}\subset\UT\HH^2_\Coefs$ containing $L_r$ fails to embed under projection onto the complex root.
As we may find an affine plane containing any projective line we like, this can happen at any point in the upper half plane.
The fact that this can only happen at a single point in a single starscape is implied by the fact that the fibers of $\UT\HH^2_\Coefs$ are pairwise skew lines in $\RP^3$, and thus no two are ever contained in a plane.

\begin{corollary}
	\label{cor:transverse}
Let $\mathcal{S}\subset\UT\HH^2_\Coefs$ be an affine 2-dimensional projective subspace, and $\RootMap|_\mathcal{S}\colon\mathcal{S}\to\CC$ be the projection onto the complex root.
Then $\RootMap|_\mathcal{S}$ is singular along at most one projective line $L\subset\mathcal{S}$.
	Furthermore, for a planar starscape (i.e. $\mathcal{S}$ has rational normal vector), this line, if it exists, corresponds to a rational polynomial with a linear factor over $\QQ$ (i.e. a quadratic point in the complex plane).  Finally, every such polynomial gives a singularity for $\RootMap_\mathcal{S}$ for some planar starscape $\mathcal{S}$.
\end{corollary}

To see the statement that the line lies over a quadratic point, suppose the projective line corresponds to roots $\{r, z, \overline{z}\}$, where $r$ is the varying real root, while $z$ and $\overline{z}$ are fixed.  Then the projective line has an expression as
\begin{equation}
	\label{eqn:L}
	[ 1 : - z - \overline{z} - r : z\overline{z} + (z + \overline{z}) r : - z\overline{z} r ].
\end{equation}
If this lies in a planar starscape, then it lies in a projective plane with rational normal, i.e. there is a rational vector $[r:s:t:u]$ such that
\[
	\begin{pmatrix}
	1 & - z - \overline{z} & z\overline{z} & 0 \\
	0 & 1 & -z - \overline{z} & z\overline{z} \\
	\end{pmatrix}
	\begin{pmatrix} r \\ s \\ t \\ u \end{pmatrix} = 0.
\]
In particular, this implies $1$, $z + \overline{z}$ and $z \overline{z}$ are $\QQ$-linearly dependent.  If they have rank two, then this implies that $[r:s:t] = [s:t:u]$, which implies $[r:s:t:u]$ has the form $[1:x:x^2:x^3]$; this further implies $z$ is rational (as the projective line in coefficient space has a fixed rational root $x$).  If they have rank one, then $z$ lies on the intersection of two rational geodesics, i.e. is quadratic (see Observation \ref{obs:rat-geo}).  

For the final statement, suppose the projective line $L$ is as in \eqref{eqn:L} above.  Then it lies on infinitely many rational geodesics; taking any two will span an appropriate $\mathcal{S}$.

\begin{figure}[h!tbp]
\centering
\centering
\begin{subfigure}[b]{0.45\textwidth} 
    \centering
	\includegraphics[width=\textwidth]{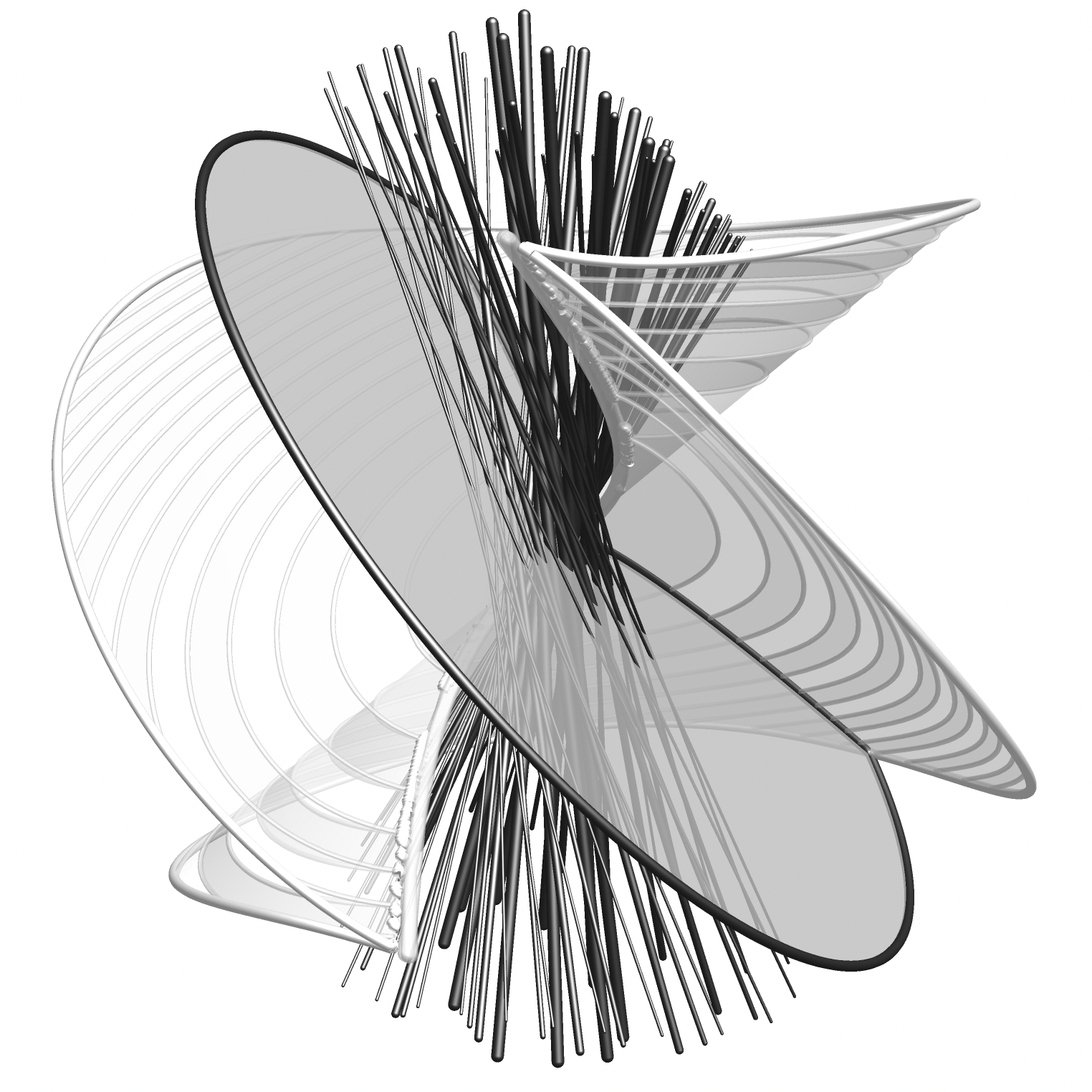}
	\\
	\includegraphics[width=\textwidth]{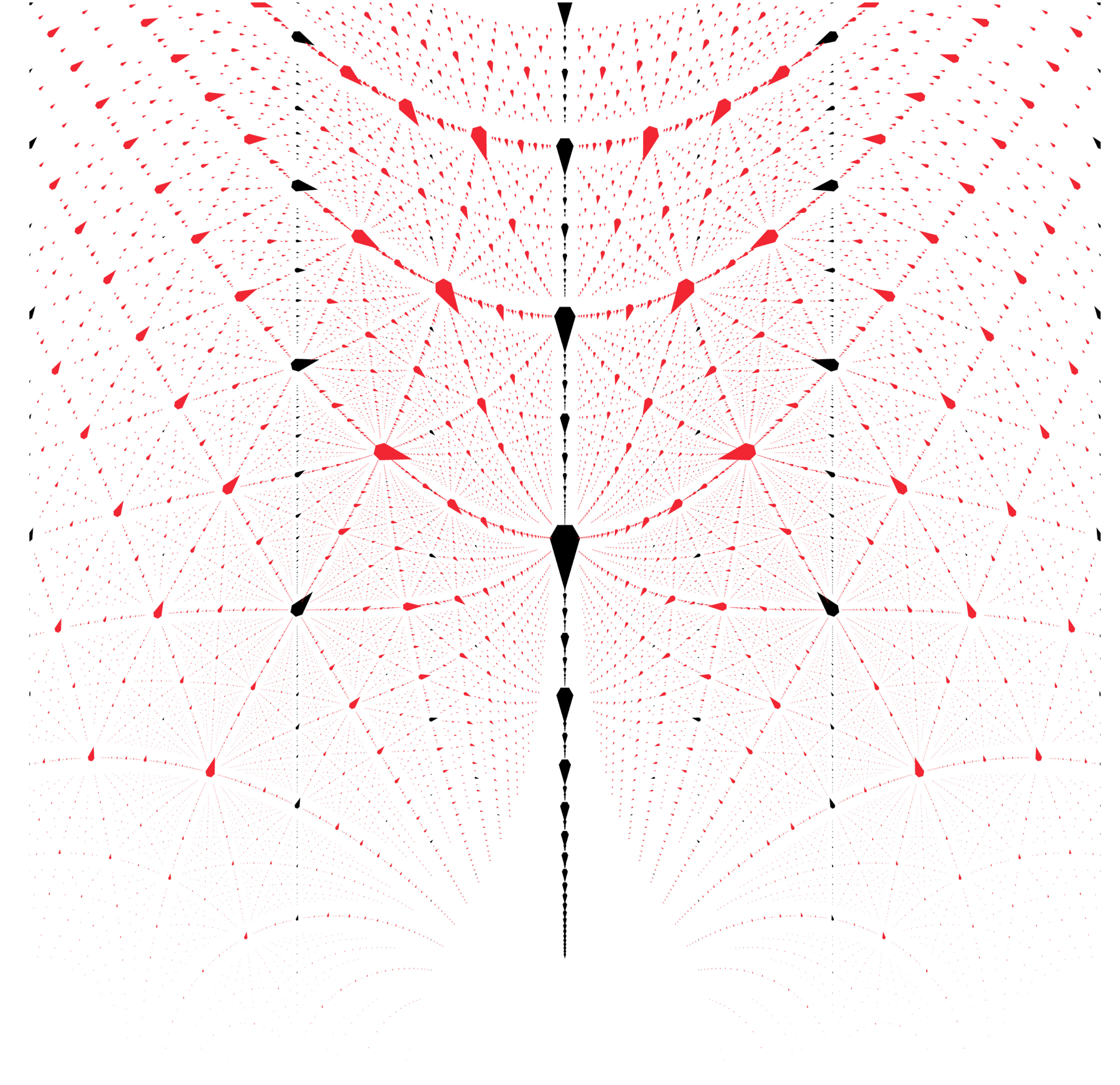}
	\caption{$ax^3+cx+d$}
    \label{fig:Fam2A}
\end{subfigure}
\begin{subfigure}[b]{0.45\textwidth} 
    \centering
	\includegraphics[width=\textwidth]{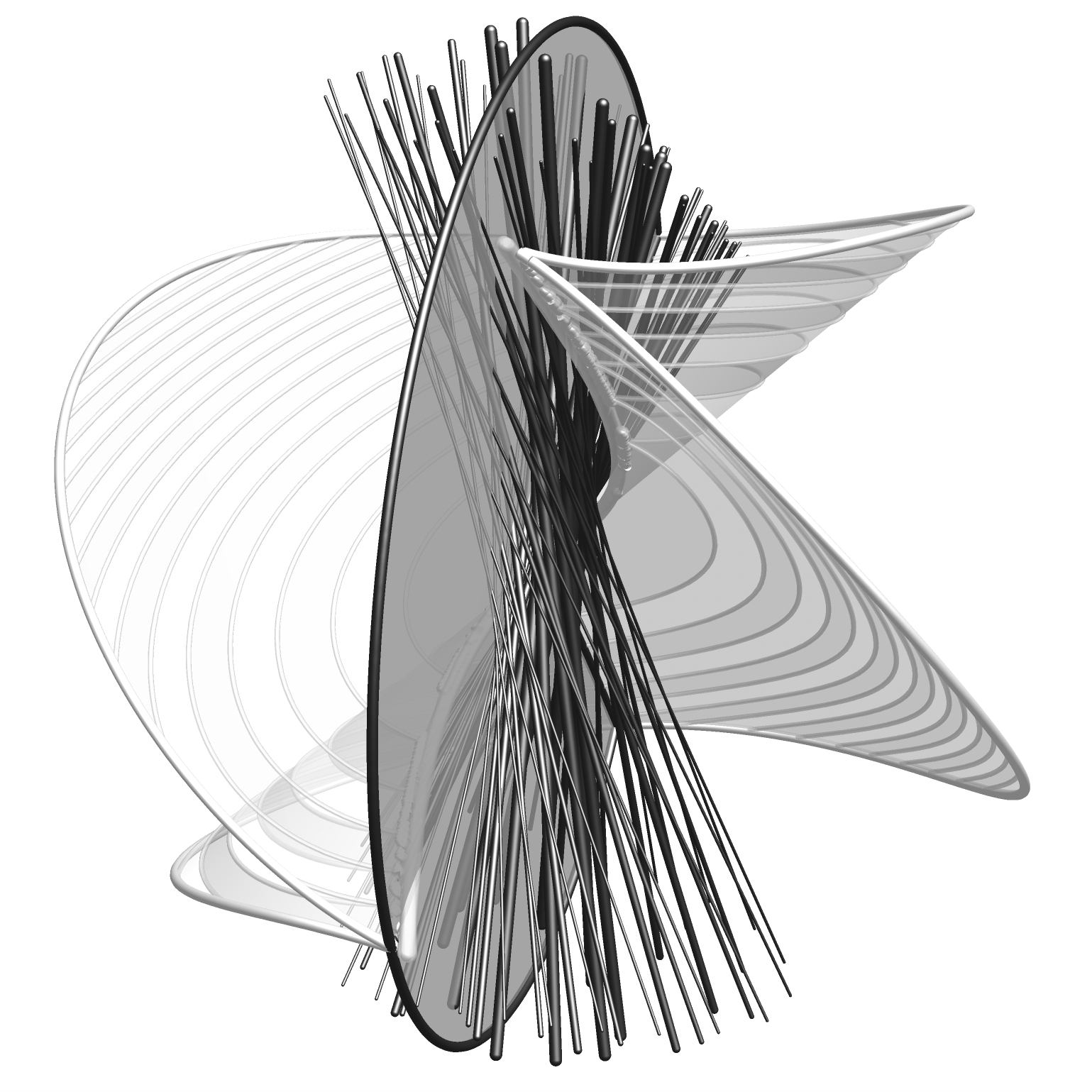}\\
	\includegraphics[width=\textwidth]{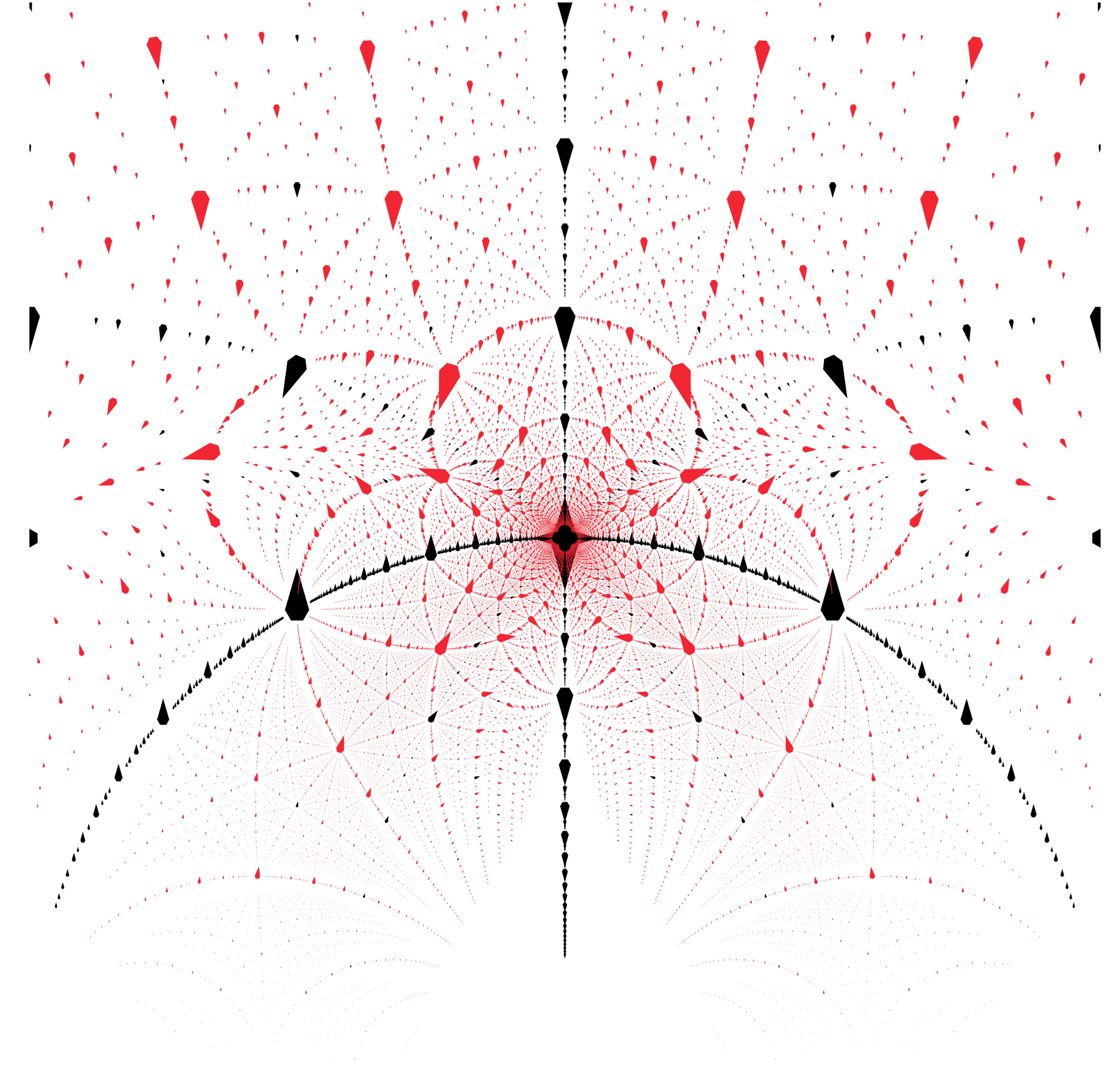}
	\caption{$ax^3+bx^2+cx+b$}
    \label{fig:Fam7C}
\end{subfigure}
	\caption{2-parameter projective families of cubics, plotted in coefficient space together with their projection onto their complex root.  The family on the left is transverse to the $\SS^1$ fibers of $\UT\HH^2_\Coefs$, and so the projection onto roots is an embedding.  The family on the right is not transverse to the $\SS^1$ foliation, and contains the fiber $(x^2+1)(px+q)$.  Thus the projection onto the complex root is not an embedding, and collapses an entire curve above $i$.
}
\label{fig:CoeffCubics}
\end{figure}

\section{Diophantine approximation}
\label{sec:DiophantineApproximation}

\subsection{Classical Diophantine approximation}
\label{ssec:ClassicDiophantineApproximation}

The study of Diophantine approximation is the study of the relative placement of real or complex numbers with regards to their arithmetic complexity.  To illustrate, we consider the unit interval $[0,1]$.  We measure the complexity of a rational number $p/q$ in lowest terms by its denominator, defining its \emph{height} to be $q$.  Then we observe a fundamental phenomenon one might call \textbf{repulsion}:  distinct rationals of low height cannot be too close to one other.  Explicitly,
\begin{equation}
  \label{eqn:basicrep}
      \left|  \frac{p_1}{q_1} - \frac{p_2}{q_2} \right| \ge \frac{1}{q_1q_2}.
    \end{equation}
    This leads one to consider the question of \textbf{good approximations}:  fix $\alpha \in [0,1]$ and ask whether there are rational $p/q$ which are surprisingly close to $\alpha$, in terms of their height.  One asks whether there are infinitely or finitely many $p/q$ such that
    \begin{equation}
      \label{eqn:goodapprox}
      \left| \alpha - \frac{p}{q} \right| < \frac{1}{q^k}.
    \end{equation}
    The behaviour with regards to exponent $k=2$ distinguishes rationals from irrationals.
    
\begin{theorem}[Dirichlet \cite{Dirichlet}]
  \label{thm:dirichlet}
  Let $\alpha \in \RR$.  Then $\alpha$ is irrational if and only if there exist infinitely many distinct $p/q \in \QQ$ such that
  \[
    \left| \alpha - p/q \right| < 1/q^2.
  \]
\end{theorem}
In other words, rationals are ``poorly approximable'' and irrationals are ``well approximable.''  This can be proven by a simple pigeonhole principle argument, which we include here for the sake of exposition, as later proofs will imitate the method.

\begin{proof}
	Let $\alpha$ be irrational.  Choose an integer $Q > 1$.  Divide the unit interval $[0,1]$ into $Q$ even subintervals.  Then, among the real numbers $0, \alpha, 2\alpha, \ldots, Q\alpha$, there must be two, say $i\alpha$ and $j\alpha$, where $0 \le i < j \le Q$, whose fractional parts fall into the same subintervals.  Then we have $| (j-i)\alpha - p | < 1/Q$ for some integer $p$.  Letting $q = j-i$, observe that $q \le Q$, and we obtain $|\alpha - p/q| < 1/qQ \le 1/q^2$.  As $\alpha$ is irrational, we may choose $Q'$ to be such that $|q\alpha - p| > 1/Q'$, and run the argument again; by construction, we discover a new, distinct rational approximation.  In this way, if $\alpha$ is irrational we discover infinitely many such approximations.  By contrast, if $\alpha$ is rational, then \eqref{eqn:basicrep} limits the ability to find good approximations.
\end{proof}

Dirichlet's Theorem is illustrated in Figure \ref{fig:Dirichlet}; if one places disks over each rational $p/q$ with radius $1/q^2$, then the irrationals are covered by infinitely many disks, while rationals by only finitely many.  We can create a more starscape-esque version by a constant scaling, in Figure \ref{fig:RationalStarscape}.  This latter version has a different feel, and more vividly illustrates the mutual repulsion of rational numbers.

The natural accompaniment to Dirichlet's elementary result is a deep one of Roth:  if the exponent $2$ is strengthened to $2+\epsilon$ for any positive $\epsilon$, then all \emph{algebraic} $\alpha$ fail to have infinitely many good approximations \cite{Roth}.  
\begin{theorem}[Roth \cite{Roth}]
	\label{thm:roth}
	Let $\epsilon > 0$.
	Let $\alpha \in \RR$ be algebraic of degree $\ge 2$.  Then there are only finitely many distinct $p/q \in \QQ$ such that
  \[
	  \left| \alpha - p/q \right| < 1/q^{2+\epsilon}.
  \]
\end{theorem}
This finiteness is in fact true for almost all real numbers $\alpha$ (in the sense of Lebesgue measure), a result due to Khintchine \cite[Theorem 29]{Khintchine}.

However, one can construct real numbers which are well-approximable to all higher exponents, called \emph{Liouville numbers} after Liouville's famous construction \cite{Liouville}.  The key to constructing a Liouville number is to artificially build something incredibly close to a series of rational numbers.  The simplest example is $\sum_{k=0}^\infty \frac{1}{10^{k!}}$: the partial sums form rational approximations that are too good to allow for a finiteness property like Theorem \ref{thm:roth}, even for any fixed positive exponent of $q$.  These were the first explicit transcendental numbers.

Having studied the exponent $k$ in approximation within $1/q^k$, we can turn to a finer question of constants.  For example, does Dirichlet's Theorem hold if $1/q^2$ is replaced with $1/Cq^2$ for various increasing values of $C$?  The theorem holds until $C = \sqrt{5}$, above which the golden ratio and certain of its relatives are no longer approximable by infinitely many rationals.  This state of affairs continues until another tipping point, $C = 2\sqrt{2}$, above which $\sqrt{2}$ is poorly approximable, and so on.  These tipping points form the beginning of the theory of the \emph{Lagrange spectrum}.

For the rich theory of Diophantine approximation, including further historical context for these results, the reader may begin with \cite{BugeaudBook, SilvermanHindryBook, SchmidtBook}.

The Diophantine approximation of the complex plane away from the real line is less well-studied.  Here we may ask the approximants to come from number fields:  in our context, it is natural to stratify these by degree, and ask:  how well-approximable is a complex number by algebraic numbers of degree $\le d$?

In order to do so, we need to generalize the notion of height, to measure the arithmetic complexity of algebraic numbers in general.

\subsubsection{Measuring arithmetic complexity with the na\"{i}ve height of a polynomial}
\label{ssec:NaiveHeightComplexity}

The simplest notion of arithmetic complexity may be with reference to the coefficients of its minimal polynomial.  Recall that an algebraic number has a unique \emph{minimal polynomial}, most often taken to be the unique monic irreducible $f(x) \in \QQ[x]$ for which it is a root.  We may take this polynomial to be in $\ZZ[x]$ by scaling up the denominator, which causes us to lose the monic condition.  We will write $f_\alpha$ for the unique scaling whose coefficients are integral but with no common factor, and positive leading coefficient, and refer to this as the \emph{minimal polynomial}, following much of the literature of Diophantine approximation (for example \cite{BugeaudEvertse}).

Consider any polynomial $f = a_d x^d + \cdots a_1 x + a_0 \in \ZZ[x]$.  Then the \emph{na\"{i}ve height} of $f$ is defined as
\[
  H(f) := \max_{0\le i\le d} |a_i|.
\]
This measure has the advantage of simplicity, and a close connection to the linear algebra of the lattice of polynomials of degree $d$ in the space of coefficients.  It is also clear that there are only finitely many polynomials of bounded degree and height.  

We will write $H(f_\alpha)$ for the height of the minimal polynomial $f_\alpha$ of $\alpha$, which measures the arithmetic complexity of $\alpha$.  

\subsubsection{Good approximations drawn from fixed degree}
\label{ssec:fixdeg}

To generalize Dirichlet's Theorem \ref{thm:dirichlet}, one might ask to approximate by algebraic numbers of bounded degree $d$ (so that Dirichlet's Theorem becomes the case of $d=1$, i.e. approximation by rationals).  One can define, following Koksma \cite{Koksma} the quantity $k_d(\alpha)$ to be the surpremum of all $k$ such that there are infinitely many algebraic numbers $\beta$ of degree $\le d$ satisfying
\[
  |\alpha - \beta| < \frac{1}{H(f_\beta)^{k}}.
\]
Dirichlet's Theorem \ref{thm:dirichlet}, in this language, states that $k_1(\alpha) \le 2$ for $\alpha$ rational and $k_1(\alpha) \ge 2$ for $\alpha$ real and irrational.  Roth's Theorem \ref{thm:roth} is that $k_1(\alpha) = 2$ for $\alpha$ real and algebraic.  

As regards the real case, Wirsing conjectured that for transcendental $\alpha \in \RR$, $k_d(\alpha) \ge d + 1$ \cite{Wirsing}.  This is known for $d=1$ (from Dirichlet's Theorem) and for $d=2$ \cite{DavenportSchmidt}.

For general degree, Sprind\u{z}uk gave an answer for almost all $\alpha$ (generalizing Khintchine's statement).  We see a qualitative difference between real and non-real $\alpha$.

\begin{theorem}[Sprind\u{z}uk, \cite{Sprindzuk}]
  \label{thm:sprindzuk}
For almost all $\alpha \in \RR$, $k_d(\alpha) = d+1$.  For almost all $\alpha \in \CC \backslash \RR$, $k_d(\alpha) = (d+1)/2$.
\end{theorem}

This gives us a better idea of the natural sizing for algebraic points in the complex plane:  a sizing of $1/H(f_\alpha)^{(d+1)/2}$ would be the natural analogue of Figure \ref{fig:RationalStarscape}.  See Figure \ref{fig:nuanced-sizings}.

Next, we may consider the generalisation of Roth's Theorem \ref{thm:roth} governing the approximability of algebraic numbers.  Schmidt showed that for algebraic $\alpha \in \RR$ of degree at least $2$, $k_d(\alpha) = \min\{ \deg(\alpha), d+1\}$ \cite{SchmidtAlgebraic}.  For complex numbers away from the real line (our concern here), it is slightly more complicated.

\begin{theorem}[Bugeaud, Evertse \cite{BugeaudEvertse}, Theorem 2.1, Corollary 2.4]
  \label{thm:bugeaud-evertse}
	For algebraic $\alpha \in \CC \backslash \RR$,
	\[
		k_d(\alpha) = \min\{ \deg(\alpha)/2, (d+1)/2 \}
	\]
	except in the case that $\deg(\alpha) \ge d+2$ and $d$ is even.  In this case, $k_d(\alpha) \in \{ (d+1)/2, (d+2)/2 \}$.  

	In particular, for $d=2$, we have $k_2(\alpha) = 2$ if and only if the quantities $1$, $\alpha \overline{\alpha}$ and $\alpha + \overline{\alpha}$ are $\QQ$-linearly dependent (and $k_2(\alpha)= 3/2$ otherwise).
\end{theorem}

In degree $d > 2$, Bugeaud and Evertse give precise conditions for determining which of the two possibilities for $k_d(\alpha)$ is correct in most cases, but were not able to compute it in all cases.

Let us consider the dichotomy given for $d=2$, where Theorem \ref{thm:bugeaud-evertse} says that some $\alpha$ are much more approximable than others, based on whether $1$, $\alpha \overline{\alpha}$ and $\alpha + \overline{\alpha}$ are $\QQ$-linearly dependent. 
To see why this is the case, we recast this characterization more geometrically:  such $\alpha$ have the property that they lie on rational hyperbolic geodesics, i.e.\ those geodesics corresponding to rational planes in coefficient space.  These are exactly the \emph{rational geodesics} discussed in Observation \ref{obs:rat-geo}.   

This dichotomy is illustrated in Figure \ref{fig:quad-quar}.  In fact, we will show in Section \ref{sec:Dio-quad} that this dichotomy holds even for non-algebraic $\alpha$.  

Note that cubics cannot lie on rational geodesics\footnote{One way to see this is to reduce to the unit circle; cubics cannot lie on the unit circle $\alpha \overline{\alpha} = 1$ unless their real root $r$ is rational, since the constant coefficient of the minimal polynomial $\alpha\overline{\alpha}r$ is rational.}, an effect which is quite prominent in Figure \ref{fig:Initial_cubics}.  This is why $k_2(\alpha) \in \{ 3/2, 2 \}$ (i.e., we do not need to allow for $k_2(\alpha) = \deg(\alpha)/2$ separately).


Based on this geometric interpretation of the $d=2$ case of Bugeaud and Evertse, one wonders if the exceptional cases all have similar geometric interpretations.  We muse on this briefly in Section \ref{sec:futurework}.

\begin{figure}[h!tbp]
\centering
\begin{subfigure}[b]{0.95\textwidth} 
    \centering
	\includegraphics[width=\textwidth]{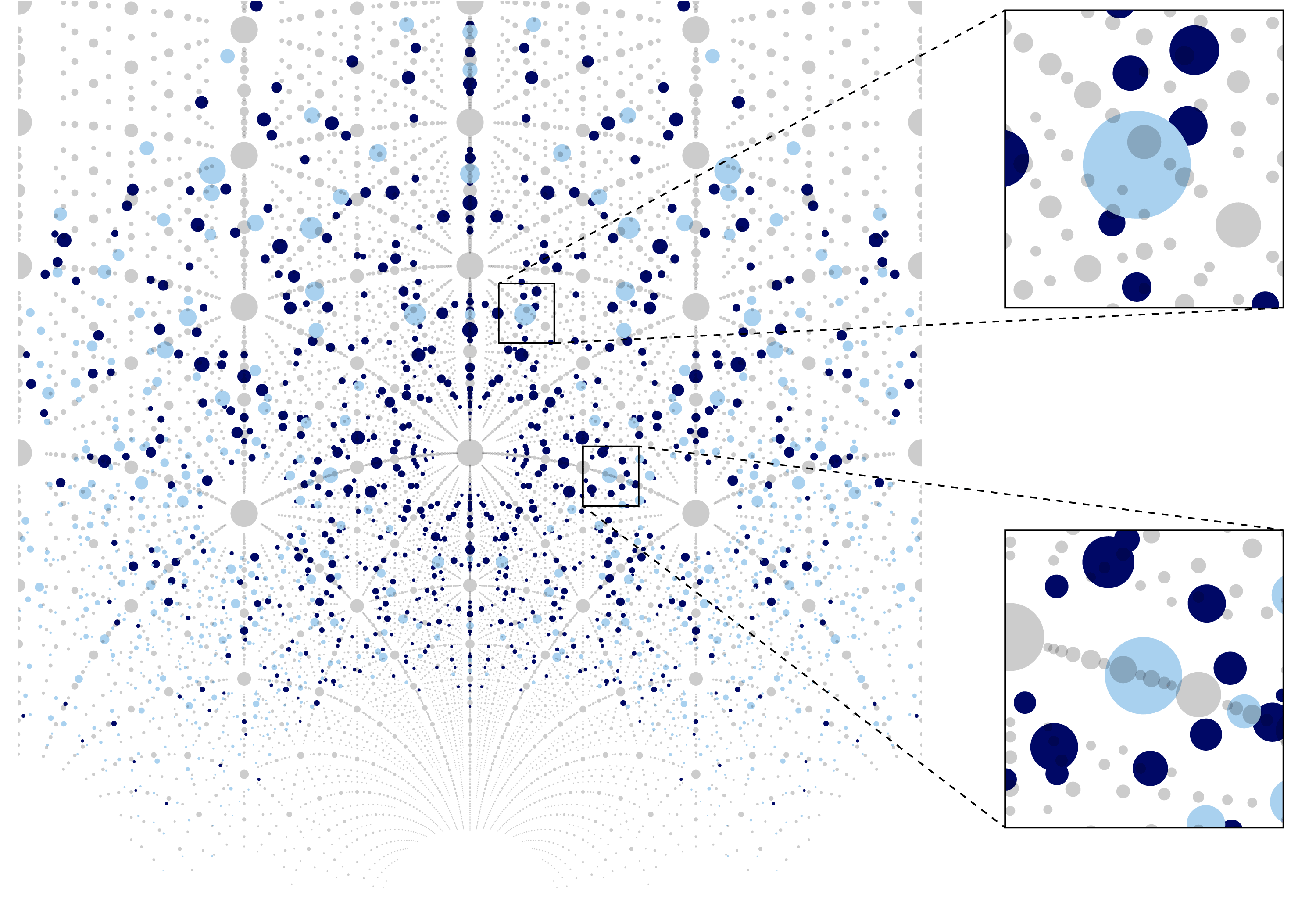}
	\caption{}
\end{subfigure}
	\caption{Quadratics shown in gray, quartics shown in  light blue (if they have no real conjugates) and dark blue (if they do). The quartics which lie off such geodesics (such as the one highlighted in the top right) are approximated only by quadratics at some distance on these geodesics.  In contrast quartics on a geodesic such as the unit circle have more quadratics near them as shown in the lower highlight.}
\label{fig:quad-quar}
\end{figure}

\subsection{Measuring approximation: heights and distances}

\subsubsection{The importance of $\PSL(2;\ZZ)$ and the hyperbolic metric}
\label{ssec:DophantineSl2R}

There is a natural symmetry of the algebraic numbers in the upper half plane:  the action of $\PSL(2;\ZZ)$.  That is, the equivariant action of $\PSL(2;\CC)$ on coefficient and root space, restricted to those elements which preserve the lattice $\ZZ^3$ in the space $\RR^3$ of coefficients (see Section \ref{sec:applquad}).  In this work, we extol the philosophy that, for Diophantine approximation away from the real line, this action should be built into our definitions.  Hyperbolic distance and complex distance are conformally equivalent\footnote{That is, locally the hyperbolic metric and the euclidean metric are very nearly multiples of each other.  This becomes exact at the level of tangent spaces for the Riemannian metric, where $ds^2_{\mathrm{Hyp}}=\frac{1}{\mathrm{Im}(z)}ds^2_{\mathrm{Euc}}$.}, and this symmetry respects the former.  Therefore the relative hyperbolic positions of the algebraic numbers are preserved under $\PSL(2;\ZZ)$.  Several of the notions of arithmetic complexity in the literature partially respect this symmetry.  If they do not, then, given any complex Diophantine approximation statement (such as those of Bugeaud and Evertse), it seems natural to translate by $\PSL(2;\ZZ)$ until the statement is strongest (by which we mean, translate by $\PSL(2;\ZZ)$, find the best approximations according to the theorem, and then transport the constellation back to the original region of interest, where perhaps a priori only worse approximations were guaranteed). 

Perhaps our position is most simply stated in terms of our visualizations:  the positions of the dots are periodic under $\PSL(2;\ZZ)$ and therefore we argue the sizings of them should be too.  
We illustrate this in Figure \ref{fig:sl2z}, showing the transport of a constellation of algebraic numbers under $\PSL(2;\ZZ)$, but shown in the euclidean metric and the height sizing.  The varying sizes and distances of the dots illustrate the changing levels of approximation from the classical perspective. 

\begin{figure}[h!tbp]
\centering
\begin{subfigure}[b]{0.9\textwidth} 
    \centering
	\includegraphics[width=\textwidth]{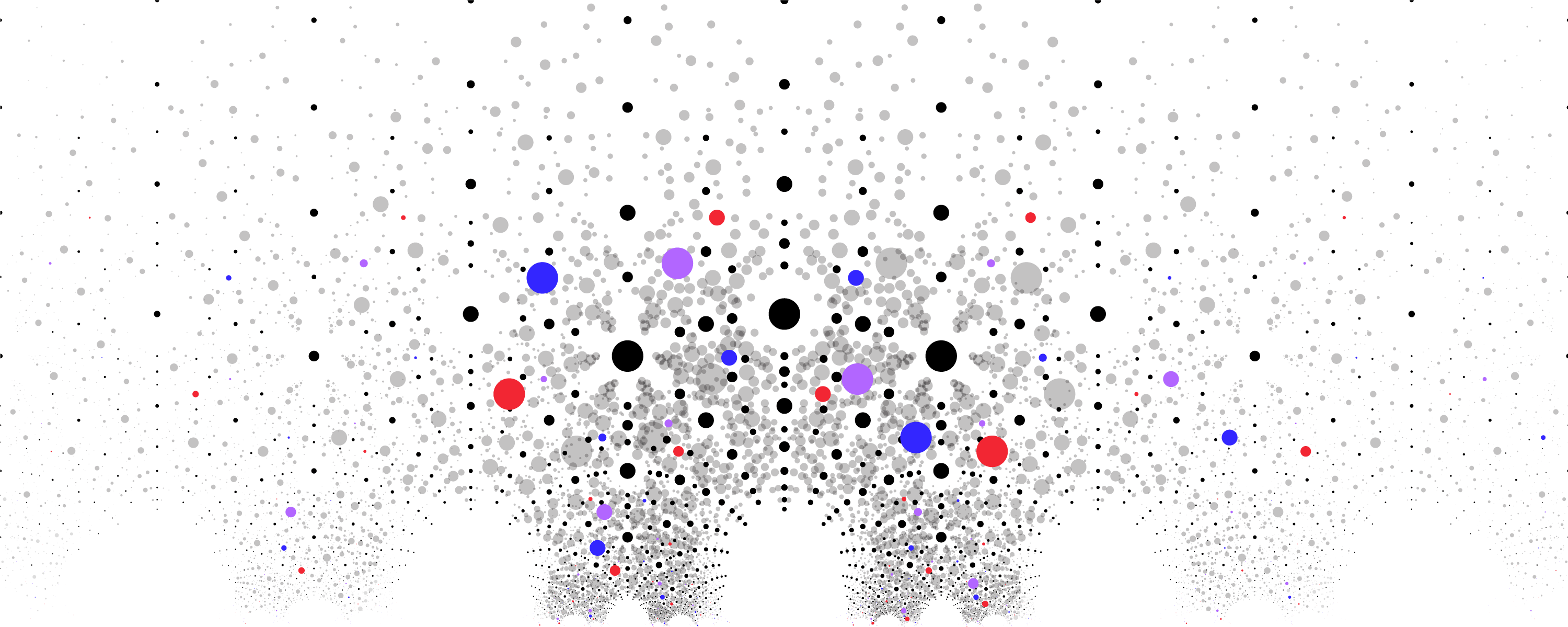}
	\caption{}
\end{subfigure}

	\caption{Cubic algebraic numbers in the na\"ive height, using the maximum coefficient, plotted in the euclidean metric. $\PSL(2;\ZZ)$ orbits are shown for the complex roots of $x^3+x-1$ (red), $x^3+x^2-1$ (purple) and $x^3+x^2+x-1$ (blue). }
\label{fig:sl2z}
\end{figure}

While $H(f_\alpha) = H(f_{1/\alpha})$, unfortunately $H(f_\alpha) \neq H(f_{\alpha+1})$.  Consequently, the na\"ive height is not invariant under the action of $\PSL(2;\ZZ)$ on $\alpha$.  For non-real numbers, the height does, however, attain a minimum on its full $\PSL(2;\ZZ)$ orbit.  Therefore we define\footnote{Although the definition, and much of the discussion, works for $\RR$ as well as $\CC$, $\PSL(2;\ZZ)$ doesn't preserve any meaningful metric on the real line, and all of $\QQ$ falls into a single orbit:  this definition would lose too much information.}
\[
  H_{\PSL}(f) = \min\{ H(f \circ A) : A \in \PSL(2;\ZZ) \}.
\]
To compute this minimum, there is a recent algorithm due to Stoll-Cremona and Hutz-Stoll which may\footnote{It generalizes the reduction theory of binary quadratic forms with respect to $\PSL(2;\ZZ)$ to general binary forms.  It is not known how small a height one is guaranteed under their algorithm, nor where the minimum is attained.} be helpful \cite{HutzStoll, StollCremona}.

\subsubsection{Weil height}
\label{ssec:WeilHeight}

A more nuanced generalisation of the notion of height is the Weil height, defined in terms of the absolute values of an ambient number field.  For an algebraic number $\alpha$ contained in a number field $K$, the Weil height is defined\footnote{We beg the reader's forgiveness for the use of $H$ for both na\"ive height of a polynomial and Weil height of a number; they do \emph{not} satisfy $H(f_\alpha) = H(\alpha)$, but the notation is standard in the literature.} as
\begin{equation}
  \label{eqn:weilht}
  H(\alpha) = \prod_{v \in M_K} \max\{ 1, ||\alpha||_v \},
\end{equation}
where the product is over the set $M_K$ of all \emph{normalized absolute values} $||\alpha||_v = |\alpha|_v^{[K_v:\QQ_v]/[K:\QQ]}$ of $K$.  Here, $K_v$ and $\QQ_v$ are the completions of $K$ and $\QQ$ at $v$.  For further details, a nice introduction to this is available in \cite[Section B.1--B.2]{SilvermanHindryBook}.  This is actually independent of the choice of $K$ containing $\alpha$.  This definition is a generalisation of the case $K=\QQ$, namely
\begin{equation}
    \label{eqn:weil-height-q}
    H(p/q) = \prod_{v \in M_\QQ} \max\{ 1, |p/q|_v \},
\end{equation}
where $| \cdot |_v$ ranges over all $p$-adic absolute values, as well as the archimedean one.  This case can be more simply and intuitively rewritten as
\[ 
 H(p/q) = \max\{ |p|, |q| \},
\]
if $p$ and $q$ are taken to be coprime and integral.  Fortunately, the Weil height and na\"{i}ve height of its minimal polynomial are closely related by a well known relationship in terms of the degree $d := [\QQ(\alpha):\QQ]$ of $\alpha$ \cite[Lemma A.2]{BugeaudBook} (note that the Mahler measure satisfies $M(f_\alpha) = H(\alpha)^d$ \cite[Proposition 1.6.6]{BombieriGubler}):
\begin{equation}
	\label{eqn:wellknown}
  \left( \begin{matrix} d \\ \lfloor d/2 \rfloor \end{matrix} \right)^{-1} H(f_\alpha)
  \le 
  H(\alpha)^d
  \le \sqrt{d+1} H(f_\alpha).
\end{equation}
Here it is important that $f_\alpha$ is minimal in the sense of coprime integer coefficients.

Note that $H(\alpha) = H(1/\alpha)$.  For the same reasons discussed in the previous section, it is natural to define
\[
  H_{\PSL}(\alpha) = \min\{ H(A . \alpha) : A \in \PSL(2;\ZZ) \}.
\]

\subsubsection{Repulsion in the complex plane in terms of Weil height}
\label{ssec:WeilHeightRepulsion}

We can now state a generalization of the repulsion statement \eqref{eqn:basicrep}.

\begin{proposition}[{\cite[Theorem 1.5.21]{BombieriGubler}}]
Suppose $\alpha \neq \beta$ are distinct algebraic numbers, and let $d \ge 1$ be the degree of a field containing both.  Then
\begin{equation}
  \label{eqn:repulsion}
  |\alpha - \beta|^2 \ge \frac{1}{2^d H(\alpha)^d H(\beta)^d}.
\end{equation}
The exponent $2$ on the left can be removed if $\alpha, \beta \in \RR$.
\end{proposition}
See \cite[Section A.2]{BugeaudBook} for a version in terms of the na\"ive height.

\subsubsection{Discriminant as a measure of arithmetic complexity}
\label{ssec:DiscriminantComplexity}

From the perspective of the geometry discussed in the previous section, and the images we've drawn, one might consider the measure of arithmetic complexity given by the discriminant.

For a polynomial $f = a_d x^d + \cdots + a_1 x + a_0 = a_d \prod_{i=1}^d (x - \alpha_i) \in \ZZ[x]$, let $\Delta_f$ denote the discriminant of $f$.  Recall that the discriminant is a measure of the differences between the roots:
\[
	\Delta_f = a_d^{2d-2} \prod_{i < j} (\alpha_i - \alpha_j)^2.
\]
We will refer to an algebraic number as having a discriminant, namely the discriminant of its minimal polynomial, and write $\Delta_\alpha := \Delta_{f_\alpha}$.  
This has the particular advantage of being invariant under $\PSL(2;\ZZ)$.

How do the previously defined heights and the discriminant relate?  Mahler proved a relationship in one direction \cite{Mahler}, namely:
\begin{equation}
  \label{eqn:delta-le-height-gen}
	|\Delta_\alpha| \le d^d H(\alpha)^{d(2d-2)}.
\end{equation}
By using \eqref{eqn:wellknown}, we obtain the related inequality:
\begin{equation}
	\label{eqn:delta-le-naive-gen}
	|\Delta_\alpha| \le d^d(d+1)^{d-1} H(f_\alpha)^{2d-2}.
\end{equation}

In general one doesn't expect a converse inequality, since the discriminant is invariant under $f(x) \mapsto f(x+1)$, while the Weil height would be expected to grow.  Even within one fundamental region of the upper half plane, one doesn't expect a tight relationship.  For example, it is possible to define a family of quadratic irrationalities $(\alpha_n)_{n \ge 1}$ such that $H(\alpha_n)^{4}/|\Delta_{\alpha_n}| \rightarrow \infty$ as $n \rightarrow \infty$.  Namely, the polynomials $x^2 + n$ have upper-half-plane roots $\alpha_n$ approaching $\infty$ along the imaginary axis, and $4H(\alpha_n)^4/|\Delta_{\alpha_n}| = 4n^2/4n = n \rightarrow \infty$.

\subsubsection{Sizing in starscape images}
\label{ssec:SizingStarscapes}

In light of the comparisons \eqref{eqn:wellknown}, \eqref{eqn:delta-le-height-gen}, and \eqref{eqn:delta-le-naive-gen}, as well as Theorems \ref{thm:sprindzuk} and \ref{thm:bugeaud-evertse}, one might compare the na\"ive sizings in Figures \ref{fig:DotSizesCoeff} and \ref{fig:DotSizesRoots} with slightly more nuanced versions given in Figure \ref{fig:nuanced-sizings}.  These latter sizings are all chosen to match Theorem \ref{thm:sprindzuk} just as Figure \ref{fig:RationalStarscape} matches Dirichlet's Theorem \ref{thm:dirichlet}.

\begin{figure}[h!tbp]
\centering
\begin{subfigure}[b]{0.6\textwidth} 
    \centering
	\includegraphics[width=\textwidth]{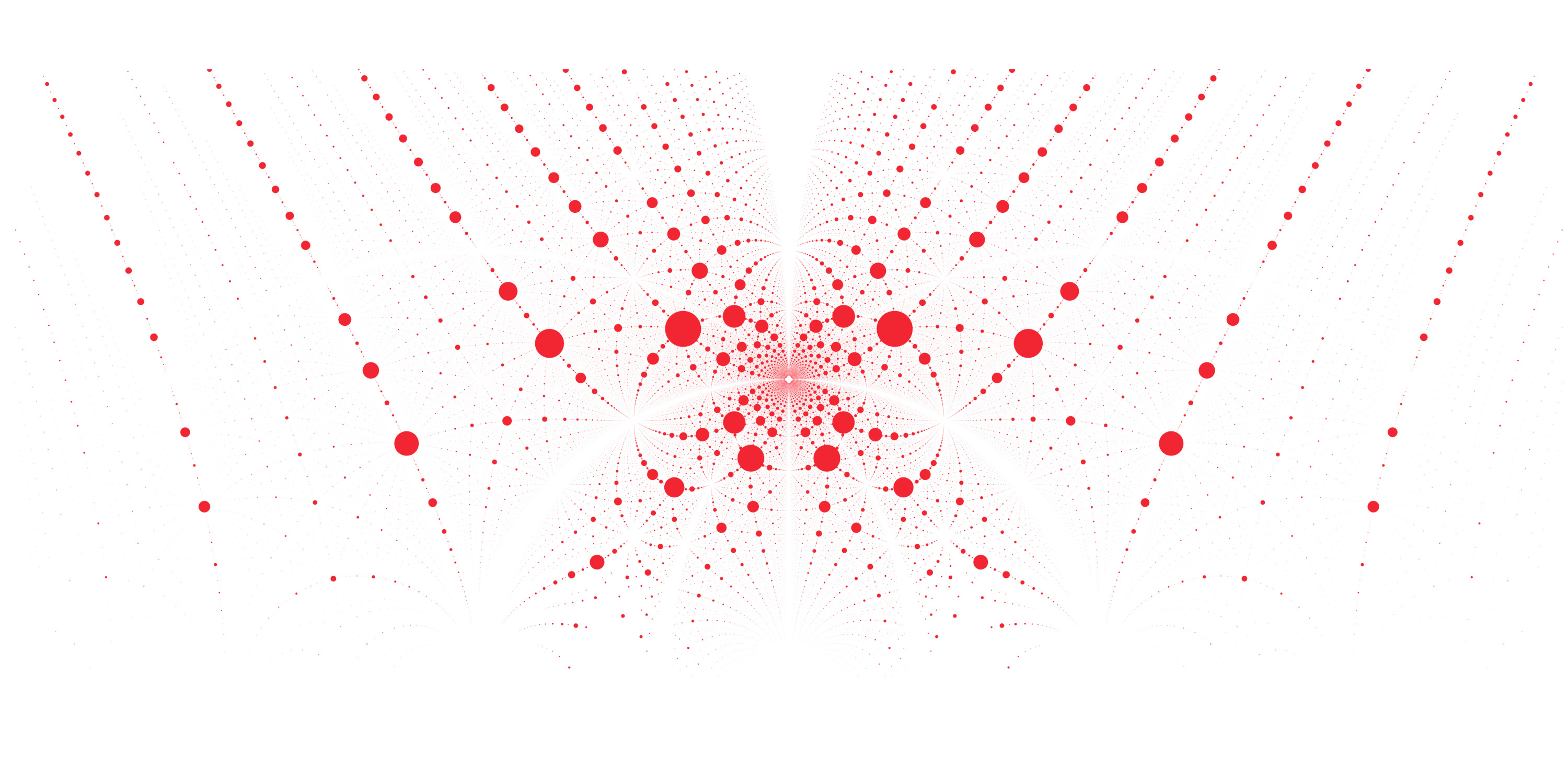}
	\caption{Discriminant: $|\Delta_\alpha|^{(d+1)/(4d-4)}$}
	\label{fig:nuancedDisc}
\end{subfigure}
\begin{subfigure}[b]{0.6\textwidth} 
    \centering
	\includegraphics[width=\textwidth]{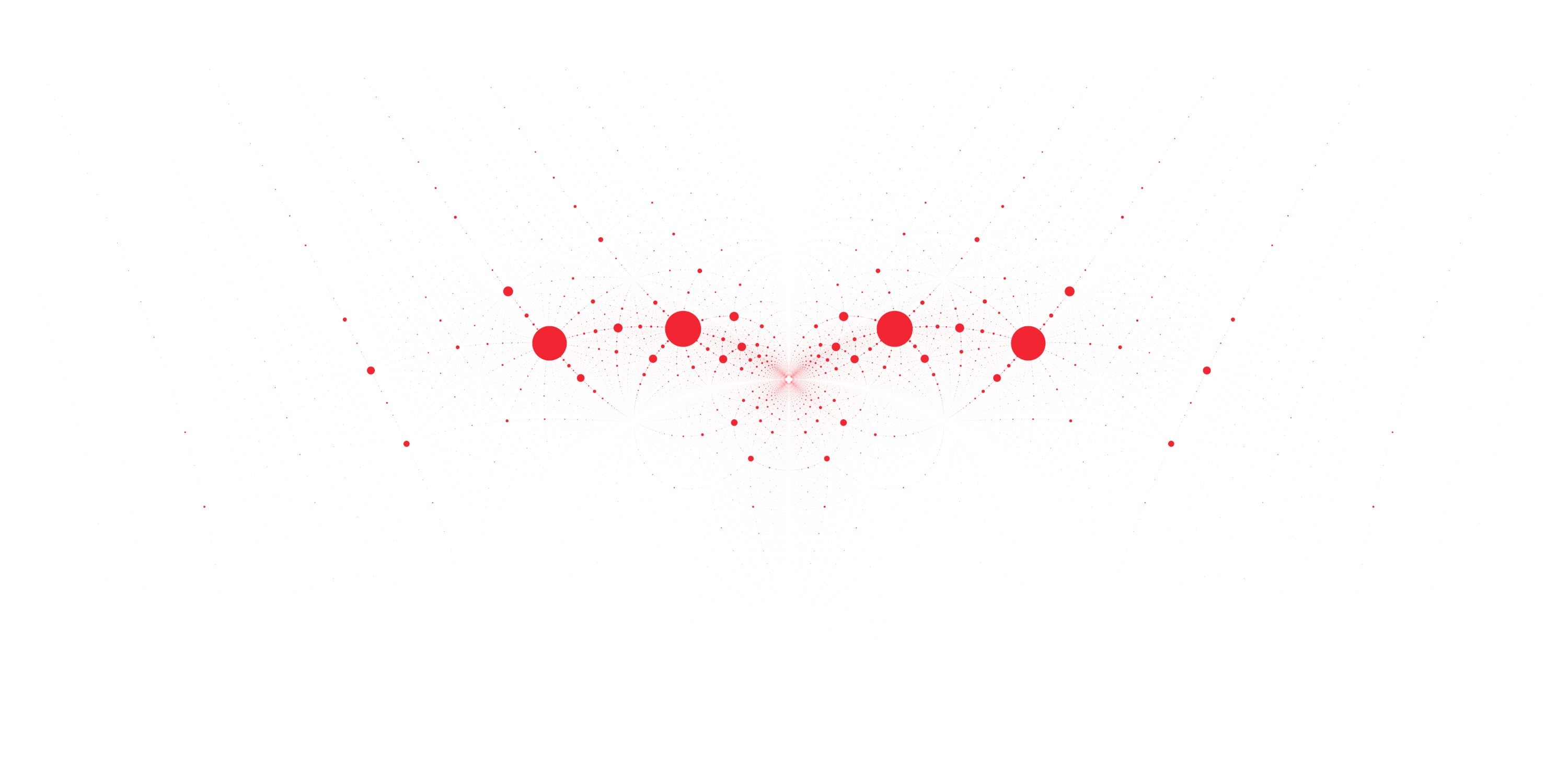}
	\caption{Na\"ive height:  $H(f_\alpha)^{(d+1)/2}$}
	\label{fig:nuancedNaive}
\end{subfigure}
\begin{subfigure}[b]{0.6\textwidth} 
    \centering
	\includegraphics[width=\textwidth]{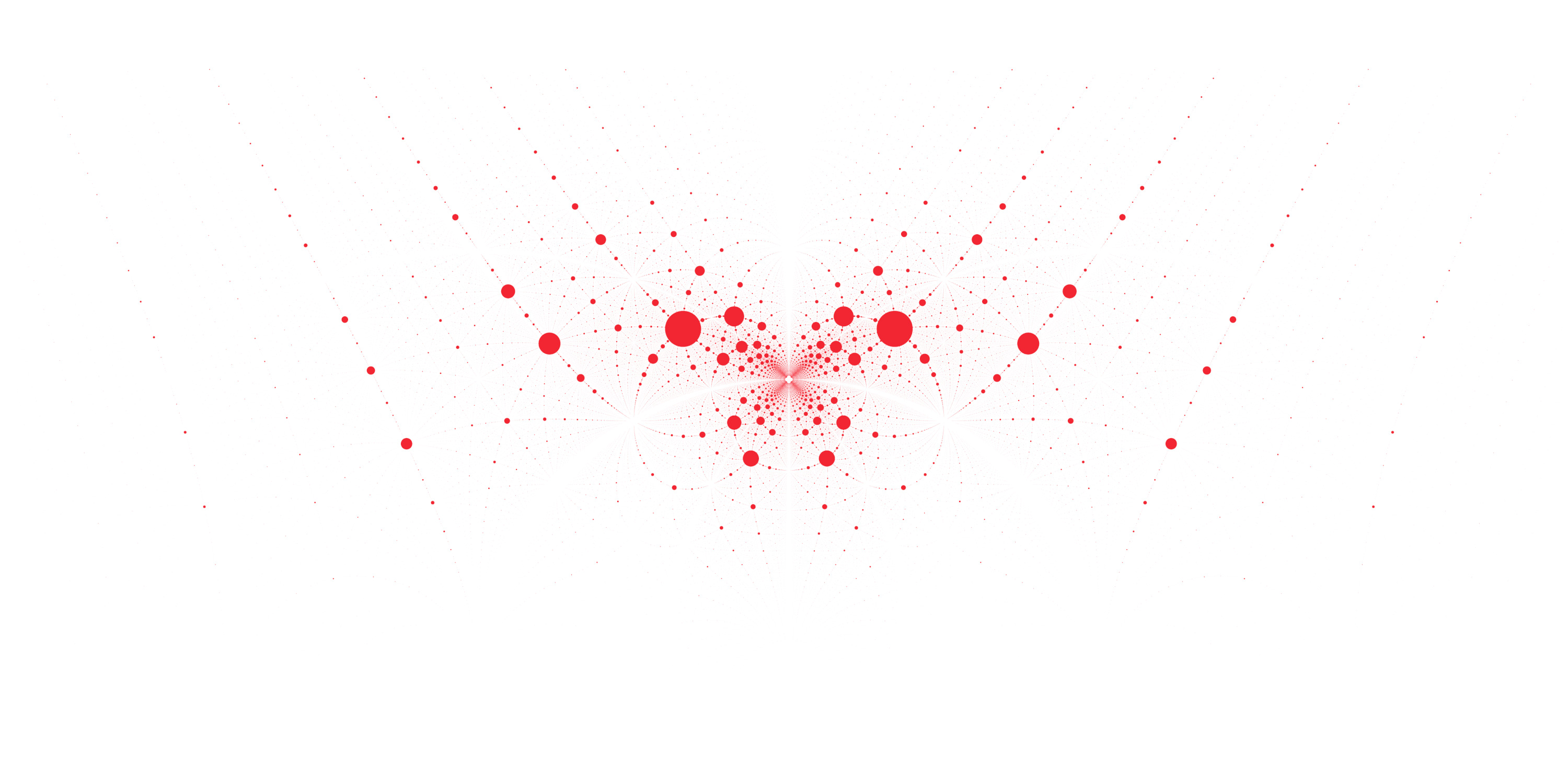}
	\caption{Weil height or Mahler measure:
	$H(\alpha)^{d(d+1)/2} = M(f)^{(d+1)/2}$}
	\label{fig:nuancedMahler}
\end{subfigure}
	\caption{
		Several natural choices of sizing by arithmetic complexity of the roots of cubics to match Theorem \ref{thm:sprindzuk}, scaled so that the complex root of $x^3+x-1$ is the same size. The roots plotted are the family $a x^3 +c x^2+b x + c$. Compare with Figures \ref{fig:DotSizesCoeff} and \ref{fig:DotSizesRoots}.}
\label{fig:nuanced-sizings}
\end{figure}

\section{Diophantine approximation in the quadratics}
\label{sec:Dio-quad}

We now revisit the basic theory of Diophantine approximation by complex quadratic irrationalities, from the starscapes perspective:  we use the \emph{hyperbolic distance} to measure distance and the \emph{discriminant} to measure arithmetic complexity.  We give analogs of repulsion (as in \eqref{eqn:basicrep} and \eqref{eqn:repulsion}), Dirichlet's Theorem \ref{thm:dirichlet} guaranteeing infinitely many approximations, and Roth's Theorem \ref{thm:roth} on the approximation of algebraic numbers, in this new situation.  In this way, we recover the $d=2$ case of Theorem \ref{thm:bugeaud-evertse}, with some additional geometric nuance (we distinguish between approximations coming from different rational geodesics).  The proofs are elementary (with the exception that they depend on Schmidt's Subspace Theorem), and are based on the geometry of Section \ref{sec:Geometry}.  We work in the coefficient space and transport the results to the complex plane afterward.   

This perspective offers a few benefits.  First, it respects the natural $\PSL(2;\ZZ)$ symmetry.  Second, the geometry nicely explains the special cases that arise in Theorem \ref{thm:bugeaud-evertse}, and we can now observe that the same dichotomy holds for non-algebraic complex numbers:  those on rational geodesics are better approximable.  In particular, Theorem \ref{thm:dirichlet-geo} implies that $k_2(\alpha) \ge 2$ for any complex number on a rational geodesic, whereas Theorem \ref{thm:sprindzuk} says we have $k_2(\alpha)=3/2$ for almost all complex non-real numbers.  Also, our method discriminates between approximations taken from fixed rational geodesics.  The relationship between our hyperbolic/discriminant statements and the classical exponents $k_2(\alpha)$ is explained in Section \ref{sec:connections}.

Throughout this section, it will be convenient to use Vinogradov notation: that is, $f \ll_\alpha g$ denotes that $f$ is bounded above by a constant multiple of $g$, where the constant may depend on $\alpha$.

\subsubsection{Repulsion amongst imaginary quadratics}

Recall that with the Weil height function we have a simple repulsion principle, \eqref{eqn:repulsion}:
\begin{equation*}
	|\alpha - \beta|^2 > \frac{1}{2^dH(\alpha)^dH(\beta)^d}.
\end{equation*}

Using the hyperbolic metric (denoted $d_{hyp}$) and the discriminant, we obtain the following version of repulsion.

\begin{theorem}
  \label{thm:delta-repulsion}
  Let $\alpha \neq \beta$ be two non-real quadratic irrationalities, of discriminants $\Delta_\alpha$ and $\Delta_\beta$ respectively.  Then the hyperbolic distance between $\alpha$ and $\beta$, considered in the upper half plane, is at least
  \[
	  d_{hyp}(\alpha, \beta) \ge \acosh\left( \sqrt{ 1 + \frac{1}{{\Delta_\alpha\Delta_\beta}} } \right).
  \]
	If $\Delta_\alpha = \Delta_\beta$, then
	\[
		d_{hyp}(\alpha,\beta) \ge \acosh\left( 1 + \frac{1}{|\Delta_\alpha|} \right).
	\]
\end{theorem}

\begin{proof}
	Suppose $\alpha$ and $\beta$ are associated to some vectors $f_\alpha, f_\beta$ in coefficient space.
	By \eqref{eqn:CoefsMetric}, the distance between them is
  \[
	  d_\Coefs(f_\alpha,f_\beta) = \acosh\left( \frac{ |\langle f_\alpha, f_\beta \rangle| }{\sqrt{|\Delta_\alpha \Delta_\beta|}} \right).
  \]
  Since $\alpha \neq \beta$, this distance is greater than $0$.  However, it lies in
  \[
    \acosh\left( \frac{\ZZ}{\sqrt{\Delta_\alpha\Delta_\beta}} \right),
  \]
  which is a discrete set of values whose smallest positive value is
  \[
    \acosh\left( \frac{n}{\sqrt{\Delta_\alpha\Delta_\beta}} \right),
  \]
  where $n$ is the smallest integer greater than $\sqrt{|\Delta_\alpha\Delta_\beta|}$.  In particular, $n^2 - \Delta_\alpha \Delta_\beta \ge 1$.  This implies that 
 \[
	 \frac{n}{\sqrt{\Delta_\alpha\Delta_\beta}} \ge \sqrt{ 1 + \frac{1}{{\Delta_\alpha\Delta_\beta}} }.
 \]
	The special case is the case that $n = |\Delta_\alpha| + 1$.
\end{proof}

\subsubsection{Fundamental geometric lemma}

Fix $\alpha \in \CC \backslash \RR$, not quadratic.  We will consider approximations by quadratic irrational $\beta$.  We can give corresponding vectors in coefficient space: 
\[
	f_\alpha = [\alpha_1: 1: \alpha_2], \quad f_\beta = [p_1: n: p_2]
\]
where $p_1, n, p_2 \in \ZZ$ and $\alpha_1, \alpha_2$ are not both rational.  More precisely, we have $\alpha_1 = 1/(\alpha + \overline{\alpha})$ and $\alpha_2 = (\alpha\overline{\alpha})/(\alpha + \overline{\alpha})$.

We wish to compute the hyperbolic distance between $f_\alpha$ and $f_\beta$.  Our proofs will rely on a fundamental lemma which relates this distance to linear forms, with coefficients depending on $\alpha$, in $f_\beta$'s coordinates.\footnote{This generalizes the rational approximation case, since there the angle $\theta$ between the projective lines $[\alpha:1]$ and $[p:q]$ satisfies
\[
	\cos \theta = 
	\frac{ \langle (\alpha,1), (p,q) \rangle }{ || (p,q) || || (\alpha,1) || }
	= 
	\frac{ |q\alpha - p| }{ || (p,q) || || (\alpha,1) || }.
\]}

\begin{lemma}
	\label{lemma:mainlemma}
	Suppose $\alpha \in \CC \backslash \RR$ is a fixed non-quadratic, with $f_\alpha = [\alpha_1:1:\alpha_2]$.

	Let $\beta$ be quadratic, with $f_\beta = [p_1: n: p_2]$, where $n, p_1, p_2 \in \ZZ$.  

	Define the linear forms:
	\begin{align*}
		L_1 := L_1(p_1,n,p_2) &= n \alpha_1 - p_1,  \\
		L_2 := L_2(p_1,n,p_2) &= -n \alpha_2 + p_2,  \\
		L_3 := L_3(p_1,n,p_2) &= \alpha_1p_2 - \alpha_2 p_1.
	\end{align*}

	\begin{enumerate}
\item We have
	\[
		d_{hyp}(\alpha,\beta) \le \acosh \left( 1 + \frac{\max\{ |L_1L_2|, L_3^2 \}}{ |\Delta_\beta| } \right).
	\]
		\item 
	Suppose $\beta$ is sufficiently close to $\alpha$, namely $d_{hyp}(\alpha,\beta) < \acosh 2$.  Then
	\[
		d_{hyp}(\alpha,\beta) \ge \acosh \left( 1 + \frac{m_\alpha |L_1L_2| }{ |\Delta_\beta| } \right),
	\]
			for some constant $m_\alpha > 0$ depending only on $\alpha$.
	\end{enumerate}
\end{lemma}

\begin{proof}
	Recall that the distance is invariant under scaling, so we will temporarily replace $f_\alpha$ with $f'_\alpha = nf_\alpha$, so the middle coordinates of the two vectors agree.
We have
\begin{align*}
	&\frac{ \langle f'_\alpha, f_\beta \rangle^2 }
	{ ||f'_\alpha||^2 ||f_\beta||^2 } - 1 \\ 
	&= \frac{ (n^2 - 2n\alpha_1 p_2 - 2n\alpha_2 p_1)^2 - (n^2 - 4n^2\alpha_1\alpha_2)(n^2-4p_1p_2) }
	{ ||f'_\alpha||^2 ||f_\beta||^2 } \\
	&= 4n^2 \frac{ (\alpha_1p_2 - \alpha_2p_1)^2 + (n\alpha_1 - p_1)(n\alpha_2 - p_2) }
	{ ||f'_\alpha||^2 ||f_\beta||^2 }  \\
	&= 4 \frac{ L_3^2 - L_1L_2 }
	{ ||f_\alpha||^2 ||f_\beta||^2 } 
\end{align*}
Evidently,
\[
	L_3^2 - L_1L_2 \le 2 \max\{ |L_1L_2|, L_3^2 \}.
\]
Therefore
\[
	\frac{ \langle f_\alpha, f_\beta \rangle^2 }
	{ ||f_\alpha||^2 ||f_\beta||^2 } 
	\le 1 + \frac{2 \max\{ |L_1L_2|, L_3^2 \} }{ ||f_\beta||^2 }.
\]
Hence,
\[
	\frac{ - \langle f_\alpha, f_\beta \rangle }
	{ ||f_\alpha|| ||f_\beta|| } 
	\le 1 + \frac{ \max\{ |L_1L_2|, L_3^2 \} }{ ||f_\beta||^2 }.
\]

Next we show
\[
	L_3^2 - L_1L_2 \gg_\alpha \min\{ |L_1L_2|, L_3^2 \}.
\]
By the definitions of the $L_i$, we have
\[
	L_1\alpha_2 + L_2\alpha_1 = L_3.
\]
In particular,
\[
	L_3^2 - L_1L_2 = (L_1\alpha_2 + L_2\alpha_1)^2 - L_1L_2 \ge (4\alpha_1\alpha_2-1)L_1L_2
\]
by the arithmetic-geometric mean inequality.  Note that $K_\alpha := 4\alpha_1\alpha_2-1 > 0$ since $\alpha$ is not real.  Thus if $L_1L_2$ is positive, we are done.  On the other hand, if it is negative, then 
\[
	L_3^2 - L_1L_2 = L_3^2 + |L_1L_2| \ge |L_1L_2|.
\]

We have shown that, for some constant $m'_\alpha > 0$ depending only on $\alpha$,
\[
	\frac{ \langle f_\alpha, f_\beta \rangle^2 }
	{ ||f_\alpha||^2 ||f_\beta||^2 } 
	\ge 1 + \frac{m'_\alpha |L_1L_2| }{ ||f_\beta||^2 }.
\]
Therefore, for $\beta$ sufficiently close to $\alpha$ (i.e., so that $\frac{ \langle f_\alpha, f_\beta \rangle^2 }{ ||f_\alpha||^2 ||f_\beta||^2 } < 2$), this implies (taking $m_\alpha = (\sqrt{2}-1)m'_\alpha$) that
\[
	\frac{ - \langle f_\alpha, f_\beta \rangle }
	{ ||f_\alpha|| ||f_\beta|| } 
	\ge 1 + \frac{ m_\alpha |L_1L_2| }{ ||f_\beta||^2 }.
\]
\end{proof}

\subsection{Quadratic Dirichlet's Theorem}

\subsubsection{Quadratic Dirichlet's Theorem on a rational geodesic}
\label{ssec:QuadraticDirichletonRationalGeodesics}
We first consider the question of Diophantine approximation on a single \emph{rational geodesic}, i.e. the image of a rational plane in coefficient space (see Section \ref{sec:applquad}).  This is motivated by the observation that each such geodesic looks, in Figure \ref{fig:Initial_quadratics}, like a copy of Figure \ref{fig:RationalStarscape}, so we expect it to have Diophantine approximation properties similar to the rationals.  It will turn out that points lying on such geodesics are better approximable than points elsewhere in $\CC$:  see Figure \ref{fig:quad-quar}.
Recall that, given $\alpha \in \CC \backslash \RR$, $\alpha$ lies on a rational geodesic if and only if $1$, $\alpha + \overline{\alpha}$ and $\alpha\overline{\alpha}$ are $\QQ$-linearly dependent (Observation \ref{obs:rat-geo}).

We begin with an analogue to Dirichlet's Theorem \ref{thm:dirichlet}, asserting the existence of infinitely many good approximations on a rational geodesic.

\begin{theorem}
	\label{thm:dirichlet-geo}
  Let $\alpha \in \CC \backslash \RR$ not be quadratic irrational, but lying on a rational geodesic.
	Then there exists a constant $K_\alpha > 0$, depending only on the $\PSL(2;\ZZ)$ orbit of $\alpha$, such that there are infinitely many quadratic irrational $\beta$ lying on that rational geodesic, with
  \[
	  d_{hyp}(\alpha, \beta) \le \operatorname{arcosh}\left( 1 + \frac{K_\alpha}{|\Delta_\beta|^{2}} \right).
  \]
\end{theorem}

\begin{proof}
	Suppose $\alpha \in \CC \backslash \RR$ is a fixed non-quadratic, with $f_\alpha = [\alpha_1:1:\alpha_2]$.
	The theorem statement is invariant under $\PSL(2;\ZZ)$ (i.e.\ replacing $\alpha$ and all candidate $\beta$ with their images under the action of some element of $\PSL(2;\ZZ)$, we preserve $\Delta_\beta$ and the hyperbolic distances).
	We assume $\alpha$ is on a rational geodesic, so $a\alpha_1 + b + c\alpha_2 = 0$.  We may translate by $\PSL(2;\ZZ)$ until $\max\{|a|,|b|,|c|\}$ is minimal; this is a constant depending only on the $\PSL(2;\ZZ)$ orbit of $\alpha$.  We can actually choose a canonical geodesic amongst these finitely many in any way we wish; say the one of smallest radius, and amongst those, center nearest the origin but to its right.  We set some such convention.

	Let $Q > 0$ be an integer.

    We use the classical method of proof of Dirichlet's Theorem \ref{thm:dirichlet} to find a solution $(n_0,q_0) \in \ZZ^2$ to
    \[
      |n_0 \alpha_1 - q_0| \le 1/Q, \quad n_0 \le Q.
    \]
	Namely, we divide the unit interval into $Q$ even subintervals, and the box principle guarantees some $i\alpha_1$ and $j\alpha_1$, for some $0 \le i < j \le Q$, lie in the same interval modulo $\ZZ$; we let $n_0=j-i$.  
    We have $an_0\alpha_1 = -bn_0 - cn_0\alpha_2$, so that
	\[
		|cn_0 \alpha_1 - cq_0| \le |c|/Q, \quad
		|cn_0 \alpha_2 + bn_0 + aq_0| = |-an_0 \alpha_1 + aq_0| \le |a|/Q.
	\]
	We let $f_\beta = [p:n:q] = [cq_0: cn_0: -bn_0 - aq_0]$.  
	The corresponding $\beta \in \CC$ is a candidate good quadratic irrational approximation to $\alpha$.

	Increasing $Q$ and finding a new approximation, we can in fact produce infinitely many such $f_\beta$ which are linearly independent, and such that the resulting infinite sequence of distinct $\beta$ approach $\alpha$.

	It remains to show that these $\beta$ are indeed good approximations.  
	We have the following observations:
	\begin{gather*}
		n \ll_\alpha Q, \quad |L_1L_2| \ll_\alpha 1/Q^2,
		\\
		L_3^2 = | L_1\alpha_2 + L_2\alpha_1 |^2 \ll_\alpha \max\{ |L_1|^2, |L_2|^2 \} \ll_\alpha 1/Q^2.
	\end{gather*}
	Here and for the entire proof, the constants in the Vinogradov notation depend on $\alpha$, but this in the canonical choice of $\alpha$ within its $\PSL(2;\ZZ)$ orbit, so the constants actually only depend on the orbit.

	Combining these with Lemma \ref{lemma:mainlemma},
	\[
		\cosh d_\Coefs(f_n,f_\beta) - 1 \ll_\alpha \frac{1}{n^{2}|\Delta_\beta|}.
	\]
	Let $C_\alpha = |\Re(\alpha)|/2|\Im(\alpha)|$.  This is a positive constant depending only on $\alpha$.
	Then for all $\beta$ sufficiently close to $\alpha$, $|\Re(\beta)| > C_\alpha |\Im(\beta)|$.  We consider only those solutions $\beta$ which are at least that close.
	Then $n/2p > C_\alpha |\Delta_\beta|^{1/2}/2p$, hence $n > C_\alpha |\Delta_\beta|^{1/2}$.  Using this fact, we obtain
	\[
		\cosh d_{\Coefs}( f_n, f_\beta ) - 1 \ll_\alpha \frac{1}{ |\Delta_\beta|^{2}}.
	\]
	This proves the theorem.
\end{proof}

It is interesting to note that the constant $K_\alpha$ depends on a particular ideal class of a real quadratic field, since by Observation \ref{obs:rat-geo}, it can only lie on one rational geodesic and that geodesic is associated to such an ideal class.  In the proof, the constant $K_\alpha$ directly depends on the class.  It is natural to wonder about a Lagrange spectrum for $\alpha$ on rational geodesics.

Also interesting is that all $\alpha$ lying on rational geodesics are exceptions to Sprind\u{z}uk's Theorem \ref{thm:sprindzuk}.  It is likely possible to prove using Khintchine's classic methods, that, within a single geodesic, almost all $\alpha$ have $k_2(\alpha) = 2$.

\subsubsection{Quadratic Dirichlet's Theorem in general}

In the general case, where $\alpha$ may not lie on a geodesic, we have weaker approximation guarantee, with a similar proof.

\begin{theorem}
	\label{thm:dirichlet-quad}
  Let $\alpha \in \CC \backslash \RR$ not be a quadratic irrational.  Let $K > 0$ be any constant.  Then, there are infinitely many quadratic irrationalities $\beta$ with
  \[
	  d_{hyp}(\alpha,\beta) \le \acosh\left( 1 + \frac{K}{|\Delta_\beta|^{3/2}} \right).
  \]
\end{theorem}

\begin{proof}
	Let $Q > 1$ be an integer.

	Let $f_i = [i\alpha_1: i: i\alpha_2] = i f_\alpha$ for positive integers $i$.
	Divide the unit interval into $Q$ equal subintervals, and consider the vectors $f_i$ modulo $[\ZZ:\ZZ:\ZZ]$.  Then, by the pigeonhole principle, there are some $0 \le i < j \le Q^2$ such that $f_i$ and $f_j$ have first and third coordinates lying in the same pair of subintervals modulo $\ZZ$.  Let $n := j-i$.  Note that $0 < n \le Q^2$.

	By construction, we have $|n\alpha_1-p_1| < 1/Q$ and $|n\alpha_2-p_2| < 1/Q$ for some integers $p_1$ and $p_2$.  Defining $f_\beta = [p_1: n: p_2]$, we have a corresponding $\beta \in \CC$:  this is the candidate good approximation we seek.

	Choose $Q'$ large enough such that $|n\alpha_1- p_1|, |n\alpha_2 - p_2| > 1/Q'$ for all $n<Q^2$, $p_1,p_2 \in \ZZ$.   Then, running this argument again with $Q = Q'$, we obtain a new solution $f_\beta'$ that is linearly independent of $f_\beta$.  By this method, there are infinitely many such solutions, with $\beta$ approaching $\alpha$.

        It remains to show that these $\beta$ are indeed good approximations.  We have the following:
	\begin{equation*}
		n \le Q^2, \quad |L_1L_2| \le 1/Q^2,
		\quad
		L_3^2 = | L_1\alpha_2 + L_2\alpha_1 |^2 \ll_\alpha \max\{ |L_1|^2, |L_2|^2 \} \le 1/Q^2.
	\end{equation*}
	Combining these with Lemma \ref{lemma:mainlemma},
	\[
		\cosh d_\Coefs(f_n,f_\beta) - 1 \ll_\alpha \frac{1}{n|\Delta_\beta|}.
	\]

	Note that the theorem statement is invariant under the action of $\PSL(2;\ZZ)$.
	Therefore, we may assume without loss of generality that $\Re(\alpha) > C \Im(\alpha)$ for any positive constant $C$, by translation by $\ZZ$.
	This implies the same fact about all $\beta$ sufficiently close to $\alpha$:  that for any fixed positive $C$, we can guarantee
	$n/2p > C |\Delta_\beta|^{1/2}/2p$, hence $n > C |\Delta_\beta|^{1/2}$.  Using this fact, we obtain
	\[
		\cosh d_\Coefs(f_n,f_\beta) - 1 \ll_\alpha \frac{1}{C|\Delta_\beta|^{3/2}}.
	\]
	This proves the theorem.
\end{proof}

\subsection{Poor approximation of algebraic numbers by quadratics}

In this section, we will give a complementary result to Theorems \ref{thm:dirichlet-geo} and \ref{thm:dirichlet-quad}, showing that algebraic numbers are not any better approximable than those theorems guarantee.

Most of the deep results in the Diophantine approximation of algebraic numbers, including Roth's Theorem \ref{thm:roth} and Theorem \ref{thm:bugeaud-evertse} of Bugeaud and Evertse, are consequences of the following far-reaching result.

\begin{theorem}[Schmidt's Subspace Theorem \cite{SchmidtSubspace, SchmidtBook}]
	\label{thm:schmidt}
	Let $n \ge 2$, and let $L_1, \ldots, L_n$ be linearly independent linear forms in $n$ variables, with real algebraic coefficients.  Let $\epsilon > 0$ be real.  Then the solutions $\mathbf{x} \in \ZZ^n$ to
	\[
		|L_1(\mathbf{x})L_2(\mathbf{x})\cdots L_n(\mathbf{x})| < \frac{1}{(\max\{1,|x_1|, \ldots, |x_n|\})^\epsilon}
		\]
		lie in finitely many proper subspaces of $\QQ^n$.
\end{theorem}

As an example, if one takes $\alpha \in \RR$ to be algebraic, Roth's Theorem \ref{thm:roth} can be reformulated as the statement that there are only finitely many solutions to
\begin{equation}
	\label{eqn:roth}
	|q||q\alpha - p| < \frac{1}{q^\epsilon}.
\end{equation}
To recover this assertion from Schmidt's Subspace Theorem, one can choose $L_1(p,q) = q$ and $L_2(p,q) = q\alpha - p$.

Theorem \ref{thm:bugeaud-evertse} of Bugeaud and Evertse uses Schmidt's Subspace Theorem on the space of coefficients.  To approximate an algebraic number $\alpha$, the coefficient vector $f$ is subject to a linear form given by $|f(\alpha)|$ (so the coefficients of the linear form are the powers of $\alpha$).  In our case, the application is different:  we work again on the space of coefficient vectors, but our linear form is given in terms of $1$, $\alpha + \overline{\alpha}$ and $\alpha\overline{\alpha}$.  That is, in terms of the coefficient vector associated to $\alpha$, instead of a vector of its powers.  We are, in effect, searching for approximations to the coefficient vector $f_\alpha$ of $\alpha$ within the coefficient space.  It is interesting to ask whether this method extends to higher degree.

We will now use Schmidt's Subspace Theorem \ref{thm:schmidt} to deduce the main result of this section.

\begin{theorem}
	\label{thm:quad-roth}
	Suppose that $\alpha \in \CC \backslash \RR$ is algebraic and non-quadratic.   Let $\eta > 0$. 
	If $\alpha$ lies on a rational geodesic, then there are only finitely many quadratic irrationals $\beta$ on that geodesic such that
	\begin{equation}
		\label{eqn:quad-roth-1}
		d_{hyp}(\alpha, \beta) \le \acosh\left( 1 + \frac{1}{|\Delta_\beta|^{2+\eta}} \right).
	\end{equation}
	Amongst quadratic irrationals $\beta$ not sharing a rational geodesic with $\alpha$, there are only finitely many such that
	\begin{equation}
		\label{eqn:quad-roth-2}
		d_{hyp}(\alpha, \beta) \le \acosh\left( 1 + \frac{1}{|\Delta_\beta|^{3/2 + \eta}} \right).
	\end{equation}
	In particular, if $\alpha$ is not on any rational geodesic, then there are only finitely many quadratic irrational $\beta$ satisfying \eqref{eqn:quad-roth-2} at all.
\end{theorem}

\begin{proof}
	Note that, by Theorem \ref{thm:delta-repulsion}, at most one $\beta$ having $\Delta_\beta = \Delta$ can satisfy \eqref{eqn:quad-roth-1} or \eqref{eqn:quad-roth-2}, for each $\Delta$.  Therefore, by throwing away at most finitely many approximations $\beta$, we can reduce to considering $\beta$ having discriminant $|\Delta_\beta|$ exceeding any fixed bound, or, consequently, to considering $\beta$ sufficiently close to $\alpha$.

	First, we will show that there are only finitely many $\beta$ which are not on the same rational geodesic as $\alpha$ (where $\alpha$ may or may not be on any rational geodesic), and which satisfy
	\[
		d_\Coefs(f_\alpha, f_\beta) \le \acosh\left( 1 + \frac{1}{|\Delta_\beta|^{3/2+\eta}} \right).
	\]
	For $\beta$ sufficiently close to $\alpha$, Lemma \ref{lemma:mainlemma} implies that any such solution $\beta$ satisfies
	\[
		|L_1L_2| \ll_\alpha 1/|\Delta_\beta|^{1/2+\eta}.
	\]
	For $\beta$ sufficiently close to $\alpha$, we also have
	\begin{equation}
		\label{eqn:im-re}
		|\Re(\beta)| \gg\ll_\alpha |\Im(\beta)|.
	\end{equation}
        In particular, 	
	\[
		n \gg\ll_\alpha |\Delta_\beta|^{1/2}.
	\]
	Therefore (altering $\eta$),	
	\[
		|n| |L_1(p_1,n,p_2)| |L_2(p_1,n,p_2)| \ll_\alpha \frac{1}{|n|^\eta}.
	\]
	These three linear forms are independent.  So by Schmidt's Subspace Theorem \ref{thm:schmidt}, these solutions lie on finitely many proper subspaces, i.e. rational geodesics.  Note that this finite collection of orbits depends only on $\alpha$.  
	Choose any one geodesic.  
	In particular, assume that $ap_1 + bn + cp_2 = 0$.  In particular,  $|a|$, $|b|$, $|c|$ are bounded above by a constant depending only on $\alpha$.

	Then, what we have is actually
	\[
		|n\alpha_1 - p_1||cn\alpha_2+ap_1+bn|
		=
		|n\alpha_1 - p_1||n\alpha_2-p_2||c|
		\ll_\alpha \frac{1}{|n|^{1+\eta}}
		< \frac{1}{|n|^{\eta}}.
	\]
	Since $\alpha$ is not on the geodesic, $a\alpha_1 + b + c\alpha_2 \neq 0$, which implies these are independent linear forms in two variables $n$ and $p_1$.
	Again by Schmidt's Subspace Theorem \ref{thm:schmidt}, the solutions lie on finitely many lines in (non-projectivized) coefficient space.  Hence there are finitely many solutions.

	Next, we show that there are only finitely many solutions $\beta$ which are on the same rational geodesic as $\alpha$ (thus we are in the case that $\alpha$ is on a rational geodesic), and which satisfy
	\[
		d_{hyp}(\alpha, \beta) \le \acosh\left( 1 + \frac{1}{|\Delta_\beta|^{2+\eta}} \right).
	\]
	For such $\beta$ sufficiently close to $\alpha$, by Lemma \ref{lemma:mainlemma},
	\[
		|L_1L_2||\Delta_\beta| \ll_\alpha 1/|\Delta_\beta|^{\eta}.
	\]
	By the same argument surrounding \eqref{eqn:im-re} (altering $\eta$),
	\[
		|n|^2 |n\alpha_1 - p_1| |n\alpha_2 - p_2| \ll_\alpha \frac{1}{|n|^\eta}.
	\]
	Assume that this geodesic is characterised by $ap_1 + b + cp_2 = 0$, and therefore $a\alpha_1 + b + c\alpha_2 = 0$.  Therefore $|n\alpha_2 - p_2| = |a/c||n\alpha_1 - p_1|$ and so this becomes (altering $\eta$):
	\[
		|n| |n\alpha_1 - p_1| \ll_\alpha \frac{1}{|n|^\eta}.
	\]
	These are independent linear forms in two variables, so by Schmidt's Subspace Theorem \ref{thm:schmidt} on the two-dimensional space in $n$ and $p_1$, we have solutions on only finitely many lines in coefficient space.   This means there are only finitely many $\beta$.
\end{proof}

\subsubsection{Complex Liouville numbers}

We demonstrate that there are non-quadratic complex numbers which are extremely well-approximable by quadratic irrationals (in particular, so as to be necessarily non-algebraic, as a result of Theorem \ref{thm:quad-roth}).

In analogy with the classical case, we will call $\alpha \in \CC$ a \emph{complex quadratic Liouville number} if, for every positive integer $m$, there are infinitely many quadratic irrational $\beta$ such that
	\begin{equation}
		\label{eqn:liouville-goal}
		d_{hyp}(\alpha, \beta) \le \acosh\left( 1 + \frac{1}{|\Delta_\beta|^{m}} \right).
	\end{equation}

To accomplish this, we will construct a Cauchy sequence of quadratic irrationals $\beta_k$.  Calling the limit $\alpha$, we will show that the hyperbolic distance between $\alpha$ and $\beta_k$ satisfies \eqref{eqn:liouville-goal}.  The construction is simple:  we only require at each stage that
\begin{equation}
	\label{eqn:liouville-require}
	d_{hyp}(\beta_k,\beta_{k+1}) < \acosh \left( 1 + \frac{1}{|\Delta_{\beta_k}|^k} \right), \quad |\Delta_{\beta_k}| \ge 2.
\end{equation}
As the quadratic irrationals of absolute discriminant $\ge 2$ are dense, this is possible.  To see that the sequence is Cauchy and the limit satisfies \eqref{eqn:liouville-goal}, we can compute, for $M > m$,
\begin{align*}
	d_{hyp}(\beta_M,\beta_m) &\le \sum_{k=m}^{M-1} d_{hyp}(\beta_k,\beta_{k+1}) \\
	&\le d_{hyp}(\beta_m,\beta_{m+1}) \sum_{k=m}^{M-1} \frac{d_{hyp}(\beta_k,\beta_{k+1})}{d_{hyp}(\beta_m,\beta_{m+1})} \\
	&\ll d_{hyp}(\beta_m,\beta_{m+1}) \\
	&\le \acosh \left( 1 + \frac{1}{|\Delta_{\beta_m}|^m} \right).
\end{align*}
Note that the step in which the sum symbol disappears follows from \eqref{eqn:liouville-require} and the series expansion as $x \rightarrow \infty$,
\begin{equation}
	\label{eqn:arcosh}
	\acosh(1 + 1/x^2) = \frac{ \sqrt{2} }{x} - \frac{1}{6\sqrt{2}x^3} + \frac{3}{80\sqrt{2}x^5} + \cdots.
\end{equation}
This implies \eqref{eqn:liouville-goal}.

As this construction has a great deal of freedom, we can construct quadratic Liouville numbers on any fixed rational geodesic, for example.  The countability of the rational geodesics also implies we can construct quadratic Liouville numbers not lying on any rational geodesic.\footnote{To see this, count the geodesics and assign a tubular neighbourhood of width $1/n$ to the $n$-th such geodesic; require all terms $\beta_k$ past the $n$-th to avoid the first $n$ tubular neighbourhoods; each term has finitely many restrictions placed upon it.  (To make this work, the geodesics must be ordered as the terms are created, i.e.\ the $n$-th geodesic is always chosen so that the closure of its neighbourhood does not include $\beta_n$.)}

\subsection{Comparison with classical results}
\label{sec:connections}

We wish to compare Theorems \ref{thm:dirichlet-quad} and \ref{thm:quad-roth} with Theorem \ref{thm:bugeaud-evertse}.  In other words, to compare the classical approach to Diophantine approximation in the complex plane with the approach we consider here, using hyperbolic metric and discriminant sizing.

We begin with the following Lemma \ref{lemma:delta-ge-height-quad}, relating the na\"ive height and the discriminant as measures of arithmetic complexity.  In the case of quadratics, and taking into account $\PSL(2;\ZZ)$ invariance, \eqref{eqn:delta-le-naive-gen} becomes
\begin{equation}
  \label{eqn:delta-le-naive-quad}
	|\Delta_\alpha| \le 12 H_{\PSL}(f_\alpha)^2.
\end{equation}
As discussed above, we don't expect an inequality in the other direction in general.  However, if we take into account the $\PSL(2;\ZZ)$ action, and loosen the exponents, then we can prove such a thing in the quadratic case.

\begin{lemma}
	\label{lemma:delta-ge-height-quad}
  Suppose $\beta$ is of degree $2$.  Then
  \begin{equation}
    \label{eqn:delta-ge-height-quad}
	  |\Delta_\beta| \ge 3 H_{\PSL}(f_\beta),
  \end{equation}
	and also,
  \begin{equation}
    \label{eqn:delta-ge-height-quad-2}
	  |\Delta_\beta| \ge \frac{3 H_{\PSL}(f_\beta)^2}{N_{\PSL}(\beta)},
  \end{equation}
	where $N_{\PSL}$ denotes the norm of the element of its $\PSL(2;\ZZ)$ orbit which lies in the standard fundamental region.
\end{lemma}

\begin{proof}
The quantities involved are all invariant under the action of $\PSL(2;\ZZ)$, so it suffices to assume $\beta$ lies in the usual $\PSL(2;\ZZ)$ fundamental region.
For $\beta$ an imaginary quadratic irrationality satisfying $ax^2 + bx + c=0$, lying in the usual fundamental region,
	\[
  H(f_\beta) = \max\{|a|,|b|,|c|\}, \quad
  	N(\beta) = |c/a| \ge 1, \quad
	|\Re(\beta)| = |b/2a| \le 1/2.
	\]
	From this we conclude that $|b| \le |a| \le |c|$ hence $H(f_\beta) = |c|$ and $b^2 < |ac|$.
Then we have
\[
  |\Delta_\beta| = 4|ac| - b^2 \ge 3|ac| \ge 3|c| = 3H(f_\beta).
\]
	Alternately, we use $N(\beta) = |c/a|$ for the second inequality.
\end{proof}

Next, we recall that the hyperbolic metric and euclidean metric are locally conformally equivalent.  More precisely, for $\beta$ sufficiently close to $\alpha$ in either metric,
\begin{equation}
	\label{eqn:hyp-euc}
	(1-\epsilon)\Im(\alpha)d_{hyp}(\alpha,\beta)
	<
  |\alpha - \beta|
  <
	(1+\epsilon)\Im(\alpha)d_{hyp}(\alpha,\beta).
\end{equation}
Finally, we need the series expansion as $x \rightarrow \infty$ for the inverse hyperbolic cosine given in \eqref{eqn:arcosh}.

Collecting the above relationships, it is a simple computation to derive the following:

\begin{lemma}
	\begin{enumerate}
		\item Suppose 
	\[
		d_{hyp}(\alpha,\beta) < 
		\acosh\left( 1 + \frac{C}{|\Delta_\beta|^k} \right).
	\]
	Then 
	\[
		|\alpha - \beta| < \frac{(1+\epsilon)\Im(\alpha)N_{\PSL}(\beta)^{k/2}\sqrt{2C}}{3^{k/2}H_{\PSL}(f_\beta)^k}.
	\]
		\item Suppose 
	\[
		d_{hyp}(\alpha,\beta) >
		\acosh\left( 1 + \frac{C}{|\Delta_\beta|^k} \right).
	\]
	Then for $|\Delta_\beta|$ sufficiently large,
	\[
		|\alpha - \beta| > \frac{(1-\epsilon)\Im(\alpha)\sqrt{2C}}{12^{k/2} H_{\PSL}(f_\beta)^k}.
	\]
	\end{enumerate}
\end{lemma}

In particular, Theorems \ref{thm:delta-repulsion}, \ref{thm:dirichlet-geo}, \ref{thm:dirichlet-quad}, and \ref{thm:quad-roth} imply statements in the euclidean metric with the choice of the na\"ive height, which we collect here for completeness.

\begin{theorem}
	\label{thm:translate}
	\begin{enumerate}
		\item 
  Let $\alpha \neq \beta$ be two non-real quadratic irrationalities.  Then for and positive $\epsilon$, and for $|\Delta_\beta|$ sufficiently large,
  \[
			|\alpha - \beta| \ge 
			\frac{(1-\epsilon)\Im(\alpha)\sqrt{2}}{2\sqrt{3} H_{\PSL}(f_\beta)^{1/2+\epsilon}H_{\PSL}(f_\alpha)^{1/2+\epsilon}}.
  \]
	If $\Delta_\alpha = \Delta_\beta$ is sufficiently large in absolute value, then
	\[
				|\alpha - \beta| \ge 
			\frac{(1-\epsilon)\Im(\alpha)\sqrt{2}}{2\sqrt{3} \min\{H_{\PSL}(f_\beta),H_{\PSL}(f_\alpha)\}}.
	\]
		\item
  Let $\alpha \in \CC \backslash \RR$ not be quadratic irrational, but lying on a rational geodesic.
	Then there exists a constant $K_\alpha > 0$, depending only on the $\PSL(2;\ZZ)$ orbit of $\alpha$, such that there are infinitely many quadratic irrational $\beta$ lying on that rational geodesic, with
  \[
		|\alpha - \beta| < \frac{K_\alpha}{H_{\PSL}(f_\beta)^2}.
  \]
		\item
  Let $\alpha \in \CC \backslash \RR$ not be a quadratic irrational.  Let $K > 0$ be any constant.  Then, there are infinitely many quadratic irrationalities $\beta$ with
  \[
	  |\alpha - \beta| < \frac{K}{H_{\PSL}(f_\beta)^{3/2}}.
  \]
		\item
	Suppose that $\alpha \in \CC \backslash \RR$ is algebraic and non-quadratic.   Let $\eta > 0$. 
	If $\alpha$ lies on a rational geodesic, then there are only finitely many quadratic irrationals $\beta$ on that geodesic such that
	\begin{equation}
		|\alpha - \beta| < \frac{1}{H_{\PSL}(f_\beta)^{2+\eta}}.
	\end{equation}
	Amongst quadratic irrationals $\beta$ not sharing a rational geodesic with $\alpha$, there are only finitely many such that
	\begin{equation}
		|\alpha - \beta| < \frac{1}{H_{\PSL}(f_\beta)^{3/2+\eta}}.
	\end{equation}
			In particular, if $\alpha$ is not on any rational geodesic, then there are only finitely many quadratic irrational $\beta$ satisfying \eqref{eqn:quad-roth-2} at all.  Note that in \eqref{eqn:quad-roth-1} and \eqref{eqn:quad-roth-2}, $H_{\PSL}$ can be weakened to $H$.
	\end{enumerate}
\end{theorem}

	In particular, up to constants depending on $\alpha$, our theorems recover the computation of $k_2(\alpha)$ given by Theorem \ref{thm:bugeaud-evertse} of Bugeaud and Evertse for the quadratic case.\footnote{To see this for items (2) and (3) in Theorem \ref{thm:translate} requires some consideration of the relationship between $H_{\PSL}(f_\alpha)$ and $H(f_\alpha)$; we need to know the quotient is bounded in terms of the height of the relevant element of $\PSL(2;\ZZ)$, which is a constant in terms of $\alpha$.}  Theorem \ref{thm:translate} offers some refinement in terms of drawing a distinction between approximations on a geodesic containing $\alpha$ or not.

	However, the constants are different.  As the imaginary part of $\alpha$ approaches $\infty$, the constant $K_\alpha$ of item (2) of Theorem \ref{thm:translate} weakens.
	By contrast item (1) of Theorem \ref{thm:translate} becomes a stronger statement as $\alpha \rightarrow \infty$.  It is natural to ask whether there is an analogue to the Lagrange spectrum for approximation by quadratic irrationals, at least for those lying on a geodesic.  The preceding discussion demonstrates that knowledge of the spectrum in one of the classical or hyperbolic/discriminant settings would not imply the other, although there would be some relationships.

\section{To boldly go where no one has gone before}
\label{sec:futurework}

The investigation of algebraic starscapes raises a wide variety of possible future research directions, many of which we intend to continue to investigate.  We invite you to join us. 

\subsection{The homogeneous geometry of higher degrees}
For higher degree polynomials, the beginnings of the geometric story remain relatively unchanged, but strong conclusions such as Theorems \ref{thm:QuadRoots} and \ref{thm:CubicFactor} become more complicated to draw, as the dimension of the space of polynomials grows.

In particular, fixing a degree $n\geq 1$, we have that $\PP\Coefs_n\cong\CP^n$ identifies with complex projective space, and $\Roots_n\cong\SP^n(\CP^1)$ is the set of unordered $n$-tuples in the Riemann sphere.
The roots map $\RootMap_n\colon\CP^n\to\SP^n(\CP^1)$ is a homeomorphism, and
is equivariant with respect to the natural $\PSL(2;\CC)$ actions on each side: given on the space of roots by precomposition with a M\"obius transformation, and on the space of coefficients by the action of the unique irreducible representation $\PSL(2;\CC)\to\PSL(n;\CC)$.

Restricting to real coefficients, the fact that $\RootMap$ is a homeomorphism implies $\RootMap_n\colon\RP^n\to\SP^n(\CP^1)$ is an embedding, equivariant with respect to the restricted $\PSL(2;\RR)$ action on each side.
The orbits of this action decompose the space of degree $n$ polynomials, whose nature depends on the degree.
For $n\leq 3$, each component of the complement of the discriminant locus comprises an entire orbit itself, and thus comes equipped with the structure of a homogeneous geometry for $\PSL(2;\RR)$.
For $n>3$, the $\PSL(2;\RR)$ orbits foliate each component.
Consider as an example the space of quartics. The complement of the discriminant locus has two components; those with two pairs of complex conjugate roots, and those with a pair of complex conjugate roots and two real roots.
The $\PSL(2;\RR)$ action decomposes the latter into a family of codimension-1 hypersurfaces, which are generically\footnote{Indeed, by the action of $\PSL(2;\RR)$, the pair of complex roots can be moved to any point in $\HH^2$.  Fixing this point, the remaining degree of freedom of the $\PSL(2;\RR)$ action acts by rotation on the ideal boundary $\partial_\infty\HH^2=\RP^1$, and each $\PSL(2;\RR)$ orbit is determined by the \emph{angle} between these two points, measured between the tangent vectors at the complex root pointing to the real roots, as in Equation \ref{eqn:Triv}.} diffeomorphic to $\PSL(2;\RR)$ itself.
This complicates the overall picture, and suggests generalisations of Theorems \ref{thm:QuadRoots} and \ref{thm:CubicFactor} will involve fiber bundles of homogeneous spaces, rather than just the spaces themselves.

Building a more robust geometric toolkit would not only allow the extension of these ideas to higher degree polynomials, but strong enough tools may provide a window into exploring their solvability.
In particular, while the next case of interest, quartics, admits a solution by radicals, the quintics and beyond do not.
It is an exciting prospect to try and understand this dichotomy geometrically along our journey.

\subsection{Starscape curves}  Revisit Figure \ref{fig:Lines}.  The images of rational planes (projective rational lines) in the coefficient space form delicate beaded necklaces (linear starscapes).  Any two algebraic numbers lie on at least one common curve (more if the point has the non-transversal property discussed in Corollary \ref{cor:transverse}).  It is possible to give algebraic equations for these curves in general dimension, in terms of the two associated minimal polynomials $f_1$ and $f_2$, namely, viewing the complex plane as $\RR^2$, the curve is $\Re(f_1(x+iy))\Im(f_2(x+iy)) = \Re(f_2(x+iy))\Im(f_1(x+iy))$.  We will call these \emph{starscape curves}.  
The intricacies of some of these curves suggest that the projective geometry of coefficient space is quite disguised by projecting onto complex roots.
Determining whether these starscape curves are related to any natural geometric structures on the complex plane may allow us to develop further extensions of the material in Section \ref{sec:Geometry} to higher degree.

Furthermore, starscape curves appear to have a repulsion effect all their own (see the whitespace surrounding geodesics in Figure \ref{fig:Initial_quadratics}, for example).  Is it possible to quantify how well approximable some complex number $\alpha$ is by these curves?  We might measure the distance between $\alpha$ and a curve.  What is the \emph{height} of a starscape curve?

\subsection{Planar starscapes}

The planar starscapes we have studied provide a two dimensional analogue to the curves described above, with the complex quadratics (Figure \ref{fig:Initial_quadratics}) providing a primary example. As the degree of polynomials increases these families can get more complicated, but the projection to the upper half-plane by complex roots is always available (although each polynomial might be represented by increasing numbers of individual roots). Even in the cubic case where there is at most one complex pair we start to see new behaviour with the ``blackhole''-like phenomena seen on the right of Figure \ref{fig:Fam7C}, where a quadratic point is the complex root of many polynomials. The starscapes shown in Figures \ref{fig:Starscapes_d}, \ref{fig:Starscapes_e}, and \ref{fig:Starscapes_f} show further intriguing behaviours that might be studied, such as seemingly denser regions and the isolated quadratic and cubic points at the top middle of Figure \ref{fig:Starscapes_f}. 

\subsection{Algorithms to draw starscapes and Farey structure}

Some of the images here have taken minutes or even hours to compute.  Can we develop faster, more intelligent algorithms? The current images are made using a brute force approach:  generate a large number of polynomials, solve them, and then plot the results which fit into a desired region.  The code is short and many of the hard subproblems, like solving polynomials, are already implemented in computer algebra environments.  However, it is very inefficient:  we don't effectively sieve for points which will end up in the desired region, and work is repeated for every point.

The geometry discussed above provides a path to improvements.  Can we efficiently predict which polynomials need to be solved for a given region?  Along each curve we see a recursive pattern resembling the rational numbers in Figure \ref{fig:RationalStarscape}. An efficient way to generate rational numbers is to use the Farey (or Stern-Brocot) tree stucture, discussed in Section \ref{sec:Stern-Brocot} and shown in Figure \ref{fig:TreeMediant}:  in essence, use the mediant operation to fill in gaps, recursively.  A variation on this can be achieved by the na\"ive addition of polynomials as coefficient vectors. This produces a new polynomial with a root ``between'' the original two along the curve.  In effect, we are asking about higher dimensional Stern-Brocot trees \cite{Lennerstad}. It should also be possible to make use of the $\PSL(2;\ZZ)$ symmetry, at least in large-scale pictures. 

\subsection{Continued fractions}
Do there exist continued fraction algorithms for approximation by algebraic numbers of fixed or bounded degree?  The Farey structure of the rationals is, in some sense, the source of the continued fraction algorithm for real numbers approximated by rationals (the continued fraction algorithm can be viewed both as a traversal of a Cayley graph for $\PSL(2;\ZZ)$ and as a geometric process of repeated mediants).  Figure \ref{fig:Initial_quadratics} wonders aloud, might it be possible to extend the continued fraction method to approximate complex numbers by quadratic irrationals?  It is possible that existing multidimensional continued fraction algorithms, which have a long history, may provide hints; a recent article from which to enter the literature is \cite{Murru}.

\subsection{Higher degree and geometry-sensitive Diophantine approximation} 
We have seen that the geometry of the roots map naturally classifies certain types of approximations.  Our example is quadratic approximations from one rational geodesic, where we saw that approximations from a geodesic containing $\alpha$ can be better than approximations not on the geodesic.  Moving to the cubic case, one might ask how well a quadratic is approximated by cubics from a particular planar starscape.  If the quadratic is a singular point in the sense of Corollary \ref{cor:transverse} and Figure \ref{fig:CoeffCubics}, then we conjecture the cubics from that planar starscape are better approximations of the quadratic than those from starscapes not having the property.

More generally, to what extent can the results of Section \ref{sec:Dio-quad} be extended to higher degree?  Could such an extension settle the oustanding cases in the work of Bugeaud and Evertse?  In higher degree, are the exceptionally well-approximable algebraic numbers all living on starscape curves, or are there other reasons to be well-approximable?  The work of Bugeaud and Evertse also indicates a difference in approximability based on how many real conjugates an algebraic number has.  What pictures would one draw to see this effect?

\begin{figure}[h!tbp]
\centering
\begin{subfigure}[b]{0.42\textwidth} 
    \centering
	\includegraphics[width=\textwidth]{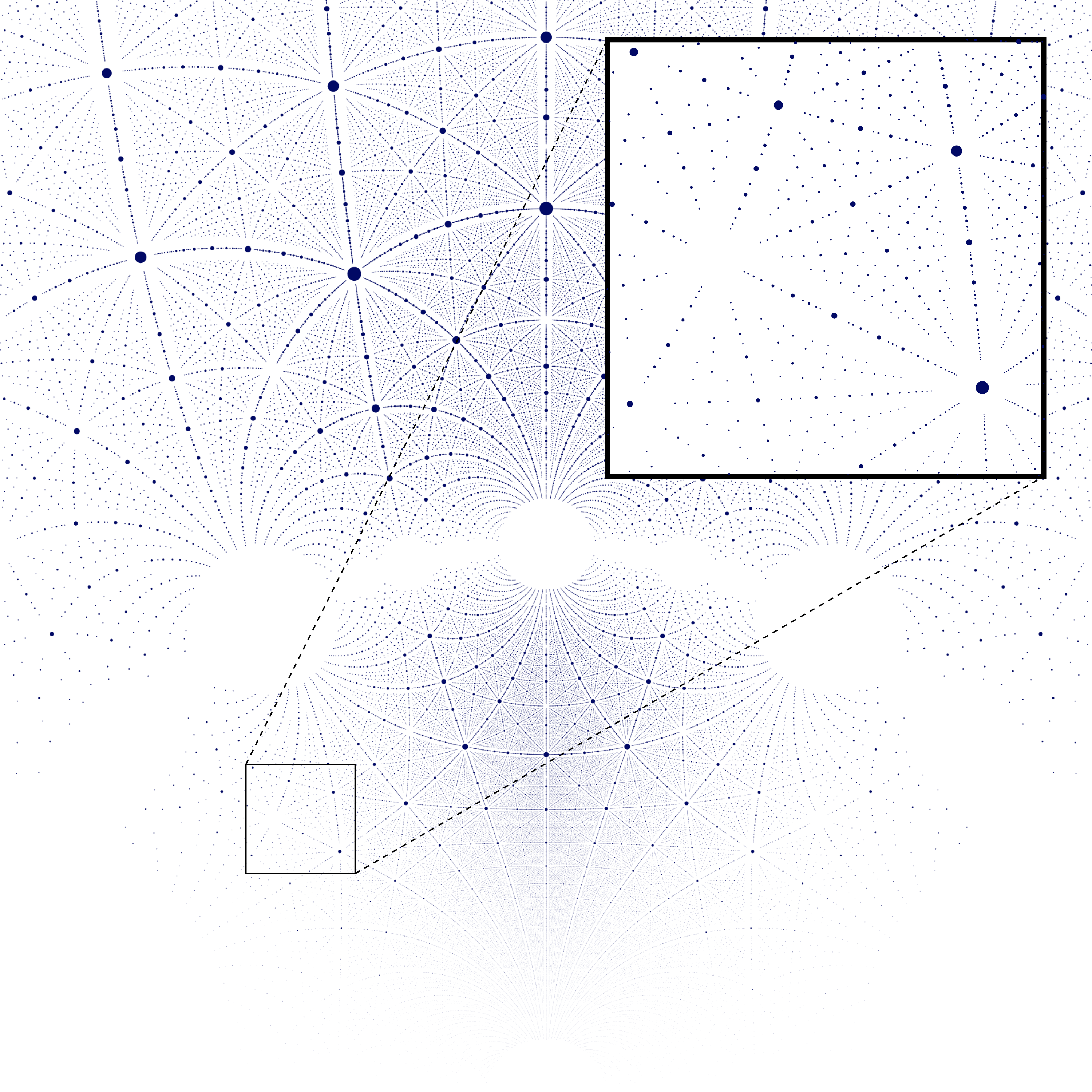}
	\caption{Quartics}
    \label{fig:reciprocal4}
\end{subfigure}
\begin{subfigure}[b]{0.42\textwidth} 
    \centering
	\includegraphics[width=\textwidth]{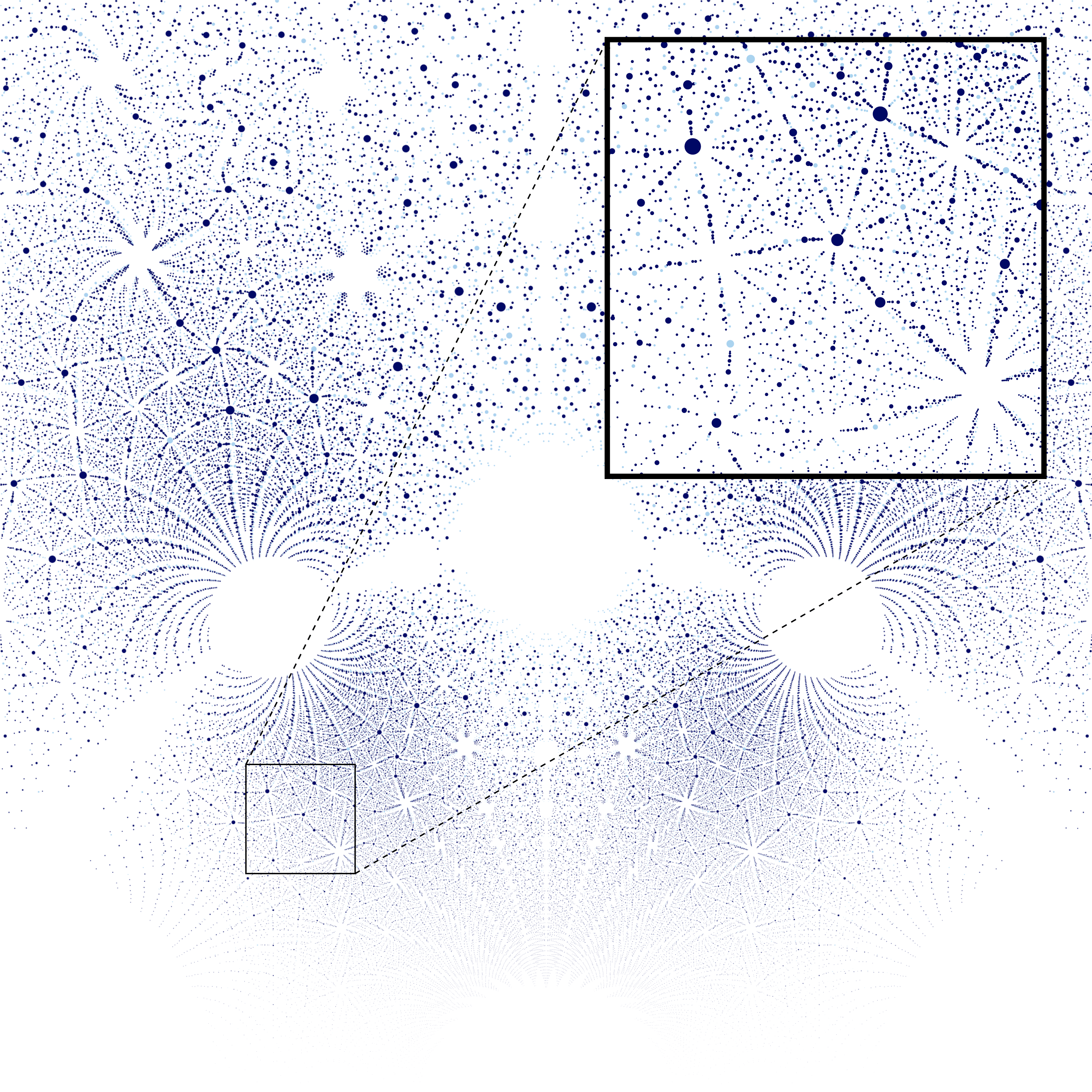}
	\caption{Sextics}
    \label{fig:reciprocal6}
\end{subfigure}
\begin{subfigure}[b]{0.42\textwidth} 
    \centering
	\includegraphics[width=\textwidth]{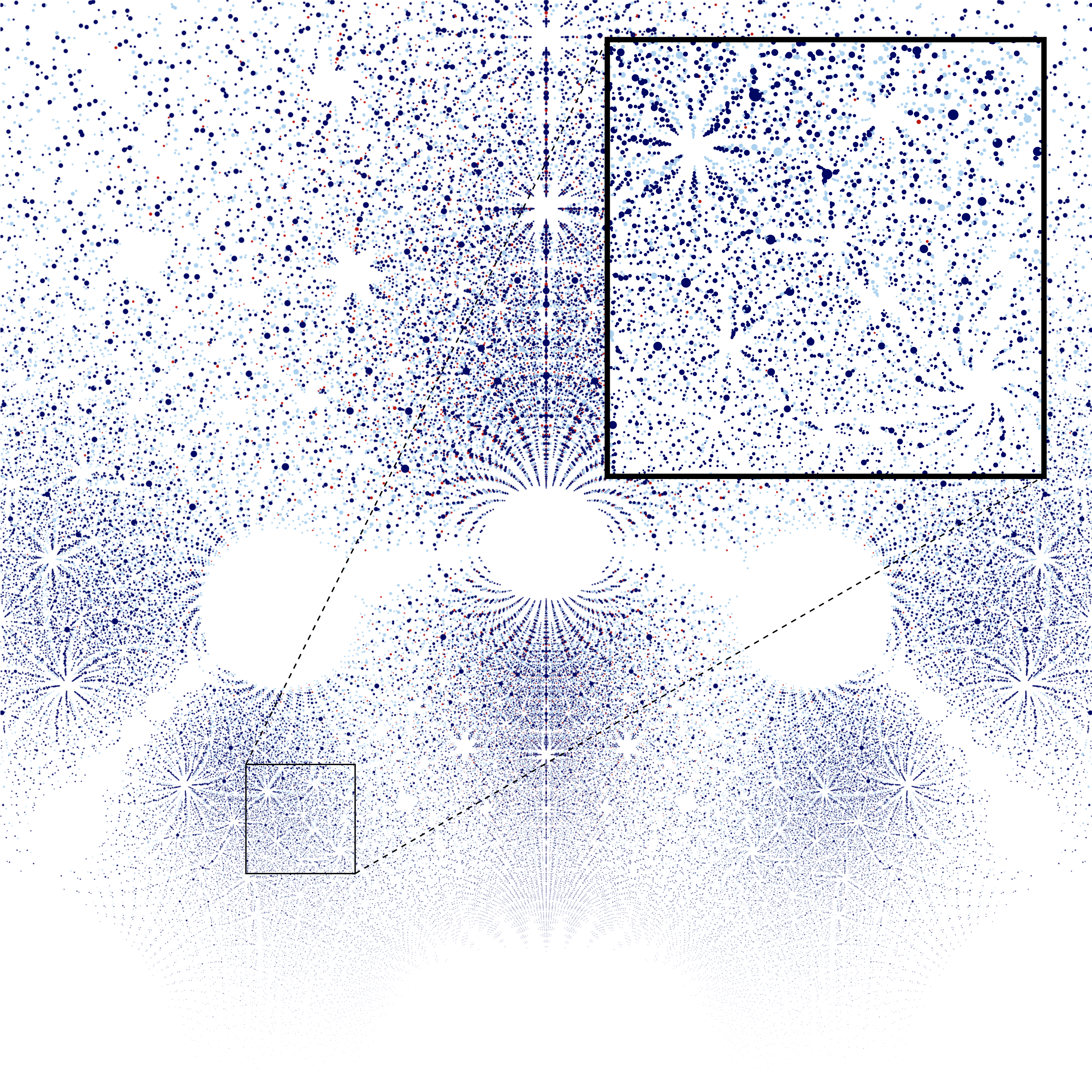}
	\caption{Octics}
    \label{fig:reciprocal8}
\end{subfigure}
\begin{subfigure}[b]{0.42\textwidth} 
    \centering
	\includegraphics[width=\textwidth]{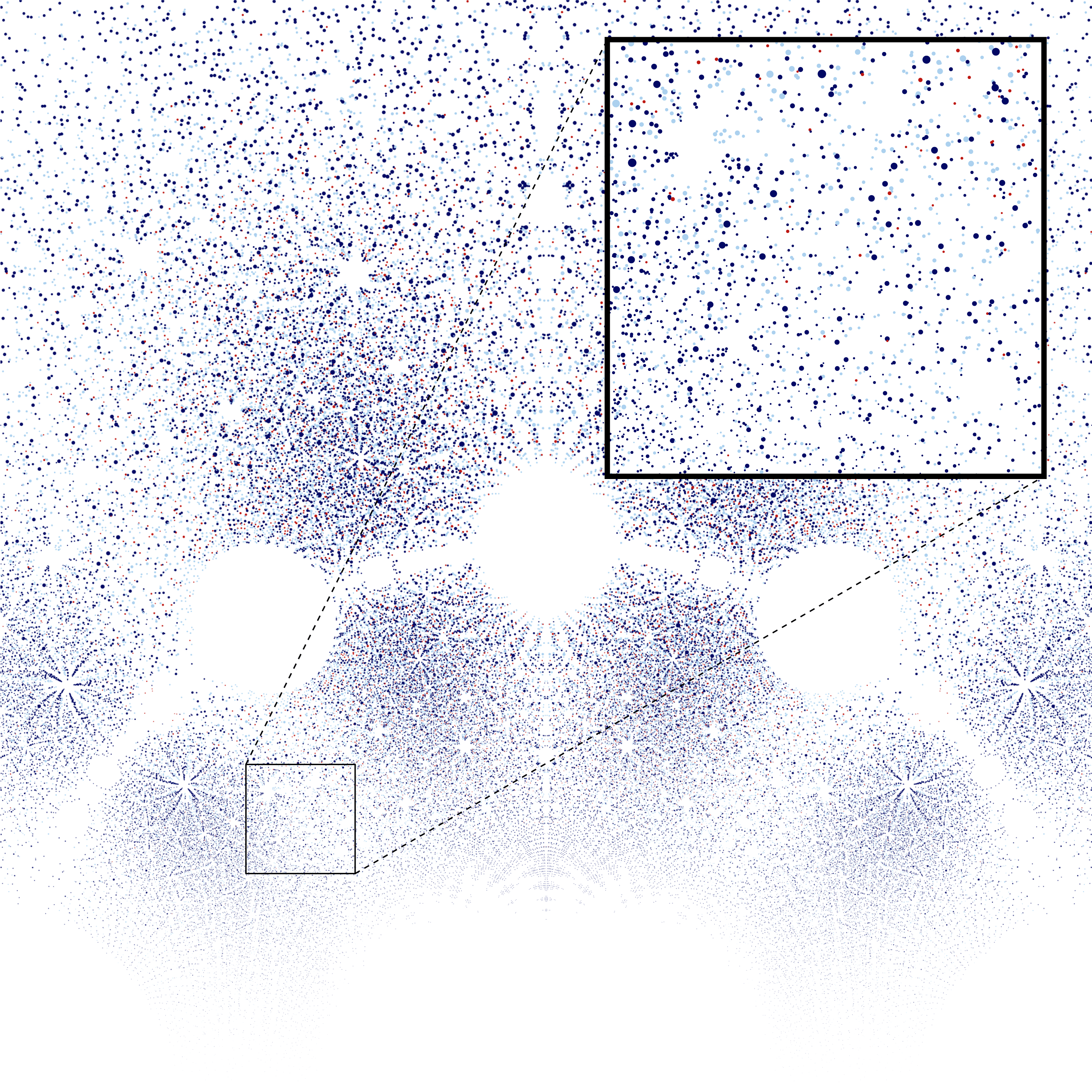}
	\caption{Decics}
    \label{fig:reciprocal10}
\end{subfigure}
	\caption{Pictures of roots of reciprocal polynomials off the unit circle (the many roots on the unit circle are not plotted). For odd degrees the reciprocal polynomials always have the root $-1$. The points are coloured by the number of conjugate real roots with dark blue for 0, light blue for 2, dark red for 4 and light red for 6 (dark and light blue are swapped from Figure \ref{fig:quad-quar} to make the image clearer). Note the strong repulsion of the unit circle in all cases, especially close to $i$ and third/sixth roots of unity. There are also denser regions in the direction of the $d$-th roots of unity for degree $d$.}
\label{fig:reciprocal}
\end{figure}

\subsection{Mahler measure and Lehmer's Conjecture}
Lehmer's famous conjecture is related to another important measure of arithmetic complexity, namely the \emph{Mahler measure} (which is not technically any type of measure).  A polynomial $f = a_d x^d + \cdots a_1 x + a_0 = a_d \prod_{i=1}^d (x - \alpha_i) \in \ZZ[x]$ has Mahler measure
\[
  M(f) = |a_d| \prod_{i=0}^d \max\{1, |\alpha_i|\}.
\]
The Weil height of an algebraic number is exactly related to the Mahler measure of its minimal polynomial: $M(f_\alpha) = H(\alpha)^d$.

Lehmer's conjecture states that there is a lower bound to $M(f)$ away from polynomials whose roots are roots of unity.  The smallest known Mahler measure of a non-root-of-unity is called \emph{Salem's number}, $1.176280818\ldots$ associated to the roots of the \emph{Lehmer polynomial} $x^{10} + x^9 - x^7 - x^6 -x ^5 -x^4 -x^3 +x+1$.\footnote{The complex roots of Lehmer's polynomial all lie on the unit circle, and its Mahler measure is a function of its real roots.}
Lehmer's Conjecture is known to hold for non-reciprocal polynomials (those whose coefficients are not palindromic) \cite{Smyth}, and Voutier \cite{Voutier} gives an explicit lower bound in terms of the degree, implying the conjecture is true if fields are restricted by degree.
Therefore to properly visualize Lehmer's conjecture would require a starscape in increasing degree, and some understanding of the real roots accompanying the complex roots.  

For these various reasons, there is a sense in which Lehmer's conjecture does not properly belong to the complex plane and the algebraic starscapes, although it is natural to wonder if it gives rise to any interesting images.  One related image which is particularly stunning is the reciprocal polynomials, shown in Figure \ref{fig:reciprocal}.  The image itself asks a variety of interesting questions.

\bibliographystyle{plain}
\bibliography{starbib}

\end{document}